\documentclass[12pt,a4paper,oneside]{amsart}
\usepackage[utf8]{inputenc}
\usepackage{amsmath,amscd, amsfonts, amssymb, amsthm}
\usepackage{enumerate,enumitem}
\usepackage{setspace,kantlipsum}
\usepackage[dvipsnames]{xcolor}
\usepackage{tikz}
\usepackage{tikz-cd}
\usepackage{mathrsfs}
\usepackage{textcomp}
\usepackage[a4paper, margin=1in]{geometry}
\usepackage{colonequals}
\usepackage{mathtools}
\usepackage{sansmath}
\usepackage[colorlinks=true, linkcolor=blue, citecolor=blue]{hyperref}
\usepackage{listings}
\usepackage{courier}
\usepackage{anyfontsize}
\theoremstyle{plain} 
\newtheorem{thm}{Theorem}[section]
\newtheorem{lemma}[thm]{Lemma}
\newtheorem{lem}[thm]{Lemma}
\newtheorem{proposition}[thm]{Proposition}
\newtheorem{prop}[thm]{Proposition}
\newtheorem{corollary}[thm]{Corollary}
\newtheorem{cor}[thm]{Corollary}

\newtheorem{conj}[thm]{Conjecture}

\theoremstyle{definition}
\newtheorem{notation}[thm]{Notation}

\newtheorem{definition}[thm]{Definition}
\newtheorem{defn}[thm]{Definition}

\newtheorem{example}[thm]{Example}
\newtheorem{remark}[thm]{Remark}
\newtheorem{rmk}[thm]{Remark}

\newtheorem{assumption}[thm]{Assumption}

\numberwithin{equation}{section}

\newcommand{\C}{\mathbb{C}}
\newcommand{\Z}{\mathbb{Z}}
\newcommand{\Q}{\mathbb{Q}}
\newcommand{\R}{\mathbb{R}}
\newcommand{\N}{\mathbb{N}}

\newcommand{\cA}{\mathcal{A}}
\newcommand{\cB}{\mathcal{B}}

\newcommand{\cG}{\mathcal{G}}
\newcommand{\cH}{\mathcal{H}}

\newcommand{\cJ}{\mathcal{J}}

\newcommand{\cM}{\mathcal{M}}
\newcommand{\cN}{\mathcal{N}}
\newcommand{\cO}{\mathcal{O}}

\newcommand{\fm}{\mathfrak{m}}

\newcommand{\sD}{\mathscr{D}}

\newcommand{\mc}[1]{{\mathcal{#1}}}
\newcommand{\mb}[1]{{\mathbb{#1}}}
\newcommand{\ms}[1]{{\mathscr{#1}}}
\newcommand{\mf}[1]{{\mathfrak{#1}}}
\newcommand{\mrm}[1]{{\mathrm{#1}}}

\newcommand{\bs}[1]{\boldsymbol{#1}}

\renewcommand{\d}{\partial}
\newcommand{\Gr}{\mathrm{Gr}}
\newcommand{\gr}{\mathrm{gr}}

\newcommand{\MHM}{\mathrm{MHM}}
\let\mhm\MHM
\newcommand{\coh}{\mathrm{Coh}}

\DeclareMathOperator{\spec}{Spec}
\DeclareMathOperator{\codim}{codim}

\DeclareMathOperator{\pro}{Pro}
\DeclareMathOperator*{\colim}{colim}

\newcommand{\Mod}{\mathrm{Mod}}

\DeclareMathOperator{\Ker}{Ker}
\DeclareMathOperator{\coker}{coker}

\DeclareMathOperator{\supp}{supp}
\newcommand{\shom}{\mathcal{H}\hspace{-1pt}\mathit{om}}

\newcommand{\id}{\mathrm{id}}

\newcommand{\ls}[2]{\vphantom{#2}^{#1}#2}

\DeclareMathOperator{\Res}{Res}
\renewcommand{\Re}{\mathop{\mrm{Re}}}

\begin{document}

\title[Multivariate $V$-filtrations and the monodromy conjecture]{Multivariate $V$-filtrations and the Strong Monodromy Conjecture for hyperplane arrangements}

\author{Dougal Davis}
\author{Ruijie Yang}

\begin{abstract}
In this work, we develop a new theory of multivariate $V$-filtration on $\sD$-modules along a simple normal crossing divisor and relate it with Sabbah's multi-filtration. We establish several new structural results and relate them with the Hodge filtration on free-monodromic local systems from geometric representation theory. As an illustrative application, we prove the Strong Monodromy Conjecture and its multivariate generalisation for hyperplane arrangements. Along the way, we confirm both the $n/d$-conjecture of Budur--Musta\c{t}\u{a}--Teitler and its multivariate form due to Budur. 
\end{abstract}
\maketitle
\tableofcontents

\section{Introduction}
In this paper, we develop a new approach to the multiply-indexed $V$-filtration of a holonomic $\ms{D}$-module relative to a divisor with simple normal crossings. Our approach is inspired by (and closely related to) the wall-crossing theory for Hodge filtrations developed by Vilonen and the first-named author \cite{DV23} in the study of real group representations. We show that our notion is equivalent to the existing theory of Sabbah \cite{Sabbah1987} and prove several new structural results that were previously inaccessible. 

As an application of the general theory, we give a proof of the Strong Monodromy Conjecture (and its multivariable generalisation) in the case of hyperplane arrangements. Unlike most previous work on this conjecture, which has attracted significant attention in singularity theory, our proof is conceptual and relies very little on explicit computation or combinatorics. As such, we hope that our methods might provide the starting point for progress on the conjecture in the general case.

\subsection{The Strong Monodromy Conjecture for hyperplane arrangements}
To provide some motivation for the more abstract parts of our work, let us begin at the end with the application. The Strong Monodromy Conjecture of Igusa \cite{Igusa} and Denef--Loeser \cite{DL92} is one of the deepest and most mysterious conjectures in singularity theory. The conjecture relates the poles of $p$-adic zeta functions of a hypersurface singularity to roots of the Bernstein-Sato polynomial and monodromy eigenvalues of the Milnor fibre, predicting a striking relationship between the arithmetic and geometry of such singularities. More precisely, if $f \in \mb{Q}[x_1, \ldots, x_n]$ is a polynomial and $p$ is a prime, then the \emph{$p$-adic zeta function} of $f$ is defined by
\begin{equation} \label{eq:p-adic zeta}
Z_{f, p}(s) \colonequals \int_{\mb{Z}_p^n}|f|_p^s d\mu,
\end{equation}
where $|\cdot|_p$ is the $p$-adic absolute value and $d\mu$ is the Haar measure on $\mb{Q}_p^n$. The integral \eqref{eq:p-adic zeta} converges when $\Re s >0$ and continues to a meromorphic function of $s \in \mb{C}$, explicitly computable in terms of an embedded resolution of the (possibly non-reduced) hypersurface $D = \{f = 0\}$ \cite{Denef87}. Taking a suitable limit as $p \to 1$, Denef and Loeser \cite{DL92} also defined the \emph{topological zeta function} $Z_{f, \mrm{top}}^{\mrm{global}}(s)$: this is a rational function in $s$, calculated in terms of an embedded resolution, which now makes sense also when $f$ has complex coefficients (see \eqref{eqn: global top zeta} below). On the other side, we have the \emph{Bernstein-Sato polynomial}: the unique monic polynomial $b_f(s)$, of minimal degree, such that
\[ b_f(s)f^s = P(s)f^{s + 1} \quad \text{for some $P \in \ms{D}_{\mb{C}^n}[s]$}.\]
Here $\ms{D}_{\mb{C}^n}$ denotes the ring of holomorphic differential operators on $\mb{C}^n$.

\begin{conj}[Strong Monodromy Conjecture] \label{conj:intro SMC}
We have 
\begin{enumerate}
\item ($p$-adic version): If $f\in \Q[x_1,\ldots,x_n]$, then, for all but finitely many $p$, if $s = s_0$ is a pole of $Z_{f, p}(s)$, then $s =\Re s_0$ is a root of $b_f(s)$.
\item (Topological version): If $f\in \C[x_1,\ldots,x_n]$ and $s = s_0$ is a pole of $Z^{\mrm{global}}_{f, \mrm{top}}(s)$, then $s = s_0$ is a root of $b_{f}(s)$.
\end{enumerate}
\end{conj}
There is also a ``weak'' form of the conjecture, which states (e.g.\ in the topological setting) that if $s_0$ is a pole of $Z^{\mrm{global}}_{f, \mrm{top}}(s)$ then $\exp(2\pi i s_0)$ is an eigenvalue of the monodromy operator on the nearby cycles complex of $f$. The strong form implies the weak form, since \cite{Malgrange,Kas83} the set of such monodromy eigenvalues is $\{ \exp(2\pi i s_0) \mid b_f(s_0) = 0\}$. These conjectures remain widely open and have attracted persistent attention from many mathematicians over several decades, a history we do not intend to survey in any detail; see \cite{Veys,Meuser} for comprehensive overviews.

Although the Weak Monodromy Conjecture is known in several settings with non-isolated singularities (see, for example \cite{BMT,LV,BV,ELT,LPS,Quek}), substantially fewer results are known for the strong version (see \cite{VeysBath} and references therein). Our main application is a proof of the Strong Monodromy Conjecture in a special case (see \S \ref{sec: single SMC}).

\begin{thm} \label{thm:intro SMC}
The Strong Monodromy Conjecture (Conjecture \ref{conj:intro SMC}) holds when $f$ defines a hyperplane arrangement, i.e.\ when $f$ is a product of linear factors with arbitrary multiplicities.
\end{thm}

Although significantly more approachable than the general case, the special case of hyperplane arrangements has received considerable interest in its own right and has proved quite difficult to establish (see \cite{Walther17,Saitohyperplanearrangement} for earlier results). One particular challenge is that the zeta functions depend only on the intersection lattice of the hyperplane arrangement \cite{BSY11}, while the Bernstein-Sato polynomial does not \cite{Walther17}. It is perhaps of interest that Theorem \ref{thm:intro SMC} therefore provides a combinatorially invariant subset of the roots of $b_f(s)$ (see also \cite{Bath20}). The Weak Monodromy Conjecture for hyperplane arrangements was proved by Budur--Musta\c{t}\u{a}--Teitler \cite{BMT}, who reduced Theorem \ref{thm:intro SMC} to the so-called $n/d$ Conjecture. We prove this below (Theorem \ref{thm: n/d text}).
\begin{thm}[$n/d$ Conjecture]\label{thm: n/d intro}
    If $f\in \mb{C}[x_1, \ldots, x_n]$ defines a degree $d$, central essential indecomposable hyperplane arrangement in $\mb{C}^n$, then $b_f(-n/d) = 0$. 
\end{thm}
In particular, this implies a slight strengthening of Theorem \ref{thm:intro SMC}: for a certain canonical resolution of singularities, every \emph{potential} pole for the local zeta function is in fact a root of $b_f(s)$ (these poles can cancel across different terms, so need not actually be poles of the zeta function itself, see \cite[Appendix]{BSY11} or Remark \ref{remark: archimedean zeta function does not work}). This conjecture was first established by Walther \cite{Walther17} for tame arrangements (and thus $n=3$), see also \cite{Walther05,BSY11,BW17,Bath20,Shi-Zuo24,XY25} for some previous work. 

The proof of Theorem \ref{thm: n/d intro} can be adapted to detect more roots.
\begin{thm}\label{thm: roots n+1/d}
With the same assumption as in Theorem \ref{thm: n/d intro}, we have $b_f(-k/d)=0$ for $n+1\leq k\leq d-1$.
\end{thm}
Let us remark briefly on some further challenges that have previously arisen in attempts to prove the $n/d$ Conjecture, beyond those already present for the Strong Monodromy Conjecture. Beyond explicit calculation, there are (broadly speaking) two main methods to access roots of the Bernstein-Sato polynomial: via jumping numbers of multiplier ideals \cite{ELSV} and poles of Archimedean zeta functions lying in $(-1,0)$ \cite{bernstein}. The root $-n/d$ in Theorem \ref{thm: n/d intro}, however, need not be either a jumping number (see \cite[Remark 3.4]{BMT}) nor a pole of the Archimedean zeta functions (see Remark \ref{remark: archimedean zeta function does not work}).

Our proof of Theorem \ref{thm:intro SMC} is inspired by the following general strategy suggested in \cite{Budur15, Budur24}. Choose a factorisation $f = f_1 \cdots f_r$ and consider the statement for the family of functions $f_1^{a_1} \cdots f_r^{a_r}$ as we vary the multiplicities $a_i \in \mb{Z}_{>0}$. This leads naturally to the definition of the \emph{Bernstein-Sato ideal} of $F = (f_1, \ldots, f_r)$:
\[ B_F = \{ b \in \mb{C}[s_1, \ldots, s_r] \mid b(s_1, \ldots, s_r) f_1^{s_1} \cdots f_r^{s_r} \in \ms{D}_{\mb{C}^n}[s_1, \ldots, s_r] \cdot f_1^{s_1 + 1} \cdots f_r^{s_r + 1}\}.\]
Versions of the (Strong) Monodromy Conjecture and $n/d$ Conjecture have been proposed for Bernstein-Sato ideals and multivariable zeta functions \cite[Conjectures 1.13 and 1.17]{Budur15}. 
One might hope to deduce from this that the single-variable conjecture holds also for $f$ by specialising $s_i = s$. The obstruction to carrying out this argument is the unfortunate fact that the Bernstein-Sato ideal does not necessarily behave well under specialisation: it is not known (or necessarily expected) that setting $s_i = s$ in $B_F$ recovers the ideal $(b_f(s))$ (see \cite[Page 6]{Budur24}). A similar issue arises for zeta functions, see e.g.\ \cite[Definition 3.21]{Budur15}. 
Instead, our proof carries out this strategy using a multivariate version of the \emph{$V$-filtration}; this turns out to have better specialisation properties than the Bernstein-Sato ideal and still carries enough information about Bernstein-Sato polynomials to prove the theorem. (We give a few more details on the proof in \S\ref{subsec:intro SMC again} below.) Our method is also robust enough to prove the multivariate conjectures for hyperplane arrangements in full generality (see \S \ref{subsec:multi SMC}).

\begin{thm}\label{thm: multivariate n/d intro}
The multivariate Strong Monodromy Conjecture and $n/d$ Conjecture hold for arbitrary factorisations of hyperplane arrangements.
\end{thm}
See \cite{maisonobe2016b,BathTrans,AAB,Wu22,Bath23,Sv23,BSZ25} for some previous work. \subsection{The multivariate $V$-filtration}

With this worthy motivation in mind, let us turn now to the technicalities of the $V$-filtration, which is really the main focus for most of this paper. The classical $V$-filtration along a smooth hypersurface \cite{Kas83, Malgrange} is a fundamental tool in $\ms{D}$-module theory, used among other things to construct nearby and vanishing cycles in this setting. It has proved very useful in the study of Bernstein-Sato roots, especially when combined with Saito's theory of mixed Hodge modules \cite{Saito88, Saito90}; see, e.g.\ \cite{BS05, Saito16, MPVfiltration, SY23}.

To fix ideas, let $X$ be a complex manifold, $D \subset X$ a smooth divisor and $\mc{M}$ a coherent (left) $\ms{D}_X$-module. The $V$-filtration of $\ms{D}_X$ along $D$ is the decreasing filtration given by
\begin{equation}\label{eqn: V filtration on DX} V^n\ms{D}_X = \{ P \in \ms{D}_X \mid P t^k \in \mc{O}_Xt^{k + n}\; \text{for all $k \geq 0$}\}\end{equation}
for $n \in \mb{Z}$, where for simplicity we fix a local coordinate $t$ on $X$ such that $D = \{t = 0\}$. A \emph{$V$-filtration} of $\mc{M}$ along $D$ (see Definition \ref{defn:V-filtration} for more details) is a decreasing $\mb{R}$-indexed filtration $V^\bullet \mc{M}$ such that
\begin{enumerate}
\item \label{itm:intro V-filtration good} $V^\bullet \mc{M}$ is a good (i.e.\ finitely generated) filtration with respect to $V^\bullet\ms{D}_X$, and
\item \label{itm:intro V-filtration eigenvalue} $\gr_V^\alpha\mc{M}:= V^\alpha\mc{M}/V^{>\alpha}\mc{M}$ is annihilated by a product of operators of the form $-\partial_t t + \gamma$ with $\Re \gamma = \alpha$.
\end{enumerate}
The first condition implies, in particular, that $V^\alpha \mc{M}$ is constant on a partition of $\mb{R}$ into intervals of the form $(\alpha_i, \alpha_{i + 1}]$, while the second implies that the jumping points $\alpha_i$ are the real parts of the eigenvalues of $\partial_t t$ on $\mc{M}$. (One can also consider a $\mb{C}$-indexed version, in which the jumps are the eigenvalues themselves, but we found the $\mb{R}$-indexed version here simpler to generalise.) The $V$-filtration is unique if it exists, which it always does if $\mc{M}$ is holonomic. Very roughly, one can think of $V^\bullet\mc{M}$ as measuring either the monodromy behaviour or the order of vanishing/pole along $D$ of the sections of $\mc{M}$; see \cite{CDM} for a nice introduction to this topic.

Let us now replace the smooth divisor in this setup with a simple normal crossings divisor $D$ with irreducible components $D_1, \ldots, D_r$, with $D_i = \{t_i = 0\}$ for some local coordinates $t_i$. In this setting, an analogue of the $V$-filtration should simultaneously measure the monodromy behaviour of $\mc{M}$ around each of the divisors $D_i$. A naive approach would be to consider the intersection
\begin{equation} \label{eq:intro naive multivariate V}
V_{D_1}^{\alpha_1} \mc{M} \cap \cdots \cap V_{D_r}^{\alpha_r}\mc{M} \quad \text{for}\; (\alpha_1, \ldots, \alpha_r) \in \mb{R}^r,
\end{equation}
where $V_{D_i}^{\bullet}\mc{M}$ is the usual $V$-filtration along the smooth divisor $D_i$. This gives an $\mb{R}^r$-indexed filtration that does satisfy an analogue of the eigenvalue condition \eqref{itm:intro V-filtration eigenvalue}, but is usually not finitely generated in any reasonable sense (see Example \ref{example: diagonal embedding}) unless the divisor $D$ is particularly well-adapted to the structure of $\mc{M}$ (see, e.g.\ \S\ref{subsec:normal crossings}).

In \cite{Sabbah1987}, Sabbah gave a more refined definition of such a filtration: in addition to the usual $V$-filtrations along the components, which filter by eigenvalues of $\partial_{t_i}t_i$, he defined an analogous filtration $\ls{\bs{L}}{V}^\bullet\mc{M}$ for every non-zero slope $\bs{L} = (L_1, \ldots, L_r) \in \mb{Q}_{\geq 0}^r$, which filters by eigenvalues of $L_1 \partial_{t_1}t_1 + \cdots + L_r \partial_{t_r}t_r$. For $R$ a finite union of cosets of $\mb{Z}^r \subset \mb{R}^r$, he then defined the \emph{canonical multi-filtration} to be the unique good $R$-indexed filtration (if it exists) satisfying
\[ V_R^{\bs{\alpha}}\mc{M} = \bigcap_{\bs{L} \in \mb{Q}_{\geq 0}^r} \ls{\bs{L}}{V}^{\bs{L}(\bs{\alpha})}\mc{M} \quad \text{and} \quad \ls{\bs{L}}{V}^{\gamma}\mc{M} = \sum_{\substack{\bs{\alpha} \in R \\\bs{L}(\bs{\alpha}) \leq \gamma}} V_R^{\bs{\alpha}}\mc{M},\]
where, for $\bs{\alpha} = (\alpha_1, \ldots, \alpha_r) \in \mb{R}^r$ and $\bs{L} \in \mb{Q}_{\geq 0}^r$, we write $\bs{L}(\bs{\alpha}) = L_1 \alpha_1 + \cdots + L_r \alpha_r$. One can think of the first condition as a refinement of \eqref{eq:intro naive multivariate V}, and the second as a consistency check. We will call $V_R^\bullet\mc{M}$ \emph{Sabbah's $V$-filtration}. Sabbah showed that his filtration exists at least when $\mc{M}$ is the direct image of the structure sheaf under the graph of a collection of holomorphic functions \cite[Theorem 3.4.1]{Sabbah1987}.\footnote{This is also asserted but not proved for $\mc{M}$ holonomic \cite[Footnote on page 286]{Sabbah1987}; see, however \cite[Corollary 4.19]{Sabbahbifiltration} for a partial result and Theorems \ref{thm:intro sabbah comparison} and \ref{thm:intro V existence} below for a general proof.}

Sabbah's approach gives, in some sense, a good generalisation of the $V$-filtration to divisors with many components; see, e.g.\ \cite{Sabbah1987b, SabbahAlexander} for some applications. However, since it is defined ``externally'' in terms of the filtrations $\ls{\bs{L}}{V}^\bullet\mc{M}$, instead of by intrinsic conditions, it has proved quite difficult to work with in practice. For example, the usual $V$-filtration has the property that every morphism of $\ms{D}$-modules is necessarily strict with respect to it; while this was claimed for $V_R^\bullet$ in \cite{Sabbah1987}, the claim was subsequently retracted in \cite[p767]{SabbahAlexander} and, to the best of our knowledge, was unknown prior to the present work.

With these drawbacks in mind, we give a different definition of a multivariate $V$-filtration of the $\ms{D}$-module $\mc{M}$ along a simple normal crossings divisor $D$. We explain the general idea here, and refer the reader to \S\ref{subsec:multivariate V definition} below for the precise definition.

First, we introduce the notion of a \emph{wall and chamber filtration}, generalising the property that the usual $V$-filtration is indexed by $\mb{R}$ but jumps only at a discrete set. A wall and chamber filtration is defined to be an $\mb{R}^r$-indexed filtration that jumps only along a discrete set of hyperplanes (the walls) and is constant on the components of their complement (the chambers). For a coherent $\ms{D}_X$-module $\mc{M}$, we define the notion of a \emph{good wall and chamber filtration} over the $\mb{Z}^r$-indexed filtration
\[ V^{\bs{n}}\ms{D}_X := V^{n_1}_{D_1}\ms{D}_X \cap \cdots \cap V^{n_r}_{D_r}\ms{D}_X \quad \text{for $\bs{n} = (n_1, \ldots, n_r) \in \mb{Z}^r$}\]
by imposing a natural finite generation condition.

To define a multivariate $V$-filtration, we impose an additional eigenvalue condition on the jumps across the walls. If $\bs{\alpha}, \bs{\beta} \in \mb{R}^r$ are separated by a wall of the form $L_1 x_1 + \cdots + L_r x_r = \gamma$, we require that the quotient $V^{\bs{\alpha}}\mc{M}/V^{\bs{\beta}}\mc{M}$ be annihilated by a product of operators of the form
\[ \gamma' - L_1\partial_{t_1}t_1 - \cdots - L_r \partial_{t_r}t_r\]
with $\Re \gamma' = \gamma$.  We prove that these conditions, which reduce to the definition of the usual $V$-filtration when $r = 1$, uniquely determine the filtration.
\begin{thm}[Theorem \ref{thm:multivariate V uniqueness}] \label{thm:intro V uniqueness}
Let $\mc{M}$ be a coherent $\ms{D}_X$-module and $D \subset X$ a divisor with simple normal crossings. Then any two multivariate $V$-filtrations on $\mc{M}$ along $D$ coincide.
\end{thm}

In view of Theorem \ref{thm:intro V uniqueness}, we simply write $V^\bullet\mc{M}$ for the multivariate $V$-filtration on $\mc{M}$ as long as it exists.

The classical proof of Theorem \ref{thm:intro V uniqueness} for $r = 1$ is a simple exercise in manipulating good filtrations and eigenvalues. However, since hyperplanes can intersect in higher dimensions, this argument does not easily generalise to $r > 1$, so new ideas are needed. Our proof hinges on the observation (Lemma \ref{lem:localised chambers}), that the defining eigenvalue condition on the quotients across the walls is equivalent to requiring that the localisation of $V^\bullet\mc{M}$ at any maximal ideal in $\mb{C}[\partial_{t_1}t_1, \ldots, \partial_{t_r}t_r]$ be constant away from a finite set of walls passing through a single point. This characterisation is used again and again throughout the paper. Using this, we are able to write down direct formulas for the localisations of $V^\bullet\mc{M}$ (Lemma \ref{lem:multivariate V localisation}) in terms of generators for $\mc{M}$, which uniquely determine the entire filtration (if it exists).

With the aid of Theorem \ref{thm:intro V uniqueness}, as well as Lemma \ref{lem:localised chambers} described above, we are able to prove that the multivariate $V$-filtration has the same good strictness and functoriality properties as in the single variable case. We summarise these here:

\begin{thm}[Proposition \ref{prop:V image and preimage}, Corollary \ref{cor:V strictness} and Theorem \ref{thm:multivariate V direct image}] \label{thm:intro V strictness and functoriality}
The multivariate $V$-filtration has the following properties.
\begin{enumerate}
\item If $\mc{M}$ admits a multivariate $V$-filtration, then any subquotient of $\mc{M}$ does as well.
\item If $\eta \colon \mc{M} \to \mc{N}$ is a morphism of coherent $\ms{D}_X$-modules and the multivariate $V$-filtrations $V^\bullet\mc{M}$ and $V^\bullet\mc{N}$ exist, then $\eta$ is strict with respect to these filtrations:
\[ \eta(V^{\bs{\alpha}}\mc{M}) = \eta(\mc{M}) \cap V^{\bs{\alpha}}\mc{N} \quad \text{for all $\bs{\alpha} \in \mb{R}^r$}.\]
\item Let $X, Y, Z$ be complex manifolds and let $\pi \colon X \to Y$ be a proper morphism. If $\mc{M}$ is a coherent $\ms{D}_{X \times Z}$-module admitting a multivariate $V$-filtration $V^\bullet\mc{M}$ along $X \times D$ for some simple normal crossings divisor $D \subset Z$, then the direct image
\[ \pi_+(\mc{M}, V^\bullet\mc{M}) \in \mrm{D}^b_{\mrm{coh}}(\ms{D}_{Y \times Z}, V^\bullet\ms{D}_{Y \times Z}) \]
is a strict complex, and the induced filtrations on the cohomology sheaves are multivariate $V$-filtrations along $Y \times D$.
\end{enumerate}
\end{thm}
These are all standard when $r=1$; for example, the statement (3) can be found for example in \cite{MS89} or \cite[Proposition 7.5.2]{SabbahnoteDmodule} (albeit stated in a slightly different form).

Of course, it is natural to ask whether our new  multivariate $V$-filtration coincides with Sabbah's. We answer this in the affirmative:

\begin{thm}[Theorem \ref{thm:sabbah comparison}]\label{thm:intro sabbah comparison}
For a coherent $\ms{D}_X$-module $\mc{M}$ and simple normal crossings divisor $D$, Sabbah's $V$-filtration $V^\bullet_R\mc{M}$ exists for some $R\subset \mb{R}^r$, which is a finite union of cosets of $\mb{Z}^r$, if and only if the multivariate $V$-filtration $V^\bullet\mc{M}$ exists. In this case, Sabbah's $V$-filtration is the restriction of $V^\bullet\mc{M}$ to $R \subset \mb{R}^r$.
\end{thm}
In particular, Theorems \ref{thm:intro V strictness and functoriality} and \ref{thm:intro sabbah comparison} resolve the confusion around strictness for Sabbah's filtration. We note that the proof of Theorem \ref{thm:intro sabbah comparison} is far from being a triviality: for example, it relies in an essential way on both Theorem \ref{thm:intro V uniqueness} and the existence of an adapted fan \cite[\S 2.2]{Sabbah1987}.

Now, in order for this theory to be interesting, one would of course like the multivariate $V$-filtration to exist in a wide variety of examples. Our next result says that this is in fact always the case for holonomic $\ms{D}$-modules:

\begin{thm}[Theorem \ref{thm:multivariate V existence}] \label{thm:intro V existence}
Let $\mc{M}$ be a holonomic $\ms{D}_X$-module and $D \subset X$ a divisor with simple normal crossings. Then the multivariate $V$-filtration of $\mc{M}$ along $D$ exists.
\end{thm}

Together with Theorem \ref{thm:intro sabbah comparison}, Theorem \ref{thm:intro V existence} also proves existence of Sabbah's filtration for arbitrary holonomic $\ms{D}$-modules; the latter is established for $r=2$ in the related setting of $\mb{R}$-specialisable $\ms{D}$-modules in \cite{Sabbahbifiltration}.

Our proof of Theorem \ref{thm:intro V existence} works roughly as follows. First, we observe that when $\mc{M}$ is a flat meromorphic connection with poles along $D$ and good formal structure \cite{Sabbahirregular}, the naive formula \eqref{eq:intro naive multivariate V} in fact gives a multivariate $V$-filtration. For general holonomic $\mc{M}$, we use the \emph{resolution of turning points}, established by Kedlaya and Mochizuki \cite{Kedlaya2021,Mochizuki2011}, together with Theorem \ref{thm:intro V strictness and functoriality} to reduce to this case. One nice feature of this approach is that it also gives some control over the walls in terms of a resolution (see e.g.\ Corollary \ref{cor:walls via resolution}) and over the singular supports of the $V^{\bs{\alpha}}\mc{M}$ (see Theorem \ref{thm:intro relative holonomic and flatness} below). 

\subsection{The Malgrange-Mellin transform}

With this exploration of the remote reaches of the $V$-filtration under our belts, let us chart a course back towards the familiar seas of the Bernstein-Sato polynomial. The first step is to bring expressions like $f^s$ into the picture.

Recall that if $f \colon X \to \mb{C}$ is a holomorphic function with the graph embedding  $\iota = (\id_X, f) \colon X \to X \times \mb{C}$, then for any coherent $\ms{D}_X$-module $\mc{M}$ on which $f$ acts invertibly, we have an isomorphism of $\ms{D}_X[s]$-modues
\begin{equation} \label{eq:intro Malgrange-Mellin}
\iota_+\mc{M} \cong \mc{M}[s]f^s,
\end{equation}
where $s$ acts on $\iota_+\mc{M}$ via $s = -\partial_t t$ for $t$ the coordinate on $\mb{C}$ and $\mc{M}[s]f^s = \mc{M} \otimes \mb{C}[s]$ as $\mc{O}_X[s]$-modules with vector fields acting by the usual formula for differentiating the symbol $f^s$ (see \eqref{eqn: action on M[s]fs}). We suggest the name \emph{Malgrange-Mellin transform} for the isomorphism \eqref{eq:intro Malgrange-Mellin}, as it was first observed by Malgrange \cite{Malgrange} and is indeed a holomorphic analogue of the Mellin transform when $\mc{M} = \mc{O}_X[f^{-1}]$ (see Remark \ref{remark: heuristic reason for MM transform}). In particular, the $V$-filtration $V^\bullet\iota_+\mc{M}$ along the smooth divisor $X \times \{0\}$ gives a filtration of $\mc{M}[s]f^s$ by $\ms{D}_X[s]$-submodules. This construction is used in many places (e.g.\ \cite{Kas83,Malgrange}) to relate the $V$-filtration to $b$-functions: for example, every jumping number of the $V$-filtration on $\iota_{+}\cO_X$ is an integer shift of a root of $b_f(s)$ and, when restricting to the lowest piece of the Hodge filtration on $\iota_{+}\cO_{\C^n}$, every jumping number between $(0,1)$ must be a root of $b_f(s)$, by the work of Ein-Lazarsfeld-Smith-Varolin \cite{ELSV} (see also \cite{BS05}). 

More generally, given a tuple $f_1, \ldots, f_r \colon X \to \mb{C}$ of holomorphic functions and $\mc{M} = \mc{M}[(f_1 \cdots f_r)^{-1}]$, we have a Malgrange-Mellin transform
\[ \iota_+\mc{M}\cong \mc{M}[s_1, \ldots, s_r]f_1^{s_1} \cdots f_r^{s_r} =: \mc{M}[\bs{s}]\bs{f^s},\]
an isomorphism of $\ms{D}_X[\bs{s}]:= \ms{D}_X[s_1, \ldots, s_r]$-modules, where now 
\[\iota = (\id_X, f_1, \ldots, f_r) \colon X \to X \times \mb{C}^r\]
denotes the graph embedding for the tuple $(f_1, \ldots, f_r)$. The multivariate $V$-filtration $V^\bullet \iota_+\mc{M}$ along the coordinate axes in $\mb{C}^r$ now defines a wall and chamber filtration of $\mc{M}[\bs{s}]\bs{f^s}$ by $\ms{D}_X[\bs{s}]$-submodules. Our next result shows that the $\ms{D}_X[\bs{s}]$-modules obtained in this way are particularly well-behaved (if $\bs{\alpha} \in \mb{R}_{\geq 0}^r$), which is even new in the single variable case.

\begin{thm}[Proposition \ref{prop:V is relative holonomic} and Theorem \ref{thm:flatness}]  \label{thm:intro relative holonomic and flatness}
Suppose that $\mc{M}$ is a holonomic $\ms{D}_X$-module on which the $f_i$ act bijectively. Then for all $\bs{\alpha} \in \mb{R}_{\geq 0}^r$, the $\ms{D}_X[\bs{s}]$-module $V^{\bs{\alpha}}\iota_+\mc{M}$ is relative holonomic (see \S \ref{subsec:relative holonomicity}) and flat over $\mb{C}[\bs{s}]$. Moreover, if $\bs{\alpha}, \bs{\beta} \in \mb{R}_{\geq 0}^r$ and $\bs{\alpha} \leq \bs{\beta}$, then $V^{\bs{\alpha}}\iota_+\mc{M}/V^{\bs{\beta}}\iota_+\mc{M}$ is flat relative to the walls separating $\bs{\alpha}$ and $\bs{\beta}$ (Definition \ref{defn:relative flatness}).
\end{thm}

In more intuitive terms, Theorem \ref{thm:intro relative holonomic and flatness} says that each piece of the multivariate $V$-filtration $V^{\bs{\alpha}}\iota_+\mc{M}$ with $\bs{\alpha} \in \mb{R}^r_{\geq 0}$ can be regarded as a flat family of holonomic $\ms{D}_X$-modules over $\mb{C}^r$, and that the quotients $V^{\bs{\alpha}}\iota_+\mc{M}/V^{\bs{\beta}}\iota_+\mc{M}$ remain as flat as they can be given the defining restriction on their support. The next result explains how these families behave under monomial transformations of the functions $f_i$.

\begin{thm}[Corollary \ref{cor: change of functions}] \label{thm:intro change of functions}
Suppose $f_j' = \prod_{i=1}^r f_i^{d_{ij}}$ for some $d_{ij} \in \mb{Z}_{\geq 0}$  and $j = 1, \ldots, r'$. If the zero loci of $f_1 \cdots f_r$ and $f_1' \cdots f_{r'}'$ coincide, then for $\bs{\alpha}' = (\alpha_1', \ldots, \alpha_{r'}') \in \mb{R}_{\geq 0}^{r'}$ and $\bs{\alpha} = (\sum_j d_{1j} \alpha_j', \ldots, \sum_j d_{rj}\alpha_{j}')$, we have an isomorphism of $\sD_X[\bs{s}']$-modules:
\[ V^{\bs{\alpha}}\iota_+\mc{M} \otimes_{\mb{C}[\bs{s}]} \mb{C}[\bs{s}'] \cong V^{\bs{\alpha}'}\iota_+'\mc{M},\]
where $\iota' = (\id_X, f_1', \ldots, f_r') \colon X \to X \times \mb{C}^{r'}$ and the tensor product is taken over $s_i \mapsto \sum_j d_{ij} s_j'$.
\end{thm}

Note that the analogous statement without the $V$-filtrations is immediate from the Malgrange-Mellin transform. In particular, when $r' = 1$, Theorem \ref{thm:intro change of functions} implies that one obtains the usual $V$-filtration (for positive indices) by specialising the multivariate $V$-filtration in an appropriate way. The flatness and specialisation properties guaranteed by Theorems \ref{thm:intro relative holonomic and flatness} and \ref{thm:intro change of functions} are the key advantages of the multivariate $V$-filtration over the Bernstein-Sato ideal in our proof of the Strong Monodromy Conjecture for hyperplane arrangements.

\begin{rmk}
The assumption $\mc{M} = \mc{M}[(f_1 \cdots f_r)^{-1}]$ in Theorems \ref{thm:intro relative holonomic and flatness} and \ref{thm:intro change of functions} is fairly inessential. We show (Corollary \ref{cor:localising V}) that
\[ V^{\bs{\alpha}}\iota_+\mc{M} = V^{\bs{\alpha}}\iota_+\mc{M}[(f_1 \cdots f_r)^{-1}]  \quad \text{for $\bs{\alpha} \in \mb{R}^r_{> 0}$}.\]
So Theorems \ref{thm:intro relative holonomic and flatness} and \ref{thm:intro change of functions} (and other similar statements below) hold for arbitrary holonomic $\mc{M}$ as long as $\bs{\alpha} \in \mb{R}^r_{> 0}$.
\end{rmk}

The proof of Theorem \ref{thm:intro relative holonomic and flatness} (on which Theorem \ref{thm:intro change of functions} relies) has two main ingredients: the description of the multivariate $V$-filtration in terms of resolutions (which we use to prove relative holonomicity) and the following result on the behaviour of the multivariate $V$-filtration under duality (which we use to prove the flatness statements).

\begin{thm}[Theorem \ref{thm:duality}] \label{thm:intro duality}
Let $\mc{M}$ be a holonomic $\ms{D}_X$-module. Then there is a natural isomorphism 
\[ \mb{D}_{\ms{D}_X[\bs{s}]}V^{\bs{\alpha}}\iota_+\mc{M} \cong (V^{-\bs{\alpha} + \epsilon \bs{1} + m\bs{1} }\iota_+\mb{D}\mc{M})_{\bs{s} \mapsto -\bs{s} - m\bs{1}},\]
for all $\bs{\alpha} \in \mb{R}^r_{> 0}$, $0 < \epsilon \ll 1$ and $m \in \mb{Z}$ such that $-\bs{\alpha} + m \bs{1} \in \mb{R}^r_{\geq 0}$. Here $\bs{1} = (1, \ldots, 1)\in \R^r$, $\mb{D}\mc{M}$ is the usual $\ms{D}_X$-module dual, $\mb{D}_{\ms{D}_X[\bs{s}]}$ denotes the duality functor for $\ms{D}_X[\bs{s}]$-modules (see \S \ref{subsec:duality}), and $(-)_{\bs{s} \mapsto -\bs{s} - m\bs{1}}$ denotes the involution on the category of $\ms{D}_X[\bs{s}]$-modules sending $s_i$ to $-s_i - m$. 
\end{thm}
When $r=1$, Theorem \ref{thm:duality} refines the classical self-duality of the nearby cycles functor for $\sD$-modules \cite{Saitodual,MM}; see Remark \ref{remark: recover usual self-duality}.

\subsection{Mixed Hodge modules}

So far, everything we have said about the multivariate $V$-filtration holds for arbitrary holonomic $\ms{D}$-modules. Our next few results refine some of these statements in the more structured setting of Saito's mixed Hodge modules. This is of intrinsic interest, and furthermore provides some essential Hodge-theoretic input into our application to the Strong Monodromy Conjecture.

Intuitively, mixed Hodge modules are to $\ms{D}$-modules as mixed Hodge structures are to vector spaces. More precisely, a mixed Hodge module (in the sense of \cite{Saito88, Saito90}) is a regular holonomic $\ms{D}$-module $\mc{M}$ equipped with a \emph{weight filtration} $W_\bullet\mc{M}$, a \emph{Hodge filtration} $F_\bullet \mc{M}$ and a $\mb{Q}$-structure on the associated perverse sheaf. These structures are required to satisfy some stringent inductive conditions, which we will not attempt to recall in this paper; see \cite{OverviewMHM} for an excellent introduction. The main examples are polarisable variations of (pure) Hodge structure, as well as mixed versions of these. The main upshot of the general theory is that the category of mixed Hodge modules inherits both the functoriality (i.e.\ six operations) of the category of holonomic $\ms{D}$-modules and the strictness properties of the category of mixed Hodge structures.

As has been observed previously \cite{MPbirational,SY23,DY25}, for many applications (including to singularity theory), one sometimes wants to relax this structure a little, and allow some $\ms{D}$-modules whose associated perverse sheaves do not admit $\mb{Q}$-structures. There are, broadly speaking, two ways to do this. The least ambitious is to formally extend scalars in Saito's theory from $\mb{Q}$ to $\bar{\mb{Q}}$, see \cite[\S 1.4]{DY25}. This gives a notion of $\bar{\mb{Q}}$-mixed Hodge module; the underlying $\ms{D}$-module of such a thing will always be a direct summand of the $\ms{D}$-module underlying a $\mb{Q}$-mixed Hodge module of a very particular type. This is essentially the approach taken in \cite{MPbirational}, for example (albeit dressed in more expensive clothing). Alternatively, one can work from the beginning in the setting of \emph{complex} mixed Hodge modules of \cite{MHMproject}, which builds a version of the entire theory without mentioning $\mb{Q}$-structures or perverse sheaves. The latter approach, taken in \cite{SY23}, has the advantage of being slightly more flexible (as well as giving a different perspective on polarisations, which can be useful, see e.g.\ \cite{DV23, DLY}), but the disadvantage that the \emph{mixed} part of the theory of complex mixed Hodge modules is not yet completely written down. We will generally state our results without specifying which theory we are using, but indicate where we must either work in a restricted context (so that the $\bar{\mb{Q}}$ theory applies) or assume some expected property of the theory of complex mixed Hodge modules.

Now suppose that $f_1, \ldots, f_r$ are holomorphic functions on a complex manifold $X$ and that $\mc{M}$ is a mixed Hodge module on $X$ such that $\mc{M} = \mc{M}[(f_1 \cdots f_r)^{-1}]$ as a $\ms{D}$-module. Then we can consider the direct image $\iota_+\mc{M}$ as a mixed Hodge module on $X \times \mb{C}^r$; the Hodge filtration $F_\bullet$ induces a filtration on $V^{\bs{\alpha}}\iota_+\mc{M}$, which we denote by the same letter. In this context, we have Hodge-filtered versions of Theorems \ref{thm:intro relative holonomic and flatness} and \ref{thm:intro change of functions}.

\begin{thm}[Theorem \ref{thm:hodge flatness} and Theorem \ref{thm:filtered change of functions}]\label{thm:intro hodge flatness}
For $\mc{M}$ as above, we have:\begin{enumerate}
\item For any $\bs{\alpha} \in \mb{R}^{r}_{\geq 0}$, the stalks of $\gr^{F}V^{\bs{\alpha}}\mc{M}$ and $V^{\bs{\alpha}}\mc{M}$ are free over $\C[\bs{s}]$. 
\item Moreover, if $\bs{\alpha}\leq \bs{\beta}\in \mb{R}^{r}_{\geq 0}$ are separated by a single wall $W \subset \mb{R}^r$, then the stalks of $\gr^{F}(V^{\bs{\alpha}}\iota_{+}\mc{M}/V^{\bs{\beta}}\iota_{+}\mc{M})$ are free relative to $W$ (Definition \ref{defn:relative freeness}).
\item In the setting of Theorem \ref{thm:intro change of functions}, we have 
\[ \gr^F V^{\bs{\alpha}}\iota_+\mc{M}(-r) \otimes_{\mb{C}[\bs{s}]} \mb{C}[\bs{s}'] = \gr^F V^{\bs{\alpha}'}\iota_+'\mc{M}(-r').\]
\end{enumerate}
\end{thm}

As a corollary of Theorem \ref{thm:intro hodge flatness}, we have the following nontrivial generalisation of a classical result of Budur-Saito \cite{BS05}, which allows us to control the lowest piece of the Hodge filtration on the multivariate $V$-filtration using birational geometry.

\begin{cor} [Corollary \ref{cor:log canonical}]\label{cor:multivariate Budur-Saito}
For $\bs{\alpha} = (\alpha_1, \ldots, \alpha_r) \in \mb{R}_{> 0}^r$, we have an equality of ideals
\[ \mc{J}\left(X,\sum_{i=1}^r (\alpha_i - \epsilon)D_i\right) =F_rV^{\bs{\alpha}}\iota_+\mc{O}_X \subset F_r\iota_+\mc{O}_X = \mc{O}_X,\]
where $D_i = \mrm{div}(f_i)$ and $\mc{J}$ denotes the multiplier ideal of an $\mb{R}$-divisor.
\end{cor}

Our proof of Theorem \ref{thm:intro hodge flatness} makes use of the following multivariate generalisation of the main result in \cite{DY25}, which one can think of as a Hodge-theoretic enhancement of the Malgrange-Mellin transform $\iota_+\mc{M} \cong \mc{M}[\bs{s}]\bs{f^s}$.

To motivate the result, observe that $\iota_+\mc{M}$ naturally inherits the structure of a mixed Hodge module from $\mc{M}$ and hence carries a Hodge filtration. On the other hand, $\mc{M}[\bs{s}]\bs{f^s}$ does not carry any {\it a priori} Hodge structure (there is no theory of mixed Hodge $\ms{D}_X[\bs{s}]$-modules, and $\mc{M}[\bs{s}]\bs{f^s}$ is not even holonomic as a $\ms{D}_X$-module). However, if we fix $\bs{\alpha} = (\alpha_1, \ldots, \alpha_r) \in \mb{R}^r$ and $n > 0$, then the $\ms{D}_X$-module
\[ \frac{\mc{M}[\bs{s}]\bs{f^s}}{((s_1 + \alpha_1)^n, \ldots, (s_r + \alpha_r)^n)} \]
\emph{is} holonomic and in fact carries a natural complex mixed Hodge module structure. These structures are compatible for different $n$, and thus define a pro-object $\mc{M}[[\bs{s} + \bs{\alpha}]]\bs{f^s}$ in the category of mixed Hodge modules on $X$. These pro-objects can be taken as stand-ins for the non-existent Hodge version of $\mc{M}[\bs{s}]\bs{f^s}$. Pro-objects of this kind play a big role in many areas, including Beilinson's construction of the nearby cycles and maximal extension functors \cite{Beilinson} and the free-monodromic local systems studied in geometric representation theory (e.g.\ \cite{bezrukavnikov-yun}).

Now, $\mc{M}[[\bs{s} + \bs{\alpha}]]\bs{f^s}$ is so far defined just as a pro-object, i.e.\ the inverse limit is interpreted in a formal categorical sense. We can of course also consider the inverse limit in the category of (not necessarily quasi-coherent) sheaves of $\ms{D}_X[\bs{s}]$-modules; unless otherwise specified, we will denote this sheaf by the same symbol. Because of its Hodge-theoretic origin, the sheaf $\mc{M}[[\bs{s} + \bs{\alpha}]]\bs{f^s}$ comes equipped with a filtration defined by the formula 
\[ F_p \mc{M}[[\bs{s} + \bs{\alpha}]] \bs{f^s} := \varprojlim_n F_p\left(\frac{\mc{M}[\bs{s}]\bs{f^s}}{((s_1 + \alpha_1)^n, \ldots, (s_r + \alpha_r)^n)}\right).\]
Note, however, that because limits and colimits do not commute, this filtration will not usually be exhaustive: the same phenomenon occurs for $\mb{C}[[s]] = \varprojlim_n \mb{C}[s]/(s^n)$ and the order filtration on $\mb{C}[s]$, for example. Thus, the union of the $F_p$'s defines an interesting subsheaf inside $\mc{M}[[\bs{s} + \bs{\alpha}]]\bs{f^s}$. The Hodge version of the Malgrange-Mellin transform is:

\begin{thm}[Theorem \ref{thm:multivariate magic formula}] \label{thm:intro magic formula}
If $\mc{M}$ is a mixed Hodge module on which $f_1 \cdots f_r$ acts bijectively, then for $\bs{\alpha} \in \mb{R}^r_{\geq 0}$ we have an isomorphism of filtered $\ms{D}_X[\bs{s}]$-modules
\[ (V^{\bs{\alpha}}\iota_+\mc{M}, F_{\bullet + r} V^{\bs{\alpha}}\iota_+\mc{M}) \cong \left(\bigcup_p F_p \mc{M}[[\bs{s} + \bs{\alpha}]]\bs{f^s}, F_\bullet \mc{M}[[\bs{s} + \bs{\alpha}]]\bs{f^s}\right).\]
\end{thm}

Note that in order to make sense of the right hand side in the theory of $\bar{\mb{Q}}$-mixed Hodge modules, one needs to assume that $\bs{\alpha} \in \mb{Q}^r$. Theorem \ref{thm:intro magic formula} applies in that context, and this suffices for our applications to the Strong Monodromy Conjecture. The assumptions about the general theory needed to prove the theorem for complex mixed Hodge modules (in which case the right hand side makes sense for $\bs{\alpha} \in \mb{R}^r$) are explained in \S \ref{sec: recollection}.

In fact, Theorem \ref{thm:intro magic formula} was really the starting point for the present work: having treated the case $r = 1$ in \cite{DY25}, we knew that the right hand side defined a filtration of the left hand side satisfying certain properties. (The inverse limit construction and its properties play a key role in the deformation and wall-crossing theory for Hodge filtrations developed by Vilonen and the first named author in \cite[\S 3]{DV23}; see also \cite[\S 2.4]{DMB} for a treatment closer to the multivariate story considered here.) This guided us to the definition of the multivariate $V$-filtration, which we were then able to show exists, is unique and is equivalent to Sabbah's filtration for arbitrary holonomic $\ms{D}$-modules. Similarly, Theorem \ref{thm:intro hodge flatness} follows quite easily from Theorem \ref{thm:intro magic formula} using the strictness properties of mixed Hodge modules on the right hand side; this implies Theorems \ref{thm:intro relative holonomic and flatness} and \ref{thm:intro change of functions} when $\mc{M}$ underlies a mixed Hodge module, which we were then able to prove directly for arbitrary holonomic $\ms{D}$-modules.

We conclude this subsection by explaining a result that expresses the jumping set of the multivariate $V$-filtration in terms of the zero loci of the Bernstein--Sato ideal. Given a tuple of holomorphic functions $f_1,\ldots,f_r:X\to \C$, let $D=\textrm{div}(f_1\cdots f_r)$ and $\iota:X\to X\times \C^r$ be the graph embedding. Consider the exponential map
\[ \textrm{Exp}: \C^r \to (\mb{C}^{\times})^r, \quad (\alpha_1,\ldots,\alpha_r)\mapsto (e^{2\pi i\alpha_1},\ldots, e^{2\pi i\alpha_r}).\]

\begin{prop}[{Proposition \ref{prop: Bernsteinideal and Vfiltration}}]
We have
\[\mathrm{Exp}(\mathrm{Zero}(B_F))=\mathrm{Exp}\overline{\left(\{-\bs{\alpha}\mid \bs{\alpha}\in \R^r_{\geq 0}, \, V^{\bs{\alpha}}\iota_{+}\cO_X(\ast D)\neq V^{\bs{\alpha}+\epsilon\bs{1}}\iota_{+}\cO_X(\ast D)\}\right)},
\]
where $\overline{(-)}$ denotes the Zariski closure in $\mb{C}^r$ and $B_F$ is the Bernstein-Sato ideal of $(f_1,\ldots,f_r)$.
\end{prop}
If $r=1$, this recovers the classical relationship between zeros of the Bernstein-Sato polynomial and the jumping values of the $V$-filtration. 
\subsection{The Strong Monodromy Conjecture again} \label{subsec:intro SMC again}
Finally, let us return to the Strong Monodromy Conjecture for hyperplane arrangements, and explain how the above theory feeds into its proof. Suppose that $f$ is the equation of a (central) hyperplane arrangement in $X = \mb{C}^n$ and consider its factorisation $f = f_1^{a_1} \cdots f_r^{a_r}$ into linear factors with multiplicities. We write $\iota \colon X \to X \times \mb{C}^r$ for the graph embedding with respect to $f_1, \ldots, f_r$. To prove the Strong Monodromy Conjecture for $f$, we first argue that, for any potential pole of the local zeta functions (with respect to the iterated blow-ups along dense edges), there is a choice of multiplicities $b_1, \ldots, b_r$ such that that pole becomes the log canonical threshold of $g = f_1^{b_1}\cdots f_r^{b_r}$ (cf.\ \cite{Shi-Zuo24}). By Corollary \ref{cor:multivariate Budur-Saito}, this gives us a particularly well-behaved jump in $F_r V^{\bs{\beta}}\iota_+\mc{O}_X$ for an appropriate $\bs{\beta} \in \mb{R}^r_{> 0}$. Using the fact that the Malgrange-Mellin transforms of sections in $F_r\iota_+\mc{O}_X$ are constant in $s_1, \ldots, s_r$, we can vary the multiplicities and use Theorems \ref{thm:intro relative holonomic and flatness} and \ref{thm:intro change of functions} to translate this to a similarly well-behaved jump in the $V$-filtration for the graph embedding of our original $f$. Since this jump in fact occurs in the interval $(0, 1)$, we are able to apply a version of a result of Ein-Lazarsfeld-Smith-Varolin \cite{ELSV} to conclude that it is also a root of the Bernstein-Sato polynomial of the appropriate multiplicity.

\subsection{Outline of the paper}
In \S \ref{sec: multivariate Vfiltration}, we give the definition of multivariate $V$-filtration, prove some basic properties and compare it with Sabbah's filtration. Two variants are introduced: the restricted and localised $V$-filtrations, which play an important role in the interaction of multivariate $V$-filtration with Hodge theory in later sections. In \S \ref{sec: multivariate V and holonomic}, we study the multivariate $V$-filtrations on holonomic $\sD$-modules and establish the existence. In \S \ref{sec: multivariate V as DXsmodule}, further properties of the multivariate $V$-filtrations on graph embeddings are proved: flatness, relative holonomicity and duality. In \S \ref{sec: Hodge and multivariate V}, we examine the interaction between the multivariate $V$-filtration and Hodge filtration on mixed Hodge modules. In \S \ref{sec: strong monodromy conjecture}, we prove the Strong Monodromy Conjecture for hyperplane arrangements and its multivariable generalisation.

\subsection{Notation and conventions} \label{subsec:conventions}

When writing anything about $\ms{D}$-modules, one must make a choice whether to work in the analytic or algebraic category. As a general rule, the latter is slightly easier than the former. All of our results go through in either setting; we have chosen to write them in the analytic context mainly to reassure the reader (and ourselves) that there are no unforeseen issues in this setting.

Thus, for $X$ a complex manifold, we write $\mc{O}_X$ for the sheaf of holomorphic functions on $X$ and $\ms{D}_X$ for the sheaf of holomorphic differential operators. We work with left $\ms{D}_X$-modules throughout, assumed at least quasi-coherent (and usually coherent or even holonomic) unless otherwise specified.

For the sake of brevity, we will often say that various statements about sheaves on a topological space $X$ hold ``locally on $X$'': this means that there exists an open cover $\{U_i\}_{i \in I}$ of $X$ such that the statement holds for the restrictions of the sheaves in question to each $U_i$. For example, if $X$ is a complex manifold, then every coherent $\mc{O}_X$-module is, locally on $X$, a quotient of $\mc{O}_X^{\oplus n}$ for some $n$.

As is usual in Hodge theory, filtrations denoted with an upper (resp.\ lower) index will be decreasing (resp.\ increasing). We will often work with filtrations indexed by the partially ordered set $\mb{R}^r$, with the partial order given by
\[ (\alpha_1, \ldots, \alpha_r) \leq (\beta_1, \ldots, \beta_r) \quad \text{if and only if} \quad \alpha_1 \leq \beta_1, \; \ldots\; , \; \alpha_r \leq \beta_r.\]
We will adopt the convention that vectors in $\mb{R}^r$ are denoted by bold symbols, and that their components are denoted by the corresponding non-bold symbols. So, for example, if $\bs{\alpha} \in \mb{R}^r$, we write $\bs{\alpha} = (\alpha_1, \ldots, \alpha_r)$ for the components. For $1 \leq i \leq r$, we write $\bs{1}_i$ for the vector with $1$ in the $i$th coordinate and $0$ elsewhere, and $\bs{1} = (1, \ldots, 1) = \sum_{i = 1}^r \bs{1}_i$.

Unless otherwise specified, $\epsilon$ will always denote a sufficiently small positive real number (where ``sufficiently small'' is allowed to depend on other quantities in a statement).

\subsection*{Acknowledgements}
The authors would like to thank Eamon Quinlan, Botong Wang, Mircea Musta\c{t}\u{a}, Christian Schnell and Claude Sabbah for helpful discussions. We would especially like to thank Bradley Dirks, both for introducing us to Sabbah's multi-filtration in the first place and for his careful reading and many valuable comments on an earlier version of this manuscript.

\section{The multivariate $V$-filtration: definition and basic properties}\label{sec: multivariate Vfiltration}

In this section, we introduce the multivariate $V$-filtration of a coherent $\ms{D}$-module on a complex manifold $X$ along a divisor $D$ with simple normal crossing, and establish some of the fundamental properties of this notion. In \S\ref{subsec:multivariate V definition} we give the definition and prove that a multivariate $V$-filtration is uniquely determined if it exists (Theorem \ref{thm:multivariate V uniqueness}) and deduce that morphisms are strict with respect to it (Corollary \ref{cor:V strictness}). We briefly recall Sabbah's approach in \S\ref{subsec:sabbah comparison} and prove that our notion is equivalent to his (Theorem \ref{thm:sabbah comparison}). In \S\ref{subsec:direct image}, we prove that the multivariate $V$-filtrations are compatible with the formation of proper direct images (Theorem \ref{thm:multivariate V direct image}).

In \S\ref{subsec:extension}, we undertake a more detailed study of the multivariate $V$-filtration on $\ms{D}_X$-modules localised away from $D$. We introduce here a crucial technical tool, which we use throughout the paper, the $V_*$-filtration. This is defined exactly as the usual multivariate $V$-filtration, but with respect to the sheaf of rings $\ms{D}_X(*D)$ instead of $\ms{D}_X$. While we show (Corollaries \ref{cor:localising V} and \ref{cor:extending V*}) that the $V$-filtration and $V_*$-filtration directly determine one another, the $V_*$-filtration is often technically simpler to work with.

Finally, in \S\ref{subsec:malgrange-mellin}, we introduce the Malgrange-Mellin transform and record some elementary properties of the multivariate $V$-filtration on a graph embedding.

\subsection{Good wall and chamber filtrations} \label{subsec:wall and chamber}

In this subsection, we define \emph{good wall and chamber filtrations}, of which the multivariate $V$-filtration will be an example, and record some basic facts about them.

To motivate the definition, we first recall the definition of the usual $V$-filtration of a holonomic $\ms{D}$-module along a smooth divisor. Let $X$ be a complex manifold, $D \subset X$ a smooth divisor and $\mc{M}$ a coherent left $\ms{D}_X$-module. We fix a local equation $t = 0$ for $D$ and a vector field $\partial_t$ such that $[\partial_t, t] = 1$. (The reader may assume that $t$ is part of a local coordinate system on $X$, and that $\partial_t$ is corresponding partial derivative.) The \emph{$V$-filtration of $\ms{D}_X$} is the decreasing $\mb{Z}$-indexed filtration given by
\[ V^n \ms{D}_X = \{ P \in \ms{D}_X \mid P t^m \in (t^{m + n}) \text{ for all }m \geq 0\}.\]

\begin{defn} \label{defn:V-filtration}
A \emph{$V$-filtration on $\mc{M}$ along $D$} is an $\mb{R}$-indexed filtration $\{V^\alpha\mc{M}\}_{\alpha \in \mb{R}}$ such that:
\begin{enumerate}
\item \label{itm:V-filtration 1}$V^\bullet\mc{M}$ is decreasing (i.e.\ $V^{\alpha}\mc{M} \supset V^\beta\mc{M}$ if $\alpha \leq \beta$), left-continuous (i.e.\ $V^{\alpha - \epsilon}\mc{M} = V^\alpha\mc{M}$ for all $\alpha$) and, locally on $X$, there exists a finite set $A \subset \mb{R}$ such that $V^{\alpha + \epsilon}\mc{M} = V^\alpha\mc{M}$ for $\alpha\not\in A + \mb{Z}$.
\item \label{itm:V-filtration 2} The filtration $V^\bullet\mc{M}$ is \emph{good} over $V^\bullet\ms{D}_X$, i.e.\ it is exhaustive, satisfies $V^n \ms{D}_X \cdot V^\alpha\mc{M} \subset V^{\alpha + n}\mc{M}$ for all $n \in \mb{Z}$, $\alpha \in \mb{R}$, each $V^\alpha\mc{M}$ is a coherent $V^0\ms{D}_X$-module, and there exists, locally on $X$, a finite set of indices $\{\beta_i \in \mb{R}\}_{i \in I}$ such that
\[ V^\alpha\mc{M} = \sum_{\substack{n \in \mb{Z}, i \in I \\ n + \beta_i \geq \alpha}} V^n\sD_X\cdot V^{\beta_i}\mc{M}, \quad \textrm{for all $\alpha \in \mb{R}$.}\]

\item \label{itm:V-filtration 3} For each $\alpha \in \mb{R}$, there exist, locally on $X$, $\gamma_1, \ldots, \gamma_k \in \mb{C}$ with $\Re \gamma_i = \alpha$ such that the operator $\prod_{i  = 1}^k (\gamma_i - \partial_t t)$ acts by zero on $\gr_V^\alpha\mc{M}$. 
\end{enumerate}
\end{defn}

Intuitively, $V^\alpha \mc{M}$ can be thought of roughly as the set of sections vanishing to order at least $\alpha - 1$ along $D$. By \cite{Kas83}, the $V$-filtration is unique if it exists, which it always does when $\mc{M}$ is holonomic.

Our aim in the next subsection is to generalise Definition \ref{defn:V-filtration} to the case where $D$ is a divisor with simple normal crossings. In this setting, we will be interested in filtrations indexed by $\mb{R}^r$, where $r$ is the number of components of $D$. In this subsection, we will consider how to generalise conditions \eqref{itm:V-filtration 1} and \eqref{itm:V-filtration 2} in the general setting of a sheaf of modules over a (sufficiently nice) $\mb{Z}^r$-filtered sheaf of rings. In what follows, we use the notation and conventions defined in \S\ref{subsec:conventions} for writing vectors in $\mb{R}^r$ and for the partial order thereon.

We first generalise \eqref{itm:V-filtration 1}.
\begin{defn}\label{defn: set of walls}
Let $r \geq 0$. A \emph{set of walls in $\mb{R}^r$} is a set $\mc{W}$ of affine linear hyperplanes in $\mb{R}^r$, finite modulo translation by $\mb{Z}^r$, such that each hyperplane $H \in \mc{W}$, called a \emph{wall}, is of the form
\[ H = \{\bs{\alpha} \in \mb{R}^r \mid L_1 \alpha_1 + \cdots + L_r \alpha_r = \beta \}\]
for some (non-zero) $\bs{L} \in \mb{Q}_{\geq 0}^r$ and some $\beta \in \mb{R}$. We will often regard $\bs{L}$ as a linear map $\mb{R}^r \to \mb{R}$ and write $H = \bs{L}^{-1}(\beta)$. Whenever we write a wall in this way, we will implicitly assume that $\bs{L} \in \mb{Q}^r_{\geq 0}$.
\end{defn}
Note that the set $\mc{W}$ is allowed to consist of infinitely many walls (and usually will). We allow the case $r = 0$ (in which case the only possible set of walls is $\mc{W} = \emptyset$) as it will sometimes be convenient as an induction base.

\begin{defn}\label{defn: chamber}
Let $\mc{W}$ be a set of walls in $\mb{R}^r$. A \emph{chamber of $\mc{W}$} is a non-empty subset $\sigma \subset \mb{R}^r$ of the form
\[ \sigma = \{\bs{\alpha} \in \mb{R}^r \mid \bs{L}(\bs{\alpha}) > \beta \text{ if }\bs{L}^{-1}(\beta) \in \mc{W}_+, \,\bs{L}(\bs{\alpha}) \leq \beta \text{ if }\bs{L}^{-1}(\beta) \in \mc{W}_-\}\]
for some partition $\mc{W} = \mc{W}_+ \cup \mc{W}_-$.
\end{defn}

When $r = 1$, a set of walls is simply a subset $\mc{W}$ of $\mb{R}$ with finite image modulo $\mb{Z}$, and a chamber is an open-closed interval $(\alpha, \beta]$ wth $\alpha, \beta$ consecutive elements in $\mc{W}$. The shaded region in the figure below depicts an example of a chamber for a set of walls with $r = 2$.

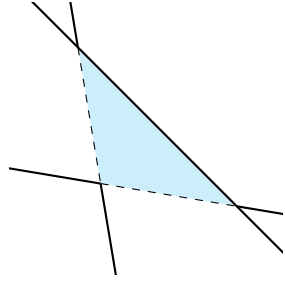
\begin{center}
\begin{tikzpicture}[scale=1.2]

\clip (-1, -1) rectangle (2, 2); 

\coordinate (A) at (-0.25, 1.5);
\coordinate (B) at (0, 0);
\coordinate (C) at (1.5, -0.25);


\coordinate (P1) at (-2, 3.25);
\coordinate (P2) at (-3, 0.5);
\coordinate (P3) at (0.5, -3);
\coordinate (P4) at (3.25, -2);
\coordinate (P5) at (3, -0.5);
\coordinate (P6) at (-0.5, 3);

\fill[color=Cyan, opacity=0.2] (A) -- (B) -- (C) -- cycle;

\draw[thick] (P1) -- (P4);

\draw[thick] (P2) -- (B);
\draw[dashed] (B) -- (C);
\draw[thick] (C) -- (P5);

\draw[thick] (P6) -- (A);
\draw[dashed] (A) -- (B);
\draw[thick] (B) -- (P3);
\end{tikzpicture}
\end{center}

In general, it is clear that the chambers of $\mc{W}$ form a partition of $\mb{R}^r$ and that the interiors of the chambers of $\mc{W}$ are precisely the connected components of $\mb{R}^r \setminus \bigcup_{H \in \mc{W}} H$.

\begin{definition}\label{defn: separated by a wall}
If $\bs{\alpha}, \bs{\beta} \in \mb{R}^r$, we say that a wall $H = \bs{L}^{-1}(\gamma) \in \mc{W}$ \emph{separates $\bs{\alpha}$ and $\bs{\beta}$} if $\bs{L}(\bs{\alpha}) \leq \gamma$ and $\bs{L}(\bs{\beta}) > \gamma$, or vice versa.
\end{definition} Note that there are only finitely many walls separating any two points and that $\bs{\alpha}$ and $\bs{\beta}$ lie in the same chamber if and only if they have no separating walls.

\begin{defn}\label{defn: wall and chamber}
Let $X$ be a topological space and $\mc{F}$ a sheaf on $X$. For $r \in \mb{Z}_{\geq 0}$, a \emph{decreasing $\mb{R}^r$-indexed filtration of $\mc{F}$} is a collection $\{U^{\bs{\alpha}}\mc{F} \subset \mc{F}\}_{\bs{\alpha} \in \mb{R}^r}$ of subsheaves such that $U^{\bs{\alpha}} \mc{F} \subset U^{\bs{\beta}} \mc{F}$ if $\bs{\alpha} \geq \bs{\beta}$. We say that $U^{\bullet} \mc{F}$ is a \emph{wall and chamber} filtration if, locally on $X$, there exists a set of walls $\mc{W}$ in $\mb{R}^r$ such that if $\bs{\alpha}$ and $\bs{\beta}$ lie in the same chamber of $\mc{W}$, then $U^{\bs{\alpha}} \mc{F} = U^{\bs{\beta}} \mc{F}$. In this setting, we will call such a set $\mc{W}$ a \emph{set of walls} of the filtration $U^{\bullet}\mc{F}$. If $\sigma$ is a chamber of $\mc{W}$, we write
 $U^\sigma \mc{F} := U^{\bs{\alpha}} \mc{F}$  for some (hence any) $\bs{\alpha}\in \sigma$.
\end{defn}

When $r = 1$, a wall and chamber filtration is nothing but an exhaustive, decreasing, left-continuous filtration with a finite set of jumps modulo $\mb{Z}$.

\begin{remark}\label{remark: containment of closure}
We record two basic properties of a wall and chamber filtration $U^\bullet\mc{F}$.
\begin{enumerate}

\item For any chamber $\sigma$, we have $U^{\sigma}\mc{F} \subset U^{\bs{\alpha}'}\mc{F}$ for any $\bs{\alpha}'\in \overline{\sigma}$. 
Indeed, one can find a sufficiently small $\epsilon > 0$, and a point $\bs{\alpha} \in \sigma$, so that $\bs{\alpha} \geq \bs{\alpha}' - \epsilon \bs{1}$ and $\bs{\alpha}',\bs{\alpha}'-\epsilon\bs{1}$ lie in the same chamber. So $U^{\sigma}\mc{F} = U^{\bs{\alpha}}\mc{F} \subset U^{\bs{\alpha}' - \epsilon \bs{1}}\mc{F} = U^{\bs{\alpha}'}\mc{F}$.

\item If $\bs{\alpha},\bs{\beta}$ are separated by a single wall $\bs{L}^{-1}(\gamma)$ and $\bs{L}(\bs{\alpha}) \leq \gamma < \bs{L}(\beta)$, then $U^{\bs{\beta}}\mc{F}\subset U^{\bs{\alpha}}\mc{F}$. Indeed, by the uniqueness of the separating wall, we must have  
\[ \sigma_{\bs{\alpha}}\cap \overline{\sigma}_{\bs{\beta}} = \bs{L}^{-1}(\gamma)\cap \overline{\sigma}_{\bs{\beta}}\neq \varnothing,\]
where $\sigma_{\bs{\alpha}},\sigma_{\bs{\beta}}$ are the chambers containing $\bs{\alpha},\bs{\beta}$, respectively. Choose $\bs{\alpha}'\in \sigma_{\bs{\alpha}}\cap \overline{\sigma}_{\bs{\beta}}$, by the property above, we have
$U^{\bs{\alpha}}\mc{F}=U^{\bs{\alpha}'}\mc{F}\supseteq U^{\bs{\beta}}\mc{F}$.
\end{enumerate}
\end{remark}

We next generalise Definition \ref{defn:V-filtration}, \eqref{itm:V-filtration 2}. To do so, suppose that $X$ is a complex manifold and that $\mc{A}$ is a sheaf of (not necessarily commutative) $\mb{C}$-algebras on $X$ equipped with a decreasing $\mb{Z}^r$-indexed filtration $U^\bullet \mc{A}$. In order to have a well-behaved theory of coherent sheaves of $\mc{A}$-modules with good filtrations, we will make the following assumptions on $(\mc{A}, U^\bullet\mc{A})$.

\begin{assumption} \label{assumption:ring filtrations}
We assume that $\mc{A}$ is equipped with an additional $\Z$-indexed increasing filtration $F_\bullet \mc{A}$ and that $F_\bullet \mc{A}$ and $U^\bullet\mc{A}$ satisfy the following conditions.
\begin{enumerate}
\item $1 \in F_0\mc{A} \cap U^{\bs{0}}\mc{A}$.
\item $F_p\mc{A} \cdot F_q\mc{A} \subset F_{p + q}\mc{A}$ for $p,q\in \Z$ and $U^{\bs{m}}\mc{A} \cdot U^{\bs{n}}\mc{A} \subset U^{\bs{m} + \bs{n}}\mc{A}$ for $\bs{m},\bs{n}\in \Z^r$.
\item We have $F_{-1}\mc{A} = 0$ and $U^{\bs{0}}\mc{A} \cap F_0\mc{A} \cong \mc{O}_X$ as sheaves of $\mb{C}$-algebras.
\item Each $U^{\bs{n}}\mc{A} \cap F_p\mc{A}$ is a coherent $\mc{O}_X$-module.
\item The $\mb{Z}^{r + 1}$-indexed filtration $ U^\bullet\mc{A} \cap F_\bullet\mc{A}$ is exhaustive and the Rees algebra
\[ R_{U, F} \mc{A} = \bigoplus_{\bs{n} \in \mb{Z}^r, p \in \mb{Z}} (U^{\bs{n}}\mc{A}\cap F_p\mc{A})\,  u_1^{-n_1} \cdots u_r^{-n_r} v^p \]
is flat over $\mb{C}[\bs{u}, v] := \mb{C}[u_1, \ldots, u_r, v]$.
\item The algebras $\Gr^F\mc{A}$ and $\Gr_U\Gr^F\mc{A}$ are commutative and (locally) finitely presented over $\mc{O}_X$. Here $\Gr_U\Gr^F\mc{A}$ is the quotient the Rees algebra $R_{U,F}\mc{A}$ by $u_1, \ldots, u_r,v$. By flatness, this coincides with taking iterated $\Gr$ in each $U^\bullet$ or $F_{\bullet}$ variable separately, in any order.
\end{enumerate}
\end{assumption}

\begin{remark}\label{remark: coherence of Un over U0}
One can deduce from the finiteness conditions (4) and (6) above that $U^{\bs{n}}\cA$ is coherent over $U^{\bs{0}}\cA$, for each $\bs{n}\in \Z^r$.
\end{remark}
In preparation for the wall and chamber case, let us recall the notion of good $\mb{Z}^r$-indexed filtrations over such a ring.

\begin{defn}\label{defn: good Zr indexed filtration}
Under the assumptions above, let $\mc{M}$ be a coherent sheaf of $\mc{A}$-module.
A \emph{good filtration over $U^\bullet\mc{A}$} is a decreasing $\mb{Z}^r$-indexed filtration $U^\bullet\mc{M}$ such that
\begin{enumerate}
\item $U^\bullet\mc{M}$ is exhaustive, i.e.\ $\mc{M} = \bigcup_{\bs{n} \in \mb{Z}^r} U^{\bs{n}}\mc{M}$, and
\item the Rees module $R_U\mc{M} = \bigoplus_{\bs{n} \in \mb{Z}^r} U^{\bs{n}}\mc{M}\, u_1^{-n_1} \cdots u_r^{-n_r}$ is coherent over $R_U\mc{A}$.
\end{enumerate}
We define good $\mb{Z}^{r + 1}$-indexed filtrations over $U^\bullet \mc{A} \cap F_\bullet\mc{A}$ similarly.
\end{defn}

Note that the Rees module construction defines an equivalence of categories between coherent $\mc{A}$-modules with good filtration over $U^\bullet\mc{A}$ (resp.\ $U^\bullet\mc{A} \cap F_\bullet\mc{A}$) and coherent $\mb{Z}^r$-graded modules over $R_U\mc{A}$ (resp.\ $\mb{Z}^{r + 1}$-graded modules over $R_{U, F}\mc{A}$) without $\mb{C}[\bs{u}]$-torsion (resp.\ $\mb{C}[\bs{u}, v]$-torsion). By classical theorems of Cartan, Frisch and Oka about the sheaf of holomorphic functions $\mc{O}_X$ (cf.\ e.g.\ \cite[\S 8.8]{MHMproject}), Assumption \ref{assumption:ring filtrations} implies that $\mc{A}$ is a coherent sheaf of rings, any coherent $\mc{A}$-module admits good filtrations (of both kinds) locally on $X$, that the categories of coherent $\mc{A}$-modules, graded $R_U\mc{A}$-modules and graded $R_{U, F}\mc{A}$-modules are abelian, and that the following Artin-Rees lemma holds. 

\begin{lem}[Artin-Rees] \label{lem:artin-rees}
Let $(\mc{A}, U^\bullet \mc{A},F_{\bullet}\mc{A})$ be as in Assumption \ref{assumption:ring filtrations}. Let $\mc{M}$ be a coherent $\mc{A}$-module equipped with a good filtration $U^\bullet \mc{M}$ over $U^\bullet\mc{A}$ (resp.\ a good filtration $U^\bullet F_\bullet\mc{M}$ over $U^\bullet\mc{A} \cap F_\bullet \mc{A}$). Then for any coherent submodule $\mc{N} \subset \mc{M}$, the filtration
\[ U^\bullet\mc{N} := \mc{N} \cap U^\bullet\mc{M} \quad \text{(resp.} \;\; U^{\bullet}F_\bullet\mc{N} := \mc{N} \cap U^\bullet F_\bullet\mc{M} \; \text{)}\]
is also good.
\end{lem}
We first prove a few standard results, for the lack of suitable references.

\begin{lemma}\label{lemma: Theorem A and B}
Let $\cB$ be a sheaf of $\cO_X$-algebras such that
\begin{itemize}
\item $\cB$ admits an exhaustive $\Z^{k}$-indexed filtration $G_{\bullet}\cB$ such that $G_{\bs{m}}\cB\cdot G_{\bs{n}}\cB\subset G_{\bs{m}+\bs{n}}\cB$, and $G_{\bs{m}}\cB$ is a coherent $\cO_X$-module for each $\bs{m}\in \Z^k$,
\item $\cB(K)$ is (left) Noetherian for any compact polycylinder $K$.
\end{itemize}
Then the following holds.
\begin{enumerate}
\item $\cB$ is coherent.
\item Let $\cM$ be a coherent $\cB$-module and let $K$ be a compact polycylinder contained in an open subset $U$ of $X$ such that $\cM$ has a good filtration on $U$ (with respect to $G_{\bullet}\cB$), then $\cM(K)$ generates $\cM|_K$ as an $\cO_K$-module, and $H^i(K,\cM)=0$ for every $i\geq 1$.
\item  A $\cB$-module $\cM$ is coherent if and only if there exists a covering $\{K_{\alpha}\}$ by compact polycylinders $K_{\alpha}$ such that $X$ is the union of the interiors of the $K_{\alpha}$ and that for each $\alpha$, $\cM(K_\alpha)$ is a $\cB(K_\alpha)$-module of finite type, and the natural map
\begin{equation}\label{eqn: eval MK} \cO_{X,x}\otimes_{\cO_{X}(K_\alpha)}\cM(K_\alpha)\to \cM_x\end{equation}
is an isomorphism, for every $x\in K_\alpha$.
\end{enumerate}
\end{lemma}
\begin{proof}
For (1), we adapt the proof of \cite[Proposition 2.1.3]{SabbahnoteDmodule}. Suppose $U\subset X$ is open and $\phi: \cB^q|_U \to \cB^p|_U$ is a morphism of left $\cB|_U$-modules, we have to prove that $\Ker \phi$ is locally of finite type. We may assume that $U$ is an open subset of $\C^n$. Let $\epsilon_1,\ldots,\epsilon_q$ be the canonical base of $\cB(U)^q$ and choose $\bs{n}\in \Z^k$ such that 
\[ \phi(\epsilon_i) \in G_{\bs{n}}\cB(U)^p, \quad \textrm{for all $1\leq i\leq q$}.\]
Then for each $\bs{m}\in \Z^k$, we have $\phi(G_{\bs{m}}\cB^q|_U)\subset G_{\bs{m}+\bs{n}}\cB^p|_U$ and $\Ker \phi\cap G_{\bs{m}}\cB^q|_U$ is the kernel of a morphism between coherent $\cO_U$-modules
\[ G_{\bs{m}}\cB^q|_U\to  G_{\bs{m}+\bs{n}}\cB^p|_U.\]
It follows that $\Ker \phi\cap G_{\bs{m}}\cB^q|_U$ is $\cO_U$-coherent .

Let $K\subset U$ be a compact polycylinder and let $x\in K$. By Theorem A of Cartan-Oka, the sheaf $\left[\Ker \phi \cap G_{\bs{m}}\cB^q|_U\right]_x$ is generated by $\Gamma(K,\Ker \phi\cap G_{\bs{m}}\cB^q|_U)\subset \Gamma(K,\Ker \phi)$. It follows that $(\Ker \phi)_x$ is generated by $\Gamma(K,\Ker \phi)$ over $\cO_{X,x}$. On the other hand, note that we have an exact sequence of left $\cB(K)$-modules:
\[ 0\to \Gamma(K,\Ker \phi) \to \cB(K)^q \xrightarrow{\Gamma(K,\phi)} \cB(K)^p.\]
Since $\cB(K)$ is left Noetherian, $\Gamma(K,\Ker \phi)$ is of finite type as a left $\cB(K)$-module. Using these two facts, one can build a surjective morphism of left $\cB|_K$-modules
\[ \cB^r|_K \to (\Ker \phi)|_K \to 0,\]
which proves that $\Ker \phi$ is locally of finite type. In other words, $\cB$ is coherent.

Arguing as in \cite[Exercise 2.2.5]{SabbahnoteDmodule}, one can show that $\cM$ admits locally a good filtration. Then the statement (2) immediately follows from the Theorems A and B of Cartain-Oka for $\cO_X$-modules using a good filtration on $\cM$ locally. 

For (3), we adapt the proof of \cite[Theorem 2.2.9]{SabbahnoteDmodule}. We first prove the $``\Rightarrow"$ direction. Let $U\subset X$ be an open subset so that there is a presentation
\[ 0\to \cN \to \cB^{p}|_U\xrightarrow{\phi} \cM|_U\to 0.\]
Since $\cB$ is coherent by (1),  $\cN$ is a coherent $\cB$-module. Then for a compact polycylinder $K\subset U$ such that $\mc{N}$ admits a good filtration on some neighbourhood of $K$, one has $H^1(K,\cN)=0$ by (2). So there is a surjection $\cB(K)^p \twoheadrightarrow \cM(K)$, i.e. $\cM(K)$ is a $\cB(K)$-module of finite type. Let $\bs{m}\in \Z^k$, since $G_{\bs{m}}\cB|_U$ is coherent over $\cO_X$, the image sheaf $G_{\bs{m}}\cM|_U\colonequals \phi \left((G_{\bs{m}}\cB)^p|_U\right)$ is also coherent. Then Cartan-Oka's Theorem A  implies that the natural map
\[ \cO_{X,x}\otimes_{\cO_X(K)}(G_{\bs{m}}\cM)(K)\to (G_{\bs{m}}\cM)_x \]
is an isomorphism for $x\in K$. Then we can prove \eqref{eqn: eval MK} is an ismorphism by using an inductive limit
\[ \cO_{X,x}\otimes_{\cO_X(K)}\cM(K)\cong \varinjlim_{\bs{m}\in \Z^k} \cO_{X,x}\otimes_{\cO_X(K)}(G_{\bs{m}}\cM)(K) \cong \varinjlim_{\bs{m} \in \mb{Z}^k} (G_{\bs{m}}\mc{M})_x \cong \mc{M}_x.\]

$``\Leftarrow"$: For each $\alpha$, since $\cM(K_\alpha)$ is an $\cB(K_\alpha)$-module of finite type, and $\cB(K_\alpha)$ is left Noetherian, there is a finite presentation
\[\cB^q(K_\alpha) \xrightarrow{\phi} \cB^p(K_\alpha) \xrightarrow{\pi} \cM(K_\alpha)\to 0.\]
This induces sheaf morphisms  $\cB^q|_{K_\alpha} \xrightarrow{\phi} \cB^p|_{K_\alpha} \xrightarrow{\pi} \cM|_{K_\alpha} \to 0$ and by the right exactness of $\cB_x\otimes_{\cB(K_\alpha)}(-)$, one also has
\[ \cB^q_x \xrightarrow{\phi_x} \cB^p_x \xrightarrow{p_x} \cB_x\otimes_{\cB(K_\alpha)}\cM(K_\alpha)\to 0.\]

Note that the canonical homomorphism $\cO_{X,x}\otimes_{\cO(K_\alpha)}\cM(K_\alpha)\to \cB_x\otimes_{\cB(K_\alpha)}\cM(K_\alpha)$ is an isomorphism. This claim is clear if $\cM(K)$ is free over $\cB(K)$ (since $\mc{B}_x = \mc{O}_{X, x} \otimes_{\mc{O}_X(K_\alpha)} \mc{M}(K_\alpha)$ as argued above) and the general case follows from the right exactness of the functors $\cO_{X,x}\otimes_{\cO(K_\alpha)}(-)$ and $\cB_x\otimes_{\cB(K_\alpha)}(-)$.  Consequently, the morphism $c_x:\cB_x\otimes_{\cB(K_\alpha)}\cM(K_\alpha) \to \cM_x$ is an isomorphism, and one has an exact sequence
\[ \cB^q_x \xrightarrow{\phi_x} \cB^p_x \xrightarrow{\pi_x} \cM_x\to 0.\]
We conclude that $\cM|_{K_\alpha}$ is finitely presented over $K_\alpha$. Since the interiors of the $K_\alpha$ cover $X$, $\cM$ is therefore coherent over $\cB$.
\end{proof}

\begin{lemma}\label{lemma: coherence of A and rees}
For any sufficiently small compact polycylinder $K$, the rings $\Gr_U\Gr^F\cA(K)$, $\cA(K),R_{U}\cA(K)$ and $R_{U,F}\cA(K)$ are left and right Noetherian. In particular, the sheaves of rings $\Gr_U\Gr^F\cA$, $\cA,R_{U}\cA$ and $R_{U,F}\cA$ are coherent by Lemma \ref{lemma: Theorem A and B}. 
\end{lemma}
\begin{proof}
Let $K\subset X$ be a small compact polycylinder on which $\Gr^F\mc{A}$ and $\Gr_U\Gr^F\mc{A}$ are finitely presented sheaves of algebras over $\mc{O}_X$. By a theorem of Frisch \cite{Frisch}, $\cO_X(K)$ is Noetherian. Since $\cO_X$ is coherent by the theorems of Cartan and Oka, it follows that $\Gr^F\cA, \Gr_U\Gr^F\cA(K)$ are all Noetherian. 

Now we argue as in \cite[Proposition 2.1.2]{SabbahnoteDmodule} to show $\cA(K)$ is left and right Noetherian. Let $I\subset \cA(K)$ be a left ideal and set $F_{\bullet}I=I\cap F_{\bullet}\cA(K)$. Then $\Gr^FI$ is an ideal in $\Gr^F\cA(K)$, thus is finitely generated, because $\Gr^F\cA(K)$ is Noetherian. Let $e_1,\ldots,e_{\ell}$ be homogeneous generators of $\Gr^FI$ and let $P_1,\ldots,P_{\ell}$ be the lifts of them in $I$ so that $P_i \in F_{p_i}\cA(K)$. By induction on the order, one can show that $I=\sum_i \cA(K)\cdot P_i$. Hence $\cA(K)$ is left Noetherian. Similar argument shows that $\cA(K)$ is right Noetherian.

Finally, let us deal with the case $R_{U,F}\cA(K)$. Consider the $\Z$-index filtration
\[ G_k\left(R_{U,F}\cA\right)=\bigoplus_{\bs{n}\in \Z^r}\left(\bigoplus_{j\leq k}(U_{\bs{n}}\cA\cap F_j\cA) v^j \oplus\bigoplus_{j>k} (U_{\bs{n}}\cA\cap F_k\cA) v^j\right) \bs{u^{-n}}.\] 
It is direct to check that
\[ \Gr^GR_{U,F}\cA\cong \bigoplus_{\bs{n}}\left(\Gr^FU_{\bs{n}}\cA\otimes \C[v]\right)\bs{u^{-n}}\cong \Gr^F(R_U\cA)\otimes \C[v].\]
Iterating this procedure, we obtain a sequence of $\cO_X$-algebras $\cB_i$ for $0 \leq i \leq r+1$ where
\[ \cB_{-1}=R_{U,F}\cA, \quad \cB_0=\Gr^F\cB_0, \quad \cB_i=\Gr_{U_i}\cB_{i-1}, \quad i\geq 1,\]
where $\Gr_{U_i}$ is the associated graded of the filtration induced by the $i$-th coordinate of $U_{\bullet}$. Furthermore by the flatness of $R_{U,F}\cA$ over $\C[\bs{u},v]$,  one has
\[ \cB_r=\Gr_U\Gr^F\cA\otimes \C[\bs{u},v].\]
By above, we know $\cB_r(K)$, being polynomial ring over $\Gr_U\Gr^F\cA(K)$, is both left and right Noetherian. So by the argument above, one can show by descending induction on $i$ that $\cB_i(K)$ is left and right Noetherian. Hence we conclude that $R_{U,F}\cA(K)$ is left and right Noetherian. Similarly, we can also show $R_{U}\cA(K)$ is left and right Noetherian.
\end{proof}

\begin{proof}[Proof of Lemma \ref{lem:artin-rees}]
Any good filtration $U^\bullet\mc{M}$ over $U^\bullet\mc{A}$ can be refined locally to a good filtration $U^\bullet F_\bullet \mc{M}$ over $U^\bullet\mc{A} \cap F_\bullet \mc{A}$, so it suffices to treat the latter case. 

By Lemma \ref{lemma: coherence of A and rees}, $R_{U,F}\cA(K)$ is left Noetherian for any compact polycylinder $K$. Furthermore, from its proof we know one can put a $\Z^k$-indexed filtration on $R_{U,F}\cA$ so that it satisfies  the assumption of Lemma \ref{lemma: Theorem A and B}. By assumption, $R_{U,F}\cM$ is coherent over $R_{U,F}\cA$, so by Lemma \ref{lemma: Theorem A and B}(3) we can cover $X$ by the interiors of compact polycylinders $K$ such that $(R_{U,F}\cM)(K)$ is an $(R_{U,F}\cA)(K)$-module of finite type and  for every $x\in K$, 
\[ \cO_{X,x}\otimes_{\cO_{X}(K)}(R_{U,F}\cM)(K)\to (R_{U,F}\cM)_{x}\]
is an isomorphism. In particular, there is an isomorphism
\[ \cO_{X,x}\otimes_{\cO_{X}(K)}(U^{\bs{n}}F_p\cM)(K)\to (U^{\bs{n}}F_p\cM)_{x}, \quad \forall \bs{n}\in \Z^r, p\in \Z.\]
Shrinking $K$ if necessary, we can assume that $(R_{U,F}\cA)(K)$ is left Noetherian by Lemma \ref{lemma: coherence of A and rees}, so the submodule $(R_{U,F}\cN)(K)$ of $(R_{U,F}\cM)(K)$ is also an $(R_{U,F}\cA)(K)$-module of finite type. Since $\cO_{X,x}$ is flat over $\cO_X(K)$ by a Theorem of Frisch \cite{Frisch}, it follows that there is an isomorphism
 \[ \cO_{X,x}\otimes_{\cO_{X}(K)}\left(U^{\bs{n}}F_p\cM\cap \cN\right)(K)\to \left(U^{\bs{n}}F_p\cM\cap \cN\right)_{x}. \]
Hence by Lemma \ref{lemma: Theorem A and B}(3) again, we conclude that the filtration $U^{\bullet}F_{\bullet}\cN$ is good.\end{proof}

We now consider the analogue of the ``good'' property for wall and chamber filtrations.

\begin{defn}\label{defn: good filtration}
Let $(\mc{A}, U^\bullet \mc{A})$ be as in Assumption \ref{assumption:ring filtrations} and let $\mc{M}$ be a coherent $\mc{A}$-module. We say that an $\mb{R}^r$-indexed wall and chamber filtration $U^\bullet\mc{M}$ is \emph{good} (over $U^\bullet\mc{A}$) if:
\begin{enumerate}
\item \label{itm:good 1}$U^\bullet\mc{M}$ is exhaustive, i.e.\ $\mc{M} = \bigcup_{\bs{\alpha} \in \mb{R}^r}U^{\bs{\alpha}}\mc{M}$,
\item \label{itm:good 2}$U^{\bs{n}}\mc{A} \cdot U^{\bs{\alpha}} \mc{M} \subset U^{\bs{n} + \bs{\alpha}}\mc{M}$ for all $\bs{n} \in \mb{Z}^r, \bs{\alpha} \in \mb{R}^r$,
\item \label{itm:good 3}$U^{\bs{\alpha}}\mc{M}$ is a coherent $U^{\bs{0}}\mc{A}$-module for all $\bs{\alpha} \in \mb{R}^r$, and
\item \label{itm:good 4}there exists, locally on $X$, a bounded interval $[A, B] \subset \mb{R}$ such that $U^\bullet\mc{M}$ is generated by its restriction to $[A, B]^r$, i.e.\ 
\[ U^{\bs{\alpha}}\mc{M} = \sum_{\substack{\bs{n} \in \mb{Z}^r, \bs{\gamma} \in [A, B]^r \\ \bs{n} + \bs{\gamma} \geq \bs{\alpha}}} U^{\bs{n}}\mc{A}\cdot U^{\bs{\gamma}}\mc{M}, \quad \textrm{for all $\bs{\alpha} \in \mb{R}^r$}.\]
\end{enumerate}
\end{defn}

The following equivalent formulation is often helpful, as it allows us to take advantage of the well-behaved theory of good $\mb{Z}^r$-filtrations.

\begin{lem}\label{lem: goodness part}
An $\mb{R}^r$-indexed wall and chamber filtration $U^{\bullet}\cM$ is good if and only if $\{U^{\bs{\alpha} + \bs{n}}\mc{M}\}_{\bs{n} \in \mb{Z}^r}$ is a good $\mb{Z}^r$-filtration over $U^\bullet\mc{A}$ for all $\bs{\alpha} \in \mb{R}^r$.
\end{lem}
\begin{proof}
Suppose first that $U^\bullet\mc{M}$ is good and fix $\bs{\alpha} \in \mb{R}^r$; we want to show that $\{U^{\bs{\alpha} + \bs{n}}\mc{M}\}_{\bs{n} \in \mb{Z}^r}$ is also good.
Since the statement is local on $X$, we may choose an interval $[A, B]$ as in \eqref{itm:good 4}. For each $i = 1, \ldots, r$, choose an interval $[\alpha_i + A_i, \alpha_i + B_i] \supset [A, B]$ with $A_i, B_i \in \mb{Z}$. Set $S = \prod_{i = 1}^r [\alpha_i + A_i, \alpha_i + B_i]$. Then by \eqref{itm:good 4}, we have
\begin{align*}
U^{\bs{\alpha} + \bs{k}}\mc{M} &= \sum_{\substack{\bs{n} \in \mb{Z}^r, \bs{\gamma} \in S \\ \bs{n} + \bs{\gamma} \geq \bs{\alpha} + \bs{k}}} U^{\bs{n}}\mc{A} \cdot U^{\bs{\alpha} + \bs{k}}\mc{M} \\
&= \sum_{\substack{\bs{n} \in \mb{Z}^r \\ S_{\geq \bs{\alpha} + \bs{k} - \bs{n}}\neq \emptyset}} U^{\bs{n}}\mc{A}\cdot U^{\min(S_{\geq \bs{\alpha} + \bs{k} - \bs{n}})}\mc{M} \\
&= \sum_{\substack{\bs{n} \in \mb{Z}^r \\ \bs{\alpha} + \bs{k} - \bs{n} \in S}} U^{\bs{n}}\mc{A} \cdot U^{\bs{\alpha} + \bs{k} - \bs{n}}\mc{M},
\end{align*}
where $S_{\geq \bs{\beta}} := \{\bs{\gamma} \in S \mid \bs{\gamma} \geq \bs{\beta}\}$, and the third equality holds because $\min(S_{\geq \bs{\alpha} + \bs{k} - \bs{n}}) \in \bs{\alpha} + \mb{Z}^r$. Since $S \cap (\bs{\alpha} + \mb{Z}^r)$ is finite and each $U^{\bs{\beta}}\mc{M}$ is a coherent $U^{\bs{0}}\mc{A}$-module, it follows that $\{U^{\bs{\alpha} + \bs{n}}\mc{M}\}_{\bs{n} \in \mb{Z}^r}$ is good as claimed.

Conversely, suppose that $\{U^{\bs{\alpha} + \bs{n}}\mc{M}\}_{\bs{n} \in \mb{Z}^r}$ is good for all $\bs{\alpha} \in \mb{R}^r$. Then clearly \eqref{itm:good 1} and \eqref{itm:good 2} are satisfied. Moreover, our assumptions on $\mc{A}$ imply that each $U^{\bs{n}}\mc{A}$ is coherent over $U^{\bs{0}}\mc{A}$ (see Remark \ref{remark: coherence of Un over U0}), so this also implies that $U^{\bs{\alpha}}\mc{M}$ is coherent over $U^{\bs{0}}\mc{A}$ for every $\bs{\alpha} \in \mb{R}^r$, i.e.\ \eqref{itm:good 3} holds. Finally, to prove \eqref{itm:good 4}, working locally on $X$, we may choose a set of walls $\mc{W}$, a finite set $\{\bs{\alpha}_1, \ldots, \bs{\alpha}_k\}$ such that every chamber of $\mc{W}$ contains a point in $\bs{\alpha}_i + \mb{Z}^r$ for some $r$, and finitely many degrees $\bs{n}_{ij}$ such that the good filtration $\{U^{\bs{\alpha}_i + \bs{n}}\mc{M}\}_{\bs{n} \in \mb{Z}^r}$ is generated in degrees $\bs{\alpha}_{i} + \bs{n}_{ij}$ for each $i$. This implies \eqref{itm:good 4} for any interval $[A, B]$ containing all coordinates of $\bs{\alpha}_i + \bs{n}_{ij}$ for each $i, j$.\end{proof}

As a corollary of Lemmas \ref{lem:artin-rees} and \ref{lem: goodness part}, we deduce that the Artin-Rees lemma also holds for wall and chamber filtrations:

\begin{cor} \label{cor:wall and chamber artin-rees}
Let $\mc{M}$ be a coherent $\mc{A}$-module equipped with a good wall and chamber filtration $U^\bullet \mc{M}$. Then for any coherent submodule $\mc{N} \subset \mc{M}$,
\[ U^\bullet \mc{N} := \mc{N} \cap U^\bullet \mc{M} \]
is also a good wall and chamber filtration.
\end{cor}

\subsection{The multivariate $V$-filtration} \label{subsec:multivariate V definition}

We now turn to the definition of a multivariate $V$-filtration of a coherent $\ms{D}$-module along a simple normal crossings divisor. 

Suppose that $X$ is a complex manifold and $D \subset X$ is a simple normal crossings divisor with irreducible components $D_1, \ldots, D_r$. Working locally, let us assume that $D_i = \{t_i = 0\}$ for some local coordinates $t_1, \ldots, t_r$. Consider the $\mb{Z}^r$-indexed filtration $V^\bullet\ms{D}_X$ given by
\[ V^{\bs{n}}\ms{D}_X := V^{n_1}_{D_1} \ms{D}_X \cap \cdots \cap V^{n_r}_{D_r}\ms{D}_X \quad \text{for $\bs{n} = (n_1, \ldots, n_r)\in \Z^r$,}\]
where $V^\bullet_{D_i}\ms{D}_X$ is the $V$-filtration along the smooth hypersurface $D_i$ (see \eqref{eqn: V filtration on DX}). Note that $\partial_{t_i} t_i \in V^{\bs{0}}\ms{D}_X$ for all $i$. It is easy to see that the filtered sheaf of rings $V^\bullet\ms{D}_X$ satisfies Assumption \ref{assumption:ring filtrations}, with $F_\bullet\ms{D}_X$ the usual filtration by order of differential operator.

In the definition below (and throughout this paper), we will write $s_i = -\partial_{t_i} t_i \in V^{\bs{0}}\ms{D}_X$ for $i = 1, \ldots, r$.

\begin{defn} \label{defn:multivariate V-filtration}
Let $\mc{M}$ be a coherent $\ms{D}_X$-module. We say that an $\mb{R}^r$-indexed filtration $V^\bullet \mc{M}$ is a \emph{multivariate $V$-filtration} with respect to $D$ if it satisfies:
\begin{enumerate}
\item \label{itm:multivariate V 1}$V^{\bullet}\cM$ is a good wall and chamber filtration over $V^{\bullet}\sD_X$ in the sense of Definition \ref{defn: good filtration}.
\item \label{itm:multivariate V 4} If $\mc{W}$ is a set of walls and $\bs{\alpha} \leq \bs{\beta}$ are separated by a single wall $H = \bs{L}^{-1}(\gamma)\in \mc{W}$, then there exist (not necessarily distinct) $\gamma_1, \ldots, \gamma_k \in \mb{C}$ with $\Re \gamma_j = \gamma$ such that the operator
\[ \prod_{j = 1}^k (\bs{L}(\bs{s}) + \gamma_j) \colon \frac{V^{\bs{\alpha}}\mc{M}}{V^{\bs{\beta}}\mc{M}} \to \frac{V^{\bs{\alpha}}\mc{M}}{V^{\bs{\beta}}\mc{M}} \]
is zero, where we write
\begin{equation}\label{eqn: Ldtt} \bs{L}(\bs{s}) = \sum_{i = 1}^r L_i s_i = -\sum_{i = 1}^r L_i \partial_{t_i} t_i=-\bs{L}(\bs{\d_tt}).\end{equation}
\end{enumerate}
If $A \subset \mb{C}$ is a $\mb{Q}$-linear subspace containing $\mb{Q}$, we will always endow $ A\subset \mb{R}$ and $A^r \subset \mb{R}^r$ with the (partial) order induced by restriction. We say that $V^\bullet \mc{M}$ is \emph{defined over $A$} if the $\gamma_j$ in \eqref{itm:multivariate V 4} can always be chosen to lie in $A$. 
\end{defn}
Note that if $r = 1$ then Definition \ref{defn:multivariate V-filtration} reduces to the usual definition of a $V$-filtration (Definition \ref{defn:V-filtration}).

\begin{prop} \label{prop:V image and preimage}
Let $f \colon \mc{M} \to \mc{N}$ be a morphism of coherent $\ms{D}_X$-modules.
\begin{enumerate}
\item[(a)] \label{itm:V image} If $f$ is surjective and $V^\bullet \mc{M}$ is a multivariate $V$-filtration on $\mc{M}$, then $V^\bullet \mc{N} := f(V^\bullet\mc{M})$ is a multivariate $V$-filtration on $\mc{N}$.
\item[(b)] \label{itm:V preimage} If $f$ is injective and $V^\bullet \mc{N}$ is a multivariate $V$-filtration on $\mc{N}$, then $V^\bullet\mc{M} := f^{-1}(V^\bullet \mc{N})$ is a multivariate $V$-filtration on $\mc{M}$.
\end{enumerate}
\end{prop}
\begin{proof}
The case (a) is immediate from the definition. For (b), property \eqref{itm:multivariate V 4} of Definition \ref{defn:multivariate V-filtration} is clear, and property \eqref{itm:multivariate V 1} follows from Corollary \ref{cor:wall and chamber artin-rees}.\end{proof}

The main result of this subsection is that multivariate $V$-filtrations are unique when they exist.

\begin{thm} \label{thm:multivariate V uniqueness}
Let $\mc{M}$ be a coherent $\ms{D}_X$-module. Then any two multivariate $V$-filtrations on $\mc{M}$ coincide.
\end{thm}

To prove Theorem \ref{thm:multivariate V uniqueness}, let us fix a coherent $\ms{D}_X$-module $\mc{M}$ and a multivariate $V$-filtration $V^\bullet\mc{M}$. Regard $\mc{M}$ as a module over the polynomial ring $\mb{C}[\bs{s}] := \mb{C}[s_1, \ldots, s_r]$ by setting $s_i = -\partial_{t_i}t_i$. We recall the following elementary lemma. 

\begin{lem}[{\cite[Proposition 3.8]{AtiyahMacdonald}}] \label{lem:localisation}
Let $R$ be a commutative ring, $M$ an $R$-module and $N_1, N_2 \subset M$ submodules. Then $N_1 = N_2$ if and only if $(N_1)_{\mf{m}} = (N_2)_{\mf{m}}$ for all maximal ideals $\mf{m} \subset R$.
\end{lem}

In view of Lemma \ref{lem:localisation}, the filtration $V^\bullet\mc{M}$ is completely characterised by the localisations $(V^\bullet\mc{M})_{\mf{m}} \subset \mc{M}_{\mf{m}}$. Throughout this paper, we will repeatedly use the following re-interpretation of condition \eqref{itm:multivariate V 4} of Definition \ref{defn:multivariate V-filtration} in these terms.

\begin{lem} \label{lem:localised chambers}
Let $U^\bullet \mc{M}$ be a good wall and chamber filtration over $V^\bullet\ms{D}_X$ with a compatible set of walls $\mc{W}$. For a maximal ideal $\mf{m} = (s_1 + \alpha_1, \ldots, s_r + \alpha_r) \subset \mb{C}[\bs{s}]$, let $\mc{W}_\mf{m}$ be the \emph{finite} set of walls
 \begin{equation}\label{eqn: localized walls}
 \mc{W}_{\mf{m}} \colonequals \{H \in \mc{W} \mid \Re \bs{\alpha} \in H\}.
 \end{equation}
Then $U^\bullet \mc{M}$ is a multivariate $V$-filtration (i.e.\ satisfies \eqref{itm:multivariate V 4} of Definition \ref{defn:multivariate V-filtration}) if and only if $(U^\bullet \mc{M})_{\mf{m}}$ is a wall and chamber filtration on $\mc{M}_\mf{m}$ with set of walls $\mc{W}_{\mf{m}}$.
\end{lem}

\begin{proof}
Suppose first that $U^\bullet\mc{M}$ is a multivariate $V$-filtration, fix a maximal ideal $\mf{m} = (s_1 + \alpha_1, \ldots, s_r + \alpha_r)$ and $\bs{\beta}, \bs{\gamma} \in \mb{R}^r$ lying in the same chamber of $\mc{W}_{\mf{m}}$. We need to show that $(U^{\bs{\beta}}\mc{M})_{\mf{m}} = (U^{\bs{\gamma}}\mc{M})_\mf{m}$.

Let $\sigma$ (resp.\ $\sigma'$) denote the chamber of $\mc{W}$ containing $\bs{\beta}$ (resp.\ $\bs{\gamma}$). We may of course replace $\bs{\beta}$ and $\bs{\gamma}$ with any other points in $\sigma$ and $\sigma'$ respectively without changing the truth or falsehood of the statement. For dimension reasons, choosing $\bs{\beta}$ and $\bs{\gamma}$ generically, we can assume that no point in the interval $I = \{(1 - u) \bs{\beta} + u \bs{\gamma}  \mid u \in [0, 1] \subset \mb{R}\}$ will lie on more than one wall of $\mc{W}$. An illustrative example is depicted below:
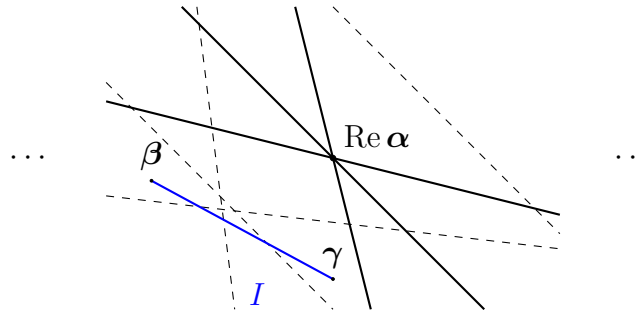
\begin{center}
\begin{tikzpicture}
\draw[thick] (-2, 2) -- (2, -2);
\draw[thick] (-3, 0.75) -- (3, -0.75);
\draw[thick] (-0.5, 2) -- (0.5, -2);
\draw[dashed] (-3, 1) -- (0, -2);
\draw[dashed] (0, 2) -- (3, -1);
\draw[dashed] (-3, -0.5) -- (3, -1.2);
\draw[dashed] (-1.8, 2) -- (-1.3, -2);
\draw[blue, thick] (-2.4,-.3) -- (0,-1.6);

\filldraw (0, 0) circle (1pt) node[above right]{$\Re \bs{\alpha}$};
\filldraw (0,-1.6) circle (0.5pt) node[above]{$\bs{\gamma}$};
\filldraw (-2.4,-.3) circle (0.5pt) node[above]{$\bs{\beta}$};
\node at (-1,-1.8) {$\textcolor{blue}{I}$};
  
  \node at (-4,0) {$\cdots$};
\node at (4,0) {$\cdots$};
\end{tikzpicture}
\end{center}
In the picture above, $\mc{W}_{\mf{m}}=\{\textrm{solid lines}\}$ and $\mc{W}$ is the union of $\mc{W}_{\mf{m}}$ and dashed lines. 

Write $\sigma = \sigma_0, \sigma_1, \ldots, \sigma_m = \sigma'$ for the sequence of chambers of $\mc{W}$ meeting $I$ as the parameter $u$ varies from $0$ to $1$.  Then $\sigma_j$ and $\sigma_{j - 1}$ are separated by a single wall of the form $\bs{L}_j^{-1}(\eta_j)$ for each $j$. It follows from Remark \ref{remark: containment of closure} that $U^{\sigma_j} \mc{M} \subset U^{\sigma_{j - 1}}\mc{M}$ or $U^{\sigma_{j - 1}}\mc{M} \subset U^{\sigma_j}\mc{M}$. Let us suppose that $U^{\sigma_j} \mc{M} \subset U^{\sigma_{j - 1}}\mc{M}$ (the converse case is identical). Then, by Definition \ref{defn:multivariate V-filtration} \eqref{itm:multivariate V 4}, there exist a finite set $\{\eta_{jk}\}_{k} \in \mb{C}$ with $\Re \eta_{jk} = \eta_j$ such that the operator
\begin{equation} \label{eq:localised chambers 1}
\prod_{k} ( \bs{L}_j(\bs{s}) + \eta_{jk}) \colon \frac{U^{\sigma_{j - 1}}\mc{M}}{U^{\sigma_j}\mc{M}} \to \frac{U^{\sigma_{j - 1}}\mc{M}}{U^{\sigma_j}\mc{M}}
\end{equation}
is zero. But since $\bs{\beta}$ and $\bs{\gamma}$ lie in the same chamber of $\mc{W}_{\mf{m}}$, none of the separating walls $\bs{L}_j^{-1}(\eta_j)$ can be walls in $\mc{W}_{\mf{m}}$: in other words, we have $0\neq \eta_{jk}-\bs{L}_j(\bs{\alpha})$ for all $k$. Hence, the operator \eqref{eq:localised chambers 1} is invertible in $\mb{C}[s_1, \ldots, s_r]_{\mf{m}}$. So $(U^{\sigma_{j - 1}}\mc{M})_\mf{m} = (U^{\sigma_j}\mc{M})_\mf{m}$ for all $j$, and hence $(U^{\bs{\beta}}\mc{M})_\mf{m} = (U^{\bs{\gamma}}\mc{M})_\mf{m}$ as claimed.

To prove the converse, suppose that $(U^{\bs{\beta}}\mc{M})_{\mf{m}} = (U^{\bs{\gamma}}\mc{M})_{\mf{m}}$ for all maximal ideals $\mf{m}$ and all $\bs{\beta}, \bs{\gamma}$ lying in the same chamber of $\mc{W}_{\mf{m}}$. Fix $\bs{\beta}\leq \bs{\gamma}$ separated by a single wall $\bs{L}^{-1}(\eta)$ of $\mc{W}$. For a local section $w \in U^{\bs{\beta}}\mc{M}/U^{\bs{\gamma}}\mc{M}$, consider the support $Z(w)$ of the $\mb{C}[\bs{s}]$-module generated by $w$. This is a closed subvariety of $\mb{C}^r$, whose image under $\bs{L} \colon \mb{C}^r \to \mb{C}$ is contained in $-\eta + \sqrt{-1}\mb{R}$ by our assumption. The image of any morphism of varieties $\bs{L} \colon Z(w) \to \mb{C}$ is either a finite set of points or contains a dense Zariski open set; by the above, we must be in the former situation. Hence, there must exist $\eta_1, \ldots, \eta_k \in \mb{C}$ with $\Re \eta_j = \eta$ such that $w$ is annihilated by $\prod_j( \bs{L}(\bs{s}) + \eta_j)$. Since $U^{\bs{\beta}}\mc{M}/U^{\bs{\gamma}}\mc{M}$ is coherent over the ring
\[ V^{\bs{0}}\ms{D}_X\big/\displaystyle\sum_{\substack{i = 1, \ldots, r \\ L_i \neq 0}} V^{\bs{1}_i}\ms{D}_X,\]
which contains $\bs{L}(\bs{s})$ in its centre, if we choose $\eta_1, \ldots, \eta_k \in \mb{C}$ as above so that $\prod_j(\bs{L}(\bs{s})+ \eta_j)$ annihilates a finite set of local generators, then $\prod_j (\bs{L}(\bs{s}) + \eta_j)$ will act by zero on $U^{\bs{\beta}}\mc{M}/U^{\bs{\gamma}}\mc{M}$ as claimed. \end{proof}

With Lemma \ref{lem:localised chambers} at hand, we can now describe $(V^{\bs{\beta}}\mc{M})_\mf{m}$. Working locally on $X$, we can suppose that $\mc{M}$ is generated by finitely many global sections $m_1, \ldots, m_k$. Let $\mc{F} \subset \mc{M}$ be the $\mc{O}_X$-coherent subsheaf generated by $m_1, \ldots, m_k$ over $\mc{O}_X$. Let $\sigma_{\mf{m}}$ be the chamber of $\mc{W}_{\mf{m}}$ containing $\bs{\beta}$. Choose an integer vector $\bs{v} \in \mb{Z}^r$ such that $\Re \bs{\alpha} + \bs{v}$ is in the interior of $\sigma_{\mf{m}}$ and, for simplicity, that $v_i \neq 0$ for all $i$.

\begin{lem} \label{lem:multivariate V localisation}
Locally on $X$, we have $(V^{\bs{\beta}}\mc{M})_\mf{m} = (V^{N \bs{v}}\ms{D}_X \cdot \mc{F})_\mf{m}$, for $N \gg 0$.
\end{lem}
\begin{proof}
Since $V^\bullet\mc{M}$ is exhaustive and $\mc{F}$ is coherent, we have locally on $X$ that $\mc{F} \subset V^{\bs{\gamma}}\mc{M}$ for some $\bs{\gamma} \in \R^r$. So $V^{N\bs{v}} \ms{D}_X \cdot \mc{F} \subset V^{\bs{\gamma} + N \bs{v}}\mc{M}$. Since $\mc{W}_{\mf{m}}$ consists only of finitely many hyperplanes passing through a single point, for $N \gg 0$, we have $\bs{\gamma} + N \bs{v} \in \sigma_{\mf{m}}$, as illustrated in the figure below. 
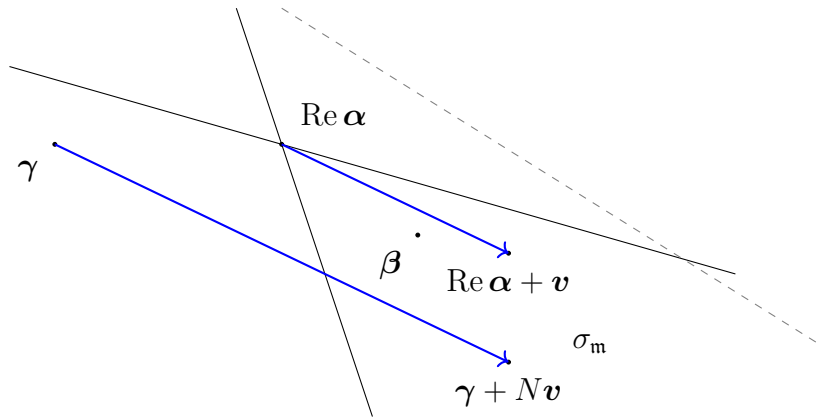
\[
\begin{tikzpicture}[scale=1.2]

\draw[thin] (-3, 6/7) -- (5, -10/7);         
\draw[thin] (-0.5, 1.5) -- (1, -3);       
\draw[dashed, gray] (0, 1.5) -- (6, -2.25);  

\filldraw (0,0) circle (0.6pt) node[above right=3pt]{$\Re \bs{\alpha}$};
\node at (3.4, -2.2) {$\sigma_{\mf{m}}$};

\filldraw (1.5,-1) circle (0.6pt) node[below left=2pt] {$\bs{\beta}$};
\filldraw (-2.5,0) circle (0.6pt) node[below left=2pt] {$\bs{\gamma}$};

\filldraw (2.5, -1.2) circle (0.6pt) node[below=2pt]{$\Re \bs{\alpha} + \bs{v}$};
\filldraw (2.5, -2.4) circle (0.6pt) node[below=2pt] {$\bs{\gamma} + N\bs{v}$};

\draw[->, blue, thick] (0, 0) -- (2.5,-1.2);
\draw[->, blue, thick] (-2.5, 0) -- (2.5,-2.4);

\end{tikzpicture}
\]
So $(V^{\bs{\gamma} + N \bs{v}}\mc{M})_\mf{m} = (V^{\bs{\beta}}\mc{M})_\mf{m}$ by Lemma \ref{lem:localised chambers}, which proves the inclusion $(V^{N \bs{v}} \ms{D}_X \cdot \mc{F})_\mf{m} \subset (V^{\bs{\beta}}\mc{M})_\mf{m}$.

To prove the reverse, fix an interval $[A, B]$ such that $V^\bullet \mc{M}$ is generated by its restriction to $[A, B]^r$. Then we have
\[ V^{\bs{\beta}'}\mc{M} = \sum_{\bs{\gamma} \in [A, B]^r} V^{\lceil \bs{\beta}' - \bs{\gamma}\rceil} \ms{D}_X \cdot V^{\bs{\gamma}}\mc{M},\quad \text{for all $\bs{\beta}' \in \R^r$}. \]
Let us fix $\bs{\beta}' = \bs{\beta} + M \bs{v}$ with $M\in \mathbb{Z}$ large enough so that $\lceil \beta'_i - \gamma_i \rceil$ and $v_i$ have the same sign for all $i$ and all $\bs{\gamma} \in [A, B]^r$. Since $V^{\bs{\beta}'}\mc{M}$ is a coherent $V^{\bs{0}} \ms{D}_X$-module, there exists, locally on $X$, $\bs{k} \in \mb{Z}^r$ such that $V^{\bs{\beta}'}\mc{M} \subset V^{\bs{k}} \ms{D}_X \cdot \mc{F}$.  Choose $N$ large enough so that $\bs{\beta}' + N \bs{v} - \bs{k} \in \sigma_\mf{m}$ and so that $v_i$ and $N v_i - k_i$ have the same sign for all $i$.  Since $V^{\bs{m}} \ms{D}_X \cdot V^{\bs{n}} \ms{D}_X = V^{\bs{m} + \bs{n}} \ms{D}_X$ as long as $m_i$ and $n_i$ have the same sign for all $i$, we have
\begin{align*}
V^{\bs{\beta}' + N\bs{v} - \bs{k}}\mc{M} &= \sum_{\bs{\gamma} \in [A, B]^r} V^{N \bs{v} - \bs{k} + \lceil \bs{\beta}' - \bs{\gamma}\rceil} \ms{D}_X \cdot V^{\bs{\gamma}}\mc{M} \\
&=  \sum_{\bs{\gamma} \in [A, B]^r} V^{N\bs{v} - \bs{k}}\ms{D}_X \cdot V^{\lceil \bs{\beta}' - \bs{\gamma}\rceil} \ms{D}_X \cdot V^{\bs{\gamma}}\mc{M} \\
& = V^{N \bs{v} - \bs{k}} \ms{D}_X \cdot V^{\bs{\beta}'}\mc{M} \subset V^{N\bs{v}} \ms{D}_X \cdot \mc{F}.
\end{align*}
Since $\bs{\beta}' + N \bs{v} - \bs{k} \in \sigma_\mf{m}$, by Lemma \ref{lem:localised chambers}, we have
\[ (V^{\bs{\beta}}\mc{M})_\mf{m}= (V^{\bs{\beta}' + N\bs{v} - \bs{k}}\mc{M})_{\mf{m}}\subset (V^{N\bs{v}} \ms{D}_X \cdot \mc{F})_\mf{m}\]
as claimed.\end{proof}

\begin{proof}[Proof of Theorem \ref{thm:multivariate V uniqueness}]
Let $U^\bullet \mc{M}$ and $V^\bullet\mc{M}$ be two multivariate $V$-filtrations. Note that, since the statement is local on $X$, we may fix a set of walls $\mc{W}$ for both $U^\bullet\mc{M}$ and $V^\bullet\mc{M}$ by taking the union of sets of walls for $U^\bullet\mc{M}$ and $V^\bullet\mc{M}$ if necessary. Since Lemma \ref{lem:multivariate V localisation} applies to both $U^\bullet\mc{M}$ and $V^\bullet \mc{M}$, we have $(U^\bullet\mc{M})_\mf{m} = (V^\bullet\mc{M})_\mf{m}$ for every maximal ideal $\mf{m}$. Applying Lemma \ref{lem:localisation}, we conclude that $U^\bullet\mc{M} = V^\bullet\mc{M}$ as claimed.
\end{proof}

\begin{cor} \label{cor:V strictness}
Let $f \colon \mc{M} \to \mc{N}$ be a morphism of coherent $\ms{D}_X$-modules such that $\mc{M}$ and $\mc{N}$ admit multivariate $V$-filtrations $V^\bullet \mc{M}$ and $V^\bullet\mc{N}$. Then $f$ is strict with respect to these filtrations.
\end{cor}
\begin{proof}
Let $\mc{P}$ be the image of $f$. Then, by Proposition \ref{prop:V image and preimage}, $f(V^\bullet\mc{M})$ and $V^\bullet \mc{N} \cap \mc{P}$ are both multivariate $V$-filtrations on $\mc{P}$. Hence, these two filtrations coincide by Theorem \ref{thm:multivariate V uniqueness}, i.e.\ $f$ is strict.
\end{proof}

\subsection{Comparison with Sabbah's filtration}\label{sec: comparison with Sabbah's filtration} \label{subsec:sabbah comparison}

In this subsection, we compare our notion of multivariate $V$-filtration with Sabbah's formalism in \cite{Sabbah1987}. We first recall (and paraphrase) some general theory from \cite{Sabbah1987}. In what follows, we fix a subset $R \subset \mb{R}^r$, a union of finitely many cosets for $\mb{Z}^r \subset \mb{R}^r$. We regard $R$ as a poset with the lexicographic partial order inherited from $\mb{R}^r$.

\begin{defn}
Let $U^\bullet \mc{M}$ be a decreasing $R$-indexed filtration such that $V^{\bs{n}}\ms{D}_X \cdot U^{\bs{\alpha}}\mc{M} \subset U^{\bs{n} + \bs{\alpha}}\mc{M}$ for all $\bs{\alpha}\in R$ and $\bs{n}\in \Z^r$.
\begin{enumerate}
\item We say that $U^\bullet\mc{M}$ is \emph{good} if $\{U^{\bs{\alpha} + \bs{n}}\mc{M}\}_{\bs{n} \in \mb{Z}^r}$ is a good $\mb{Z}^r$-indexed filtration over $V^\bullet\ms{D}_X$ for all $\bs{\alpha} \in R$, in the sense of Definition \ref{defn: good Zr indexed filtration}.
\item The \emph{saturation} of $U^\bullet \mc{M}$ is
\[ \bar{U}^{\bs{\alpha}}\mc{M} = \bigcap_{\bs{L} \in \mb{Q}_{\geq 0}^r} \ls{\bs{L}}{U}^{\bs{L}(\bs{\alpha})}\mc{M}, \qquad where \quad \ls{\bs{L}}{U}^{\beta}\mc{M} := \sum_{\substack{\bs{\alpha} \in R \\ \bs{L}(\bs{\alpha}) \geq \beta}} U^{\bs{\alpha}}\mc{M}.\]
\item We say that $U^\bullet\mc{M}$ is \emph{saturated} if $U^\bullet\mc{M} = \bar{U}^\bullet\mc{M}$.
\end{enumerate}
\end{defn}

It is elementary to see that $\ls{\bs{L}}{\bar{U}}^\beta \mc{M} = \ls{\bs{L}}{U}^\beta \mc{M}$ for all $\bs{L}$ and $\beta$. In particular, $\bar{U}^\bullet\mc{M}$ is saturated. We also note the following fact from \cite[2.2.3]{Sabbah1987}.

\begin{prop}\label{prop:saturation is good}
If $U^\bullet\mc{M}$ is a good filtration then its saturation $\bar{U}^\bullet\mc{M}$ is also good.
\end{prop}

Sabbah's proof of this fact uses the existence of an adapted fan \cite[\S 2.2]{Sabbah1987}. While it has been noted elsewhere that the proof that such a fan exists in \cite[A.1.1]{Sabbah1987} is incorrect, complete proofs can be found in \cite{Bahloul04,CJG07}, for example.

\begin{defn} \label{defn:sabbah}
We say that a decreasing $R$-indexed filtration $V_R^\bullet\mc{M}$ is \emph{Sabbah's $V$-filtration} if:
\begin{enumerate}
\item \label{itm:sabbah 1} $V_R^\bullet\mc{M}$ is good,
\item \label{itm:sabbah 2} $V_R^\bullet\mc{M}$ is saturated, and
\item \label{itm:sabbah 3} for all $\bs{L} \in \mb{Q}_{\geq 0}^r$ and all $\beta \in \mb{R}$, there exist finitely many $\beta_j \in \mb{C}$ with $\Re \beta_j = \beta$ such that the operator $\prod_j(\beta_j - \bs{L}(\bs{\partial_t t}))$ annihilates $\gr^\beta_{\ls{\bs{L}}{V_R}} \mc{M}$.
\end{enumerate}
\end{defn}

We record the following elementary properties from \cite{Sabbah1987}. First, Sabbah's $V$-filtration is unique if it exists. For $\bs{L}\in \Q^r_{\geq 0}$, consider the $\Z$-indexed filtration $\ls{\bs{L}}{V}^{\bullet}\sD_X$:
\[ \ls{\bs{L}}{V}^{k}\sD_X\colonequals \sum_{\bs{n}\in \Z^r, \,\bs{L}(\bs{n})\geq k} V^{\bs{n}}\sD_X, \quad \textrm{for all $k\in \Z$}.\]
The $\ls{\bs{L}}{V}$-filtration of $\cM$ is the unique $\Z$-indexed filtration $\ls{\bs{L}}{V}^{\bullet}\cM$ that is good over $\ls{\bs{L}}{V}^\bullet\ms{D}_X$ and whose associated graded $\gr^{\bullet}_{\ls{\bs{L}}{V}}\cM$ satisfies the nilpotency condition \eqref{itm:sabbah 3}. It follows from the conditions \eqref{itm:sabbah 1},\eqref{itm:sabbah 3} that $\ls{\bs{L}}{V_R}^\bullet\mc{M} = \ls{\bs{L}}{V}^\bullet\mc{M}$. By saturation, \eqref{itm:sabbah 2}, this shows that $V_R^\bullet\mc{M}$ is determined by the $\ls{\bs{L}}{V}$-filtrations of $\cM$ for all $\bs{L}\in \Q^r_{\geq 0}$, so the uniqueness follows.

Second, if $U^\bullet \mc{M}$ is an $R$-indexed filtration satisfying \eqref{itm:sabbah 1} and \eqref{itm:sabbah 3}, then its saturation $\bar{U}^\bullet\mc{M}$ is Sabbah's $V$-filtration, by Proposition \ref{prop:saturation is good}.

Finally, suppose that $R \subset R' \subset \mb{R}^r$ are finite unions of cosets and that Sabbah's $V$-filtration $V_R^\bullet\mc{M}$ exists. Then $V^{\bullet}_{R'}\cM$ exists by the following argument: set 
\[ U_{R'}^{\bs{\alpha}'}\mc{M}\colonequals \sum_{\bs{\alpha} \geq \bs{\alpha}', \,\bs{\alpha} \in R} V_R^{\bs{\alpha}}\mc{M}, \quad \textrm{for $\bs{\alpha}' \in R'$}.\]
Clearly $U_{R'}^\bullet\mc{M}$ is an $R'$-indexed filtration satisfying \eqref{itm:sabbah 1} and \eqref{itm:sabbah 3}. So by above $V_{R'}^\bullet\mc{M}= \bar{U}_{R'}^\bullet\mc{M}$. Furthermore, $V_{R'}^{\bs{\alpha}}\mc{M} = V_R^{\bs{\alpha}}\mc{M}$ for $\bs{\alpha} \in R$, because $V_R^\bullet\mc{M}$ is saturated.

Here is the main result of this section.

\begin{thm}\label{thm:sabbah comparison}
The following are equivalent.
\begin{enumerate}[label=(\alph*)]
\item The multivariate $V$-filtration $V^\bullet\mc{M}$ exists.
\item Sabbah's $V$-filtration $V_R^\bullet\mc{M}$ exists for some $R$.
\end{enumerate}
Moreover, we have $V^{\bs{\alpha}}\mc{M} = V_R^{\bs{\alpha}}\mc{M}$ for $\bs{\alpha} \in R$ provided $V_R^\bullet\mc{M}$ exists.
\end{thm}
\begin{proof}
We first assume Sabbah's $V$-filtration $V_R^\bullet\mc{M}$ exists for some union of cosets $R \subset \mb{R}^r$. Define an $\mb{R}^r$-indexed filtration $U^\bullet\mc{M}$ using $\ls{\bs{L}}{V}$-filtrations of $\cM$ by
\[ U^{\bs{\alpha}}\mc{M} = \bigcap_{\bs{L} \in \mb{Q}_{\geq 0}^r} \ls{\bs{L}}{V}^{\bs{L}(\bs{\alpha})}\mc{M}, \quad \bs{\alpha}\in \R^r.\]
We show that this is a multivariate $V$-filtration. We first check that $U^\bullet\mc{M}$ is a wall and chamber filtration. To do so, choose an adapted fan $\Sigma$ for $V_R^\bullet\mc{M}$ in the sense of \cite[\S 2.2]{Sabbah1987}. Let $\ms{L}(\Sigma)$ be the (finite) set of $\bs{L}\in \mb{Z}_{\geq 0}^r$ such that $\bs{L}$ is primitive and $\mb{R}_{> 0}\cdot\bs{L}$ is a ray in the $1$-skeleton of $\Sigma$. Then, as observed in the remark following \cite[2.2.3]{Sabbah1987}, we have
\begin{equation} \label{eq:sabbah comparison 1}
U^{\bs{\alpha}}\mc{M} = \bigcap_{\bs{L} \in \ms{L}(\Sigma)} \ls{\bs{L}}{V}^{\bs{L}(\bs{\alpha})}\mc{M}.
\end{equation}
For $\bs{L} \in \ms{L}(\Sigma)$, let $\textrm{JN}(\bs{L}) \subset \mb{R}$ denote the set of jumping numbers for $\ls{\bs{L}}{V}^{\bullet}\mc{M}$; since the latter is a good filtration with respect to $\ls{\bs{L}}{V}^\bullet\ms{D}_X$, $\textrm{JN}(\bs{L})$ is finite modulo integer translations. So by \eqref{eq:sabbah comparison 1}, $U^\bullet\mc{M}$ is a wall and chamber filtration, with a set of walls
\[ \mc{W} := \{\bs{L}^{-1}(\beta) \mid \bs{L} \in \ms{L}(\Sigma), \beta \in \textrm{JN}(\bs{L}) \}.\]
To check that the wall and chamber filtration $U^\bullet\mc{M}$ is good, we apply Lemma \ref{lem: goodness part}. For any $\bs{\alpha} \in \mb{R}^r$, we can assume $\bs{\alpha} \in R$ by enlarging $R$ if necessary by saturation. Since the restriction of $U^\bullet\mc{M}$ to $R$ is the good filtration $V_R^\bullet\mc{M}$, we have that $\{U^{\bs{\alpha} + \bs{n}}\mc{M}\}_{\bs{n} \in \mb{Z}^r}$ is good. Hence $U^\bullet\mc{M}$ is a good wall and chamber filtration as claimed.

Now, if $\bs{\alpha} \leq \bs{\beta}$ are separated by a single wall $\bs{L}^{-1}(\gamma)$, then by \eqref{eq:sabbah comparison 1}, we have
\begin{align*}
U^{\bs{\alpha}}\mc{M} \cap \ls{\bs{L}}{V}^{>\gamma} \mc{M} &= \bigcap_{\bs{L}' \in \ms{L}(\Sigma) \setminus\{\bs{L}\}} \ls{\bs{L}'}{V}^{\bs{L}'(\bs{\alpha})}\mc{M} \cap \ls{\bs{L}}V^{>\gamma}\mc{M} \\
&=  \bigcap_{\bs{L}' \in \ms{L}(\Sigma) \setminus\{\bs{L}\}} \ls{\bs{L}'}{V}^{\bs{L}'(\bs{\beta})}\mc{M} \cap \ls{\bs{L}}V^{\bs{L}(\bs{\beta})}\mc{M} \\
&= U^{\bs{\beta}}\mc{M},
\end{align*}
where the second equality holds because, by the assumption that $\bs{\alpha} \leq \bs{\beta}$ are separated by the single wall $\bs{L}^{-1}(\gamma)$ of $\mc{W}$, there are no jumping numbers of $\ls{\bs{L}'}{V}^\bullet\mc{M}$ (resp.\ $\ls{\bs{L}}{V}^\bullet\mc{M}$) between $\bs{L}'(\bs{\alpha})$ and $\bs{L}'(\bs{\beta})$ (resp.\ $\gamma + \epsilon$ and $\bs{L}(\bs{\beta})$). So $U^{\bs{\alpha}}\mc{M}/U^{\bs{\beta}}\mc{M}$ is a submodule of $\gr_{\ls{\bs{L}}{V}}^\gamma\mc{M}$ and hence condition \eqref{itm:multivariate V 4} of Definition \ref{defn:multivariate V-filtration} is satisfied. So by Theorem \ref{thm:multivariate V uniqueness}, $V^\bullet\mc{M}$ exists and is equal to $U^\bullet\mc{M}$ as claimed.

Conversely, suppose that the multivariate $V$-filtration $V^\bullet\mc{M}$ exists. We will show that there exists a union of cosets $R$ and an $R$-indexed filtration $U^\bullet_R\mc{M}$ satisfying \eqref{itm:sabbah 1} and \eqref{itm:sabbah 3} of Definition \ref{defn:sabbah}; $V_{R}^\bullet\mc{M}$ will then be given as the saturation of $U^\bullet_R\mc{M}$. 

Fix $\mc{W}$ a set of walls for $V^\bullet\mc{M}$. Enlarging $\mc{W}$ if necessary, we can assume for simplicity that the chambers of $\mc{W}$ are all bounded subsets of $\mb{R}^r$. Call a point $\bs{\alpha} \in \mb{R}^r$ a \emph{vertex of $\mc{W}$} if $\bs{\alpha}$ is the intersection of $r$ hyperplanes in $\mc{W}$. Note that there are finitely many vertices modulo $\mb{Z}^r$. Fix any finite union of cosets $R$ such that $R$ contains all vertices of $\mc{W}$ and define 
\[U^{\bs{\alpha}}_R\mc{M} \colonequals V^{\bs{\alpha}}\mc{M}, \quad \textrm{for $\bs{\alpha} \in R$}.\]
By Lemma \ref{lem: goodness part}, the filtration $U_R^\bullet\mc{M}$ is good. To prove \eqref{itm:sabbah 3} of Definition \ref{defn:sabbah}, we will use the following lemma.

\begin{lem} \label{lem:LV from lattice}
For any $\bs{L} \in \mb{Q}_{\geq 0}^r$ and any $\beta \in \mb{R}$, we have
\[ \sum_{\substack{\bs{\alpha} \in R \\ \bs{L}(\bs{\alpha}) \geq \beta}} U_R^{\bs{\alpha}} \mc{M} = \sum_{\substack{\bs{\alpha} \in \mb{R}^r \\ \bs{L}(\bs{\alpha}) \geq \beta}} V^{\bs{\alpha}} \mc{M}.\]
\end{lem}
\begin{proof}
Clearly the left hand side is contained in the right. To prove the reverse, suppose that $\bs{\alpha} \in \mb{R}^r$ satisfies $\bs{L}(\bs{\alpha}) \geq \beta$. Let $\sigma$ be the chamber of $\mc{W}$ containing $\bs{\alpha}$. Since the chambers of $\mc{W}$ are bounded by assumption, the closure $\bar{\sigma}$ is compact. Let 
\[\gamma \colonequals \max_{\bs{\eta}\in \bar{\sigma}} \bs{L}(\bs{\eta}).\]
The set $\bar{\sigma} \cap \bs{L}^{-1}(\gamma)$ is a bounded polytope in $\bs{L}^{-1}(\gamma) \cong \mb{R}^{r - 1}$; let $\bs{\alpha}'$ be one of the vertices of this polytope. We claim that $\bs{\alpha}'$ is a vertex of $\mc{W}$. Indeed, $\bs{\alpha}'$ lies on at least $r - 1$ hyperplanes in $\mc{W}$ by construction; we need to show that it lies on at least $r$. Suppose to the contrary that it lies on exactly $r - 1$ hyperplanes $H_1, \ldots, H_{r - 1}$. Then there exists a vector $\bs{v} \in \mb{R}^r$ parallel to the $H_i$ for all $i$ such that $\bs{L}(\bs{v}) > 0$. In particular, for $0 < \epsilon \ll 1$, we have $\bs{L}(\bs{\alpha}' + \epsilon\bs{v}) > \gamma$ and $\bs{\alpha}' + \epsilon\bs{v} \in \bar{\sigma}$, contradicting maximality of $\gamma$. So $\bs{\alpha}'$ is indeed a vertex of $\mc{W}$ as claimed and hence $\bs{\alpha}' \in R$. Moreover, since $\bs{\alpha}' \in \bar{\sigma}$, by Remark \ref{remark: containment of closure} we have
\[ V^{\bs{\alpha}}\cM\subset V^{\bs{\alpha}'}\mc{M} =U_R^{\bs{\alpha}'}\mc{M}.\]
As $\bs{L}(\bs{\alpha}') = \gamma \geq \bs{L}(\bs{\alpha})\geq \beta$, this proves the converse containment.
\end{proof}

Continuing with the proof of Theorem \ref{thm:sabbah comparison}, fix $\bs{L} \in \mb{Q}_{> 0}^r$ and write
\[ \ls{\bs{L}}{U}_R^\beta \mc{M} = \sum_{\substack{\bs{\alpha} \in R, \\ \bs{L}(\bs{\alpha}) \geq \beta}} U_R^{\bs{\alpha}}\mc{M} = \sum_{\substack{\bs{\alpha} \in \mb{R}^r, \\ \bs{L}(\bs{\alpha}) \geq \beta}} V^{\bs{\alpha}}\mc{M}.\]
We need to show that $\gr_{\ls{\bs{L}}{U}_R}^\beta \mc{M}$ is annihilated by some operator of the form $\prod_j(\bs{L}(\bs{s}) + \beta_j)$ for $\Re \beta_j = \beta$, where $s_i=-\d_{t_i}t_i$. Arguing as in the final paragraph of the proof of Lemma \ref{lem:localised chambers}, it suffices to show that, for every $\bs{\gamma} \in \mb{C}^r$ with corresponding maximal ideal $\mf{m} = (s_1 + \gamma_1, \ldots, s_r + \gamma_r) \subset \mb{C}[\bs{s}]$, we have
\begin{equation} \label{eq:sabbah comparison 3}
(\gr_{\ls{\bs{L}}{U}_R}^\beta\mc{M})_{\mf{m}} = 0 \quad \text{if $\Re \bs{L}(\bs{\gamma}) \neq \beta$}.
\end{equation}
To prove \eqref{eq:sabbah comparison 3}, note that by Lemma \ref{lem:localised chambers}, we have
\begin{equation} \label{eq:sabbah comparison 2}
(\ls{\bs{L}}{U}_R^{\beta}\mc{M})_{\mf{m}} = \sum_{\sigma_\mf{m} \cap \bs{L}^{-1}(\mb{R}_{\geq \beta}) \neq \emptyset}(V^{\sigma_\mf{m}}\mc{M})_{\mf{m}},
\end{equation}
where the sum is over all chambers $\sigma_{\mf{m}}$ of $\mc{W}_{\mf{m}}$, the set of walls in $\mc{W}$ passing through $\Re\bs{\gamma}$. If $\Re \bs{L}(\bs{\gamma}) \neq \beta$, then any chamber of $\mc{W}_\mf{m}$ containing a point $\bs{\alpha}$ with $\bs{L}(\bs{\alpha}) \geq \beta$ also contains a point $\bs{\alpha}'$ with $\bs{L}(\bs{\alpha}') > \beta$ (e.g.\ we can take $\bs{\alpha}' = \bs{\alpha} + t (\bs{\alpha} - \Re\bs{\gamma})$ for some small $t$). So from \eqref{eq:sabbah comparison 2}, we see that
\[ (\ls{\bs{L}}{U}_R^{\beta}\mc{M})_{\mf{m}} = (\ls{\bs{L}}{U}_R^{>\beta}\mc{M})_{\mf{m}},\]
so \eqref{eq:sabbah comparison 3} follows.

Finally, by Lemma \ref{lem:LV from lattice} and the proof of the first direction, we conclude that for $\bs{\alpha}\in R$,
\[ V^{\bs{\alpha}}_R\cM=\bar{U}^{\bs{\alpha}}_R\cM=\bigcap_{\bs{L}\in \Q^r_{\geq 0}}\ls{\bs{L}}{U}_R^{\bs{L}(\bs{\alpha})}\cM=\bigcap_{\bs{L}\in \Q^r_{\geq 0}}\ls{\bs{L}}{V}^{\bs{L}(\bs{\alpha})}\cM=V^{\bs{\alpha}}\cM.\qedhere\] 
\end{proof}

Theorem \ref{thm:sabbah comparison} implies, in particular, that the ordinary $V$-filtrations $V_{D_i}^\bullet\mc{M}$ can be recovered from the multivariate $V$-filtration $V^\bullet\mc{M}$.  For future reference, we record the following related statement.

\begin{prop} \label{prop:V-filtration for sub divisor}
Suppose that $D' \subset D$ is a sub-divisor, given by, say, $D' = D_1 \cup \cdots \cup D_{r'}$ for $r' < r$. If $V^\bullet\mc{M}$ is a multivariate $V$-filtration along $D$, then the filtration
\[ U^{\alpha_1, \ldots, \alpha_{r'}}\mc{M} := \bigcup_{\alpha_{r' + 1}, \ldots, \alpha_r \in \mb{R}} V^{\alpha_1, \ldots, \alpha_{r'}, \alpha_{r' + 1}, \ldots, \alpha_r}\mc{M} \]
is a multivariate $V$-filtration along $D'$.
\end{prop}

\begin{proof}
To see that $U^\bullet\mc{M}$ is a wall and chamber filtration, suppose that $\mc{W}$ is a set of walls for $V^\bullet\mc{M}$ and let $\mc{W}'$ be the set of walls
\[ \mc{W}' = \{(\bs{L}')^{-1}(\gamma) \mid \bs{L}' = (L_1', \ldots, L_{r'}') \in \mb{Q}_{\geq 0}^{r'}, (L_1', L_2', \ldots, L_{r'}', 0, \ldots, 0)^{-1}(\gamma) \in \mc{W}\}\]
in $\mb{R}^{r'}$. We claim that $\mc{W}'$ is a set of walls for $U^\bullet\mc{M}$. To see this, fix $\bs{\alpha}', \bs{\beta}' \in \mb{R}^{r'}$ in the same chamber of $\mc{W}'$; we need to show that $U^{\bs{\alpha}'}\mc{M} = U^{\bs{\beta}'}\mc{M}$. By symmetry, it is enough to check that if $m \in U^{\bs{\alpha}'}\mc{M}$ is a local section then $m \in U^{\bs{\beta}'}\mc{M}$ also. Since $m \in U^{\bs{\alpha}'}\mc{M}$, we have $m \in V^{\bs{\alpha}}\mc{M}$ for some $\bs{\alpha} \in \mb{R}^r$ with $\alpha_i = \alpha_i'$ for $i \leq r'$. Choose any $\bs{\beta} \in \mb{R}^r$ with $\beta_i = \beta_i'$ for $i \leq r'$. Since $\bs{\alpha}'$ and $\bs{\beta}'$ are in the same chamber of $\mc{W}'$, the walls $\bs{L}^{-1}(\gamma)$ separating $\bs{\alpha}$ and $\bs{\beta}$ all satisfy $L_i > 0$ for some $i > r'$. So modifying our choice of $\bs{\beta}$ if necessary, we can suppose that $\bs{L}(\bs{\alpha}) > \gamma$ and $\bs{L}(\bs{\beta}) \leq \gamma$ for all such walls. This implies that $V^{\bs{\alpha}}\mc{M} \subset V^{\bs{\beta}}\mc{M}$, so $m \in V^{\bs{\beta}}\mc{M} \subset U^{\bs{\beta}'}\mc{M}$ as claimed.

Now, if $\bs{\alpha} \in \mb{R}^r$ and $\bs{\alpha'} = (\alpha_1, \ldots, \alpha_{r'})$, then, by construction,
\[ R_{U^{\bs{\alpha}' + \bullet}} \mc{M} = R_{V^{\bs{\alpha} + \bullet}}\mc{M} \otimes_{\mb{C}[\bs{u}]}\mb{C}[\bs{u}'],\]
where the tensor product is over the map sending $u_i$ to $u_i$ for $i \leq r'$ and $1$ for $i > r'$. In particular, since $R_{V^{\bs{\alpha} + \bullet}}\mc{M}$ is coherent over $R_{V_D}\ms{D}_X$, we know $R_{U^{\bs{\alpha}' + \bullet}} \mc{M}$ is coherent over $R_{V_{D'}}\ms{D}_X$. So $U^\bullet\mc{M}$ is good by Lemma \ref{lem: goodness part}.

To conclude that $U^\bullet\mc{M}$ is a multivariate $V$-filtration, we apply Lemma \ref{lem:localised chambers}: if $\mf{m} = (s_1 + \gamma_1, \ldots, s_r + \gamma_r) \subset \mb{C}[\bs{s}]$ is a maximal ideal, then we have
\[ (U^{\bs{\alpha}'}\mc{M})_{\mf{m}} = \bigcup_{\alpha_{r' + 1}, \ldots, \alpha_r\in \mb{R}} (V^{\bs{\alpha}', \alpha_{r' + 1}, \ldots, \alpha_{r'}}\mc{M})_{\mf{m}}\]
for all $\bs{\alpha}' \in \mb{R}^{r'}$. Since $(V^\bullet\mc{M})_{\mf{m}}$ is a wall and chamber filtration with set of walls 
\[ \mc{W}_\mf{m}=\{ H\in \mc{W} \mid \Re(\bs{\gamma})\in H \} \] 
by Lemma \ref{lem:localised chambers}, we can apply the above argument to conclude that $(U^\bullet\mc{M})_{\mf{m}}$ is a wall and chamber filtration with set of walls
\[ \mc{W}'_\mf{m} := \{H \in \mc{W}' \mid \Re(\gamma_1, \ldots, \gamma_{r'}) \in H\}.\]
Now, if we write $\mf{m}' = (s_1 + \gamma_1, \ldots, s_{r'} + \gamma_{r'}) \subset \mb{C}[s_1, \ldots, s_{r'}]$, then
\[ \mc{W}'_\mf{m} = (\mc{W}')_{\mf{m}'},\]
where the right hand side is defined as in \eqref{eqn: localized walls} from Lemma \ref{lem:localised chambers}. 
 So, letting $\gamma_{r' + 1}, \ldots, \gamma_r$ vary and applying Lemma \ref{lem:localisation} to the $\mb{C}[\bs{s}]_{\mf{m}'}$-module $(U^\bullet\mc{M})_{\mf{m}'}$, we conclude that $(U^\bullet\mc{M})_{\mf{m}'}$ is a wall and chamber filtration with set of walls $(\mc{W}')_{\mf{m}'}$ and hence $U^\bullet \mc{M}$ is a multivariate $V$-filtration by Lemma \ref{lem:localised chambers}. 
\end{proof}

\begin{remark}
One can also argue using the theory of adapted fans that the restriction $U_{R'}^\bullet\mc{M}$ of $U^\bullet\mc{M}$ to a sufficiently large union of cosets $R'$ satisfies the conditions defining Sabbah's filtration and apply Theorem \ref{thm:sabbah comparison}. We leave the details to the interested reader.\end{remark}

\subsection{Proper direct images} \label{subsec:direct image}

In this subsection, we prove that multivariate $V$-filtrations behave well under proper direct images.

We first discuss some generalities about filtered complexes and derived categories. Suppose that $(\mc{A}, U^\bullet\mc{A})$ is a filtered sheaf of rings on a complex manifold $X$ satisfying Assumption \ref{assumption:ring filtrations}. Recall that the Rees module construction defines an equivalence of categories between (coherent) sheaves of $\mc{A}$-modules $\mc{M}$ equipped with an exhaustive (good) $\mb{Z}^r$-indexed filtration $U^\bullet\mc{M}$ and (coherent) $\mb{Z}^r$-graded modules over the Rees algebra
\[ R_U\mc{A} = \bigoplus_{\bs{n} \in \mb{Z}^r} U^{\bs{n}}\mc{A} \, \bs{u}^{-\bs{n}}\]
without $\mb{C}[\bs{u}]$-torsion. Motivated by this, one can define the filtered derived category
\[ \mrm{D}(\mc{A}, U^\bullet\mc{A}) \colonequals \mrm{D}(\Mod_{\mb{Z}^r}(R_U\mc{A})),\]
to be the derived category of $\mb{Z}^r$-graded $R_U\mc{A}$-modules, and the full subcategory of good filtrations
\[ \mrm{D}^b_{\mrm{coh}}(\mc{A}, U^\bullet\mc{A}) \subset \mrm{D}(\mc{A}, U^\bullet\mc{A}) \]
to be the subcategory of bounded objects whose cohomology sheaves are coherent $R_U\mc{A}$-modules. We have a forgetful functor
\[ \mrm{D}(\mc{A}, U^\bullet\mc{A}) \to \mrm{D}(\mc{A}) \colonequals \mrm{D}(\Mod(\mc{A}))\]
sending a complex $\mc{N} = \bigoplus_{\bs{n} \in \mb{Z}^r}\mc{N}^{\bs{n}}$ of graded $R_U\mc{A}$-modules to the complex of $\mc{A}$-modules
\[ \mc{M} = \colim_{\bs{n} \in \mb{Z}^r} \mc{N}^{\bs{n}}.\]
For such an $\mc{N} \in \mrm{D}(\mc{A}, U^\bullet\mc{A})$, we usually imagine that it is the Rees module of a filtration $U^\bullet\mc{M}$ on $\mc{M}$ (and indeed one can arrange that this is indeed the case by replacing $\mc{N}$ with a quasi-isomorphic complex if necessary). Accordingly, we often write $\mc{N} = (\mc{M}, U^\bullet\mc{M})$ for a general object in $\mrm{D}(\mc{A}, U^\bullet\mc{A})$, $R_{U}\mc{M} \colonequals \mc{N}$ for $\mc{N}$ regarded as an object in the derived category of $R_{U}\mc{A}$-modules, and, for $\bs{n} \in \mb{Z}^r$, we write $U^{\bs{n}}\mc{M} \colonequals \mc{N}^{\bs{n}}$. 

We have a similar story for $\mb{R}^r$-indexed filtrations. If $(\mc{M}, U^\bullet\mc{M})$ is an $\mc{A}$-module equipped with an exhaustive $\mb{R}^r$-indexed filtration, we can form the $\mb{R}^r$-graded Rees module
\[ R_{U, \mb{R}}\mc{M} = \bigoplus_{\bs{\alpha} \in \mb{R}^r} U^{\bs{\alpha}}\mc{M} \, \bs{u}^{-\bs{\alpha}} \]
over the sheaf of rings
\[ R_{U, \mb{R}}\mc{A} := R_U\mc{A} \otimes_{\mb{C}[\bs{u}]} \mb{C}[\bs{u}^{\mb{R}_{\geq 0}}],\]
where $\mb{C}[\bs{u}^{\mb{R}_{\geq 0}}]$ is the generalised polynomial algebra
\[ \mb{C}[\bs{u}^{\mb{R}_{\geq 0}}] = \mb{C}\textnormal{-span} \{\bs{u}^{\bs{\alpha}} \colonequals u_1^{\alpha_1} \cdots u_r^{\alpha_r} \mid \bs{\alpha} = (\alpha_1, \ldots, \alpha_r) \in \mb{R}_{\geq 0}^r\}\]
with the obvious multiplication and $\mb{R}^r$-grading for which $\bs{u}^{\bs{\alpha}}$ has degree $-\bs{\alpha}$. This defines an equivalence between the category of $\mc{A}$-modules equipped with an exhaustive $\mb{R}^r$-indexed filtration and the category of $\mb{R}^r$-graded $R_{U, \mb{R}}\mc{A}$-modules without $\mb{C}[\bs{u}^{\mb{R}_\geq 0}]$-torsion. By analogy with the definition for filtered modules, we will say that an $\mb{R}^r$-graded $R_{U, \mb{R}}\mc{A}$-module $\mc{N}$ (possibly with $\mb{C}[\bs{u}^{\mb{R}_{\geq 0}}]$-torsion) is \emph{wall and chamber} if, locally on $X$, there exists a set of walls $\mc{W}$ such that
\[ u^{\bs{\beta}} \colon \mc{N}^{\bs{\alpha}} \to \mc{N}^{\bs{\alpha} - \bs{\beta}} \]
is an isomorphism whenever $\bs{\alpha}$ and $\bs{\alpha} - \bs{\beta}$ are in the same chamber of $\mc{W}$. We say that $\mc{N}$ is \emph{good} if in addition, for all $\bs{\alpha} \in \mb{R}^r$, the $\mb{Z}^r$-graded $R_U\mc{A}$-module
\[ \bigoplus_{\bs{n} \in \mb{Z}^r} \mc{N}^{\bs{\alpha} + \bs{n}} \]
is coherent. We define the $\mb{R}^r$-filtered derived category
\[ \mrm{D}_\mb{R}(\mc{A}, U^\bullet\mc{A}) \colonequals \mrm{D}(\Mod_{\mb{R}^r}(R_{U, \mb{R}}\mc{A})) \]
and write
\[ \mrm{D}_{\mb{R}, \mrm{coh}}^b(\mc{A}, U^\bullet\mc{A}) \subset \mrm{D}^b_{\mb{R}, \mrm{wc}}(\mc{A}, U^\bullet\mc{A}) \subset \mrm{D}_{\mb{R}}(\mc{A}, U^\bullet\mc{A}) \]
for the full subcategories of bounded objects whose cohomology sheaves are good wall and chamber modules (resp.\ wall and chamber modules).

As in the $\mb{Z}^r$-indexed case, we will use notation for objects in $\mrm{D}_\mb{R}(\mc{A}, U^\bullet\mc{A})$ as if they were the Rees modules of $\mb{R}^r$-filtered complexes of $\mc{A}$-modules. So we write $(\mc{M}, U^\bullet\mc{M}) \in \mrm{D}_\mb{R}(\mc{A}, U^\bullet\mc{A})$, $R_{U, \mb{R}}\mc{M}$ for the underlying $\mb{R}^r$-graded $R_{U, \mb{R}}\mc{A}$-module, etc.

In the lemma below, we recall that an object $(\mc{M}, U^\bullet\mc{M}) \in \mrm{D}(\mc{A}, U^\bullet\mc{A})$ (resp.\ $\mrm{D}_\mb{R}(\mc{A}, U^\bullet\mc{A})$) is \emph{strict} if the cohomology sheaves $\mc{H}^i(R_U \mc{M})$ (resp.\ $\mc{H}^i(R_{U, \mb{R}}\mc{M})$) are without $\mb{C}[\bs{u}]$-torsion (resp.\ $\mb{C}[\bs{u}^{\mb{R}_{\geq 0}}]$-torsion) for all $i$. Equivalently, $(\mc{M}, U^\bullet\mc{M})$ is strict if the map
\[ \mc{H}^i(U^{\bs{\alpha}}\mc{M}) \to \mc{H}^i(U^{\bs{\beta}}\mc{M}) \]
is injective for all $i$ and all $\bs{\alpha} \geq \bs{\beta}$. In this case, we have an induced filtration $U^\bullet\mc{H}^i(\mc{M})$ such that $U^\bullet\mc{H}^i(\mc{M})$ satisfies that $R_U(\mc{H}^i(\mc{M}))= \mc{H}^i(R_U\mc{M})$ such that
\[ \frac{U^{\bs{\beta}}\mc{H}^i(\mc{M})}{U^{\bs{\alpha}}\mc{H}^i(\mc{M})} = \mc{H}^i(\mrm{Cone}(U^{\bs{\alpha}}\mc{M} \to U^{\bs{\beta}}\mc{M}))\]
for all $\bs{\alpha} \geq \bs{\beta}$.

\begin{lem} \label{lem:derived multivariate}
Let $(\mc{M}, U^\bullet\mc{M}) \in \mrm{D}^b_{\mb{R}, \mrm{coh}}(\ms{D}_X, V^\bullet\ms{D}_X)$. Assume that $\mc{W}$ is a compatible set of walls for $U^\bullet\mc{M}$ and that for $\bs{\alpha} \leq \bs{\beta} \in \mb{R}^r$ separated by a single wall $H = \bs{L}^{-1}(\gamma) \in \mc{W}$, there exist $\gamma_1, \ldots, \gamma_k \in \mb{C}$ with $\Re \gamma_j = \gamma$ such that the operator
\[ \prod_{j = 1}^k (\bs{L}(\bs{s}) + \gamma_j) \colon \mrm{Cone}(U^{\bs{\beta}}\mc{M} \to U^{\bs{\alpha}}\mc{M}) \to \mrm{Cone}(U^{\bs{\beta}}\mc{M} \to U^{\bs{\alpha}}\mc{M}) \]
is zero. Then $(\mc{M}, U^\bullet\mc{M})$ is strict and, for all $i$, the induced filtration $U^\bullet\mc{H}^i(\mc{M})$ is a multivariate $V$-filtration.
\end{lem}
\begin{proof}
Let us first show that $(\mc{M}, U^\bullet\mc{M})$ is strict, i.e.\ that $\mc{H}^i(R_{U, \mb{R}}\mc{M})$ has no $\mb{C}[\bs{u}^{\mb{R}_{\geq 0}}]$-torsion for all $i$. By assumption, $\mc{H}^i(R_{U, \mb{R}}\mc{M})$ is a good wall and chamber object. In particular, for all $\bs{\alpha} \in \mb{R}^r$, the $\mb{Z}^r$-graded $R_V\ms{D}_X$-module $\mc{H}^i(R_{U, \bs{\alpha} + \mb{Z}^r}\mc{M})$ given by restricting to the graded pieces with grading in $\bs{\alpha} + \mb{Z}^r \subset \mb{R}^r$ is coherent. So for each $\bs{\alpha}$, there exists, locally on $X$, $\bs{k} \in \mb{Z}_{\geq 0}^r$ such that $\bs{u}^{\bs{k}}$ annihilates all $\mb{C}[\bs{u}]$-torsion in $\mc{H}^i(R_{U, \bs{\alpha} + \mb{Z}^r}\mc{M})$. Since $\mc{H}^i(R_{U, \mb{R}}\mc{M})$ is a wall and chamber object, $\mc{H}^i(R_{U, \bs{\alpha} + \mb{Z}^r}\mc{M})$ depends only on the chamber of $\mc{W}$ containing $\bs{\alpha}$ modulo translations by $\mb{Z}^r$. Since there are only finite many of these, we can choose $\bs{k}$ independent of $\bs{\alpha}$, i.e.\ so that $\bs{u}^{\bs{k}}$ annihilates all $\mb{C}[\bs{u}^{\mb{R}_{\geq 0}}]$-torsion in $\mc{H}^i(R_{U, \mb{R}}\mc{M})$. We will deduce from this that in fact no such torsion can exist.

To do so, fix a maximal ideal $\mf{m} \subset \mb{C}[\bs{s}]$. Suppose that we have a torsion element $w \in \mc{H}^i(U^{\bs{\beta}}\mc{M})_{\mf{m}} \subset \mc{H}^i(R_{U, \mb{R}}\mc{M})_{\mf{m}}$ for some $\bs{\beta}\in \mb{R}^r$. Denote by $\sigma_{\mf{m}}$ the chamber of the localised set of walls $\mc{W}_\mf{m}$ containing $\bs{\beta}$. Since $\mc{W}_{\mf{m}}$ consists only of finitely many walls passing through a single point, we can choose $\bs{\gamma}\in \mb{R}^r$ so that both $\bs{\gamma},\bs{\gamma} - \bs{k}$ lie in $\sigma_\mf{m}$. By our assumption on $(\mc{M}, U^\bullet\mc{M})$, the proof of Lemma \ref{lem:localised chambers} gives isomorphisms
\begin{equation} \label{eq:derived multivariate 1}
\mc{H}^i(U^{\bs{\gamma}}\cM)_{\mf{m}} \cong \mc{H}^i(U^{\bs{\beta}}\cM)_{\mf{m}} \cong \mc{H}^i(U^{\bs{\gamma}-\bs{k}}\cM)_{\mf{m}},
\end{equation}
where the map from left to right is $\bs{u}^{\bs{k}}$. In particular, from the first isomorphism, $w$ yields an element $w_{\bs{\gamma}} \in \mc{H}^i(U^{\bs{\gamma}}\mc{M})_{\mf{m}}$ such that $\bs{u}^{\bs{n}}w_{\bs{\gamma}} = \bs{u}^{\bs{n} + \bs{\beta} - \bs{\gamma}}w$ for $\bs{n} \in \mb{Z}_{\geq 0}^r$ sufficiently large. Since $w$ is torsion, we conclude that $w_{\bs{\gamma}}$ is also torsion.
So $\bs{u}^{\bs{k}}w= 0$ and hence $w = 0$ by \eqref{eq:derived multivariate 1}. So $\mc{H}^i(U^{\bs{\beta}}\mc{M})_{\mf{m}}$ is without $\mb{C}[\bs{u}^{\mb{R}_{\geq 0}}]$-torsion for all $\mf{m}$, and hence so is $\mc{H}^i(U^{\bs{\beta}}\mc{M})$ itself. So $(\mc{M}, U^\bullet\mc{M})$ is strict as claimed.

To conclude that $U^\bullet\mc{H}^i(\mc{M})$ is a multivariate $V$-filtration, we note that, by strictness,
\[ R_{U, \mb{R}}\mc{H}^i(\mc{M}) = \mc{H}^i(R_{U, \mb{R}}\mc{M}) \]
is a good wall and chamber object, i.e.\ $U^\bullet\mc{H}^i(\mc{M})$ is a good wall and chamber filtration. We deduce that it is a multivariate $V$-filtration by Lemma \ref{lem:localised chambers}.
\end{proof}

We now consider the behaviour of the multivariate $V$-filtration under direct images. Let us first fix a context in which the relevant direct images make sense. Suppose we are given complex manifolds $X$, $Y$ and $Z$, a divisor $D \subset Z$ with simple normal crossings and a morphism $\pi \colon X \to Y$. Recall that if $\mc{M}$ is a $\ms{D}_{X \times Z}$-module, then the direct image along $\pi \times \id \colon X \times Z \to Y \times Z$ is defined by
\[ (\pi \times \id)_+\mc{M} = \mrm{R}(\pi \times \id)_*(\ms{D}_{Y \times Z \leftarrow X \times Z} \overset{\mrm{L}}\otimes_{\ms{D}_{Y \leftarrow X}} \mc{M}).\]
Since
\[ \ms{D}_{Y \times Z \leftarrow X \times Z} = (\ms{D}_{Y \leftarrow X} \boxtimes \ms{D}_Z) \otimes_{\ms{D}_X \boxtimes \ms{D}_Z} \ms{D}_{X \times Z}\]
this simplifies to 
\[ (\pi \times \id)_+\mc{M} = \mrm{R}(\pi \times \id)_*(\mrm{pr}_1^{-1}\ms{D}_{Y \leftarrow X} \overset{\mrm{L}}\otimes_{\mrm{pr}_1^{-1}\ms{D}_X} \mc{M})\]
where $\mrm{pr}_1$ is the first projection. The latter formula also makes sense at the level of $V$-filtrations. More precisely, define a $\mb{Z}^r$-indexed filtration $V^\bullet\ms{D}_{Y \times Z \leftarrow X \times Z}$ by
\[ V^\bullet\ms{D}_{Y \times Z \leftarrow X \times Z} \colonequals  (\ms{D}_{Y \leftarrow X} \boxtimes V^\bullet\ms{D}_Z) \otimes_{\ms{D}_X \boxtimes \ms{D}_Z} \ms{D}_{X \times Z}.\]
This filtration is compatible in the obvious sense with the $V$-filtrations on $\ms{D}_{X \times Z}$ and $\ms{D}_{Y \times Z}$ (along $X\times D$ and $Y\times D$, respectively) under the actions on the right and left, so we have a functor between filtered derived categories
\[ \pi_+ \colon \mrm{D}(\ms{D}_{X \times Z}, V^\bullet \ms{D}_{X \times Z}) \to \mrm{D}(\ms{D}_{Y \times Z}, V^\bullet\ms{D}_{Y \times Z}) \]
given by
\begin{align*}
 \pi_+(\mc{M}, U^\bullet) &:= \mrm{R}(\pi \times \id)_*((\ms{D}_{Y \times Z \leftarrow X \times Z}, V^\bullet) \overset{\mrm{L}}\otimes_{(\ms{D}_{Y \leftarrow X}, V^\bullet)} (\mc{M}, U^\bullet)) \\
 &= \mrm{R}(\pi \times \id)_*(\mrm{pr}_1^{-1}\ms{D}_{Y \leftarrow X} \overset{\mrm{L}}\otimes_{\mrm{pr}_1^{-1}\ms{D}_X} (\mc{M}, U^\bullet)).
 \end{align*}
Here the derived tensor product of filtered complexes is defined as a derived tensor product of the corresponding Rees modules. This makes sense for both $\mb{Z}^r$-indexed and $\mb{R}^r$-indexed filtrations.

In the analytic setting, it is unfortunately not known in general if the proper direct image of a coherent $\ms{D}$-module is again coherent. This is known to be the case under the technical assumption that the input is a \emph{good} coherent $\ms{D}$-module (see, e.g.\ \cite[Definition 2.2.6]{SabbahnoteDmodule}); we will not recall the definition here, but suffice it to say that the condition is automatic in the algebraic setting for arbitrary coherent $\ms{D}$-modules and in the analytic setting for holonomic $\ms{D}$-modules \cite[Remark 2.2.7]{SabbahnoteDmodule}. Under this assumption, we can show that the multivariate $V$-filtration also behaves well under direct images.

\begin{thm} \label{thm:multivariate V direct image}
Let $X$, $Y$ and $Z$ be complex manifolds, $D \subset Z$ a divisor with simple normal crossings and $\pi \colon X \to Y$ a morphism. Suppose that $\mc{M}$ is a good coherent $\ms{D}_{X \times Z}$-module with a multivariate $V$-filtration $V^\bullet\mc{M}$ along $X \times D$, and $\pi \times \id_Z$ is proper on the support of $\mc{M}$. Then the $\mb{R}^r$-filtered complex $\pi_+(\mc{M}, V^\bullet\mc{M})$ is strict, and the induced filtration $V^{\bullet}\mc{H}^i(\pi_+\mc{M})$ on $\mc{H}^i(\pi_+\mc{M})$ is a multivariate $V$-filtration along $Y \times D$ for all $n$. Moreover, any set of walls for $V^{\bullet}\cM$ is also a set of walls for $V^{\bullet}\cH^i(\pi_{+}\cM)$.
\end{thm}
\begin{proof}
We will show that the filtered complex $\pi_+(\mc{M}, V^\bullet\cM)$ satisfies the conditions of Lemma \ref{lem:derived multivariate}. It is clear from the construction that the cohomology sheaves $\mc{H}^i(R_{V, \mb{R}}\pi_+\mc{M})$ are wall and chamber objects such that any set of walls for $V^\bullet\mc{M}$ is a set of walls for $\mc{H}^i(R_{V, \mb{R}}\pi_+\mc{M})$. Note that, by definition, $R_{V, \mb{R}}\pi_+\mc{M}=\pi_+R_{V, \mb{R}}\mc{M}= \pi_{+}(\cM,V^{\bullet}\cM)$. Since the $\ms{D}_{X \times Z}$-module $\mc{M}$ is good, we may argue as in the proof of \cite[Theorem 7.5.2 (1)]{SabbahnoteDmodule} to conclude by Grauert's theorem that $\pi_+R_{V, \mb{R}}\mc{M}$ has bounded cohomology and that for all $\bs{\alpha} \in \mb{R}^r$, each $\mc{H}^i(\pi_+R_{V, \bs{\alpha} + \mb{Z}^r}\mc{M})$ is a coherent graded $R_V\ms{D}_{Y \times Z}$-module. So $\pi_+(\mc{M}, V^\bullet\mc{M}) \in \mrm{D}^b_{\mb{R}, \mrm{coh}}(\ms{D}_{Y \times Z}, V^\bullet\ms{D}_{Y \times Z})$.  Finally, since the direct image is defined by operations linear over $\mb{C}[\bs{s}]$, it follows immediately from Definition \ref{defn:multivariate V-filtration} \eqref{itm:multivariate V 4} that the assumption on the walls in Lemma \ref{lem:derived multivariate} is also satisfied. The conclusion of the theorem now follows.\end{proof}

\subsection{Restricted and localised $V$-filtrations} \label{subsec:extension}

In preparation for our study of multivariate $V$-filtrations on holonomic $\ms{D}$-modules and mixed Hodge modules, in this subsection we study  the behaviour of multivariate $V$-filtrations on localisations $\mc{M}(*D_i)$, and how to reconstruct the multivariate $V$-filtration from its restriction to $\bs{\alpha} \in \mb{R}_{\geq 0}^r$. We also introduce the $V_*$-filtration on a $\ms{D}_X(*D)$-module and explain its relation to the usual multivariate $V$-filtration.

The starting point is the following generalisation of an elementary property of the single-variable $V$-filtration.

\begin{prop} \label{prop:weak specialisation}
Let $V^\bullet\mc{M}$ be a multivariate $V$-filtration and $\bs{\alpha} \in \mb{R}^r$. For $\alpha_i > 0$, the morphism
\begin{equation} \label{eq:weak specialisation 1}
t_i \colon V^{\bs{\alpha}}\mc{M} \to V^{\bs{\alpha} + \bs{1}_i}\mc{M}
\end{equation}
is an isomorphism, and for $\alpha_i < 0$, we have 
\begin{equation} \label{eq:weak specialisation 2}
V^{\bs{\alpha}}\mc{M} = \partial_{t_i}(V^{\bs{\alpha} +  \bs{1}_i}\mc{M}) + V^{\bs{\alpha} + \epsilon \bs{1}_i}\mc{M}.
\end{equation}
\end{prop}
\begin{proof}
First note that if we fix any maximal ideal $\mf{m} = (s_1 + \beta_1, \ldots, s_r + \beta_r) \subset \mb{C}[\bs{s}]$ and set $\mf{m}' = (s_1 + \beta_1, \ldots, s_i + \beta_i + 1, \ldots, s_r + \beta_r)$, then we have morphisms
\begin{equation} \label{eq:weak specialisation 3}
t_i \colon (V^{\bs{\alpha}}\mc{M})_{\mf{m}} \to (V^{\bs{\alpha} + \bs{1}_i}\mc{M})_{\mf{m}'},
\end{equation}
and
\begin{equation} \label{eq:weak specialisation 4}
\partial_{t_i} \colon (V^{\bs{\alpha} + \bs{1}_i}\mc{M})_{\mf{m}'} \to (V^{\bs{\alpha}}\mc{M})_{\mf{m}}.
\end{equation}
If $\beta_i \neq 0$, then $\partial_{t_i}t_i = -s_i$ and $t_i\partial_{t_i} = -s_i - 1$ are invertible in $\mb{C}[\bs{s}]_{\mf{m}}$ and $\mb{C}[\bs{s}]_{\mf{m}'}$ respectively, so \eqref{eq:weak specialisation 3} and \eqref{eq:weak specialisation 4} are isomorphisms for any $\bs{\alpha}$.

Now, to prove that \eqref{eq:weak specialisation 1} is an isomorphism for $\alpha_i > 0$, it suffices to show that \eqref{eq:weak specialisation 3} is an isomorphism for any $\mf{m}$; by the remark above, we need only consider the case $\beta_i = 0$. 
Now, by the definition of good wall and chamber filtrations, there exists $B \in \mb{R}$ (locally on $X$) such that \eqref{eq:weak specialisation 1} (and hence \eqref{eq:weak specialisation 3}) is surjective if we replace $\bs{\alpha}$ with a vector $\bs{\alpha}'$ with $\alpha_i' > B$. Since $\alpha_i > 0$ and $\beta_i = 0$, for $T > 0$ sufficiently large, the vector $\bs{\alpha}' = T(\bs{\alpha} - \Re \bs{\beta}) + \Re \bs{\beta}$ will satisfy this condition. But $\bs{\alpha}$ and $\bs{\alpha}'$ lie in the same chamber of $\mc{W}_{\mf{m}}=\{H \in \mc{W}\mid \Re \bs{\beta}\in H\}$, so we conclude from Lemma \ref{lem:localised chambers} that \eqref{eq:weak specialisation 3} is surjective for $\bs{\alpha}$ itself.

To see that \eqref{eq:weak specialisation 3} is also injective, suppose that $m \in V^{\bs{\alpha}}\mc{M}$ is such that its image in $(V^{\bs{\alpha}}\mc{M})_{\mf{m}}$ satisfies $t_i m = 0$. Consider the submodule $\mc{N} = \ms{D}_X\cdot m \subset \mc{M}$. Since $t_i m = 0$, we clearly have that $V^{\bs{n}} \ms{D}_X \cdot m = 0$ if $n_i > 0$. By Proposition \ref{prop:V image and preimage}, $V^\bullet\mc{N} = V^\bullet\mc{M} \cap \mc{N}$ is a multivariate $V$-filtration on $\mc{N}$. Applying Lemma \ref{lem:multivariate V localisation} with $\mc{F} = \mc{O}_X\cdot m$, we conclude that $(V^{\bs{\alpha}}\mc{N})_\mf{m} = 0$, and hence that $m = 0$ in $(V^{\bs{\alpha}}\mc{M})_{\mf{m}}$.

Now suppose $\alpha_i < 0$ and consider \eqref{eq:weak specialisation 2}. This is clear after localising at any $\mf{m}$ with $\beta_i \neq 0$, since \eqref{eq:weak specialisation 4} is an isomorphism in this case. For the localisation at $\mf{m}$ with $\beta_i = 0$, again by definition of a good wall and chamber filtration there exists $A \in \mb{R}$ (locally on $X$) such that \eqref{eq:weak specialisation 2} holds if we replace $\bs{\alpha}$ with any vector $\bs{\alpha}'$ with $\alpha_i' < B$. Since $\alpha_i < 0$ and $\beta_i = 0$, this is satisfied by $\bs{\alpha}' = T(\bs{\alpha} - \Re \bs{\beta}) + \Re \bs{\beta}$ with $T > 0$ sufficiently large. Since $\bs{\alpha}'$ (resp.\ $\bs{\alpha}' + \epsilon \bs{1}$) lies in the chamber of $\mc{W}_{\mf{m}}$ containing $\bs{\alpha}$ (resp.\ $\bs{\alpha} + \epsilon \bs{1}$) by Lemma \ref{lem:localised chambers}, we conclude that
\[
(V^{\bs{\alpha}}\mc{M})_{\mf{m}} = \partial_{t_i}\left((V^{\bs{\alpha} + \bs{1}_i}\mc{M})_{\mf{m}'}\right) + (V^{\bs{\alpha} + \epsilon\bs{1}_i}\mc{M})_{\mf{m}} = \left(\partial_{t_i}(V^{\bs{\alpha} + \bs{1}_i}\mc{M}) + V^{\bs{\alpha} + \epsilon \bs{1}_i}\mc{M}\right)_{\mf{m}}.
\]
So \eqref{eq:weak specialisation 2} holds after localising at every maximal ideal in $\mb{C}[\bs{s}]$ and hence holds on the nose.
\end{proof}

\begin{rmk}
In the single-variable case, it is elementary to show that the morphism
\[ \partial_t \colon \gr_V^{\alpha + 1}\mc{M} \to \gr_V^\alpha\mc{M}\]
is an isomorphism for $\alpha < 0$; this implies \eqref{eq:weak specialisation 2} but is a priori stronger. We will prove a multivariate generalisation of this stronger statement at the end of this subsection (see Proposition \ref{prop:strong specialisation}).
\end{rmk}

Motivated by Proposition \ref{prop:weak specialisation}, we will sometimes consider multivariate $V$-filtrations in the ``localised'' setting of $\ms{D}_X(*D)$-modules. Consider the $\mb{Z}^r$-indexed filtration on $\ms{D}_X(*D)$ given by
\[ V^{\bs{n}}\ms{D}_X(*D) = t_1^{n_1} \cdots t_r^{n_r} V^{\bs{0}}\ms{D}_X(*D) \quad \text{for $\bs{n} \in \mb{Z}^r$}.\]
It is easy to see that the filtered ring $(\ms{D}_X(*D), V^\bullet\ms{D}_X(*D))$ satisfies Assumption \ref{assumption:ring filtrations}. We make the following definition for a coherent $\ms{D}_X(*D)$-module $\mc{M}$.

\begin{defn}\label{definition: V*}
Let $\mc{M}$ be a coherent $\ms{D}_X(*D)$-module. A \emph{(multivariate) $V_*$-filtration on $\mc{M}$} is an $\mb{R}^r$-indexed wall and chamber filtration $V_*^\bullet\mc{M}$ such that:
\begin{enumerate}
\item $V_*^\bullet\mc{M}$ is good over $V^\bullet\ms{D}_X(*D)$ (Definition \ref{defn: good filtration}).
\item \label{itm:V* 2} If $\mc{W}$ is a set of walls (Definition \ref{defn: set of walls}) for $V_*^\bullet\mc{M}$  and $\bs{\alpha} \leq \bs{\beta} \in \mb{R}^r$ are separated by a single wall $H = \bs{L}^{-1}(\gamma)$ of $\mc{W}$, then there exist $\gamma_1, \ldots, \gamma_k \in \mb{C}$ with $\Re \gamma_j = \gamma$ such that the operator $\prod_{j = 1}^k (\gamma_j + \bs{L}(\bs{s}))$ acts by zero on  $V^{\bs{\alpha}}_{*}\mc{M}/V^{\bs{\beta}}_{*}\mc{M}$. 
\end{enumerate}
\end{defn}

Since $V^{\bs{m}} \ms{D}_X(*D) \cdot V^{\bs{n}} \ms{D}_X(*D) = V^{\bs{m} + \bs{n}} \ms{D}_X(*D)$ as long as $m_i$ and $n_i$ have the same sign for all $i$, one can apply the exact same argument as Theorem \ref{thm:multivariate V uniqueness} to deduce:

\begin{thm} \label{thm:V* uniqueness}
A $V_*$-filtration on a coherent $\ms{D}_X(*D)$-module $\mc{M}$ is unique if it exists.
\end{thm}

Similarly, direct analogues of Corollary \ref{cor:V strictness} and Theorem \ref{thm:multivariate V direct image} hold for $V_*$-filtrations, with the same proofs.

\begin{notation}\label{notation: V* for arbitrary M}
If $\mc{M}$ is a coherent $\ms{D}_X$-module, we will sometimes write 
\[V^{\bullet}_*\mc{M} \colonequals V^{\bullet}_*\mc{M}(*D)\]
 for the unique $V_*$-filtration on the coherent $\ms{D}_X(*D)$-module $\mc{M}(*D)$, assuming it exists.
\end{notation}

The following observation is often helpful for computing the multivariate $V_*$-filtration.

\begin{prop} \label{prop:good V* criterion}
Let $\mc{M}$ be a coherent $\ms{D}_X(*D)$-module and $U^\bullet \mc{M}$ an exhaustive wall and chamber filtration such that $V^{\bs{n}}\ms{D}_X(*D) \cdot U^{\bs{\alpha}}\mc{M} \subset U^{\bs{\alpha} + \bs{n}}\mc{M}$ for all $\bs{n}, \bs{\alpha}$. Then $U^\bullet \mc{M}$ is good if and only if $U^{\bs{\alpha}}\mc{M}$ is coherent over $V^{\bs{0}}\ms{D}_X(*D) = V^{\bs{0}}\ms{D}_X$ for all $\bs{\alpha} \in \mb{R}^r$.
\end{prop}
\begin{proof}
Observe that the Rees algebra of $V^\bullet\ms{D}_X(*D)$ is given by
\[ R_V\ms{D}_{X}(*D) =\bigoplus_{\bs{n}\in \Z^r} V^{\bs{n}} \ms{D}_X(*D) \bs{u}^{-\bs{n}}=V^{\bs{0}}\ms{D}_X[\bs{v}, \bs{v}^{-1}]\]
where $v_i = u_i/t_i$.  So a $\mb{Z}^r$-graded $R_V\ms{D}_X(*D)$-module is coherent if and only if its zeroth graded piece is a coherent $V^{\bs{0}}\ms{D}_X$-module. The proposition therefore follows from Lemma \ref{lem: goodness part}.
\end{proof}

We have the following corollary of Proposition \ref{prop:weak specialisation}.

\begin{cor} \label{cor:localising V}
Assume that a coherent $\ms{D}_X$-module $\mc{M}$ admits a multivariate $V$-filtration $V^\bullet\mc{M}$. Then the $\ms{D}_X(*D)$-module $\mc{M}(*D)$ admits a multivariate $V_*$-filtration $V_*^\bullet\mc{M}(*D)$ satisfying
\[ V_*^{\bs{\alpha}}\mc{M}(*D) = V^{\bs{\alpha}}\mc{M} \quad \text{for $\bs{\alpha} \in \mb{R}_{> 0}^r$}.\]
\end{cor}
\begin{proof}
For $\bs{\alpha} \in \mb{R}^r$, define
\begin{equation}\label{eqn: def of V*} U_*^{\bs{\alpha}}\mc{M}(*D) = t_1^{-n_1}\cdots t_r^{-n_r}V^{\bs{\alpha} + \bs{n}}\mc{M} \end{equation}
for $\bs{n} \in \mb{Z}^r$ such that $\bs{\alpha} + \bs{n} \in \mb{R}_{> 0}^r$; by Proposition \ref{prop:weak specialisation}, this is independent of the choice of such an $\bs{n}$. The filtration $U_*^\bullet \mc{M}(*D)$ clearly satisfies the defining properties of $V_*^\bullet\mc{M}(*D)$, so the corollary follows.
\end{proof}

\begin{remark} \label{rmk:V versus V*}
When $r = 1$ we in fact have $V^\alpha\mc{M}(\ast D)=V_*^\alpha\mc{M}(\ast D)$ for any $\alpha$, provided both filtrations exist. Indeed, by the argument above, this is equivalent to the condition that $t V^\alpha\mc{M}(*D) = V^{\alpha + 1}\mc{M}(*D)$ for all $\alpha \in \mb{R}$; since the statement holds for $\alpha > 0$, this reduces to the claim that $t \colon \gr_V^\alpha \mc{M}(*D) \to \gr_V^{\alpha + 1}\mc{M}(*D)$ is an isomorphism for all $\alpha$, which in turn reduces to the standard fact that $t \colon \gr_V^0 \mc{M}(*D) \to \gr_V^1\mc{M}(*D)$ is an isomorphism. When $r > 1$, however, this usually fails: see Example \ref{example: diagonal embedding} below for an example where $V^{\bs{\alpha}}\mc{M}(*D) \neq V_*^{\bs{\alpha}}\mc{M}(*D)$.
\end{remark}

We will prove a converse to Corollary \ref{cor:localising V} below (see Corollary \ref{cor:extending V*}), which allows us to extend a $V_*$-filtration to a genuine multivariate $V$-filtration. The key technical tool is the following theory of restricted $V$-filtrations. Observe that by Proposition \ref{prop:weak specialisation}, in particular \eqref{eq:weak specialisation 2}, the multivariate $V$-filtration $V^\bullet\mc{M}$ is always generated by its restriction to the positive orthant $\bs{\alpha} \in \mb{R}_{\geq 0}^r$. Our key technical result below characterises precisely when an $\mb{R}_{\geq 0}^r$-indexed filtration is the restriction of a multivariate $V$-filtration on $\mc{M}$.

\begin{defn}
    Let $U_{\mathit{res}}^\bullet\mc{M}$ be a decreasing filtration indexed by $\mb{R}_{\geq 0}^r$. We say that $U_{\mathit{res}}^\bullet\mc{M}$ is a \emph{restricted wall and chamber filtration} if there exists a set of walls $\mc{W}_{\mathit{res}}$ in $\mb{R}^r$ (Definition \ref{defn: set of walls}), such that for every chamber $\sigma$ of $\mc{W}_{\mathit{res}}$, $U_{\mathit{res}}^{\bs{\alpha}}\mc{M}$ is constant for $\bs{\alpha} \in \sigma \cap \mb{R}_{\geq 0}^r$. We call a set $\mc{W}_{\mathit{res}}$ as above a \emph{set of restricted walls}.
\end{defn}

\begin{defn} \label{defn:restricted V-filtration}
We say that a restricted wall and chamber filtration $V^\bullet_{\mathit{res}}\mc{M}$ of $\mc{M}$ is a \emph{restricted multivariate $V$-filtration} if it satisfies:
\begin{enumerate}
\item \label{itm:restricted V-filtration 1} $V^{\bs{0}}_{\mathit{res}}\mc{M}$ generates $\mc{M}$ as a $\ms{D}_X$-module.
\item \label{itm:restricted V-filtration 2} Each $V^{\bs{\alpha}}_{\mathit{res}}\mc{M}$ is a coherent $V^{\bs{0}}\ms{D}_X$-submodule of $\mc{M}$.
\item \label{itm:restricted V-filtration 3} For all $\bs{\alpha} \in \mb{R}_{\geq 0}^r$, we have
\[ t_i (V^{\bs{\alpha}}_{\mathit{res}}\mc{M}) \subset V^{\bs{\alpha} + \bs{1}_i}_{\mathit{res}}\mc{M} \quad \text{and} \quad \partial_{t_i}(V^{\bs{\alpha} + \bs{1}_i}_{\mathit{res}}\mc{M}) \subset V^{\bs{\alpha}}_{\mathit{res}}\mc{M},\]
\item \label{itm:restricted V-filtration 4} There exists $N$ such that if $\alpha_i > N$ then the map
\[ t_i \colon V^{\bs{\alpha}}_{\mathit{res}}\mc{M} \to V^{\bs{\alpha} + \bs{1}_i}_{\mathit{res}}\mc{M} \]
is surjective.
\item \label{itm:restricted V-filtration 5} Locally on $X$, there exists a restricted set of walls $\mc{W}_{\mathit{res}}$ such that if $\bs{\alpha} \leq \bs{\beta}$ are separated by walls $\bs{L}_1^{-1}(\gamma_1)$, $\ldots$, $\bs{L}_N^{-1}(\gamma_N)$ in $\mc{W}_{\mathit{res}}$, then there exist $\gamma_{j, 1}, \ldots, \gamma_{j, M_j} \in \mb{C}$ with $\Re \gamma_{j, k} = \gamma_j$ such that the operator $\prod_{j = 1}^N\prod_{k = 1}^{M_j} (\gamma_{j, k} - \bs{L}_j(\bs{\partial_t t}))$ acts by zero on $V_{\mathit{res}}^{\bs{\alpha}}\mc{M}/V_{\mathit{res}}^{\bs{\beta}}\mc{M}$. 
\end{enumerate}
\end{defn}
\begin{rmk}
In the setting of the usual multivariate $V$-filtration, the analogue of the eigenvalue condition \eqref{itm:restricted V-filtration 5} is imposed only when $\bs{\alpha}$ and $\bs{\beta}$ are separated by a single wall. This implies the more complicated statement here, as, given $\bs{\alpha} \leq \bs{\beta} \in \mb{R}^r$, there exists a sequence $\bs{\gamma}^1 \leq \bs{\gamma}^2 \leq \cdots \leq \bs{\gamma}^k \in \mb{R}^r$ such that $\bs{\gamma}^1$ (resp.\ $\bs{\gamma}^k$) lies in the same chamber as $\bs{\alpha}$ (resp.\ $\bs{\beta}$) and for each $i$, $\bs{\gamma}^i$ and $\bs{\gamma}^{i + 1}$ are separated by a single wall. This can fail in the restricted setting, however: for example, if $r = 2$ and $\mc{W}_{\mathit{res}}$ contains the walls $\{\alpha_1 = 0\}$, $\{\alpha_2 = 0\}$ and $\{\alpha_1 + \alpha_2 = 0\}$, then $(0, 0)$ is separated by at least two walls from all other vectors in $\mb{R}^2_{\geq 0}$. So in this setting we must impose the more complicated condition from the start.
\end{rmk}

It is clear from the definitions (and Proposition \ref{prop:weak specialisation}) that the restriction of a multivariate $V$-filtration is a restricted multivariate $V$-filtration. Our main result on restricted multivariate $V$-filtrations is the following converse.

\begin{proposition} \label{prop:multivariate V extension}
Let $V_{\mathit{res}}^{\bullet}\mc{M}$ be a restricted multivariate $V$-filtration. Then the multivariate $V$-filtration $V^\bullet\mc{M}$ exists, and satisfies $V^{\bs{\alpha}}\mc{M} = V_{\mathit{res}}^{\bs{\alpha}}\mc{M}$ for $\bs{\alpha} \in \mb{R}_{\geq 0}^r$.
\end{proposition}
\begin{proof}
Define an $\mb{R}^r$-indexed filtration $U^\bullet\mc{M}$ by the formula
\begin{equation}\label{eqn: whole V filtration from Vres} U^{\bs{\alpha}}\mc{M} = \sum_{\substack{\bs{n} \in \mb{Z}^r, \bs{\beta} \in \mb{R}_{\geq 0}^r \\ \bs{n} + \bs{\beta} \geq \bs{\alpha}}} V^{\bs{n}}\ms{D}_X \cdot V^{\bs{\beta}}_{\mathit{res}}\mc{M}, \quad \textrm{for $\bs{\alpha} \in \mb{R}^r$}.\end{equation}
We will show that $U^{\bs{\alpha}}\mc{M} = V^{\bs{\alpha}}_{\mathit{res}}\mc{M}$ for $\bs{\alpha} \in \mb{R}_{\geq 0}^r$ and that $U^\bullet\mc{M}$ is a multivariate $V$-filtration.

Consider the first statement. For $\bs{\alpha} \in \mb{R}_{\geq 0}^r$, we clearly have $V^{\bs{\alpha}}_{\mathit{res}}\mc{M} \subset U^{\bs{\alpha}}\mc{M}$. To prove the converse, fix $\bs{\beta} \in \mb{R}_{\geq 0}^r$, $\bs{n} \in \mb{Z}^r$ such that $\bs{\beta} + \bs{n} \geq \bs{\alpha}$; we need to show that $V^{\bs{n}}\ms{D}_X \cdot V_{\mathit{res}}^{\bs{\beta}}\mc{M} \subset V_{\mathit{res}}^{\bs{\alpha}}\mc{M}$. If $n_i < 0$ for some $i$, then $\beta_i \geq \alpha_i - n_i \geq 1$, so
\begin{align*}
V^{\bs{n}}\ms{D}_X \cdot V_{\mathit{res}}^{\bs{\beta}}\mc{M} &= (V^{\bs{n} + \bs{1}_i}\ms{D}_X \cdot \partial_{t_i}+V^{\bs{n}+\bs{1}_i}\sD_X)\cdot V_{\mathit{res}}^{\bs{\beta}}\mc{M} \\
&=V^{\bs{n} + \bs{1}_i}\ms{D}_X \cdot (\partial_{t_i} \cdot V_{\mathit{res}}^{\bs{\beta}}\mc{M}) +V^{\bs{n}+\bs{1}_i}\sD_X\cdot V_{\mathit{res}}^{\bs{\beta}}\mc{M}\\
&\subset V^{\bs{n} + \bs{1}_i}\ms{D}_X \cdot V_{\mathit{res}}^{\bs{\beta} - \bs{1}_i}\mc{M}. \end{align*}
Thus we can reduce to the case where $n_i \geq 0$ for all $i$. Similarly, if $n_i \geq 1$, we have
\[ V^{\bs{n}}\ms{D}_X \cdot V_{\mathit{res}}^{\bs{\beta}}\mc{M} = V^{\bs{n} - \bs{1}_i}\ms{D}_X\cdot (t_i \cdot V_{\mathit{res}}^{\bs{\beta}}\mc{M}) \subset V^{\bs{n} - \bs{1}_i}\ms{D}_X \cdot V_{\mathit{res}}^{\bs{\beta} + \bs{1}_i}\mc{M}.
\]
So we may reduce further to the case $\bs{n} = \bs{0}$, which is clear.

We now prove that $U^\bullet\mc{M}$ is a multivariate $V$-filtration. Since $V^{\bs{0}}_{\mathit{res}}\mc{M}$ generates $\mc{M}$, $U^\bullet\mc{M}$ is exhaustive. If  $U^\bullet\mc{M}$ is a wall and chamber filtration, it is clear from the definition that $U^\bullet\mc{M}$ is good over $V^{\bullet}\sD_X$ as in Definition \ref{defn: good filtration}.  It remains to establish the wall-chamber property and to verify that condition \eqref{itm:multivariate V 4} in Definition \ref{defn:multivariate V-filtration} holds.

First, let us rewrite the definition of $U^\bullet\mc{M}$ slightly. For $\bs{\alpha} \in \mb{R}^r$, let $\bs{\alpha}^+$ be the vector with coordinates $\alpha_i^+ = \max(0, \alpha_i)$. Then 
\[ \sum_{\bs{\beta} \in \mb{R}_{\geq 0}^r,\, \bs{\beta} \geq \bs{\alpha}} V_{\mathit{res}}^{\bs{\beta}}\mc{M} = V_{\mathit{res}}^{\bs{\alpha}^+} \mc{M},\]
and hence
\begin{equation} \label{eq:multivariate V extension 2}
U^{\bs{\alpha}}\mc{M} = \sum_{\bs{n} \in \mb{Z}^r} V^{\bs{n}}\ms{D}_X \cdot V_{\mathit{res}}^{(\bs{\alpha} - \bs{n})^+}\mc{M}.
\end{equation}
Thus, $U^{\bs{\alpha}}\mc{M} \neq U^{\bs{\alpha} + \epsilon \bs{1}}\mc{M}$ only if $V_{\mathit{res}}^{(\bs{\alpha} - \bs{n})^+}\mc{M} \neq V_{\mathit{res}}^{(\bs{\alpha} - \bs{n} + \epsilon \bs{1})^+}\mc{M}$ for some $\bs{n} \in \mb{Z}^r$, which in turn holds only if $(\bs{\alpha} - \bs{n})^+$ lies on a wall in $\mc{W}_{\mathit{res}}$. It follows that $U^{\bs{\alpha}}\mc{M}$ is indeed a wall and chamber filtration, with a set of walls given by
\begin{equation} \label{eq:multivariate V extension 4}
\mc{W} = \{ (\bs{L}^I)^{-1}\left(\beta + \bs{L}^I(\bs{n})\right) \mid \bs{L}^{-1}(\beta) \in \mc{W}_{\mathit{res}}, I \subset \{1, \ldots, r\}, \bs{n} \in \mb{Z}^r\}
\end{equation}
where we write $\bs{L}^I$ for the linear form with coordinates
\[ (L^{I})_{i} = \begin{cases} L_i, & \text{if $i \in I$}, \\ 0, & \text{otherwise}.\end{cases}\]

Finally, we show that $U^\bullet\mc{M}$ satisfies Definition \ref{defn:multivariate V-filtration}, \eqref{itm:multivariate V 4}. Fix $\bs{\alpha} \in \mb{R}^r$ and consider $U^{\bs{\alpha}}\mc{M}/U^{\bs{\alpha} + \epsilon \bs{1}}\mc{M}$. We need to show that this is annihilated by an operator of the form
\begin{equation} \label{eq:multivariate V extension 1}
 \prod_{j = 1}^N\prod_{k = 1}^{M_j} (\gamma_{j, k} + \bs{L}_j(\bs{s}))
\end{equation}
where $\bs{L}_j^{-1}(\gamma_j)$ are the walls of $\mc{W}$ containing $\bs{\alpha}$ and $\Re \gamma_{j, k} = \gamma_j$. If $\bs{\alpha} \in \mb{R}_{\geq 0}^r$, then this holds by definition since $U^{\bs{\alpha}}\mc{M} = V^{\bs{\alpha}}_{\mathit{res}}\mc{M}$. So suppose that $\alpha_i < 0$ for some $i$. We argue inductively on the number of such $i$ and the sizes of the $\lfloor \alpha_i\rfloor$.

Observe that, since $\alpha_i < 0$, we have 
 \begin{equation}\label{eqn: Ualpha with a negative entry} U^{\bs{\alpha}}\mc{M} = \partial_{t_i}(U^{\bs{\alpha} + \bs{1}_i}\mc{M}) + U^{\bs{\alpha} + \bs{1}_i}\mc{M}+ U^{\bs{\alpha} - \alpha_i \bs{1}_i}\mc{M}.\end{equation}
Indeed, it is clear that the right hand side is contained in the left, and by \eqref{eq:multivariate V extension 2} that
\begin{align*}
U^{\bs{\alpha}}\mc{M} &= \sum_{\bs{n} \in \mb{Z}^r, n_i < 0} V^{\bs{n}}\ms{D}_X \cdot V^{(\bs{\alpha} - \bs{n})^+}_{\mathit{res}}\mc{M} + \sum_{\bs{n} \in \mb{Z}^r, n_i \geq 0} V^{\bs{n}}\ms{D}_X \cdot V_{\mathit{res}}^{(\bs{\alpha} - \bs{n})^+}\mc{M} \\
&=  \sum_{\bs{n} \in \mb{Z}^r, n_i < 0} \partial_{t_i}\cdot (V^{\bs{n} + \bs{1}_i}\ms{D}_X \cdot V_{\mathit{res}}^{(\bs{\alpha} - \bs{n})^+}\mc{M}) + \sum_{\bs{n} \in \mb{Z}^r, n_i < 0} V^{\bs{n} + \bs{1}_i}\ms{D}_X \cdot V_{\mathit{res}}^{(\bs{\alpha} - \bs{n})^+}\mc{M} \\
&+\sum_{\bs{n} \in \mb{Z}^r, n_i \geq 0} V^{\bs{n}}\ms{D}_X \cdot V_{\mathit{res}}^{(\bs{\alpha} - \alpha_i \bs{1}_i - \bs{n})^+}\mc{M} \\
&\subset \partial_{t_i}(U^{\bs{\alpha} + \bs{1}_i}\mc{M}) + U^{\bs{\alpha} + \bs{1}_i}\mc{M}+ U^{\bs{\alpha} - \alpha_i \bs{1}_i}\mc{M}.
\end{align*}
By induction on the size of $\lfloor\alpha_i \rfloor$, it is clear that the image of
\begin{equation} \label{eq:multivariate V extension 3}
\partial_{t_i} \colon \frac{U^{\bs{\alpha} + \bs{1}_i}\mc{M}}{U^{\bs{\alpha} + \epsilon\bs{1} + \bs{1}_i}\mc{M}} \to \frac{U^{\bs{\alpha}}\mc{M}}{U^{\bs{\alpha} + \epsilon\bs{1} }\mc{M}}
\end{equation}
is annihilated by an operator of the form \eqref{eq:multivariate V extension 1}. So it remains to check this also for the cokernel of \eqref{eq:multivariate V extension 3}. Since by \eqref{eqn: Ualpha with a negative entry} this cokernel is a quotient of $U^{\bs{\alpha} - \alpha_i \bs{1}_i}\mc{M}$, by induction on the number of $i$ such that $\alpha_i < 0$, it is annihilated by an operator
\[ T = \prod_{j = 1}^N \prod_{k = 1}^{M_j} (\gamma_{j, k} + \bs{L}_{j}^{I_j}(\bs{s})),\]
where the $(\bs{L}_{j}^{I_j})^{-1}(\gamma_j)$ are the walls of $\mc{W}$ containing $\bs{\alpha} - \alpha_i \bs{1}_i$ (here $\bs{L}_j^{-1}(\gamma_j)$ is an integer translate of a wall of $\mc{W}_{\mathit{res}}$ and $I_j \subset \{1, \ldots, r\}$ as in \eqref{eq:multivariate V extension 4}) and $\Re \gamma_{j, k} = \gamma_j$. 
Now, recalling that $s_i = -\partial_{t_i}t_i$, we have
\[ s_i(U^{\bs{\alpha} - \alpha_i \bs{1}_i}\mc{M}) \subset \partial_{t_i}(U^{\bs{\alpha} + (1 - \alpha_i)\bs{1}_i}\mc{M}) \subset \partial_{t_i}(U^{\bs{\alpha} + \bs{1}_i}\mc{M}).\]
So $s_i$ acts by $0$ on the cokernel of \eqref{eq:multivariate V extension 3}. Hence, as operators here, we have
\[ T = \prod_{j = 1}^N \prod_{k = 1}^{M_j} (\gamma_{j, k} + \bs{L}_{j}^{I_j \setminus \{i\}}(\bs{s})).\]
But now we have 
\[ \bs{L}_{j}^{I_j \setminus\{i\}}(\bs{\alpha}) = \bs{L}_{j}^{ I_j \setminus\{i\}}(\bs{\alpha} -\alpha_i \bs{1}_i) = \bs{L}_{j}^{I_j}(\bs{\alpha} -\alpha_i \bs{1}_i)=\gamma_j,\]
so $(\bs{L}_{j}^{I_j \setminus\{i\}})^{-1}(\gamma_j)$ is a wall containing $\bs{\alpha}$, so we are done.
\end{proof}

\begin{corollary}\label{cor: V filtration for vector with integer coefficients}
Let $V^{\bullet}\cM$ be a multivariate $V$-filtration. Then for $\bs{n}\in \Z^r$, one has
\[ V^{\bs{n}}\cM= V^{\bs{n}-\bs{n}^{+}}\sD_X\cdot V^{\bs{n}^{+}}\cM, \quad \textrm{where $\bs{n}^{+}_i=\max(n_i,0)$.}\]
\end{corollary}
\begin{proof}
If $n_i\geq 0$, there is nothing to prove. If $n_i<0$, we prove it by induction on $\sum_i |n_i|$. It suffices to show that
\[ V^{\bs{n}}\cM \subset V^{\bs{n}-\bs{n}^{+}}\sD_X\cdot V^{\bs{n}^{+}}\cM\]
By Proposition \ref{prop:multivariate V extension}, $V^{\bullet}\cM$ can be computed using \eqref{eqn: whole V filtration from Vres}. Then by \eqref{eqn: Ualpha with a negative entry}, we have
 \[ V^{\bs{n}}\mc{M} = \partial_{t_i}(V^{\bs{n} + \bs{1}_i}\mc{M}) + V^{\bs{n} + \bs{1}_i}\mc{M}+ V^{\bs{n} - n_i \bs{1}_i}\mc{M}.\]
By induction on $\sum_i |n_i|$ and $(\bs{n}+\bs{1}_i)^{+}=\bs{n}^+$, one has
\begin{align*}
\d_{t_i}(V^{\bs{n}+\bs{1}_i}\cM)+V^{\bs{n} + \bs{1}_i}\mc{M}\subset &\d_{t_i}\cdot V^{\bs{n}-\bs{n}^{+}+\bs{1}_i}\sD_X\cdot V^{\bs{n}^{+}}\cM+V^{\bs{n}-\bs{n}^{+}+\bs{1}_i}\sD_X\cdot V^{\bs{n}^{+}}\cM\\
\subset & V^{\bs{n}-\bs{n}^{+}}\sD_X\cdot V^{\bs{n}^{+}}\cM.\end{align*}
Note that $n_i<0$, so $(\bs{n} - n_i\bs{1}_i)^{+}=\bs{n}^{+}$ and we also have
\[ V^{\bs{n} - n_i \bs{1}_i}\mc{M}=V^{\bs{n} - n_i \bs{1}_i-\bs{n}^{+}}\sD_X\cdot V^{\bs{n}^{+}}\cM\subset V^{\bs{n} -\bs{n}^{+}}\sD_X\cdot V^{\bs{n}^{+}}\cM\]
This finishes the proof.
\end{proof}

As a consequence of Proposition \ref{prop:multivariate V extension}, we have the following criterion for a $\ms{D}$-module to be equal to its localisation along one of the components $D_i \subset D$. 

\begin{prop} \label{prop:* criterion}
Let $\mc{M}$ be a coherent $\ms{D}_X$-module admitting a multivariate $V$-filtration $V^\bullet\mc{M}$. Then $\mc{M} = \mc{M}(*D_i)$ if and only if the morphism
\begin{equation} \label{eq:* criterion 1}
t_i \colon V^{\bs{0}}\mc{M} \to V^{\bs{1}_i}\mc{M}
\end{equation}
is an isomorphism.
\end{prop}
\begin{proof}
To prove the ``if'' direction, suppose that \eqref{eq:* criterion 1} is an isomorphism. Consider the kernel $\mc{K}$ of the morphism $\mc{M} \to \mc{M}(*D_i)$; this is a coherent submodule of $\mc{M}$ and therefore has a multivariate $V$-filtration $V^{\bs{\alpha}}\mc{K} = V^{\bs{\alpha}}\mc{M} \cap \mc{K}$ by Proposition \ref{prop:V image and preimage}. Since \eqref{eq:* criterion 1} is an isomorphism by assumption, it follows by Proposition \ref{prop:weak specialisation} that $V^{\bs{0}}\mc{M}$ has no $t_i^n$-torsion for any $n > 0$. Since every element of $\mc{K}$ is annihilated by $t_i^n$ for some $n$ by construction, we therefore have $V^{\bs{0}}\mc{K}= 0$ and hence $\mc{K} = 0$ by Proposition \ref{prop:weak specialisation} again. So we have $\mc{M} \subset \mc{M}(*D_i)$, i.e.\ $t_i$ acts injectively on $\mc{M}$.

To conclude that $\mc{M} = \mc{M}(*D_i)$, we claim that the injective morphism
\[ t_i \colon V^{\bs{\alpha}}\mc{M} \to V^{\bs{\alpha}+ \bs{1}_i}\mc{M} \]
is also surjective for all $\bs{\alpha} \in \mb{Z}_{\leq 0}^r$ with $\alpha_i = 0$. Taking the union over all such $\bs{\alpha}$ and applying Proposition \ref{prop:V-filtration for sub divisor}, this implies that
\[ t_i \colon V_{D_i}^0\mc{M} \to V_{D_i}^1\mc{M} \]
is an isomorphism, where $V_{D_i}$ denotes the usual (single variable) $V$-filtration along $D_i$; it is a standard (and elementary) fact that this is equivalent to $\mc{M} = \mc{M}(*D_i)$.

We prove the claim by induction on $\sum_i |\alpha_i|$. The claim is true by assumption when $\bs{\alpha} = \bs{0}$. So suppose by induction that we are given $\bs{\alpha}$ as above and $j \neq i$ with $\alpha_j \leq -1$ such that
\[ t_i \colon V^{\bs{\alpha} + \bs{1}_j}\mc{M} \to V^{\bs{\alpha} + \bs{1}_i + \bs{1}_j}\mc{M} \]
is surjective. Then by \eqref{eq:weak specialisation 2}, we have
\begin{align*}
 V^{\bs{\alpha} + \bs{1}_i}\mc{M} &= \partial_{t_j}(V^{\bs{\alpha} + \bs{1}_i + \bs{1}_j}\mc{M}) + V^{\bs{\alpha} + \bs{1}_i + \epsilon\bs{1}_j}\mc{M}  \\
 &= t_i(\partial_{t_j}(V^{\bs{\alpha} + \bs{1}_j}\mc{M}) + V^{\bs{\alpha} + \epsilon\bs{1}_j}\mc{M}) = t_iV^{\bs{\alpha}}\mc{M}
 \end{align*}
as well. So by induction, the claim is proved.

Now consider the ``only if'' direction, i.e.\ suppose that $\mc{M} = \mc{M}(*D_i)$. Define a restricted wall and chamber filtration on $\mc{M}$ by $U_{\mathit{res}}^{\bs{\alpha}}\mc{M} = t_i^{-1}V^{\bs{\alpha} + \bs{1}_i}\mc{M}$. Clearly $V^{\bs{\alpha}}\mc{M} \subset U_{\mathit{res}}^{\bs{\alpha}}\mc{M}$ for $\bs{\alpha} \in \mb{R}_{\geq 0}$, so in particular $U^{\bs{0}}_{\mathit{res}}\mc{M}$ generates $\mc{M}$ as a $\ms{D}_X$-module since $V^{\bs{0}}\mc{M}$ does (by Proposition \ref{prop:weak specialisation}). So it is immediate from the definitions and Proposition \ref{prop:weak specialisation} that $U_{\mathit{res}}^{\bs{\alpha}}\mc{M}$ is a restricted $V$-filtration on $\mc{M}$. So in fact $U_{\mathit{res}}^{\bs{\alpha}}\mc{M} = V^{\bs{\alpha}}\mc{M}$ for $\bs{\alpha} \in \mb{R}_{\geq 0}^r$ by Proposition \ref{prop:multivariate V extension}. We conclude the desired surjectivity, as $t_i \colon U_{\mathit{res}}^{\bs{0}}\mc{M}=V^{\bs{0}}\cM \to V^{\bs{1}_i}\mc{M}$ is an isomorphism by construction.
\end{proof}

\begin{cor} \label{cor:extending V*}
Assume that $\mc{M}(*D)$ is a coherent $\ms{D}_X(*D)$-module admitting a multivariate $V_*$-filtration $V^{\bullet}_*\mc{M}(*D)$. Then $\mc{M}(*D)$ is coherent as a $\ms{D}_X$-module and admits a multivariate $V$-filtration satisfying
\[ V^{\bs{\alpha}}\mc{M}(*D) = V_*^{\bs{\alpha}}\mc{M}(*D) \quad \text{for $\bs{\alpha} \in \mb{R}^r_{\geq 0}$}.\]
\end{cor}
\begin{proof}
Define $\mc{N} = \ms{D}_X \cdot V^{\bs{0}}_*\mc{M}(*D) \subset \mc{M}(*D)$; this is coherent by construction. Then $\{V^{\bs{\alpha}}_*\mc{M}(*D)\}_{\bs{\alpha} \in \mb{R}_{\geq 0}^r}$ defines a restricted $V$-filtration on $\mc{N}$, so by Proposition \ref{prop:V image and preimage} we conclude that $\mc{N}$ admits a multivariate $V$-filtration $V^\bullet\mc{N}$ satisfying 
\[ V^{\bs{\alpha}}\mc{N} = \mc{N}\cap V^{\bs{\alpha}}\cM(\ast D)= V_*^{\bs{\alpha}}\mc{M}(*D), \quad \textrm{for $\bs{\alpha} \in \mb{R}_{\geq 0}^r$}.\]
But since $t_i^{-1} \in V^{-\bs{1}_i}\ms{D}_X(*D)$, we have that $t_i \colon V^{\bs{0}}\mc{N} \to V^{\bs{1}_i}\mc{N}$ is an isomorphism for all $i$, by Definition \ref{definition: V*} (1). So by Proposition \ref{prop:* criterion}, we have $\mc{N} = \mc{N}(*D) = \mc{M}(*D)$, so we are done.
\end{proof}

Arguing in a similar way, we have:

\begin{cor} \label{cor:localising at a component}
Let $\mc{M}$ be a coherent $\ms{D}_X$-module admitting a multivariate $V$-filtration along $D$. Then $\mc{M}(*D_i)$ is also coherent and admits a multivariate $V$-filtration along $D$.
\end{cor}

\begin{example}\label{example: diagonal embedding}
Let us illustrate the passage between $V_{\mathit{res}}^\bullet$, $V^\bullet$ and $V_*^\bullet$ in an example. Consider the case where $X = \mb{C}^2$, $t_1$ and $t_2$ are the coordinates and
\[ \mc{M} = \frac{\mc{O}_{\mb{C}^2}[(t_1 - t_2)^{-1}]}{\mc{O}_{\mb{C}^2}} \]
is the direct image of the structure sheaf on $\mb{C}$ under the diagonal embedding. We claim that
\[ V^{\alpha_1, \alpha_2}_{\mathit{res}}\mc{M} \colonequals \mc{O}_{\mb{C}^2}\textnormal{-span}\left\{\left.\frac{t_1^m t_2^n}{(t_1 - t_2)^k} \, \right|\, m + n - k \geq \alpha_1 + \alpha_2 - 2, \,m, n, k \in \mb{Z}_{\geq 0}\right\}\]
defines a restricted multivariate $V$-filtration on $\mc{M}$, along $D=\textrm{div}(t_1t_2)$. To see this, let us check the conditions of Definition \ref{defn:restricted V-filtration}. First, note that $V^\bullet_{\mathit{res}}\mc{M}$ is a restricted wall and chamber filtration with restricted set of walls
\[ \mc{W}_{\mathit{res}} = \{H_n = \{\alpha_1 + \alpha_2 = n\} \mid n \in \mb{Z}_{\geq 0}\}.\]
We have
\[ \mc{M} = \ms{D}_{\mb{C}^2} \cdot \frac{1}{(t_1 - t_2)^2} = \ms{D}_{\mb{C}^2} \cdot V^{0, 0}_{\mathit{res}}\mc{M},\]
so \eqref{itm:restricted V-filtration 1} is satisfied. Moreover,
\[ V^{\alpha_1, \alpha_2}_{\mathit{res}}\mc{M} = \begin{cases} V^{0, 0}\ms{D}_{\mb{C}^2}  \cdot \frac{t_1^{\lceil \alpha_1 + \alpha_2\rceil - 1}}{t_1 - t_2}, & \text{if $\alpha_1 + \alpha_2 > 0$}, \\ V^{0, 0}\ms{D}_{\mb{C}^2}\cdot \frac{1}{(t_1 - t_2)^2}, & \text{if $\alpha_1 = \alpha_2 = 0$},\end{cases} \]
is a coherent $V^{0, 0}\ms{D}_{\mb{C}^2}$-module, so \eqref{itm:restricted V-filtration 2} is also satisfied. Property \eqref{itm:restricted V-filtration 3} is immediate. To see \eqref{itm:restricted V-filtration 4}, it suffices to check that for $\alpha_2 > 1$,
\[ \frac{t_1^m}{(t_1 - t_2)^k} \in t_2V_{\mathit{res}}^{\alpha_1, \alpha_2 - 1}\mc{M} \]
as long as $m - k \geq \alpha_1 + \alpha_2 - 2$, and similarly with $t_1$ and $t_2$ interchanged. This follows from the calculation
\begin{align*}
\frac{t_1^m}{(t_1 - t_2)^k} = \frac{((t_1 - t_2) + t_2)^m}{(t_1 - t_2)^k} & = \sum_{i = 0}^m \binom{m}{i} \frac{(t_1 - t_2)^it_2^{m - i}}{(t_1 - t_2)^k} \\
&= t_2 \sum_{i = 0}^{k - 1} \binom{m}{i} \frac{t_2^{m - i - 1}}{(t_1 - t_2)^{k - i}} \in t_2 V^{\alpha_1, \alpha_2 - 1}_{\mathit{res}}\mc{M}.
\end{align*}
Here we have used the fact that $m - k \geq \alpha_1 + \alpha_2 - 2 > - 1$ and hence $m \geq k$. Finally, to see \eqref{itm:restricted V-filtration 5}, note that if $(\alpha_1, \alpha_2) \leq (\beta_1, \beta_2)$ are separated by walls $H_i, H_{i + 1}, \cdots, H_{j}$, then $V^{\alpha_1, \alpha_2}_{\mathit{res}}\mc{M}/V^{\beta_1, \beta_2}_{\mathit{res}}\mc{M}$ is spanned over $\mb{C}$ by
\[\left\{ \frac{t_1^mt_2^n}{(t_1 - t_2)^k} \mid i - 2 \leq m + n - k  \leq j - 2\right\}.\]
Since
\[ (s_1 + s_2)\frac{t_1^mt_2^n}{(t_1 - t_2)^k} = -(t_1 \partial_{t_1} + t_2 \partial_{t_2} + 2)\frac{t_1^mt_2^n}{(t_1 - t_2)^k} = -(m + n - k + 2) \frac{t_1^m t_2^n}{(t_1 - t_2)^k}, \]
it follows that $V^{\alpha_1, \alpha_2}_{\mathit{res}}\mc{M}/V^{\beta_1, \beta_2}_{\mathit{res}}\mc{M}$ is annihilated by $(s_1 + s_2 + i) \cdots (s_1 + s_2 + j)$ and hence \eqref{itm:restricted V-filtration 5} is satisfied.

According to Proposition \ref{prop:multivariate V extension}, the full multivariate $V$-filtration $V^\bullet\mc{M}$ therefore exists and satisfies
\[ V^{\alpha_1, \alpha_2}\mc{M} = V^{\alpha_1, \alpha_2}_{\mathit{res}}\mc{M} \quad \text{for $\alpha_1 \geq 0$ and $\alpha_2 \geq 0$}.\]
The proof of the proposition provides the set of walls
\[ \mc{W} = \{ \{\alpha_1 + \alpha_2 = n\} \mid n \in \mb{Z} \} \cup \{\{\alpha_1 = n\} \mid n \in \mb{Z}\}\cup \{\{\alpha_2 = n\} \mid n \in \mb{Z}\}.\]
Figure \ref{fig:diagonal embedding} depicts the set of walls $\mc{W}$ (in dashed lines) together with the actual points at which $V^\bullet\mc{M}$ jumps (solid lines) and generators for $V^{\alpha_1, \alpha_2}\mc{M}$ over $V^{0, 0}\ms{D}_{\mb{C}^2}$ in the different regions.

\begin{figure}

\begin{tikzpicture}
\draw[thick] (3, 5) -- (5, 3);
\draw[thick] (1, 5) -- (5, 1);
\draw[thick] (-1, 5) -- (5, -1);
\draw[thick] (-3, 5) -- (5, -3);
\draw[thick] (-5, 5) -- (5, -5);
\draw[thick] (-5, 3) -- (3, -5);
\draw[thick] (-5, 1) -- (1, -5);
\draw[thick] (-5, -1) -- (-1, -5);
\draw[thick] (-5, -3) -- (-3, -5);

\draw[thick] (-5, 4) -- (0, 4);
\draw[dashed, opacity=0.5] (0, 4) -- (5, 4);
\draw[thick] (-5, 2) -- (0, 2);
\draw[dashed, opacity=0.5] (0, 2) -- (5, 2);
\draw[dashed, opacity=0.5] (-5, 0) -- (5, 0);
\draw[dashed, opacity=0.5] (-5, -2) -- (5, -2);
\draw[dashed, opacity=0.5] (-5, -4) -- (5, -4);

\draw[dashed, opacity=0.5] (4, 5) -- (4, 0);
\draw[thick] (4, 0) -- (4, -5);
\draw[dashed, opacity=0.5] (2, 5) -- (2, 0);
\draw[thick] (2, 0) -- (2, -5);
\draw[dashed, opacity=0.5] (0, 5) -- (0, -5);
\draw[dashed, opacity=0.5] (-2, 5) -- (-2, -5);
\draw[dashed, opacity=0.5] (-4, 5) -- (-4, -5);

\draw[fill=black] (0, 0) circle (0.05);
\node at (0.35, 0.25) {\tiny $(0, 0)$};

\node at (3.5, 3.5) {\small $\frac{t_1^3}{t_1 - t_2}$};
\node at (2.5, 2.5) {\small $\frac{t_1^2}{t_1 - t_2}$};
\node at (1.5, 1.5) {\small $\frac{t_1}{t_1 - t_2}$};
\node at (0.5, 0.75) {\small $\frac{1}{t_1 - t_2}$};
\node at (-0.6, -0.5) {\small $\frac{1}{(t_1 - t_2)^2}$};
\node at (-1.4, -1.5) {\small $\frac{1}{(t_1 - t_2)^3}$};
\node at (-2.6, -2.5) {\small $\frac{1}{(t_1 - t_2)^4}$};
\node at (-3.4, -3.5) {\small $\frac{1}{(t_1 - t_2)^5}$};
\node at (2.8, -1.5) {\small $\frac{t_1}{(t_1 - t_2)^2}$};
\node at (2.8, -3.5) {\small $\frac{t_1}{(t_1 - t_2)^3}$};
\node at (-1.4, 2.7) {\small $\frac{t_2}{(t_1 - t_2)^2}$};
\node at (-3.4, 2.7) {\small $\frac{t_2}{(t_1 - t_2)^3}$};
\end{tikzpicture}

\caption{The multivariate $V$-filtration of the diagonal in $\mb{C}^2$: $V^{\bullet}\cM$}
\label{fig:diagonal embedding}
\end{figure}
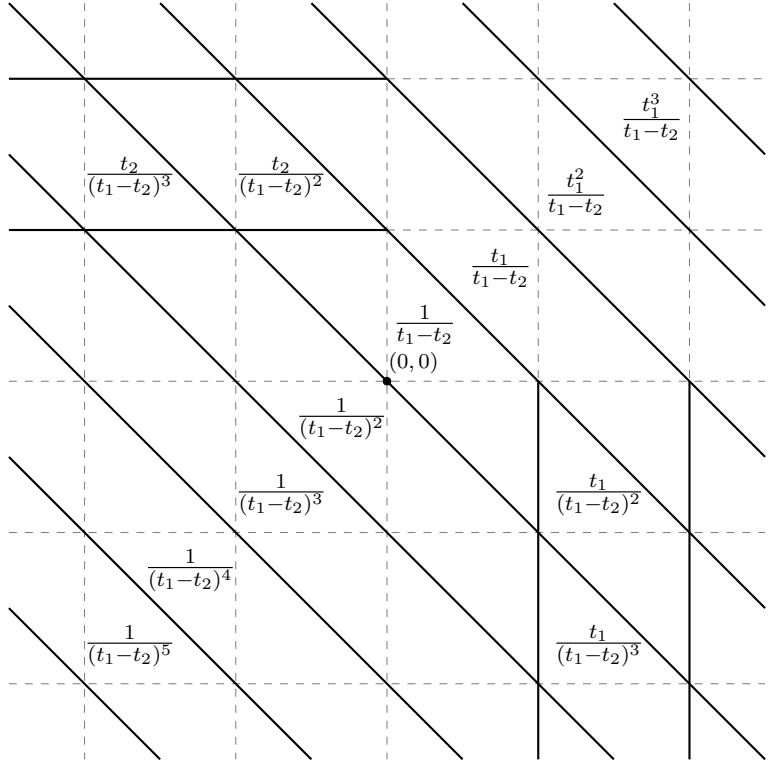

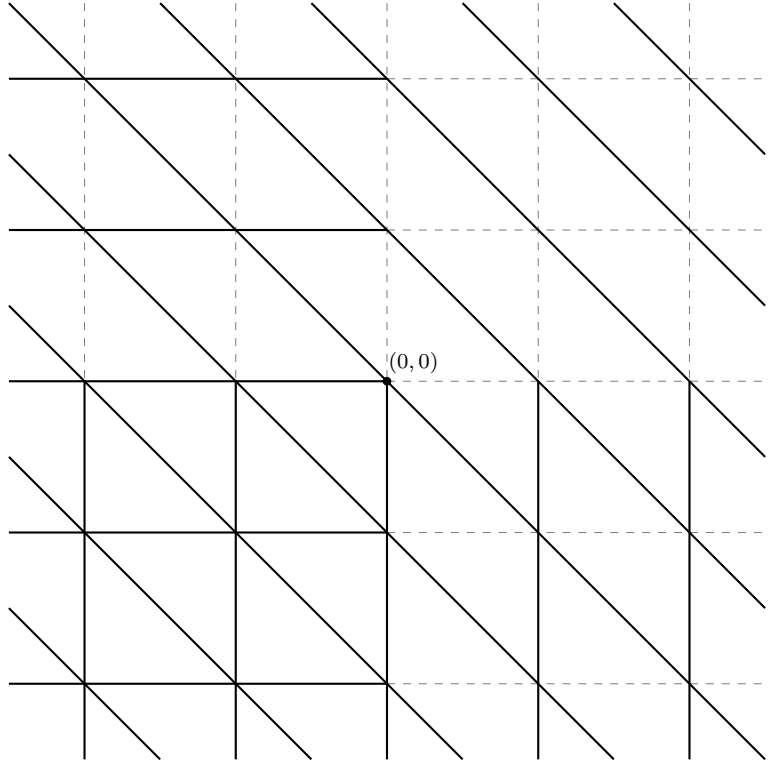
\begin{figure}

\begin{tikzpicture}
\draw[thick] (3, 5) -- (5, 3);
\draw[thick] (1, 5) -- (5, 1);
\draw[thick] (-1, 5) -- (5, -1);
\draw[thick] (-3, 5) -- (5, -3);
\draw[thick] (-5, 5) -- (5, -5);
\draw[thick] (-5, 3) -- (3, -5);
\draw[thick] (-5, 1) -- (1, -5);
\draw[thick] (-5, -1) -- (-1, -5);
\draw[thick] (-5, -3) -- (-3, -5);

\draw[thick] (-5, 4) -- (0, 4);
\draw[dashed, opacity=0.5] (0, 4) -- (5, 4);
\draw[thick] (-5, 2) -- (0, 2);
\draw[dashed, opacity=0.5] (0, 2) -- (5, 2);
\draw[thick] (-5, 0) -- (0, 0);
\draw[dashed, opacity=0.5] (0, 0) -- (5, 0);
\draw[thick] (-5, -2) -- (0, -2);
\draw[dashed, opacity=0.5] (0, -2) -- (5, -2);
\draw[thick] (-5, -4) -- (0, -4);
\draw[dashed, opacity=0.5] (0, -4) -- (5, -4);

\draw[dashed, opacity=0.5] (4, 5) -- (4, 0);
\draw[thick] (4, 0) -- (4, -5);
\draw[dashed, opacity=0.5] (2, 5) -- (2, 0);
\draw[thick] (2, 0) -- (2, -5);
\draw[dashed, opacity=0.5] (0, 5) -- (0, 0);
\draw[thick] (0, 0) -- (0, -5);
\draw[dashed, opacity=0.5] (-2, 5) -- (-2, 0);
\draw[thick] (-2, 0) -- (-2, -5);
\draw[dashed, opacity=0.5] (-4, 5) -- (-4, 0);
\draw[thick] (-4, 0) -- (-4, -5);

\draw[fill=black] (0, 0) circle (0.05);
\node at (0.35, 0.25) {\tiny $(0, 0)$};
\end{tikzpicture}
\caption{Jumping walls for the multivariate $V$-filtration of the diagonal in $\mb{C}^2$ localised at $0$: $V^{\bullet}\cM(\ast D)$.}
\label{fig:localised diagonal embedding}
\end{figure}

Let's see in this example why the na\"ive filtration defined by $V^{\alpha_1}\mc{M} \cap V^{\alpha_2}\mc{M}$ fails to be a multivariate $V$-filtration. Assume $\alpha_1,\alpha_2>0$. Let $D_i=\textrm{div}(t_i)$. Regarding the diagonal $\mb{C} \to \mb{C}^2$ as the graph of the identity, it is standard that 
\[ V_{D_2}^{\alpha_2}\mc{M} = \sum_{j \geq 0} \partial_{t_1}^j \cdot (V^{\alpha_2}\mc{O}_\mb{C} \otimes \delta) = \sum_{j \geq 0} \partial_{t_1}^j \mc{O}_{\mb{C}^2}\frac{t_2^{\lceil \alpha_2\rceil - 1}}{t_1 - t_2} = \sum_{j \geq 0} \mc{O}_{\mb{C}^2} \frac{t_2^{\lceil \alpha_2 \rceil - 1}}{(t_1 - t_2)^j},\]
where $\delta$ is the class of $\frac{1}{t_1-t_2}$ and $V^{\bullet}\cO_{\C}$ is the $V$-filtration along the origin in $\mb{C}$. In particular, $V^{1}_{D_2}\cM=\cM$. Similarly, $V^{1}_{D_1}\cM=\cM$, so $V^{1}_{D_1}\cM\cap V^{1}_{D_2}\cM=\cM$, which is not coherent over $V^{0,0}\sD_{\C^2}$. The same statement holds for any $V^{\alpha_1}_{D_1}\cM\cap V^{\alpha_2}_{D_2}\cM$.

Next, let us compare $V^{\bullet}_{\ast}\cM=V^{\bullet}_{\ast}\cM(\ast D)$ and $V^{\bullet}\cM(\ast D)$.  By \eqref{eqn: def of V*}, if $n,m\in \Z_{\geq 0}$, 
\[ V^{-n,-m}_{\ast}\cM=t_1^{-n-1}t_2^{-m-1}\cdot V^{1,1}_{res}\cM=t_1^{-n-1}t_2^{-m-1}\cdot V^{0,0}\sD_X\cdot a=V^{0,0}\sD_X\cdot (t_1^{-n-1}t_2^{-m-1}a),\]
where $a\in \cM$ is a generator of $V^{1,1}_{res}\cM$. In particular,
\[ V^{-1,0}_{\ast}\cM=V^{0,0}\sD_{\C^2}\cdot \frac{t_1^{-1}t_2^{-1}}{t_1-t_2}, \quad V^{0,0}_{\ast}\cM=V^{0,0}\sD_{\C^2}\cdot \frac{t_2^{-1}}{t_1-t_2}.\]
As  $V^{\alpha_1,\alpha_2}\cM(\ast D)=V^{\alpha_1,\alpha_2}_{\ast}\cM$ if $\bs{\alpha}\in \R^2_{\geq 0}$ by Corollary \ref{cor:extending V*}, using Corollary \ref{cor: V filtration for vector with integer coefficients}  we have
\begin{align*}
V^{-1,0}\cM(\ast D)=V^{-1,0}\sD_{\C^2}\cdot V^{0,0}_{\ast}\cM
=V^{-1,0}\sD_{\C^2}\cdot \frac{t_2^{-1}}{t_1-t_2}=V^{0,0}\sD_{\C^2}\cdot \frac{t_2^{-1}}{(t_1-t_2)^2}.
\end{align*}
It follows that $\frac{t_1^{-1}t_2^{-1}}{t_1-t_2}\not\in  V^{-1,0}\cM(\ast D)$ and thus $V^{-1,0}_{\ast}\cM(\ast D)\neq V^{-1,0}\cM(\ast D)$.

Another way to see this discrepancy is to observe that there is an exact sequence
\[ 0 \to \mc{M} \to \mc{M}(*D) \to \delta_0 \to 0,\]
where $\delta_0$ is the skyscraper $\sD$-module at $0$. It is very easy to see that the multivariate $V$-filtration of $\delta_0$ jumps non-trivially at $(\alpha_1, \alpha_2)$ if and only if $(\alpha_1, \alpha_2) \in \mb{R}_{\leq 0}^2$ and either $\alpha_1 \in \mb{Z}$ or $\alpha_2 \in \mb{Z}$. By Corollary \ref{cor:V strictness}, the jumping values for $V^\bullet \mc{M}(*D)$ are therefore the union of these jumping values for $V^\bullet\delta_0$ and the jumping values for $V^\bullet \mc{M}$, as depicted in Figure \ref{fig:localised diagonal embedding}. But $V_*^\bullet\mc{M}(*D)$ jumps only along the hyperplanes $\{(\alpha_1, \alpha_2) \mid \alpha_1 + \alpha_2 = n\}$ for $n \in \mb{Z}$, so we cannot have $V^\bullet \mc{M}(*D) = V_*^\bullet \mc{M}(*D)$.
\end{example}

In the classical theory of the $V$-filtration, one has the formula
\[ \mc{M}(*D) = \ms{D}_X \otimes_{V^0\ms{D}_X} V^0\mc{M}(*D) \]
for the localisation of a $\ms{D}$-module along a smooth divisor in terms of its $V$-filtration. This is important in the theory of mixed Hodge modules, for example, where the right hand side is used to construct the Hodge filtration on $\mc{M}(*D)$. The next proposition generalises this to the multivariate setting.
\begin{prop} \label{prop:* formula}
Let $\mc{M}(*D)$ be a coherent $\ms{D}_X(*D)$-module with a multivariate $V_*$-filtration $V_*^\bullet\mc{M}(*D)$. Then 
\begin{equation} \label{eq:* formula 1}
\mc{M}(*D) \cong \ms{D}_X \overset{\mrm{L}}\otimes_{V^{\bs{0}}\ms{D}_X} V^{\bs{0}}_*\mc{M}(*D)
\end{equation}
and for any $\bs{n} \in \mb{Z}^r$,
\begin{equation} \label{eq:* formula 2}
V^{\bs{n}}\mc{M}(*D) \cong V^{\bs{n}}\ms{D}_X \overset{\mrm{L}}\otimes_{V^{\bs{0}}\ms{D}_X} V^{\bs{0}}_*\mc{M}(*D).
\end{equation}
In particular, the derived tensor products are equal to the corresponding underived tensor products, i.e.\ there are no higher Tors.
\end{prop}

\begin{proof}
First observe that for $\bs{n} \in \mb{Z}_{\geq 0}^r$, we have $V^{\bs{n}}\ms{D}_X = \bs{t}^{\bs{n}}V^{\bs{0}}\ms{D}_X \cong V^{\bs{0}}\ms{D}_X$ as right $V^{\bs{0}}\ms{D}_X$-modules. So in this case the identification \eqref{eq:* formula 2} holds by Propositions \ref{prop:weak specialisation} and \ref{prop:* criterion} and Corollary \ref{cor:extending V*}. Now, for any $\bs{n} \in \mb{Z}^r$, we have an isomorphism of right $V^{\bs{0}}\ms{D}_X$-modules 
\begin{equation}\label{eqn: Vn-1i/Vn} \frac{V^{\bs{n} - \bs{1}_i}\ms{D}_X}{V^{\bs{n}}\ms{D}_X} \cong \frac{V^{\bs{0}}\ms{D}_X}{t_iV^{\bs{0}}\ms{D}_X}.\end{equation}
Since $t_i \colon V^{\bs{0}}_*\mc{M}(*D) \overset{\sim}\to V^{\bs{1}_i}_*\mc{M}(*D) \subset V^{\bs{0}}_*\mc{M}(*D)$ is injective, we conclude that
\begin{equation} \label{eq:* formula 3}
\mrm{Tor}_j^{V^{\bs{0}}\ms{D}_X}\left(\frac{V^{\bs{n} - \bs{1}_i}\ms{D}_X}{V^{\bs{n}}\ms{D}_X}, V^{\bs{0}}_*\mc{M}(*D)\right) = 0 \quad \text{for $j > 0$}.
\end{equation}
So the morphism
\[ V^{\bs{n}}\ms{D}_X \otimes_{V^{\bs{0}}\ms{D}_X} V^{\bs{0}}_*\mc{M}(*D)\to V^{\bs{n} - \bs{1}_i}\ms{D}_X \otimes_{V^{\bs{0}}\ms{D}_X} V^{\bs{0}}_*\mc{M}(*D) \]
is injective for all $\bs{n}$. Taking the colimit over $\bs{n}$, it follows that
\begin{equation}\label{eq:* formula 3.5} V^{\bs{n}}\ms{D}_X \otimes_{V^{\bs{0}}\ms{D}_X} V^{\bs{0}}_*\mc{M}(*D) \to \ms{D}_X \otimes_{V^{\bs{0}}\ms{D}_X}V^{\bs{0}}_*\mc{M}(*D)\end{equation}
is also injective for all $\bs{n} \in \mb{Z}^r$. But now the right hand side
\[ \mc{N} \colonequals \ms{D}_X \otimes_{V^{\bs{0}}\ms{D}_X}V^{\bs{0}}_*\mc{M}(*D) \]
is a $\ms{D}_X$-module containing $V^{\bs{0}}_*\mc{M}(*D) = V^{\bs{0}}\ms{D}_X \otimes_{V^{\bs{0}}\ms{D}_X}V^{\bs{0}}_*\mc{M}(*D)$ and generated by it, and hence $\{V^{\bs{\alpha}}_*\mc{M}(*D)\}_{\bs{\alpha} \in \mb{R}_{\geq 0}^r}$ is a restricted multivariate $V$-filtration for $\mc{N}$. 
So by Proposition \ref{prop:multivariate V extension}, $\mc{N}$ admits a multivariate $V$-filtration $V^\bullet \mc{N}$ satisfying
\[ V^{\bs{\alpha}}\mc{N} =V_*^{\bs{\alpha}}\mc{M}(*D) \quad \text{for $\bs{\alpha} \in \mb{R}_{\geq 0}^r$}.\]
In particular, the tautological morphism $\mc{N} \to \mc{M}(*D)$ extending the identity on $V^{\bs{0}}_*\mc{M}(*D)$ is an isomorphism on $V^{\bs{0}}$. By Proposition \ref{prop:V image and preimage}, the kernel and cokernel admit multivariate $V$-filtrations with $V^{\bs{0}} = 0$ and are therefore zero by Proposition \ref{prop:weak specialisation}. So in fact
\[ \ms{D}_X \otimes_{V^{\bs{0}}\ms{D}_X}V^{\bs{0}}_*\mc{M}(*D) =\mc{N}  = \mc{M}(*D).\]
Moreover, since $V^{\bs{n}}\mc{M}(*D) = \bs{t}^{\bs{n}}V_*^{\bs{0}}\mc{M}(*D)$ for $\bs{n} \in \mb{Z}_{\geq 0}^r$, it follows from Proposition \ref{prop:weak specialisation}, Corollary \ref{cor: V filtration for vector with integer coefficients} and the injectivity of \eqref{eq:* formula 3.5}  that 
\begin{equation*}\label{eq:* formula for Vinteger n of M*D} V^{\bs{n}}\ms{D}_X \otimes_{V^{\bs{0}}\ms{D}_X} V^{\bs{0}}_*\mc{M}(*D) = V^{\bs{n}}\ms{D}_X \cdot V^{\bs{0}}\mc{M}(*D) =V^{\bs{n}}\mc{M}(*D)\end{equation*}
for all $\bs{n} \in \mb{Z}^r$ as claimed. 

Finally, to prove the desired isomorphisms at the derived level, we need to show that
\begin{equation} \label{eq:* formula 4}
\mrm{Tor}_j^{V^{\bs{0}}\ms{D}_X} (V^{\bs{n}}\ms{D}_X, V^{\bs{0}}_*\mc{M}(*D)) = 0 \quad \text{for $j > 0$},
\end{equation}
for all $\bs{n} \in \mb{Z}^r$. This holds for $\bs{n} \in \mb{Z}_{\geq 0}^r$ since $V^{\bs{n}}\ms{D}_X \cong V^{\bs{0}}\ms{D}_X$ as right $V^{\bs{0}}\ms{D}_X$-modules, and for $\bs{n}$ arbitrary we have
\[ \mrm{Tor}_j^{V^{\bs{0}}\ms{D}_X}\left(V^{\bs{n}- \bs{1}_i}\ms{D}_X, V_*^{\bs{0}}\mc{M}(*D)\right) \cong \mrm{Tor}_j^{V^{\bs{0}}\ms{D}_X}(V^{\bs{n}}\ms{D}_X, V_*^{\bs{0}}\mc{M}(*D)) \quad \text{for $j > 0$}\]
by \eqref{eq:* formula 3}. Hence, \eqref{eq:* formula 4} holds for all $\bs{n}$ as claimed.
\end{proof}

We conclude this subsection with the following strengthening of Proposition \ref{prop:weak specialisation}.
\begin{prop} \label{prop:strong specialisation}
Let $\mc{M}$ be a coherent $\ms{D}_X$-module with a multivariate $V$-filtration $V^\bullet\mc{M}$. Then for all $\bs{n} \in \mb{Z}^r$ with $n_i \leq 0$, the morphism
\begin{equation} \label{eq:strong specialisation 1}
\partial_{t_i} \colon \frac{V^{\bs{n}}\mc{M}}{V^{\bs{n} + \bs{1}_i}\mc{M}} \to \frac{V^{\bs{n} - \bs{1}_i}\mc{M}}{V^{\bs{n}}\mc{M}}
\end{equation}
is an isomorphism.
\end{prop}
\begin{proof}
First suppose $\mc{M}$ is supported in one of the components $D_j$. Then the claim follows by induction on the number of components of $D$ using Lemma \ref{lem:supported on a component} below. 

Next, consider the case where $\mc{M} = \mc{M}(*D)$. 
In this case, by Corollaries \ref{cor:localising V} and \ref{cor:extending V*} and Proposition \ref{prop:* formula}, $\mc{M}(*D)$ is a coherent $\ms{D}_X$-module with multivariate $V$-filtration satisfying 
\[ \frac{V^{\bs{n}}\mc{M}(*D)}{V^{\bs{n} + \bs{1}_i}\mc{M}(*D)} = \frac{V^{\bs{n}}\ms{D}_X}{V^{\bs{n} + \bs{1}_i}\ms{D}_X}\otimes_{V^{\bs{0}}\ms{D}_X} V^{\bs{0}}_*\mc{M}(*D) \quad \text{for $\bs{n} \in \mb{Z}^r$}.\]
For $n_i \leq 0$, since the morphism
\[ \partial_{t_i} \colon \frac{V^{\bs{n}}\ms{D}_X}{V^{\bs{n} + \bs{1}_i} \ms{D}_X} \to \frac{V^{\bs{n} - \bs{1}_i}\ms{D}_X}{V^{\bs{n}} \ms{D}_X} \]
is an isomorphism of right $V^{\bs{0}}\ms{D}_X$-modules, we deduce that
\[ \partial_{t_i} \colon \frac{V^{\bs{n}}\mc{M}(*D)}{V^{\bs{n} + \bs{1}_i} \mc{M}(*D)} \to \frac{V^{\bs{n} - \bs{1}_i}\mc{M}(*D)}{V^{\bs{n}} \mc{M}(*D)} \]
is an isomorphism as well.

Finally, consider the general case. Consider the sequence of $\ms{D}_X$-modules
\[ \mc{M}_j = \mc{M}(*(D_1 + \cdots + D_j))\]
for $j = 0, 1, \ldots, r$. By Corollary \ref{cor:localising at a component}, each $\mc{M}_j$ is coherent and admits a multivariate $V$-filtration. We argue by descending induction on $j$ that the statement of the proposition holds for $\mc{M}_j$. As argued above, the result holds for $\mc{M}_r = \mc{M}(*D)$. For any $j > 0$, we have an exact sequence
\[ 0 \to \mc{K}_j \to \mc{M}_{j - 1} \to \mc{M}_{j} \to \mc{C}_j \to 0\]
where $\mc{K}_j$ and $\mc{C}_j$ are supported on $D_j$. By Proposition \ref{prop:V image and preimage}, $\mc{K}_j$ and $\mc{C}_j$ admit multivariate $V$-filtrations. Since the result holds for $\mc{K}_j$, $\mc{M}_j$ and $\mc{C}_j$, and since each morphism is strict with respect to $V^\bullet $ by Corollary \ref{cor:V strictness}, we conclude that the result holds for $\mc{M}_{j - 1}$. So by induction, the result holds for all $\mc{M}_j$. Setting $j = 0$ we obtain the result for $\mc{M}$.
\end{proof}

\begin{lem} \label{lem:supported on a component}
Let $\mc{M}$ be a coherent $\ms{D}_X$-module supported on the component $D_r \subset D$. Write $\mc{M} = \mc{N}[\partial_{t_r}]$ for some coherent $\ms{D}_{D_r}$-module $\mc{N}$ using Kashiwara's equivalence. Then $\mc{M}$ admits a multivariate $V$-filtration along $D$ if and only if $\mc{N}$ admits a multivariate $V$-filtration along $D' = D_1 \cap D_r + \cdots +  D_{r - 1} \cap D_r$. If they exist, the two filtrations are related by
\begin{equation} \label{eq:supported on a component 1}
V^{\bs{\alpha}, \alpha_r}\mc{M} = \sum_{k \leq -\alpha_r} \partial_{t_r}^k V^{\bs{\alpha}}\mc{N} \quad \text{for $\bs{\alpha} \in \mb{R}^{r - 1}$, $\alpha_r \in \mb{R}$}.
\end{equation}
\end{lem}
\begin{proof}
It is immediate from the definitions that if $V^\bullet\mc{N}$ exists then \eqref{eq:supported on a component 1} defines a multivariate $V$-filtration on $\mc{M}$. Conversely, if $V^\bullet \mc{M}$ exists, consider
\[ U^{\bs{\alpha}}\mc{M} = \bigcup_{\alpha_r} V^{\bs{\alpha}, \alpha_r}\mc{M} \quad \text{for $\bs{\alpha} \in \mb{R}^{r - 1}$}.\]
Since $U^{\bs{\alpha}}\mc{M} \subset \mc{M} = \mc{N}[\partial_{t_r}]$ is stable under $\mb{C}[t_r, \partial_{t_r}] \subset \ms{D}_X$, it follows that
\[ U^{\bs{\alpha}}\mc{M} = (U^{\bs{\alpha}}\mc{N})[\partial_{t_r}] \]
for some $U^{\bs{\alpha}}\mc{N} \subset \mc{N}$. It is now straightforward to check that $U^{\bs{\alpha}}\mc{N}$ defines a multivariate $V$-filtration of $\mc{N}$ along $D'$.
\end{proof}

\subsection{Graph embeddings and the Malgrange-Mellin transform}\label{sec: Malgrange iso} \label{subsec:malgrange-mellin}
We conclude this section with a discussion of the behaviour of multivariate $V$-filtrations under graph embeddings. We also introduce the Malgrange-Mellin transform, which plays a big role throughout the rest of the paper.

Consider the following setting. Suppose we are given an arbitrary sequence $f_1, \ldots, f_r$ of holomorphic functions on $X$ and let $\iota \colon X \to X \times \mb{C}^r$ be the graph embedding given by $\iota(x) = (x, f_1(x), \ldots, f_r(x))$. We write $D = \{f_1 \cdots f_r = 0\}$ for the union of the zero loci of the $f_i$. For any coherent $\ms{D}_X$-module $\mc{M}$, we may consider the direct image
\[\iota_+\mc{M} = \mc{M} \otimes \mb{C}[\bs{\partial_t}]\delta_{\bs{t} = \bs{f}}, \]
a coherent $\ms{D}_{X \times \mb{C}^r}$-module. Here $\delta_{\bs{t} = \bs{f}}$ is a formal symbol to remind us how the differential operators act on $\iota_+\mc{M}$. Unless otherwise specified, we write $t_i$ for the $i$th coordinate on $\mb{C}^r$ and consider multivariate $V$-filtrations with respect to the simple normal crossings divisor $\{t_1 \cdots t_r = 0\}$. Setting $s_i = -\partial_{t_i} t_i$ as usual, we may regard $\iota_+\mc{M}$ as a coherent $\ms{D}_X[\bs{s}, \bs{t}]$-module on which $t_i - f_i$ acts nilpotently. (Here the ring $\ms{D}_X[\bs{s}, \bs{t}]$ is defined so that the $s$ and $t$ variables satisfy the usual relations $[s_i, t_j] = -\delta_{ij}t_j$.)

We recall the following alternative description of the $\ms{D}_X[\bs{s}, \bs{t}]$-module $\iota_+\mc{M}$, which holds as long as $f_i$ acts invertibly on $\mc{M}$ (i.e.\ when $\mc{M}$ is in fact a $\ms{D}_X(*D)$-module). Consider the $\ms{D}_X[\bs{s}]\colonequals \ms{D}_X[s_1,\ldots,s_r]$-module 
\[\mc{M}[\bs{s}]\bs{f^s}\colonequals \mc{M}[s_1,\ldots,s_r]f_1^{s_1}\cdots f_r^{s_r}\]
given by $\mc{M}[\bs{s}]\bs{f^s}=\mc{M}\otimes \C[\bs{s}]$ as an $\cO_X[\bs{s}]$-module, with $\ms{D}_X$-action defined by
\begin{equation}\label{eqn: action on M[s]fs} \xi \cdot \left(u p(\bs{s}) \bs{f^s}\right)=\left(\xi(u)p(\bs{s})+\sum_{i=1}^r \frac{\xi(f_i)}{f_i}up(\bs{s})s_i\right)\bs{f^s} \end{equation}
for a vector field $\xi \in T_X$, $u \in \mc{M}$ and $p(\bs{s})\in \C[\bs{s}]$. The formula makes sense since $f_i$ acts invertibly on $\mc{M}$. We equip $\mc{M}[\bs{s}]\bs{f^s}$ with the structure of a $\ms{D}_X[\bs{s}, \bs{t}]$-module by setting
\begin{equation}\label{eqn: ti action} t_i \cdot (up(\bs{s})\bs{f^s}) = f_iu p(s_1,\ldots,s_i+1,\ldots,s_r) \bs{f^s}.\end{equation}

\begin{proposition}\label{prop: Malgrange isomorphism}
We have an isomorphism of $\ms{D}_X[\bs{s}, \bs{t}]$-modules
\begin{equation}\label{eqn: Malgrange isomorphism}
 \mc{M}[\bs{s}]\bs{f^s} \xrightarrow{\sim}\iota_+\mc{M},
\end{equation}
sending $u \bs{f^s}$ to $u \otimes \delta_{\bs{t} = \bs{f}} \in \iota_+\mc{M}$. 
\end{proposition}
Proposition \ref{prop: Malgrange isomorphism} was originally observed (in a special case) by Malgrange \cite{Malgrange}; see also \cite[Proposition 2.5]{MPVfiltration} and \cite[(6)]{CDMO}, for example. We will refer to the isomorphism \eqref{eqn: Malgrange isomorphism} as the \emph{Malgrange-Mellin transform}.

\begin{remark}\label{remark: heuristic reason for MM transform}
Our rationale behind the terminology ``Malgrange-Mellin transform'' is as follows. Suppose that $f$ is a smooth real-valued function on a real manifold $X$ and consider a distribution on $X \times \mb{R}$ of the form $g(x) \partial_t^n \delta_{t = f}$ for some other smooth function $g$. Then taking the Mellin transform on the $\mb{R}$ factor gives
\begin{align*}
\int (g(x)\partial_t^n \delta_{t = f})t^s dt &= (-1)^n s(s - 1) \cdots (s - n + 1) \int g(x)\delta_{t = f} t^{s - n} dt \\
&= (-1)^n s(s - 1) \cdots (s - n + 1) g(x)f^{s - n}.
\end{align*}
When $r = 1$, the inverse to \eqref{eqn: Malgrange isomorphism} is given by essentially the same formula.
\end{remark}
 
We will make extensive use of Proposition \ref{prop: Malgrange isomorphism} throughout this paper. As a first application, we observe that in this presentation, the multivariate $V$-filtration changes in a very simple way under monomial transformations in the $f_i$, keeping the underlying divisor $D$ fixed. Let $\cM$ be a coherent $\sD_X$-module where $f_i$ acts invertibly. Suppose that $f_j' = \prod_{i = 1}^r f_i^{d_{ij}}$, $j = 1, \ldots, r'$, for some $d_{ij} \in \mb{Z}_{\geq 0}$. Write $\iota' \colon X \to X \times \mb{C}^{r'}$ for the graph embedding of the $f_j'$. Then we have a $\sD_X[\bs{s}', \bs{t}']$-module isomorphism 
\begin{equation} \label{eq:basic change of functions}
 \iota'_+\mc{M} = \mc{M}[s_1', \ldots, s_{r'}'](f_1')^{s_1'} \cdots (f_r')^{s_{r'}'} \cong \mc{M}[\bs{s}]\bs{f^s}\otimes_{\mb{C}[\bs{s}]} \mb{C}[\bs{s}']=\iota_+\mc{M} \otimes_{\mb{C}[\bs{s}]} \mb{C}[\bs{s}'],
\end{equation}
where $\mb{C}[\bs{s}] \to \mb{C}[\bs{s}']$ is given by $s_i \mapsto \sum_j d_{ij}s_j'$. In the statement below, for $\bs{\alpha}' \in \mb{R}^{r'}$ we let $\bs{\alpha} \in \mb{R}^r$ be the vector with $i$th coordinate $\alpha_i = \sum_j d_{ij} \alpha_j'$.

\begin{prop} \label{prop:change of functions}
In the setting above, assume that the zero loci of $f_1 \cdots f_r$ and $f_1' \cdots f_{r'}'$ coincide. If the multivariate $V$-filtration $V^\bullet\iota_+\mc{M}$ exists, then so does $V^\bullet \iota'_+\mc{M}$, and for $\bs{\alpha}' \in \mb{R}_{\geq 0}^{r'}$, $V^{\bs{\alpha}'}\iota_+'\mc{M}$ is the image of the map
\begin{equation} \label{eq:change of functions 1}
V^{\bs{\alpha}}\iota_+\mc{M} \otimes_{\mb{C}[\bs{s}]} \mb{C}[\bs{s}'] \to \iota'_+\mc{M}.
\end{equation}
\end{prop}

\begin{proof}
Observe that since $\mc{M}$ is a $\ms{D}_X(*D)$-module and $D$ is the zero locus of both $f_1 \cdots f_r$ and $f_1' \cdots f_{r'}'$, $\iota_+\mc{M}$ and $\iota'_+\mc{M}$ are $\ms{D}_{X \times \mb{C}^r}(*\{t_1 \cdots t_r=0\})$ and $\ms{D}_{X \times \mb{C}^{r'}}(*\{t_1' \cdots t_{r'}'=0\})$-modules respectively. So by Corollaries \ref{cor:localising V} and \ref{cor:extending V*}, the claim is equivalent to showing that
\[ V_*^{\bs{\alpha}'}\iota'_+\mc{M} \colonequals \mrm{im}(V_*^{\bs{\alpha}}\iota_+\mc{M} \otimes_{\mb{C}[\bs{s}]} \mb{C}[\bs{s}'] \to \iota'_+\mc{M}), \quad \text{for $\bs{\alpha}' \in \mb{R}^{r'}$},\]
defines a multivariate $V_*$-filtration on $\iota'_+\mc{M}$. This is immediate from the definitions.
\end{proof}

\begin{rmk}
We will see in \S\ref{subsec:flatness} (Corollary \ref{cor: change of functions}) that, when $\mc{M}$ is holonomic, the map \eqref{eq:change of functions 1} is actually injective, so that $V^{\bs{\alpha}'}\iota'_+\mc{M}$ can be identified with the left hand side.
\end{rmk}

Now assume that $D \subset X$ is a divisor with simple normal crossings and that $f_i = x_i$ for some local coordinates $x_i$ on $X$. Let $D_i = \{ x_i = 0\}$ be the irreducible components of $D$. 
\begin{prop} \label{prop:reduction to graph embedding}
Let $\mc{M}$ be a coherent $\ms{D}_X$-module where $x_i$ acts invertibly. Then the multivariate $V$-filtration $V^\bullet\mc{M}$ with respect to $D$ exists if and only if the multivariate $V$-filtration $V^\bullet \iota_+\mc{M}$ with respect to $\{t_1 \cdots t_r = 0\}$ exists, and any set of walls for one is a set of walls for the other.
\end{prop}
\begin{proof}
For any point in $X$, we may identify a neighbourhood in $X \times \mb{C}^r$ with an open subset $U \subset X' \times \mb{C}^r_u \times \mb{C}^r_v$, where $X' = \{x_1 = \cdots = x_r = 0\}$ and we write the coordinates as $u_i = t_i$ and $v_i = t_i - x_i$. In these coordinates, the intersection with $\iota(X)$ is identified with $U \cap (X' \times \mb{C}^r_u \times \{0\})$ and $D_i$ with $\iota(X) \cap \{u_i = 0\}$. So, restricting to this neighbourhood, we may write,
\[ \iota_+\mc{M} = \mc{M}[\partial_{v_1}, \ldots, \partial_{v_r}].\]
It is elementary to see that, by analogy with Kashiwara's equivalence, we have a bijection $\mc{N} \mapsto \mc{N}_0$ between $V^{\bs{0}}\ms{D}_{X \times \mb{C}^r}$-submodules $\mc{N} \subset \iota_+\mc{M}$ and $V^{\bs{0}}\ms{D}_X$-submodules $\mc{N}_0 \subset \mc{M}$ such that
\[ \mc{N} = \mc{N}_0[\partial_{v_1}, \ldots, \partial_{v_r}] \quad \text{and} \quad \mc{N}_0 = \mc{M} \cap \mc{N} = \bigcap_i \ker(v_i \colon \mc{N} \to \mc{N}).\]
Given one of $V^\bullet\mc{M}$ and $V^\bullet\iota_+\mc{M}$, if we define the other by the formula
\[ V^{\bs{\alpha}}\iota_+\mc{M} = V^{\bs{\alpha}}\mc{M}[\partial_{v_1}, \ldots, \partial_{v_r}],\]
it is immediate to see that one satisfies the definition of multivariate $V$-filtration if and only if the other does. So the proposition follows.
\end{proof}

It will also be useful to have the following more explicit description of the relationship between the multivariate $V_*$-filtrations on $\mc{M}$ and $\iota_+\mc{M}$ in terms of the Malgrange-Mellin transform. In the statement below, we write $\ms{D}_X[\bs{s}]\bs{x^s} = \ms{D}_X[\bs{s}] \otimes \bs{x^s}$ for the $(V^{\bs{0}}\ms{D}_{X \times \mb{C}^r}, V^{\bs{0}}\ms{D}_X)$-bimodule with left $V^{\bs{0}}\ms{D}_{X \times \mb{C}^r}$-action given by
\[ P(\bs{s}) \cdot Q(\bs{s})\bs{x^s} = P(\bs{s})Q(\bs{s}) \bs{x^s} \quad \text{for $P(\bs{s}) \in \ms{D}_X[\bs{s}] \subset V^{\bs{0}}\ms{D}_{X \times \mb{C}^r}$},\]
\[ t_i \cdot Q(\bs{s})\bs{x^s} = Q(\bs{s} + \bs{1}_i)\bs{x}^{\bs{s} + \bs{1}_i} = Q(\bs{s} + \bs{1}_i)x_i\bs{x}^{\bs{s}}  \]
and right $V^{\bs{0}}\ms{D}_X$-action given by
\[ Q(\bs{s})\bs{x^s} g = Q(\bs{s}) g \bs{x^s} \quad \text{for $g \in \mc{O}_X$},\]
\[ Q(\bs{s})\bs{x^s} \partial_{x_i}x_i = Q(\bs{s}) (\partial_{x_i}x_i - s_i ) \bs{x^s}.\]

\begin{prop} \label{prop:V-filtration on graph}
For $\bs{\alpha} \in \mb{R}^r$, we have an isomorphism of $V^{\bs{0}}\ms{D}_{X \times \mb{C}^r}$-modules,
\begin{equation} \label{eq:V-filtration on graph 1}
V^{\bs{\alpha}}_*\iota_+\mc{M} = \ms{D}_X[\bs{s}] \bs{x^s} \otimes_{V^{\bs{0}}\ms{D}_X} V^{\bs{\alpha}}_*\mc{M},
\end{equation}
where $V^{\bullet}_*\iota_+\mc{M}$ and $V^{\bullet}_*\mc{M}$ denote the multivariate $V_*$-filtrations with respect to the coordinate axes in $\mb{C}^r$ and the divisor $D=\{x_1\cdots x_r=0\}$, respectively. 
\end{prop}
\begin{proof}
We first argue that the right hand side of \eqref{eq:V-filtration on graph 1} defines an exhaustive wall and chamber filtration on $\iota_+\mc{M} = \mc{M}[\bs{s}]\bs{x^s}$. To see this, observe that the right $V^{\bs{0}}\ms{D}_X$-module $\ms{D}_X[\bs{s}]\bs{x^s}$ is free with basis $\{\bs{\partial_x^n}\mid \bs{n} \in \mb{Z}^r\}$. So the right hand side of \eqref{eq:V-filtration on graph 1} defines an exhaustive filtration of $\ms{D}_X[\bs{s}] \bs{x^s} \otimes_{V^{\bs{0}}\ms{D}_X} \mc{M}$. Since each $x_i$ acts invertibly on $\mc{M}$, we have, localising at $D$,
\[ \ms{D}_X[\bs{s}]\bs{x^s} \otimes_{V^{\bs{0}}\ms{D}_X} \mc{M} = \ms{D}_X(*D)[\bs{s}]\bs{x^s} \otimes_{\ms{D}_X(*D)} \mc{M} = \mc{M}[\bs{s}]\bs{x^s}.\]
Thus, we have an exhaustive filtration
\[ U^{\bs{\alpha}}\iota_+\mc{M} \colonequals \ms{D}_X[\bs{s}]\bs{x^s}\otimes_{V^{\bs{0}}\ms{D}_X} V_*^{\bs{\alpha}}\mc{M},\]
which is clearly a wall and chamber filtration since $V_*^{\bs{\alpha}}\mc{M}$ is one. Moreover,
\begin{align*}
 t_i U^{\bs{\alpha}}\iota_+\mc{M} &= \ms{D}_X[\bs{s}]x_i\bs{x^s} \otimes_{V^{\bs{0}}\ms{D}_X} V_*^{\bs{\alpha}}\mc{M} \\
 &= \ms{D}_X[\bs{s}]\bs{x^s} \otimes_{V^{\bs{0}}\ms{D}_X} x_iV_*^{\bs{\alpha}}\mc{M} \\
 &= \ms{D}_X[\bs{s}]\bs{x^s} \otimes_{V^{\bs{0}}\ms{D}_X} V_*^{\bs{\alpha} + \bs{1}_i}\mc{M} = U^{\bs{\alpha} + \bs{1}_i}\iota_+\mc{M},
 \end{align*}
so $U^\bullet\iota_+\mc{M}$ is compatible with the action of $V^\bullet\ms{D}_{X \times \mb{C}^r}[(t_1 \cdots t_r)^{-1}]$. 

Next, by Proposition \ref{prop:good V* criterion}, to check that the wall and chamber filtration $U^\bullet\iota_+\mc{M}$ is good, by Lemma \ref{lem: goodness part}, we need only check that each $U^{\bs{\alpha}}\iota_+\mc{M}$ is a coherent $V^{\bs{0}}\ms{D}_{X \times \mb{C}^r}$-module. Since it is already coherent over $\ms{D}_X[\bs{s}]$ by construction, this is immediate.

Finally, let us check the eigenvalue condition \eqref{itm:V* 2} of Definition \ref{definition: V*}. Suppose that $\bs{\alpha} \leq \bs{\beta}$ are separated by a single wall $\bs{L}^{-1}(\gamma)$. Then there exist $\gamma_j$ with $\Re \gamma_j = \gamma$ such that
\[ \prod_j (-\bs{L}(\bs{\partial_x x}) + \gamma_j) \colon \frac{V_*^{\bs{\alpha}}\mc{M}}{V_*^{\bs{\beta}}\mc{M}} \to \frac{V_*^{\bs{\alpha}}\mc{M}}{V_*^{\bs{\beta}}\mc{M}} \]
is zero. Observe that if $L_i \neq 0$ then $x_iV_*^{\bs{\alpha}}\mc{M} = V_*^{\bs{\alpha} + \bs{1}_i}\mc{M} \subset V_*^{\bs{\beta}}\mc{M}$, so for $m \in V_*^{\bs{\alpha}}\mc{M}/V_*^{\bs{\beta}}\mc{M}$, one has
\begin{align*}
s_i Q(\bs{s}) \bs{x^s} \otimes m &= Q(\bs{s})\partial_{x_i}x_i \bs{x^s} \otimes m - Q(\bs{s}) \bs{x^s}\partial_{x_i} x_i \otimes m \\
&= Q(\bs{s}) \partial_{x_i} \bs{x^s} \otimes x_i m - Q(\bs{s})\bs{x^s} \otimes (\partial_{x_i} x_i) m\\
&= -Q(\bs{s})\bs{x^s} \otimes (\partial_{x_i} x_i) m.
\end{align*} Hence, as operators on
\[ \frac{U^{\bs{\alpha}}\iota_+\mc{M}}{U^{\bs{\beta}}\iota_+\mc{M}} = \ms{D}_X[\bs{s}]\bs{x^s} \otimes_{V^{\bs{0}}\ms{D}_X}\frac{V_*^{\bs{\alpha}}\mc{M}}{V_*^{\bs{\beta}}\mc{M}}\]
we have
\[  \prod_j (\bs{L}(\bs{s}) + \gamma_j) = \id \otimes \prod_j (-\bs{L}(\bs{\partial_x x}) + \gamma_j) = 0.\]
Thus, $U^\bullet\iota_+\mc{M}$ is a multivariate $V_*$-filtration, which proves the proposition.\end{proof}

\section{The multivariate $V$-filtration on holonomic $\sD$-modules}\label{sec: multivariate V and holonomic}

In this section, we study the multivariate $V$-filtrations on holonomic $\ms{D}$-modules in more detail. In \S\ref{subsec:normal crossings}, we write down the multivariate $V$-filtration in the case of a meromorphic flat connection with good formal structure. We use this in \S\ref{subsec:existence} to prove that every holonomic $\ms{D}$-module admits a multivariate $V$-filtration by a resolution of singularities argument, and relate the walls to the numerical data specified by the resolution. We conclude the section with a discussion of categories of fully $A$-specialisable $\ms{D}$-modules, whose $V$-filtrations are always defined over a restricted set $E \subset \mb{C}$, and show that the multivariate $V$-filtrations of these are also defined over $A$.

\subsection{The normal crossings case} \label{subsec:normal crossings}

In this subsection, we write down the multivariate $V$-filtration explicitly for the basic local models for holonomic $\ms{D}$-modules: flat meromorphic connections with unramified good normal crossings singularities. We note that in the case of regular singularities, these are nothing but flat meromorphic connections with poles along a divisor with normal crossings. In the irregular case, further conditions need to be imposed on the singularity of the connection to obtain simple local models.

Let $X$ be a complex manifold, $D \subset X$ a divisor with simple normal crossings and $\mc{M}$ a flat meromorphic connection on $X$ with poles along $D$. In other words, $\mc{M}$ is a $\ms{D}_X(*D)$-module that is coherent (hence locally free) over $\mc{O}_X(*D)$. Following terminology used by Sabbah and Kedlaya \cite[Definition 3.4.6]{Kedlaya2010} (see also \cite[Definition I.2.1.5]{Sabbahirregular}, and \cite[Definitions 2.2.3 and 2.3.5]{Mochizuki2011}), we say that $\mc{M}$ has \emph{unramified good formal structure} if, for every $x \in X$, the completion
\[ \widehat{\mc{M}}_x := \widehat{\mc{O}}_{X, x} \otimes_{\mc{O}_X} \mc{M}\]
admits a $\widehat{\ms{D}}_{X, x}(*D)$-decomposition
\[ \widehat{\mc{M}}_x = \bigoplus_{j \in J}\mc{U}_j\exp(\phi_j)\]
where $J$ is some finite set, each $\mc{U}_j$ is a $\widehat{\ms{D}}_{X, x}(*D)$-module with regular singularities along $D$, $\phi_j \in \widehat{\mc{O}}_{X, x}(*D)$ and $\exp(\phi_j)$ denotes the exponential twist by $\phi_j$. The twists $\phi_j$ are further assumed to satisfy:
\begin{enumerate}
\item \label{itm:unramified good 1} If $j \in J$ then either $\phi_j \in \widehat{\mc{O}}_{X, x}$ or $\phi_j$ is of the form
\begin{equation} \label{eq:unramified good 1}
\phi_j = ut_1^{-n_1} \cdots t_r^{-n_r},
\end{equation}
for some $n_i \in \mb{Z}_{\geq 0}$ and some unit $u \in \widehat{\mc{O}}_{X, x}$. Here $t_i = 0$ are local equations for the irreducible components of $D$ containing $x$, 
\item If $j, k \in J$ then either $\phi_j - \phi_k \in \widehat{\mc{O}}_{X, x}$ or $\phi_j - \phi_k$ is of the form 
\[ \phi_j - \phi_k = ut_1^{-n_1} \cdots t_r^{-n_r},\]
for some $n_i \in \mb{Z}_{\geq 0}$ and some unit $u \in \widehat{\mc{O}}_{X, x}$.
\end{enumerate}
For our purposes, only condition \eqref{itm:unramified good 1} will play a role. We say that $\mc{M}$ has \emph{good formal structure} if, locally on $X$, there exists a finite covering $\pi \colon \tilde X \to X$, ramified along $D$, such that $\pi^*\mc{M}$ has unramified good formal structure. Here $\pi^*$ is the usual pullback of a meromorphic connection.

In general, if $D$ is any divisor on $X$ and $\mc{M}$ is a flat meromorphic connection on $X$ with poles along $D$, we will sometimes say that \emph{$\mc{M}$ has (unramified) good normal crossings singularities} if $D$ has simple normal crossings and $\mc{M}$ has (unramified) good formal structure. By Hironaka, one can always resolve the divisor $D$ by one with simple normal crossings; a key theorem of Kedlaya (in the analytic case) and Mochizuki (in the algebraic case) states similarly that any flat meromorphic connection can be resolved by one with good normal crossings singularities:

\begin{thm}[{\cite[Theorems 4.4.7 and 4.5.1]{Kedlaya2021} \cite[Theorem 1.3.3]{Mochizuki2011}}] \label{thm:resolution of turning points}
Let $\mc{M}$ be a flat meromorphic connection on $X$ with poles along an arbitrary divisor $D \subset X$. Then there exists a projective birational morphism $\pi \colon \tilde X \to X$, an isomorphism outside $D$, such that $\tilde X$ is smooth, $\pi^{-1}(D)$ has simple normal crossings and $\pi^*\mc{M}$ has good formal structure. 
\end{thm}

A resolution as in Theorem \ref{thm:resolution of turning points} is called a \emph{resolution of turning points} (a turning point is a point at which $\cM$ does not have good formal structure). Roughly speaking, resolutions of turning points play the role of resolutions of singularities in the theory of irregular $\ms{D}$-modules. 

\begin{example}\label{example: resolution of turning points}
Let $X = \mb{C}^2$ with $t_1, t_2$ the usual coordinates and consider the rank $1$ connection $\cM=\left(\cO_X(*D),d+d(t_2/t_1)\right)$; as a $\ms{D}_X(*D)$-module, this is the exponential twist $\mc{O}_X(*D)\exp(t_2/t_1)$. Since $t_2/t_1$ does not have the form \eqref{eq:unramified good 1} in the definition of good formal structure above, we see that $\mc{M}$ has a turning point at $(0, 0)$. Letting $\pi \colon \tilde{X} \to X$ be the blowup at $(0, 0)$, we have $\pi^*\mc{M} = \mc{O}_X(*D)\exp(\pi^*(t_2/t_1))$; on the blowup, the function $\pi^*(t_2/t_1)$ does have the form \eqref{eq:unramified good 1} at every point, so $\pi$ is a resolution of turning points.
Later, this example will demonstrate that resolving turning points is necessary for constructing the multivariate $V$-filtration (see Example \ref{example: V filtration of d+dy/x}).
\end{example}

To prove existence of the multivariate $V$-filtration, we will also need the notion of the \emph{Deligne-Malgrange lattice}. For simplicity, we restrict our discussion to the case of unramified good normal crossings singularities, although the notion can be defined more generally (see, e.g.\ \cite[\S 2.7.2.1]{Mochizuki2011}).

\begin{defn} \label{defn:Deligne-Malgrange}
Let $\mc{M}$ be a flat meromorphic connection on $X$, with poles along a simple normal crossings divisor $D$ and unramified good formal structure. A \emph{Deligne-Malgrange lattice} is a coherent $\mc{O}_X$-submodule $M \subset \mc{M}$, with induced connection $\nabla \colon M \to M (*D)\otimes \Omega_X^1$, such that, $M\otimes_{\cO_X}\cO_X(*D)=\mc{M}$ and, for every point $x \in X$, we have a decomposition
\begin{equation} \label{eq:Deligne-Malgrange decomposition}
(\widehat{M}_x, \widehat{\nabla}_x) = \bigoplus_{j \in J} (U_j, \nabla_j + d \phi_j),
\end{equation}
where $\phi_j$ are as in the definition of unramified good formal structure and $(U_j, \nabla_j)$ are (formal) vector bundles with flat connections with logarithmic singularities along $D$ such that, for each irreducible component $D_i$ of $D$ containing $x$, every eigenvalue of the residue $\Res_{D_i} \nabla_j$ has real part in $[0, 1)$.
\end{defn}
Note that when $\mc{M}$ has regular singularities, Deligne's canonical extension gives a Deligne-Malgrange lattice. In general, we have the following result.

\begin{thm}[{e.g.\ \cite[Theorem 5.3.4]{Kedlaya2010}} and Malgrange \cite{Malgrange1996}] \label{thm:Deligne-Malgrange lattice}
Let $\mc{M}$ be a flat meromorphic connection with unramified good normal crossings singularities. Then there exists a unique Deligne-Malgrange lattice $M \subset \mc{M}$.
\end{thm}

We will also use a slight variant of this construction. Fix local equations $t_1, \ldots, t_r$ for the irreducible components $D_1, \ldots, D_r$. For $\bs{\alpha} \in \mb{R}^r$, consider the meromorphic connection
\[ \mc{M}\bs{t}^{-\bs{\alpha}} = \mc{M}t_1^{-\alpha_1}\cdots t_r^{-\alpha_r} = \frac{\mc{M}[\bs{s}]\bs{t}^{\bs{s}}}{(s_1 + \alpha_1, \ldots, s_r + \alpha_r)}.\]
This is just $\mc{M}$ as an $\mc{O}_X(*D)$-module, but with a modified connection. If $\mc{M}$ has unramified good formal structure, then so does $\mc{M}\bs{t}^{-\bs{\alpha}}$. If we write $M^{\bs{\alpha}} \subset \mc{M}$ for the $\mc{O}_X$-submodule such that $M^{\bs{\alpha}} \bs{t}^{-\bs{\alpha}}$ is the Deligne-Malgrange lattice of $\mc{M}\bs{t}^{-\bs{\alpha}}$, then $M^{\bs{\alpha}}$ is characterised by the same conditions as Definition \ref{defn:Deligne-Malgrange} except that the residues $\Res_{D_i}\nabla_j$ have real part in the interval $[\alpha_i, \alpha_i + 1)$ instead of $[0, 1)$.

\begin{prop} \label{prop:normal crossings V-filtration}
Let $\mc{M}$ be a flat meromorphic connection with poles along a simple normal crossings divisor $D = \bigcup_{i = 1}^r D_i$ and good formal structure. Then the multivariate $V$-filtration $V^\bullet\mc{M}$ along $D$ exists and is given by \begin{equation}\label{eq:normal crossings}
V^{\bs{\alpha}}\mc{M} = V_{D_1}^{\alpha_1}\mc{M} \cap \cdots \cap V_{D_r}^{\alpha_r}\mc{M}, \quad \textrm{for all $\bs{\alpha}\in \R^r$}. \end{equation}
\end{prop}
\begin{proof}
Consider the filtration on $\mc{M}$ defined by the right hand side of \eqref{eq:normal crossings}. Clearly this is a wall and chamber filtration satisfying \eqref{itm:multivariate V 4} of Definition \ref{defn:multivariate V-filtration}, so it remains to check that it is a good filtration over $V^{\bullet}\sD_X$.

We consider first the case when $\mc{M}$ has unramified good formal structure. Fix local equations $t_1, \ldots, t_r$ for the irreducible components $D_1, \ldots, D_r$ and, for $\bs{\alpha}\in \mb{R}^r$, define $M^{\bs{\alpha}} \subset \mc{M}$ so that $M^{\bs{\alpha}}\bs{t}^{-\bs{\alpha}}$ is the Deligne-Malgrange lattice of $\mc{M}\bs{t}^{-\bs{\alpha}}$ as above. Set 
\[ U^{\bs{\alpha}}\mc{M} := V^{\bs{0}}\ms{D}_X \cdot t_1^{-1} \cdots t_r^{-1}M^{\bs{\alpha}}.\]
We claim that $U^{\bs{\alpha}} \mc{M}$ is equal (as a subsheaf of $\mc{M}$) to the right hand side of \eqref{eq:normal crossings}. 
To check this, it suffices to check that their completions every point $x \in X$ are equal as sub-objects of $\widehat{\mc{M}}_x$. So fix $x \in X$ and a decomposition 
\[ (\widehat{M}_x, \widehat{\nabla}_x) = \bigoplus_{j \in J} (U_j, \nabla_j + d \phi_j).\]
Refining the decomposition if necessary, we may assume further that for each $i$ and each $j$, $\Res_{D_i} \nabla_j$ has a single eigenvalue $\beta_i^j$. Then by the uniqueness in Theorem \ref{thm:Deligne-Malgrange lattice} (which holds also for formal meromorphic connections) we have
\begin{equation} \label{eq:normal crossings 1}
\widehat{M}_x^{\bs{\alpha}} = \bigoplus_{j \in J} t_1^{\lceil \alpha_1 - \Re \beta_1^j\rceil} \cdots t_r^{\lceil \alpha_r - \Re \beta_r^j \rceil} U_j.
\end{equation}
Now, for each $j \in J$, let us write $I_j = \{i \in \{1, \ldots, r\} \mid \text{$\phi_j$ has a pole along $D_i$}\}$ and define
\[ D^j_{\mathrm{irreg}} \colonequals \bigcup_{i \in I_j} D_i.\]
Then, by direct calculation using \eqref{eq:normal crossings 1}, we have 
\begin{equation} \label{eq:normal crossings 2}
(U^{\bs{\alpha}}\mc{M})^{\wedge}_x = \bigoplus_{j \in J} t_1^{\lceil\alpha_1 - \Re \beta_1^j \rceil - 1} \cdots t_r^{\lceil \alpha_r - \Re \beta_r^j \rceil - 1}U_j(*D^j_{\mathrm{irreg}}),
\end{equation}
while the $V$-filtration along $D_i$ is given by
\[ V_{D_i}^0\ms{D}_X \cdot t_1^{-1} \cdots t_r^{-1}M^{\bs{\alpha}} \]
with completion
\begin{equation} \label{eq:normal crossings 3}
(V_{D_i}^\alpha \mc{M})^\wedge_x = \bigoplus_{j \in J} t_i^{\lceil \alpha_i - \Re \beta_i^j\rceil - 1}U_j(*(D^j_{\mrm{irreg}} + \sum_{j \neq i} D_j)).
\end{equation}
Since \eqref{eq:normal crossings 2} is indeed the intersection over all $i$ of \eqref{eq:normal crossings 3}, we conclude that $U^\bullet \mc{M}$ is indeed the intersection of the $V$-filtrations along the components as claimed.

Now, $U^{\bs{\alpha}}\mc{M}$ is manifestly coherent over $V^{\bs{0}}\ms{D}_X$ and, by a further calculation using \eqref{eq:normal crossings 2}, we have
\[ U^{\bs{\alpha} + \bs{1}_i}\mc{M} = t_i(U^{\bs{\alpha}}\mc{M}) \quad \text{for all $\bs{\alpha}$}, \quad U^{\bs{\alpha} - \bs{1}_i}\mc{M} = \partial_{t_i}(U^{\bs{\alpha}}\mc{M}) + U^{\bs{\alpha}-\bs{1}_i+\epsilon\bs{1}_i}\mc{M} \quad \text{for $\alpha_i < 0$}.\]
Consequently,
\[ U^{\bs{\alpha}}\mc{M} = \sum_{\substack{\bs{n} \in\mb{Z}^r, \bs{\gamma} \in [0, 1]^r \\ \bs{n} + \bs{\gamma}\geq \bs{\alpha}}} V^{\bs{n}}\ms{D}_X \cdot U^{\bs{\gamma}}\mc{M}. \]
So the wall and chamber filtration $U^{\bs{\alpha}}\mc{M}$ is good,
and hence is a multivariate $V$-filtration.

Next consider the general case, where $\mc{M}$ has possibly ramified good formal structure. Working locally, we may fix a finite cover $\pi \colon \tilde X \to X$, ramified along $D$, such that $\pi^*\mc{M}$ has unramified good formal structure. In local coordinates, we may assume that $\tilde{X}$ is the ramified Galois cover given by $t_i = y_i^{n_i}$, $i = 1, \ldots, r$ for some $n_1,\ldots, n_r > 0$.

We first claim that
\begin{equation} \label{eq:normal crossings 5}
V_{D_i}^\alpha \mc{M} = \pi_*(V_{\tilde{D}_i}^{n_i \alpha + 1 - n_i}\pi^*\mc{M})^\Gamma,\end{equation}
where $\tilde{D}_i = \pi^{-1}(D_i)$ and $\Gamma= \mu_{n_1} \times \cdots \times \mu_{n_r}$ is the Galois group. Using the relations
\begin{equation} \label{eq:normal crossings 4}
\partial_{y_i} = n_iy_i^{n_i - 1} \partial_{t_i},
\end{equation}
observe that we have
\[ \pi_*(\ms{D}_{\tilde{X}})^\Gamma \subset \ms{D}_X \quad \text{and} \quad U^m_{D_i} \pi_*(\ms{D}_{\tilde{X}})^\Gamma \colonequals \pi_*(V^{m n_i}_{\tilde{D}_i}\ms{D}_{\tilde X})^\Gamma \subset V^m_{D_i}\ms{D}_X,\]
for all $m \in \mb{Z}$. Since $V_{\tilde{D}_i}^\bullet\pi^*\mc{M}$ is good over $V_{\tilde{D}_i}^\bullet\ms{D}_{\tilde{X}}$, it follows that the filtration
\[ U^\alpha_{D_i} \mc{M} \colonequals \pi_*(V^{n_i\alpha + 1 - n_i}_{\tilde{D}_i}\pi^*\mc{M})^\Gamma \]
is good over $U^\bullet_{D_i} \pi_*(\ms{D}_{\tilde X})^\Gamma$. Moreover, from \eqref{eq:normal crossings 3} that, in fact,
\[ V^m_{D_i}\ms{D}_X \cdot U_{D_i}^\alpha\mc{M} \subset U_{D_i}^{m + \alpha}\mc{M}.\]
Since $U_{D_i}^\bullet\mc{M}$ is already good over $U_{D_i}^\bullet \pi_*(\ms{D}_{\tilde{X}})^\Gamma \subset V_{D_i}^\bullet \ms{D}_X$, it is good over $V_{D_i}^\bullet \ms{D}_X$. Moreover, since
\[ \partial_{y_i} y_i = n_i \partial_{t_i} t_i + 1 - n_i,\]
it follows that $\partial_{t_i} t_i$ acts on $\gr_{U_{D_i}}^\alpha \mc{M}$ with generalised eigenvalue $\alpha$. So $U_{D_i}^\bullet \mc{M} = V_{D_i}^\bullet \mc{M}$, which proves \eqref{eq:normal crossings 5}.

Since the functor $\pi_*(-)^\Gamma$ is exact, it follows that
\begin{equation} \label{eq:normal crossings 6}
\begin{aligned}
V_{D_1}^{\alpha_1}\mc{M} \cap \cdots \cap V_{D_r}^{\alpha_r} \mc{M} &= \pi_*( V_{D_1}^{n_1 \alpha + 1 - n_1}\pi^*\mc{M} \cap \cdots \cap V_{D_r}^{n_r \alpha + 1 - n_r}\pi^*\mc{M})^\Gamma\\
&= \pi_*(V^{\bs{n}\bs{\alpha} + \bs{1} - \bs{n}}\pi^*\mc{M})^\Gamma.
\end{aligned}
\end{equation}
Here we have used the fact that we know the proposition for the connection $\pi^*\mc{M}$ with unramified good formal structure. Since $V^{\bullet}\pi^*\mc{M}$ is a good wall and chamber filtration over $V^\bullet \ms{D}_{\tilde X}$, it follows that the right hand side of \eqref{eq:normal crossings 6} is a good wall and chamber filtration over
\[ \pi_*(V^{\bs{n}\bullet}\ms{D}_{\tilde X})^\Gamma \subset V^\bullet \ms{D}_X \]
and hence over $V^\bullet \ms{D}_X$ as claimed.\end{proof}

\subsection{Existence and calculation via resolutions}\label{subsec:existence}

In this subsection, we use resolutions of turning points (Theorem \ref{thm:resolution of turning points}) to prove the following theorem.

\begin{thm} \label{thm:multivariate V existence}
Let $\mc{M}$ be a holonomic $\ms{D}$-module on a complex manifold $X$ and $D \subset X$ a simple normal crossing divisor. Then the multivariate $V$-filtration of $\mc{M}$ along $D$ exists.
\end{thm}

\begin{proof}
By uniqueness, the question is local, so we may as well assume that $D_i = \{t_i = 0\}$ for some local coordinates $t_i$ on $X$.

We first consider the case where $\mc{M} = \mc{M}(*E)$ for some divisor $E \supset D$ on $X$ and $\mc{M}(*E)$ is the direct image of a flat meromorphic connection on an irreducible subvariety $Z \subset X$ such that $Z \setminus E$ is smooth. By Theorem \ref{thm:resolution of turning points}, we can choose a resolution of turning points $\pi \colon \tilde{Z} \to Z \subset X$ for this meromorphic connection. Write $\tilde{E} = \pi^{-1}(E)$ and $\tilde{D} = \pi^{-1}(D)$. Then
\[ \mc{M}(*E) = \pi_+\tilde{\mc{M}}(*\tilde{E}),\]
where $\tilde{\mc{M}}(*\tilde{E})$, the pullback of $\pi^{\ast}(\mc{M}(\ast E))$ as a meromorphic connection, is a flat meromorphic connection on $\tilde{Z}$ with poles along $\tilde{E}$ and good formal structure.
In particular, by Proposition \ref{prop:normal crossings V-filtration}, $\tilde{\mc{M}}(*\tilde{E})$ admits a multivariate $V$-filtration along $\tilde{E}$ and thus along the subdivisor $\tilde{D}$ by Proposition \ref{prop:V-filtration for sub divisor}.  Since the local defining equations of the irreducible components of $\tilde{D}$ are monomials in the local coordinates of $\tilde{Z}$, we can use Proposition \ref{prop:reduction to graph embedding} and Proposition \ref{prop:change of functions} to obtain that $\tilde{\iota}_+\tilde{\mc{M}}(*\tilde E)$ admits a multivariate $V$-filtration along the coordinate axes in $\mb{C}^r$, where $\tilde{\iota} \colon \tilde Z \to \tilde Z \times \mb{C}^r$ is the graph embedding for $t_1 \circ \pi, \ldots, t_r \circ \pi$. Then by Theorem \ref{thm:multivariate V direct image}, 
\[\iota_+\mc{M}(*E) = \pi_+\tilde{\iota}_+\tilde{\mc{M}}(*\tilde E) \]
also admits a multivariate $V$-filtration along $D$, where $\iota$ is the graph embedding for $t_1, \ldots, t_r$. So by Proposition \ref{prop:reduction to graph embedding}, we conclude that $\mc{M} = \mc{M}(*E)$ admits a multivariate $V$-filtration along $D$.

Now consider the general case. If $g \colon X \to \mb{C}$ is any holomorphic function, then by Beilinson's gluing theory \cite{Beilinson}, we may identify $\mc{M}$ with the cohomology of the complex
\begin{equation} \label{eq:multivariate V existence 1}
\mc{M} \cong [\psi_{g,1} \mc{M} \to \Xi_g \mc{M} \oplus \phi_{g,1}\mc{M} \to \psi_{g,1}\mc{M}]
\end{equation}
where $\psi_{g,1}$, $\phi_{g,1}$ and $\Xi_g$ are the unipotent nearby cycles, unipotent vanishing cycles and maximal extension functors. If we write $Z \colonequals \supp \mc{M}$, then the $\ms{D}$-modules $\psi_{g,1}\mc{M}$ and $\phi_{g,1}\mc{M}$ are supported in $Z \cap g^{-1}(0)$, while $\Xi_g\mc{M}$ is supported in $\overline{Z \setminus g^{-1}(0)}$. We will use this to prove that the multivariate $V$-filtration on $\mc{M}$ exists by induction on $\dim Z$. 

As a base case, if $Z = \emptyset$, then $\mc{M} = 0$, so the theorem is true trivially. So suppose that $Z \neq \emptyset$. Observe that, working locally, we can choose a divisor $F = g^{-1}(0)$ for some holomorphic function $g$ such that $Z \setminus F$ is smooth and dense in $Z$, and $\mc{M}(*F)$ is the direct image of a flat meromorphic connection on it. In particular, $\dim (Z \cap F) < \dim Z$, so the multivariate $V$-filtrations on $\phi_{g, 1}\mc{M}$ and $\psi_{g, 1}\mc{M}$ exist by induction. By \eqref{eq:multivariate V existence 1}, it therefore suffices to prove that the multivariate $V$-filtration exists for the maximal extension $\Xi_g \mc{M} = \Xi_g(\mc{M}(*F))$. The $\ms{D}_X$-module $\mc{M}(*F)$, and hence $\Xi_g\mc{M}$, decomposes as a direct sum according to the irreducible components of $Z \setminus F$; replacing $\mc{M}$ with such a summand if necessary, we reduce to proving the theorem when $Z$ is irreducible and $\mc{M}(*F)$ is the direct image of a flat meromorphic connection on the smooth subvariety $Z \setminus F$.

Now, if $Z \subset D_i$ for some $i$ then we may write $\mc{M} = \mc{M}'[\partial_{t_i}]$ for some holonomic $\ms{D}_{D_i}$-module $\mc{M}'$. By induction on $\dim X$ (noting that the theorem is trivially true when $\dim X = 0$), we may suppose that $\mc{M}'$ admits a multivariate $V$-filtration $V^\bullet\mc{M}'$ along the simple normal crossings divisor $D' = \bigcup_{j \neq i} D_i \cap D_j \subset D_i$. By Lemma \ref{lem:supported on a component}, $\mc{M}$ therefore admits a multivariate $V$-filtration along $D$.

If $Z \not\supset D_i$ for any $i$, on the other hand, then setting $E = D + F$ we have that $\dim Z \cap E < \dim Z$, so choosing a local equation $h$ for $E$, the theorem holds for $\phi_{h, 1}\mc{M}$ and $\psi_{h, 1}\mc{M}$ by induction. The maximal extension $\Xi_{h}\mc{M}$ is by construction a quotient of $\frac{\mc{M}[s]h^s}{s^n}(*E)$ for some large $n$. As argued above, the flat connection $\frac{\mc{M}[s]h^s}{s^n}(*E)$ admits a multivariate $V$-filtration, so we conclude using Proposition \ref{prop:V image and preimage} that $\Xi_{h}\mc{M}$ admits a multivariate $V$-filtration and hence so does $\mc{M}$ by \eqref{eq:multivariate V existence 1}.
\end{proof}

The proof of Theorem \ref{thm:multivariate V existence} often gives an effective method for calculating features of the multivariate $V$-filtration using resolutions. For example, we have the following interpretation for the walls in the case of a flat meromorphic connection.

\begin{cor} \label{cor:walls via resolution}
Suppose that $\mc{M} = \mc{M}(*D)$ is the direct image of a flat meromorphic connection on an irreducible analytic subvariety $Z \subset X$ such that $Z \setminus D$ is smooth. Fix a resolution of turning points $\pi \colon \tilde{Z} \to Z$ with $\tilde{D} = \bigcup_j E_j = \pi^{-1}(D)$ and define $L_{i, j} \in \mb{Z}_{\geq 0}$ by $\pi^*D_i = \sum_j L_{i, j} E_j$. Then every wall for $V_*^\bullet\mc{M}$ is of the form
\[ \left\{ \bs{\alpha} \in \mb{R}^r \,\left|\, \sum_i L_{i, j} \alpha_i = \gamma\right.\right\} \]
for some $j$ and some jumping number $\gamma \in \mb{R}$ for the $V$-filtration of $\pi^*\mc{M}$ along $E_j$.
\end{cor}

Note that, in the context of Corollary \ref{cor:walls via resolution}, we can deduce a set of walls for $V^\bullet\mc{M}$ from the walls for $V^\bullet_*\mc{M}$ using Proposition \ref{prop:multivariate V extension}.

\begin{example}\label{ex: set of walls via log resolution}
Consider a complex manifold $X$ and a sequence of holomorphic functions $f_1, \ldots, f_r \colon X \to \mb{C}$ and set $D = \{f_1 \cdots f_r = 0\}$. Consider the $\ms{D}_{X \times \mb{C}^r}$-module $\mc{M} = \iota_+\mc{O}_X(*D)$, where $\iota \colon X \to X \times \mb{C}^r$ is the graph embedding of $f_1, \ldots, f_r$. Fix a log resolution $\pi \colon \tilde X \to X$ of $(X, D)$ such that $\pi$ is an isomorphism over $X \setminus D$. Then $\pi$ is also a resolution of turning points for $\mc{M}$ (the trivial connection). Since the jumping numbers of the $V$-filtration of the trivial connection along a smooth divisor are integers, the multivariate $V_*$-filtration of $\mc{M}$ along the coordinate axes in $\mb{C}^r$ has a set of walls of the form
\[\mc{W}=\{H_{E,k}\colonequals \{ \bs{\alpha}\in \R^r \mid \sum_{i=1}^r \mathrm{ord}_E(f_i)\alpha_i=k\} \mid \textrm{$E$ an exceptional divisor of $\pi$, $k\in \Z$}\}.\]
\end{example}

\begin{example}\label{example: V filtration of d+dy/x}
This example shows that for irregular $\ms{D}$-modules, resolutions of turning points, rather than just resolutions of singularities, are indeed necessary in order to apply Corollary \ref{cor:walls via resolution}. Consider the meromorphic connection $\mc{M} = \mc{O}_{\mb{C}^2}[(t_1t_2)^{-1}]\exp(t_2/t_1)$ from Example \ref{example: resolution of turning points} and the resolution of turning points $\pi \colon \tilde{X} \to X = \mb{C}^2$ given by blowing up the origin. Writing $D_i = \{t_i = 0\}$, we have $\pi^*D_1 = E_1 + E_3$ and $\pi^*D_2 = E_2 + E_3$, where $E_i$, $i = 1, 2$ are the strict transforms of $D_i$ and $E_3$ is the exceptional divisor. The pulled back connection $\pi^*\mc{M}$ is an exponential twist of the trivial connection along $E_1$, so its $V$-filtration along $E_1$ is constant and its $V$-filtrations along $E_2$ and $E_3$ jump at the integers. So $V_*^\bullet\mc{M}$ has a set of walls given by
\[ \{ \{\bs{\alpha} \in \mb{R}^r \mid \alpha_2 = k\}, \{\bs{\alpha} \in \mb{R}^r \mid \alpha_1 + \alpha_2 = k\} \mid k \in \mb{Z}\}.\]
We leave it as an exercise to the reader to check using the method outlined in the proof of Theorem \ref{thm:multivariate V existence} that all of these walls are required: if $\bs{\alpha} \leq \bs{\beta}$ are separated by one of the above walls, then $V_*^{\bs{\alpha}}\mc{M} \neq V_*^{\bs{\beta}}\mc{M}$.
\end{example}

\subsection{Full $A$-specialisability}\label{sec: full specialisability}

We conclude this section with a discussion of multivariate $V$-filtrations defined over a $\mb{Q}$-linear subspace $A \subset \mb{C}$ containing $\mb{Q}$.

As stated, Theorem \ref{thm:multivariate V existence} ensures the existence of a multivariate $V$-filtration in which eigenvalues are allowed to be arbitrary complex numbers. In the theory of the classical $V$-filtration, however, one often finds in practice that eigenvalues can be assumed to lie in a much smaller set, such as $\mb{Q}$ or $\mb{R}$. In this subsection, we use this idea to define a precise class of holonomic $\ms{D}$-modules, called \emph{fully $A$-specialisable}, and show that the multivariate $V$-filtrations of such $\ms{D}$-modules are always defined over $A$, in the sense of Definition \ref{defn:multivariate V-filtration}.

The inductive definition below is loosely inspired by the definition of mixed Hodge modules or twistor $\ms{D}$-modules \cite{Saito88,MHMproject, Sabbah05, Mochizuki2011}. Roughly speaking, the fully $A$-specialisable $\ms{D}$-modules form the largest class of holonomic $\ms{D}$-modules that is closed under standard functors and for which (single-variable) $V$-filtrations are always defined over $A$.

\begin{definition}\label{defn: A specialisability}
Let $\mc{M}$ be a holonomic $\ms{D}_X$-module. We define the condition that $\mc{M}$ be \emph{fully $A$-specialisable} inductively on the dimension as follows.
\begin{enumerate}
\item If $\dim \supp \mc{M} = 0$, then $\mc{M}$ is fully $A$-specialisable.
\item If $\dim \supp \cM\geq 1$ and $\mc{M}$ is simple, then $\mc{M}$ is fully $A$-specialisable if for every open subset $U\subset X$ and every holomorphic function $f$ on $U$ such that $\supp \mc{M} \not\subset f^{-1}(0)$, the $V$-filtration $\iota_{+}(\cM|_U)$ is defined over $A$ and $\gr^\alpha_V\iota_+(\mc{M}|_U)$ is fully $A$-specialisable for all $\alpha \in \mb{R}$. Here $\iota\colon U\to U\times \C$ is the graph embedding of $f$. If $\mc{M}$ is not simple, then $\mc{M}$ is defined to be fully $A$-specialisable if each of its composition factors is.
\end{enumerate}
Note that the simplicity assumption in (2) is imposed for convenience of the inductive definition to ensure that $\dim \supp \Gr_V^\alpha \mc{M} < \dim \supp \mc{M}$ for all $\alpha$. If $\cM^{\bullet}$ is a complex of $\sD$-modules with holonomic cohomologies, we say it is fully $A$-specialisable if each cohomology $\cH^j(\cM^{\bullet})$ is fully $A$-specialisable.
\end{definition}

Some basic properties immediately follow from the properties of the (single-variable) $V$-filtration and how they interact with standard functors.  In below, any $\sD$-module is assumed to be holonomic.
\begin{enumerate}
\item \label{itm: E-specialisable V-filtration} If $\mc{M}$ is fully $A$-specialisable, then the $V$-filtration along any smooth divisor is defined over $A$ and $\Gr_V^\alpha\mc{M}$ is fully $A$-specialisable for all $\alpha \in \mb{R}$.
\item \label{itm: morphism between E-specialisable} The fully $A$-specialisability  is closed under subquotients and extensions. Furthermore, if $\cM_1^{\bullet}\to \cM_2^{\bullet}\to \cM_3^{\bullet} \xrightarrow{+1}$ is a distinguished triangle of holonomic $\sD$-complexes and $\cM_1^{\bullet}$ and $\cM_2^{\bullet}$ are both fully $A$-specialisable, then so is $\cM_3^{\bullet}$.
\item \label{itm: duality} $\cM$ is fully $A$-specialisable if and only if the same is true for $\mathbb{D}\cM$.
\item \label{itm: stable under open pullback} If $j \colon Z\to X$ is the embedding of a smooth submanifold, and $\cM$ is a fully $A$-specialisable $\sD_X$-module, then $j^{\ast}\cM$ and $j^{!}\cM$ are fully $A$-specialisable. 
\item \label{itm: * and ! extension} Let $i \colon D \to X$ be an embedding of a divisor. If $\cM$ is a fully $A$-specialisable $\sD$-module on $X$, then $\cM(\ast D)$ and $\cM(!D)$ are fully $A$-specialisable. (This follows, e.g.\ from \eqref{itm: morphism between E-specialisable},\eqref{itm: duality}, \eqref{itm: stable under open pullback} and the distinguished triangle 
\[ \cM\to \cM(\ast D) \to i_+i^{!}\cM\xrightarrow{+1}.)\]

\item \label{itm: direct image} If $\pi\colon Y\to X$ is a proper map and $\cM$ is a fully $A$-specialisable $\sD_Y$-module, then $\cH^i\pi_{+}\cM$ is fully $A$-specialisable for each $i$. (This follows from e.g. \cite[Theorem 7.5.2]{SabbahnoteDmodule} or Theorem \ref{thm:multivariate V direct image}.)
\end{enumerate}

\begin{lemma}\label{lemma: nearby cycle E specialisable}
If $\cM$ is fully $A$-specialisable and $D=\textrm{div}(g)$ for a holomorphic function $g$, then then the unipotent nearby cycles $\psi_{g,1}\cM$, unipotent vanishing cycles $\varphi_{g,1}\cM$ and maximal extension $\Xi_g\cM$ are all fully $A$-specialisable.
\end{lemma}

\begin{proof}
The assertion for $\phi_{g, 1}\mc{M}$ and $\varphi_{g, 1}\mc{M}$ follow immediately from the definitions of these functors in terms of $V$-filtrations. For the maximal extension, recall that by definition \cite{Beilinson}
\[ \Xi_g \mc{M} = \coker\left(\frac{\mc{M}(*D)[s]}{(s^n)}(!D) \to \frac{\mc{M}(*D)[s]}{(s^n)}(*D)\right) \]
for $n$ sufficiently large. By property \eqref{itm: morphism between E-specialisable}, it therefore suffices to prove that $\frac{\mc{M}(*D)[s]}{(s^n)}(*D)$ is fully $A$-specialisable for all $n$. But $\frac{\mc{M}(*D)[s]}{(s^n)}(*D)$ is an iterated extension of $n$ copies of $\mc{M}(*D)$ so is fully $A$-specialisable by \eqref{itm: morphism between E-specialisable} and \eqref{itm: * and ! extension}.
\end{proof}

\begin{prop}\label{prop: E definability}
If a holonomic $\sD_X$-module $\cM$ is fully $A$-specialisable, then its multivariate $V$-filtration along any simple normal crossing divisor $D$ is defined over $A$.
\end{prop}
\begin{proof}
By the basic properties of the class of fully $A$-specialisable $\ms{D}$-modules above, the proof of Theorem \ref{thm:multivariate V existence} reduces the proposition to the case where $\mc{M}$ is a flat meromorphic connection with good formal structure. In this case, the proposition follows from Proposition \ref{prop:normal crossings V-filtration}, since the $V$-filtration of $\mc{M}$ along every irreducible component of $D$ is defined over $A$.
\end{proof}

\begin{rmk}\label{rmk: examples of fully specialisable}
In order for Proposition \ref{prop: E definability} to be useful in practice, one needs a supply of examples of fully $A$-specialisable $\ms{D}$-modules. In this paper, we will obtain such examples (for $A = \mb{Q}$ or $\mb{R}$) from the $\ms{D}$-modules underlying mixed Hodge modules (see Corollary \ref{cor: MHM fully A specializable}). More generally, one can show that the $\ms{D}$-module associated to a flat meromorphic connection with good formal structure is fully $A$-specialisable if and only if the $V$-filtration along each irreducible component of $D$ is defined over $A$. Applying standard functors, this gives a wide class of fully $A$-specialisable $\ms{D}$-modules: for example, starting from the fact that the $V$-filtration of the trivial connection along a smooth divisor is always defined over $\mb{Q}$, one sees in this way that all holonomic $\ms{D}$-modules of geometric origin are fully $\mb{Q}$-specialisable. 
\end{rmk}

\section{The multivariate $V$-filtration as a $\ms{D}_X[\boldsymbol{s}]$-module}\label{sec: multivariate V as DXsmodule}

In this section, we study further properties of the multivariate $V$-filtration $V^{\bs{\alpha}}\iota_+\mc{M}$ on the graph embedding of a holonomic $\ms{D}_X$-module $\mc{M}$ with respect to a sequence of holomorphic functions, regarded via the Malgrange-Mellin transform (Proposition \ref{prop: Malgrange isomorphism}) as a module over $\ms{D}_X[\bs{s}]$. In \S\ref{subsec:relative holonomicity}, we review the theory of relative holonomicity and show (Proposition \ref{prop:V is relative holonomic}) that $V^{\bs{\alpha}}\iota_+\mc{M}$ is always a relative holonomic $\ms{D}_X[\bs{s}]$-module. In \S\ref{subsec:duality}, we prove a self-duality theorem for the $V_*$-filtration (Theorem \ref{thm:duality}), generalising the usual self-duality of the nearby cycles functor. Finally, in \S\ref{subsec:flatness}, we prove (Theorem \ref{thm:flatness}) that the $\ms{D}_X[\bs{s}]$-modules $V_*^{\bs{\alpha}}\iota_+\mc{M}$ are always flat over $\mb{C}[\bs{s}]$, as well as a similar statement for the quotients $V_*^{\bs{\alpha}}\iota_+\mc{M}/V_*^{\bs{\beta}}\iota_+\mc{M}$. These flatness statements play a key role in our study of the Strong Monodromy Conjecture for hyperplane arrangements in \S\ref{sec: strong monodromy conjecture}.

Unless otherwise specified, we work throughout in the following setting: fix a complex manifold $X$, holomorphic functions $f_1, \ldots, f_r \colon X \to \mb{C}$, and a holonomic $\ms{D}_X$-module $\mc{M}$. 

\subsection{Relative holonomicity} \label{subsec:relative holonomicity}

In this subsection, we recall the notion of a relative holonomic $\ms{D}_X$-module over a smooth affine variety from \cite{ZerolociBernsteinSatoideal}. We review some general properties of such modules (which we will need later on) and show that the multivariate $V$-filtration on a graph embedding gives an example.

\begin{definition}\label{defn: singular support}
Let $R$ be the coordinate ring of a smooth algebraic variety $\spec R$ over $\mb{C}$. If $\mc{N}$ is a coherent $\sD_X\otimes_{\C}R$-module, the \emph{singular support}\footnote{It is referred to as the \emph{relative characteristic variety} of $\cN$ in \cite[\S 3.2]{ZerolociBernsteinSatoideal}.} of $\cN$ is defined by
\[ \mrm{SS}(\mc{N}) \colonequals \supp \Gr^F \mc{N} \subset T^*X \times (\spec R)^{\textrm{an}},\]
for any good filtration $F_\bullet \mc{N}$ compatible with the filtration $F_p(\ms{D}_X\otimes_{\C}R)\colonequals (F_p\ms{D}_X)\otimes_{\C}R$. We say that $\mc{N}$ is \emph{relative holonomic} \cite[Definition 3.2.3]{ZerolociBernsteinSatoideal} if, locally around any point in $X$, we have
\[ \mrm{SS}(\mc{N}) = \bigcup_i L_i \times S_i^{\mrm{an}},\]
for some finite collection of conical Lagrangians $L_i \subset T^*X$ and algebraic subvarieties $S_i \subset \spec R$. 
\end{definition}

Here is a useful criterion for relative holonomicity. The proof is a fairly straightforward consequence of Gabber's integrability theorem \cite{gabber}; see \cite[Proposition 8]{Maisonobe16} (which treats the case $R = \mb{C}[\bs{s}]$) and \cite[Proposition 3.2.5]{ZerolociBernsteinSatoideal} (which treats the case of algebraic $\ms{D}$-modules).

\begin{prop} \label{prop:relative holonomic criterion}
Assume that $\mc{N}$ is a coherent $\ms{D}_X \otimes R$-module such that $\mrm{SS}(\mc{N}) \subset L \times (\spec R)^{\mathrm{an}}$ for some conical Lagrangian $L \subset T^*X$. Then $\mc{N}$ is relative holonomic.
\end{prop}

The relevance to the multivariate $V$-filtration is as follows. Suppose that $\mc{M}$ is a holonomic $\ms{D}_X$-module and $\iota \colon X \to X \times \mb{C}^r$ is the graph embedding for holomorphic functions $f_1, \ldots, f_r \colon X \to \mb{C}$ as in our running setting. Consider the multivariate $V_*$-filtration $V_*^{\bs{\alpha}}\iota_+\mc{M}$ (see Notation \ref{notation: V* for arbitrary M}), regarded as a $\ms{D}_X[\bs{s}] = \ms{D}_X \otimes \mb{C}[\bs{s}]$-module in the usual way. Recall from Corollary \ref{cor:localising V} that $V_*^{\bs{\alpha}}\iota_+\mc{M} = V^{\bs{\alpha}}\iota_+\mc{M}$ if $\bs{\alpha} \in \mb{R}_{>0}^r$, and $V_*^{\bs{\alpha}}\iota_+\mc{M} = (t_1 \cdots t_r)^{-N} V^{\bs{\alpha} + N\bs{1}}\iota_+\mc{M}$ for $N \gg 0$ in general.

\begin{prop} \label{prop:V is relative holonomic}
For any $\bs{\alpha}\in \mb{R}^r$, the $\ms{D}_X[\bs{s}]$-module $V^{\bs{\alpha}}_*\iota_+\mc{M}$ is relative holonomic.
\end{prop}

\begin{rmk}
For $\bs{\alpha} \in \mb{R}^r_{>0}$, we have $V^{\bs{\alpha}}\iota_+\mc{M} = V^{\bs{\alpha}}_*\iota_+\mc{M}$ by Corollary \ref{cor:localising V}, so Proposition \ref{prop:V is relative holonomic} implies that $V^{\bs{\alpha}}\iota_+\mc{M}$ is relative holonomic. Although we will not need this fact, the latter claim also holds for arbitrary $\bs{\alpha} \in \mb{R}^r$. Indeed, if $\mc{M} \subset \mc{M}(*D)$, then
\[ V^{\bs{\alpha}}\iota_+\mc{M} \subset  V^{\bs{\alpha}}\iota_+\mc{M}(\ast D)=\bs{t}^{-\bs{n}} V^{\bs{\alpha} + \bs{n}}\iota_+\mc{M} = \bs{t}^{-\bs{n}}V_*^{\bs{\alpha} + \bs{n}}\iota_+\mc{M}\]
for large enough $\bs{n} \in \mb{Z}^r$. Since the right hand side is relative holonomic, so is the left. In general, since the class of relative holonomic $\ms{D}_X$-modules is clearly closed under extensions, it therefore remains to check that the kernel of $\mc{M} \to \mc{M}(*D)$ is relative holonomic. But this follows by a straightforward induction argument using Lemma \ref{lem:supported on a component}.
\end{rmk}

We first check coherence:

\begin{lemma}\label{lem: DXs coherence}
For all $\bs{\alpha}\in \R^r$, $V^{\bs{\alpha}}\iota_{+}\cM$ is $\sD_X[\bs{s}]$-coherent.
\end{lemma}
\begin{proof}
Choose local generators $v_1, \ldots, v_n$ for $V^{\bs{\alpha}}\iota_+\mc{M}$ over $V^{\bs{0}}\ms{D}_{X \times \mb{C}^r}$. Since $\iota_+\mc{M}$ is supported on the graph of $(f_1,\ldots,f_r)$, there exists $m \geq 0$ such that $(t_i - f_i)^m v_j = 0$ for all $i$, $j$. So $V^{\bs{\alpha}}\iota_+\mc{M}$ is locally generated over $\sD_X[\bs{s}]$ by the sections $t^k_i v_j$ for $k < m$, and is therefore $\ms{D}_X[\bs{s}]$-coherent as claimed.
\end{proof}

\begin{proof}[Proof of Proposition \ref{prop:V is relative holonomic}]

We may as well assume for simplicity that $\bs{\alpha} \in \mb{R}^r_{>0}$, so that $V^{\bs{\alpha}}_*\iota_+\mc{M} = V^{\bs{\alpha}}\iota_+\mc{M}$. By Lemma \ref{lem: DXs coherence} and Proposition \ref{prop:relative holonomic criterion}, it suffices to show that $\mrm{SS}(V^{\bs{\alpha}}\iota_+\mc{M}) \subset L \times \mb{C}^r$ for some conical Lagrangian $L \subset T^*X$. 

We first consider the case where $D\colonequals \textrm{div}(f_1\ldots f_r)_{\textrm{red}}$ has simple normal crossings and  $\mc{M}$ is a flat meromorphic connection on $X$ with poles along a simple normal crossings divisor $E \supset D$ and unramified good formal structure (see \S \ref{subsec:normal crossings}). In this case, we claim that the lemma holds with $L$ equal to the union of conormal bundles to strata of $E$. 
To see this, first observe that, since in this case $f_i$ are monomials in the coordinates of $X$, by Proposition \ref{prop:change of functions} 
it suffices to consider the case where $f_1 = x_1, \ldots, f_r = x_r$ are local coordinates cutting out the irreducible components of $D$. Here we use the elementary fact that if $\spec R' \to \spec R$ is a morphism of smooth affine varieties and $\mc{N}$ is a $\ms{D}_X\otimes R$-module, then 
\[\mrm{SS}(R' \otimes_R \mc{N}) \subset \mrm{SS}(\mc{N}) \times_{(\spec R)^{\mrm{an}}} (\spec R')^{\mrm{an}}.\]
Let $x_{r + 1}, \ldots, x_{r + k}$ be further local coordinates cutting out the remaining irreducible components of $E$. By Proposition \ref{prop:V-filtration on graph}, 
\[ \mrm{SS}(V^{\bs{\alpha}}\iota_+\mc{M}) \subset \left(\mrm{SS}(V^{\bs{\alpha}}\mc{M}) \times_{T^*X(\log D)} T^*X\right) \times \mb{C}^r.\]
Here $\mrm{SS}(V^{\bs{\alpha}}\mc{M}) \subset T^*X(\log D)$ is the singular support of $V^{\bs{\alpha}}\cM$ as a $V^{\bs{0}}\ms{D}_X$-module, defined to be the support of the $\gr^FV^{\bs{0}}\ms{D}_X = \mrm{Sym}_{\mc{O}_X}(TX(\log D))$-module $\gr^F V^{\bs{\alpha}}\mc{M}$ on $T^*X(\log D)$, for any choice of good filtration $F_\bullet V^{\bs{\alpha}}\mc{M}$. Hence, the claim follows from the assertion that, for all $i$,
\begin{equation} \label{eq:V is relative holonomic 1}
\mrm{SS}(V^{\bs{\alpha}}\mc{M}) \subset \begin{cases} \{(\partial_{x_i}x_i)x_i = 0\}, & \text{if $i = 1, \ldots, r$}, \\ \{\partial_{x_i} x_i = 0\}, & \text{if $i = r + 1, \ldots, r + k$}.\end{cases}
\end{equation}
Note that $\partial_{x_i}x_i$ is a local coordinate on $T^*X(\log D)$ if $i = 1, \ldots, r$. Since the map $\cO_{X,x}\to \widehat{\mc{O}}_{X, x}$ is faithfully flat, \eqref{eq:V is relative holonomic 1} can be checked after completing at each point, where it is clear by direct calculation using \eqref{eq:normal crossings 2}.

Next, observe that if $\pi \colon \tilde X \to X$ is a proper map, $\tilde{\iota} \colon \tilde X \to \tilde X \times \mb{C}^r$ is the graph embedding for $f_1 \circ \pi, \ldots, f_r \circ \pi$ and $\mc{M} = \pi_+\tilde{\mc{M}}$ for some holonomic $\ms{D}_{\tilde X}$-module satisfying the conclusion of Proposition \ref{prop:V is relative holonomic}, then since $V^{\bs{\alpha}}\iota_+\mc{M} = \pi_+V^{\bs{\alpha}}\tilde{\iota}_+\tilde{\mc{M}}$ by Theorem \ref{thm:multivariate V direct image}, by the analogue of Kashiwara's estimate on the singular support for $\sD_X[\bs{s}]$-modules, we have
\[ \mrm{SS}(V^{\bs{\alpha}}\iota_+\mc{M}) \subset (p \times \id)(q \times \id)^{-1}(\mrm{SS}(V^{\bs{\alpha}}\tilde{\iota}_+\tilde{\mc{M}})),\]
where $p$ and $q$ are the morphisms
\[ T^*\tilde{X} \xleftarrow{q} T^*X \times_X \tilde X \xrightarrow{p} T^*X.\]
Since the operation $p\circ q^{-1}$ sends Lagrangian subvarieties to Lagrangian subvarieties, the conclusion of Proposition \ref{prop:V is relative holonomic} therefore holds for $\mc{M}$. Applying this in the case of a ramified covering, we deduce that Proposition \ref{prop:V is relative holonomic} holds when $\mc{M}$ is a flat meromorphic connection with (possibly ramified) good normal crossings singularities.

Finally, consider the general case. Since $\bs{\alpha}\in \R^r_{>0}$, by Corollary \ref{cor:localising V} we can assume that $\mc{M} = \mc{M}(*D)$. Since the singular support takes extensions to unions, we may as well assume that $\mc{M} = i_+\mc{N}(*E)$ for some closed analytic subset $i \colon Z \hookrightarrow X$ and some flat meromorphic connection $\mc{N}$ on $Z$ with singularities along a divisor $E \supset D$ such that $Z \setminus E$ is smooth. By Theorem \ref{thm:resolution of turning points}, we can take a resolution of turning points $\pi \colon \tilde Z \to Z \subset X$ and write $\mc{M} = \pi_+\tilde{\mc{N}}$ for a flat meromorphic connection $\tilde{\mc{N}}$ with good normal crossings singularities. Since the proposition holds for $\tilde{\mc{N}}$, it also does for $\mc{M}$ as argued the above.
\end{proof}

\begin{remark}
Although we will not need this fact, one can show using arguments similar to the proof of Theorem \ref{thm:flatness} below that $\mrm{SS}(V^{\bs{\alpha}}_*\iota_+\mc{M}) = \mrm{SS}(\mc{M}(*D))\times \mb{C}^r$ for all $\bs{\alpha} \in \mb{R}^r$.
\end{remark}
It will also be useful to consider the following generalisation of the notion of relative holonomicity and singular support.

\begin{defn} \label{defn:localised singular support}
Let $R$ be a localisation of the coordinate ring of a smooth affine variety over $\mb{C}$. We say that a coherent $\ms{D}_X \otimes R$-module $\mc{N}$ is \emph{relative holonomic} if, locally on $X$, it is of the form $\mc{N} = \mc{N}' \otimes_{R'} R$ for some smooth affine variety $\spec R'$ where $R$ is a localisation of $R'$ and some relative holonomic $\ms{D}_X \otimes R'$-module $\mc{N}'$. In this case, we define the \emph{singular support} of $\mc{N}$ to be
\[ \mrm{SS}(\mc{N}) = \bigcup_i L_i \times (S_i \times_{\spec R'} \spec R) \subset T^*X \times \spec R \quad \text{if}\;\; \mrm{SS}(\mc{N}') = \bigcup_i L_i \times S_i^{\mrm{an}},\]
for any choice of $R'$ and $\mc{N}'$ as above, where $S_i\subset \spec R'$. Note that $\mrm{SS}(\mc{N})$ is independent of this choice since taking singular support in the sense of Definition \ref{defn: singular support} commutes with taking finite localisations of $R'$ (cf.\ \cite[Remark 3.2.1]{ZerolociBernsteinSatoideal}).
\end{defn}

In this generality, the relative holonomic modules form a Serre subcategory of the category of coherent $\ms{D}_X \otimes R$-modules. 
When $R$ is the coordinate ring of a smooth affine variety, this follows from \cite[Lemma 3.2.4]{ZerolociBernsteinSatoideal}, for example. When $R$ is a localisation of the coordinate ring of a smooth affine variety, it follows from this plus the fact that any finite diagram of coherent $\ms{D}_X \otimes R$-modules is, locally on $X$, obtained as the localisation of a diagram of coherent $\ms{D}_X \otimes R'$-modules for some smooth affine variety $\spec R'$.

Finally, for certain $R$, we can also define \emph{holonomic} $\ms{D}_X \otimes R$-modules; as we will see in the next subsection, these are especially well-behaved under duality. If $R$ is a commutative ring, we say that $R$ is \emph{equidimensional} if every maximal ideal $\mf{m} \subset R$ has the same codimension (necessarily equal to the dimension of $R$). For example, any regular local ring is equidimensional, while the localisation of $\mb{C}[\bs{s}]$ at the set of functions that do not vanish at the generic points of a given collection of subvarieties of different dimensions is not.

\begin{defn} \label{defn:localised holonomic}
Suppose that $R$ is a localisation of the coordinate ring of a smooth affine variety and that $R$ is equidimensional of dimension $d$. We say that a coherent $\ms{D}_X \otimes R$-module $\mc{N}$ is \emph{holonomic} if it is relative holonomic and, locally on $X$,
\[ \mrm{SS}(\mc{M}) = \bigcup_i L_i \times S_i,\]
a finite union, where each $S_i \subset \spec R$ is of dimension $0$.
\end{defn}

\begin{rmk}
Note that, by \cite[Lemma 3.4.1]{ZerolociBernsteinSatoideal}, if $\mc{N}$ is relative holonomic then $\mc{N}$ is holonomic if and only if its support as an $R$-module is zero dimensional.
\end{rmk}

\subsection{Duality} \label{subsec:duality}

We now turn to the duality theory for relative holonomic modules in general, and the multivariate $V$-filtration on a graph embedding in particular.

Continuing as in the previous subsection, let $R$ be a localisation of the coordinate ring of a smooth affine variety.

\begin{defn} \label{defn:duality}
If $\mc{N}$ is a coherent $\ms{D}_X\otimes R$-module, we define the \emph{dual} of $\mc{N}$ to be 
\[ \mb{D}\mc{N} := \mrm{R}\shom_{\ms{D}_X \otimes R}(\mc{N}, (\ms{D}_X \otimes_{\mc{O}_X} \omega_X^{-1}) \otimes_{\mb{C}} R)[\dim X] \in \mrm{D}^b_{\mrm{coh}}(\ms{D}_X \otimes R).\]
More generally, we define the dual of any object in $\mrm{D}^b_{\mrm{coh}}(\ms{D}_X \otimes R)$ by the same formula. We will sometimes write $\mb{D} = \mb{D}_{\ms{D}_X\otimes R}$ if we need to make the coefficient ring explicit.
\end{defn}

We have the following basic properties of the duality functor.

\begin{prop} \label{prop:basic duality}
For $\mc{N} \in \mrm{D}^b_{\mrm{coh}}(\ms{D}_X \otimes R)$, we have the following.
\begin{enumerate}
\item If $R'$ is a localisation of $R$, then $\mb{D}(\mc{N} \otimes_R R') = \mb{D}\mc{N} \otimes_R R'$.
\item If the cohomology sheaves of $\mc{N}$ are relative holonomic, then the same holds for $\mb{D}\mc{N}$.
\item We have $\mb{D}\mb{D}\mc{N} \cong \mc{N}$.
\end{enumerate}
\end{prop}
\begin{proof}
Assertions (1) and (3) are clear, for example, by computing the dual (locally) using a resolution of $\mc{N}$ by modules of the form $\ms{D}_X \otimes P$ for $P$ a projective $R$-module. Assertion (2) holds when $R$ is the coordinate ring of a smooth affine variety by \cite[Lemma 3.2.4]{ZerolociBernsteinSatoideal}, and hence in general by (1). 
\end{proof}

As in the case of $\ms{D}_X$-modules, duality defines an exact anti-equivalence on the category of holonomic modules.

\begin{prop} \label{prop:holonomic duality}
Assume that $R$ is equidimensional of dimension $d$. Then for any relative holonomic $\ms{D}_X \otimes R$-module $\mc{N}$, we have
\begin{enumerate}
\item $\dim \mrm{SS}\left(\mrm{H}^i(\mb{D}\mc{N})\right) \leq d + \dim X - i$,
\item $\mrm{H}^i(\mb{D}\mc{N}) = 0$ unless $d + \dim X - \dim\mrm{SS}(\mc{N}) \leq i \leq d$, and
\item if $\mc{N}$ is holonomic, then $\mrm{H}^i(\mb{D}\mc{N}) = 0$ unless $i = d$, and $\mrm{H}^d(\mb{D}\mc{N})$ is holonomic.
\end{enumerate}
\end{prop}

\begin{remark}
In view of Proposition \ref{prop:holonomic duality}, we write $\mb{D}_{hol}\mc{N} = \mrm{H}^d(\mb{D}\mc{N}) = \mb{D}\mc{N}[d]$ if $\mc{N}$ is a holonomic $\ms{D}_X\otimes R$-module.
\end{remark}
\begin{proof}
We first consider (1) and the lower bound in (2) in the case where $R$ is the coordinate ring of a smooth affine variety. By a standard spectral sequence argument (cf.\ \cite[Proposition 4.2.2]{ZerolociBernsteinSatoideal}) it suffices to prove the assertions with the $\gr^F\ms{D}_X \otimes R$-module
\[ \mrm{Ext}^{i + \dim X}_{\gr^F \ms{D}_X \otimes R}(\gr^F \mc{N}, \gr^F\ms{D}_X \otimes R)\]
in place of $\mrm{H}^i(\mc{N})$, where $F_\bullet \mc{N}$ is any good filtration with respect to $F_{\bullet}(\sD_X\otimes R)$. By the discussion in \cite[\S3.6]{ZerolociBernsteinSatoideal}, we may further replace this with
\[ \mrm{Ext}^{i + \dim X}_A(M, A),\]
where $A$ is the localisation at the unique graded maximal ideal of the stalk of $\gr^F\ms{D}_X \otimes \mc{O}_{(\spec R)^{\textrm{an}}}$ at any point of $X \times (\spec R)^{\textrm{an}}$, and $M = \gr^F\mc{N} \otimes_{\gr^F\ms{D}_X \otimes R} A$. The ring $A$ is a regular local ring of dimension $d + 2 \dim X$ and the module $M$ is finitely generated. So by standard commutative algebra (e.g.\ \cite[Corollary 3.5.11]{BH})
\[ \dim \supp \mrm{Ext}^{i + \dim X}_A(M, A) \leq (d+2 \dim X) - i - \dim X = d+\dim X - i \]
and
\[ \mrm{Ext}^{i + \dim X}_A(M, A) = 0 \quad \text{if $i + \dim X < (d + 2\dim X) - \dim \supp M$}.\]
Since $\dim \supp M \leq \dim \mrm{SS}(\mc{N})$, this gives (1) and the lower bound in (2) in this case.

Next consider (1) and the lower bound in (2) in the general case. Working locally on $X$, write $R$ as the localisation of the coordinate ring $R'$ of a smooth affine variety, say of dimension $n$, and $\mc{N}$ as the localisation of a relative holonomic $\ms{D}_X \otimes R'$-module $\mc{N}'$. Since $R$ is equidimensional dimension $d$, it follows that for any irreducible subvariety $S' \subset \spec R'$, the fibre product $S = S' \times_{\spec R'} \spec R$ has dimension $\dim S = \dim S' + d - n$. Since singular support and duality commute with localisation, (1) and the lower bound in (2) for $\mc{N}$ follow from the same statements for $\mc{N}'$.

To get the upper bound in (2), we note that since $\mrm{H}^i(\mb{D}\mc{N})$ is relative holonomic, we have $\dim \mrm{SS}(\mrm{H}^i(\mb{D}\mc{N})) \geq \dim X$ if $\mrm{H}^i(\mb{D}\mc{N}) \neq 0$, so this follows from (1).

Finally, to get (3), note that $\mc{N}$ is holonomic if and only if $\dim \mrm{SS}(\mc{N}) \leq \dim X$, so the vanishing follows from (2) and the holonomicity of $\mrm{H}^d(\mb{D}\mc{N})$ follows from (1).
\end{proof}

Return now to a holonomic $\ms{D}_X$-module $\mc{M}$ and the graph embedding $\iota \colon X \to X \times \mb{C}^r$ of a sequence of holomorphic functions $f_1, \ldots, f_r$, with associated divisor $D = \{f_1 \cdots f_r = 0\}$. We have the following duality theorem for the multivariate $V$-filtration on $\iota_+\mc{M}$. Recall from Notation \ref{notation: V* for arbitrary M} that $V^{\bullet}_{\ast}\iota_{+}\cM\colonequals V^{\bullet}_{\ast}\iota_{+}\cM(\ast D)$.

\begin{thm} \label{thm:duality}
We have a canonical isomorphism
\[ \mb{D}V_*^{\bs{\alpha}}\iota_+\mc{M} \cong (V_*^{-\bs{\alpha} + \epsilon \bs{1}}\iota_+\mb{D}\mc{M})_{\bs{s} \mapsto -\bs{s}},\]
where $(-)_{\bs{s} \mapsto -\bs{s}}$ denotes the twist of a $\ms{D}_X[\bs{s}]$-module by the automorphism of $\ms{D}_X[\bs{s}]$ sending $s_i$ to $-s_i$. In particular, $\mb{D}V_*^{\bs{\alpha}}\iota_{+}\mc{M}$ has cohomology in degree $0$ only.
\end{thm}

\begin{rmk}\label{remark: recover usual self-duality}
When $r = 1$, Theorem \ref{thm:duality} refines the self-duality of the nearby cycles functor $\bigoplus_{\alpha \in (0, 1]} \gr_V^\alpha \iota_+(-)$ \cite{SabbahVfiltration} (see also \cite[Theorem 7.1(i),(iii),(iv)]{CDM}) as follows. If $\alpha \in (0, 1)$, then by Corollary \ref{cor:localising V} we have
\[\mb{D}_{\ms{D}_X[s]}V^{\alpha}\iota_+\mc{M} = \mb{D}_{\ms{D}_X[s]}V^\alpha_*\iota_+\mc{M} \cong \left(V^{>-\alpha}_*\iota_+\mb{D}\mc{M} \right)_{s\mapsto -s}\cong \left(V^{>1 - \alpha}\iota_+\mb{D}\mc{M}\right)_{s\mapsto -s},\]
where the first isomorphism is Theorem \ref{thm:duality} and the second is multiplication by $t$. Taking the cone of $V^{>\alpha} \to V^\alpha$, we therefore get an isomorphism
\[ \mb{D}_{\ms{D}_X}\gr_V^\alpha\iota_+\mc{M} \cong \mb{D}_{\ms{D}_X[s]}\gr_V^\alpha \iota_+\mc{M}[1] \cong \gr_V^{1 - \alpha}\iota_+\mb{D}\mc{M},\]
where the first isomorphism holds since $s + \alpha$ acts nilpotently on $\gr_V^\alpha \iota_+\mc{M}$. If $\alpha = 1$, since $V_*^0\iota_+\mb{D}\mc{M} \neq V^0\iota_+\mb{D}\mc{M}$ in general, we have to multiply by an extra factor of $t$ to get $V_*^0\iota_+\mb{D}\mc{M} \cong V^1\iota_+\mb{D}\mc{M}$; so we get a self-duality on $\gr_V^1\iota_{+}(-)$. The theorem says nothing, on the other hand, about the (unipotent) vanishing cycles functor $\gr_V^0\iota_+(-)$. 
\end{rmk}

\begin{proof}[Proof of Theorem \ref{thm:duality}]
It suffices to check that the theorem holds after localising at the divisor $D$. Using the Malgrange-Mellin transform \eqref{eqn: Malgrange isomorphism}, we may write
\begin{equation}\label{eqn:MM for *D} \iota_+\mc{M}(*D) = \mc{M}(*D)[\bs{s}]\bs{f^s}\end{equation}
as $\ms{D}_X(*D)[\bs{s}]$-modules. Now, observe that we have an automorphism $\mu$ of the sheaf of rings $\ms{D}_X(*D)[\bs{s}]$ given by
\[ \mu(g) = g, \quad \mu(s_i) = s_i, \quad \text{and} \quad \mu(\xi) = \xi + \sum_{i = 1}^r s_i \frac{\xi f_i}{f_i}\]
for $g \in \mc{O}_X(*D)$ and $\xi$ a vector field,
such that $\mc{M}(*D)[\bs{s}]\bs{f^s} = \mu(\mc{M}(*D)[\bs{s}])$ is the module obtained from $\mc{M}(*D)[\bs{s}]$ by composing the action with $\mu$. Similarly, it is straightforward to check that the map $\mu \otimes \mathrm{id}$ defines an isomorphism
\[ \mu \otimes \mrm{id} \colon \ms{D}_X(*D)[\bs{s}] \otimes_{\mc{O}_X}\omega_X^{-1} = \left(\ms{D}_X(*D) \otimes_{\mc{O}_X} \omega_X^{-1}\right) [\bs{s}] \cong (\mu, \mu^{-1})\left(\ms{D}_X(*D)[\bs{s}]\otimes_{\mc{O}_X}\omega_X^{-1}\right)\]
of the $\ms{D}_X(*D)[\bs{s}] \otimes \ms{D}_X(*D)[\bs{s}]$-module $(\ms{D}_X(*D) \otimes_{\mc{O}_X}\omega_X^{-1})[\bs{s}]$ with its $(\mu, \mu^{-1})$-twist.

So if we write $\mb{D}_{\ms{D}_X(*D)[\bs{s}]}$ for the duality on the derived category of $\ms{D}_X(*D)[\bs{s}]$-modules defined by
\[\mb{D}_{\ms{D}_X(*D)[\bs{s}]}\mc{N} := \mrm{R}\shom_{\ms{D}_X(*D)[\bs{s}]}(\mc{N}, (\ms{D}_X(*D)\otimes_{\mc{O}_X} \omega_X^{-1})[\bs{s}])[\dim X] \]
by analogy with Definition \ref{defn:duality}, then
\[ \mb{D}_{\ms{D}_X(*D)[\bs{s}]}\mu(\mc{N}) \cong \mu^{-1}\mb{D}_{\ms{D}_X(*D)[\bs{s}]}\mc{N}.\]
Since
\[ \mb{D}_{\ms{D}_X(*D)[\bs{s}]}(\mc{M}(*D)[\bs{s}]) = (\mb{D}\mc{M})(*D)[\bs{s}],\]
we therefore by \eqref{eqn:MM for *D} have
\begin{equation}\label{eq:duality 2}
\begin{aligned}
\mb{D}_{\ms{D}_X(*D)[\bs{s}]}\left(\iota_+\mc{M}(*D)\right) &= \mb{D}_{\ms{D}_X(*D)[\bs{s}]}(\mu(\mc{M}(*D)[\bs{s}])) \\
&= \mu^{-1}\left((\mb{D}\mc{M})(*D)[\bs{s}]\right) = (\iota_+\mb{D}\mc{M}(*D))_{\bs{s} \mapsto -\bs{s}}.
\end{aligned}
\end{equation}

To extend this to the $V$-filtration, let us fix an injective resolution $\mc{I}^\bullet$ for $(\ms{D}_X \otimes \omega_X^{-1})[\bs{s}]$ as a $(\ms{D}_X \otimes \ms{D}_X)[\bs{s}, \bs{t}, \bs{t}^{-1}]$-module, where $t_i$ acts by \eqref{eqn: ti action}. Since $(\ms{D}_X \otimes \ms{D}_X)[\bs{s}, \bs{t}, \bs{t}^{-1}]$ is flat over $(\ms{D}_X \otimes \ms{D}_X)[\bs{s}]$, the terms of $\mc{I}^\bullet$ are also injective as $(\ms{D}_X \otimes \ms{D}_X)[\bs{s}]$-modules. So for all $\bs{\alpha} \in \mb{R}^r$, we have
\[ (\mb{D}V_*^{-\bs{\alpha} + \epsilon \bs{1}}\iota_+\mc{M})_{\bs{s} \mapsto -\bs{s}} \cong \shom_{\ms{D}_X[\bs{s}]}(V_*^{-\bs{\alpha} + \epsilon \bs{1}}\iota_+\mc{M}, \mc{I}^\bullet[\dim X])_{\bs{s} \mapsto -\bs{s}} =: U^{\bs{\alpha}}\]
as objects in $\mrm{D}^b_{\mrm{coh}}(\ms{D}_X[\bs{s}])$. We need to show that
\[ U^{\bs{\alpha}} \cong V_*^{\bs{\alpha}}\iota_+\mb{D}\mc{M}.\]
For $\bs{\alpha} \leq \bs{\beta}$, the inclusion $V_*^{-\bs{\alpha} + \epsilon \bs{1}}\iota_+\mc{M} \to V_*^{-\bs{\beta} + \epsilon\bs{1}}\iota_+\mc{M}$ induces a map $\bs{u}^{\bs{\beta} - \bs{\alpha}} \colon U^{\bs{\beta}} \to U^{\bs{\alpha}}$, while for each $i$, we have an isomorphism
\[ u_i^{-1}t_i \colon U^{\bs{\alpha}} \overset{\sim}\to (U^{\bs{\alpha} + \bs{1}_i})_{\bs{s} \mapsto \bs{s} - \bs{1}_i} \]
sending a morphism
\[ h \colon V^{-\bs{\alpha} + \epsilon \bs{1}}_*\iota_+\mc{M} \to \mc{I}^\bullet[\dim X] \]
to the morphism
\[ t_i^{-1} h t_i \colon V^{-(\bs{\alpha} + \bs{1}_i) + \epsilon \bs{1}}_*\iota_+\mc{M} \to \mc{I}^\bullet[\dim X]. \]
Putting these together, we obtain a complex of $\mb{R}^r$-graded modules
\[ R_{U, \mb{R}} := \bigoplus_{\bs{\alpha} \in \mb{R}^r} U^{\bs{\alpha}} \bs{u}^{-\bs{\alpha}}\]
over
\[ R_{V, \mb{R}}\ms{D}_X[\bs{s}, \bs{t}, \bs{t}^{-1}] = \ms{D}_X[\bs{s}][\bs{u}^{\mb{R}_{\geq 0}}][\bs{u}^{-1}\bs{t}, \bs{u}\bs{t}^{-1}].\]

Thus, in the language of \S\ref{subsec:direct image}, $R_{U, \mb{R}}$ defines an object $(\mc{N}, U^\bullet\mc{N})$ in the filtered derived category $\mrm{D}_\mb{R}(\ms{D}_X[\bs{s}, \bs{t}, \bs{t}^{-1}], V^\bullet)$ with $R_{U, \mb{R}}\mc{N} = R_{U, \mb{R}}$. By construction, it is a wall and chamber object with set of walls
\[ -\mc{W} = \{\bs{L}^{-1}(-\gamma) \mid \bs{L}^{-1}(\gamma) \in \mc{W}\}\]
for any set of walls $\mc{W}$ of $V_*^\bullet\iota_+\mc{M}$. Moreover, the cohomology of each $U^{\bs{\alpha}}$ is coherent over $\ms{D}_X[\bs{s}]$ and hence over $V^{\bs{0}}\ms{D}_X[\bs{s}, \bs{t}, \bs{t}^{-1}]$, so $R_{U, \mb{R}}$ is a good wall and chamber object by the same argument as Proposition \ref{prop:good V* criterion}. Finally, if $\bs{\alpha} \leq \bs{\beta}$ are separated by a single wall $\bs{L}^{-1}(-\gamma)$ of $-\mc{W}$, then $-\bs{\beta} + \epsilon \bs{1} \leq -\bs{\alpha} + \epsilon \bs{1}$ are separated by the single wall $\bs{L}^{-1}(\gamma)$ of $\mc{W}$, so there exist $\gamma_j$ with $\Re \gamma_j = \gamma$ such that
\[ \prod_j (\bs{L}(\bs{s}) + \gamma_j) \colon \frac{V_*^{-\bs{\beta} + \epsilon \bs{1}}\iota_+\mc{M}}{V_*^{-\bs{\alpha} + \epsilon \bs{1}}\iota_+\mc{M}} \to \frac{V_*^{-\bs{\beta} + \epsilon \bs{1}}\iota_+\mc{M}}{V_*^{-\bs{\alpha} + \epsilon \bs{1}}\iota_+\mc{M}} \]
is zero. So taking homs into $\mc{I}^\bullet$ and applying $\bs{s} \mapsto -\bs{s}$, we deduce that
\[ \prod_j (\bs{L}(\bs{s}) - \gamma_j) = \pm \prod_j (\bs{L}(-\bs{s}) + \gamma_j) \colon \mrm{Cone}(U^{\bs{\beta}} \to U^{\bs{\alpha}})\]
is also zero. So applying Lemma \ref{lem:derived multivariate} (whose statement and proof apply equally well to $V_*$-filtrations over $\ms{D}_X[\bs{s}, \bs{t}, \bs{t}^{-1}]$), we deduce that the object $(\mc{N}, U^\bullet\mc{N})$ is strict and induces the $V_*$-filtration on its cohomologies.

To complete the proof, it remains to identify the underlying complex of $\ms{D}_X[\bs{s}, \bs{t}, \bs{t}^{-1}]$-modules $\mc{N} := \varinjlim_{\bs{\alpha} \in \mb{R}^r} U^{\bs{\alpha}}$ with $\iota_+\mb{D}\mc{M}(*D)$ as objects in the derived category. Since every local section of $V_*^{ - \bs{\alpha} - \epsilon \bs{1}}\iota_+\mc{M}$ is annihilated by a power of $t_i - f_i$ for all $\bs{\alpha}$, one can, locally on $X$, construct a resolution of $V_*^{ - \bs{\alpha} - \epsilon \bs{1}}\iota_+\mc{M}$ by $\ms{D}_X[\bs{s}, \bs{t}]$-modules with the same property that are free over $\ms{D}_X[\bs{s}]$. Taking duals, we conclude that, for all $n$, every local section of $\mc{H}^n(U^{\bs{\alpha}}\mc{N})$ and hence $\mc{H}^n(\mc{N})$ is annihilated by some power of $t_i - f_i$. Since $t_i$ acts bijectively on $\mc{H}^n(\mc{N})$, it therefore follows that $f_i$ acts bijectively on $\mc{H}^n(\mc{N})$ also. So by \eqref{eq:duality 2},
\begin{align*}
\mc{H}^n(\mc{N}) = \mc{H}^n(\mc{N})(*D) &= \varinjlim_{\bs{\alpha} \in \mb{R}^r} \mc{H}^n(U^{\bs{\alpha}}(*D)) \\
&\quad= \mc{H}^n\mb{D}_{\ms{D}_{X}(*D)[\bs{s}]}\left(\iota_+\mc{M}(*D)\right)_{\bs{s} \mapsto -\bs{s}} = \mc{H}^n(\iota_+\mb{D}\mc{M}(*D)).
\end{align*}
Since $\mc{M}$ is holonomic, the right hand side is zero unless $n = 0$, so we are done.\end{proof}

\begin{rmk}
A version of Theorem \ref{thm:duality} also holds, with essentially the same proof, for $\mb{D}_{V^{\bs{0}}\ms{D}_X}V_*^{\bs{\alpha}}\mc{M}$, 
where $\mb{D}_{V^{\bs{0}}\ms{D}_X}$ now denotes duality for $V^{\bs{0}}\ms{D}_X$-modules
\[ \mb{D}_{V^{\bs{0}}\ms{D}_X}\mc{N} := \mrm{R}\shom_{V^{\bs{0}}\ms{D}_X}(\mc{N}, V^{\bs{0}}\ms{D}_X\otimes_{\mc{O}_X} \omega_X^{-1})[\dim X].\]
\end{rmk}

\subsection{Flatness} \label{subsec:flatness}
In this subsection, we use the duality theorem (Theorem \ref{thm:duality}) together with the formalism of relative holonomicity to deduce flatness results for the multivariate $V$-filtration on a graph embedding. As before, we let $\mc{M}$ be a holonomic $\ms{D}$-module on a complex manifold $X$, $f_1, \ldots, f_r \colon X \to \mb{C}$ be holomorphic functions, $\iota \colon X \to X \times \mb{C}^r$ be the associated graph embedding. Consider the multivariate $V_*$-filtration $V_{\ast}^\bullet\iota_+\mc{M}$ along the coordinate axes in $\mb{C}^r$, where $V^{\bullet}_{\ast}\iota_{+}\cM=V^{\bullet}_{\ast}\iota_{+}\cM(\ast D)$ and $D=\textrm{div}(f_1\cdots f_r)$ (see  Notation \ref{notation: V* for arbitrary M}).

\begin{thm} \label{thm:flatness}
For any $\bs{\alpha} \in \mb{R}^{r}$, $V^{\bs{\alpha}}_{\ast}\iota_+\mc{M}$ is flat over $\C[\bs{s}]$. Moreover, if $\bs{\alpha}\leq \bs{\beta}$ are separated by a single wall $W \subset \mb{R}^r$, then $V_{\ast}^{\bs{\alpha}}\iota_{+}\mc{M}/V_{\ast}^{\bs{\beta}}\iota_{+}\mc{M}$ is flat relative to $W$.
\end{thm}

Here we have used the following terminology.

\begin{defn} \label{defn:relative flatness}
Let $M$ be a $\mb{C}[\bs{s}]$-module and $W = \bs{L}^{-1}(\gamma) \subset \mb{R}^r$. We say that $M$ is \emph{flat relative to $W$} if $M$ is flat over $\mb{C}[U]$ for any linear subspace $U \subset \mrm{span}\{s_1, \ldots, s_r\}$ not containing $\bs{L}(\bs{s})$.
\end{defn}

\begin{remark}
By Corollaries \ref{cor:localising V} and \ref{cor:extending V*}, the same statements also hold for $V^{\bs{\alpha}}\iota_{+}\cM$ if $\bs{\alpha}\in \R^r_{>0}$, or when $\cM=\cM[(f_1\cdots f_r)^{-1}]$ and $\bs{\alpha}\in \R^r_{\geq 0}$. However, the flatness result can fail if $\bs{\alpha} \not\in \mb{R}^{r}_{>0}$. For example, if $\mc{M}$ is supported on the divisor $D_i$, then $s_i$ acts by zero on $V^{\bs{0}}\iota_{+}\mc{M}$, so $V^{\bs{0}}\iota_+\mc{M}$ is not a flat $\mb{C}[\bs{s}]$-module.
\end{remark}

Before giving the proof of Theorem \ref{thm:flatness}, we first note two key corollaries. 
\begin{corollary}\label{cor: subtle injectivity}
Assume $\{d_{ij}\}_{1\leq i\leq r,1\leq j\leq r'}\subset \Q_{\geq 0}$. If $\bs{\beta}\leq \bs{\alpha}$ are separated by a wall $\bs{L}^{-1}(\gamma)$ such that $\sum_i L_id_{ij}\neq 0$ for some $j$, then the map
\begin{equation}\label{eqn: tensor with Cs' is injective}V_{\ast}^{\bs{\alpha}}\iota_{+}\cM\otimes_{\mb{C}[\bs{s}]} \mb{C}[\bs{s}']\to V_{\ast}^{\bs{\beta}}\iota_{+}\cM\otimes_{\mb{C}[\bs{s}]} \mb{C}[\bs{s}']
\end{equation}
is injective, where the tensor products are taken over $s_i\mapsto \sum_jd_{ij}s_j'$.
\end{corollary} 
\begin{proof}
Let $U\subset \textrm{span}\{s_1,\ldots,s_r\}$ be the kernel of the linear map $s_i\mapsto \sum_j d_{ij}s_j'$. Let $W\subset \textrm{span}\{s_1',\ldots,s_r'\}$ be a linear subspace complementary to the image of $s_i\mapsto \sum_j d_{ij}s_j'$, and set $K=\C[W]$, which can be viewed as a $\C[U]$-module. Then there is a $\C[\bs{s}]$-module isomorphism $\C[\bs{s}]\overset{\mrm{L}}\otimes_{\C[U]}K\xrightarrow{\sim} \C[\bs{s}']$. It follows that
\[ V_{\ast}^{\bs{\beta}}/V_{\ast}^{\bs{\alpha}}\overset{\mrm{L}}\otimes_{\mb{C}[\bs{s}]} \mb{C}[\bs{s}'] = V_{\ast}^{\bs{\beta}}/V_{\ast}^{\bs{\alpha}}\overset{\mrm{L}}\otimes_{\mb{C}[\bs{s}]} \mb{C}[\bs{s}] \overset{\mrm{L}}\otimes_{\mb{C}[U]} K = V_{\ast}^{\bs{\beta}}/V_{\ast}^{\bs{\alpha}}\overset{\mrm{L}}\otimes_{\mb{C}[U]} K,\]
The assumption on the wall $\bs{L}^{-1}(\gamma)$ implies that $\bs{L}(\bs{s})\not\in U$. So by Theorem \ref{thm:flatness},
\[ \textrm{Tor}^{\C[\bs{s}]}_1(V_{\ast}^{\bs{\beta}}\iota_{+}\cM/V_{\ast}^{\bs{\alpha}}\iota_{+}\cM,\C[\bs{s}'])= \textrm{Tor}^{\C[U]}_1(V_{\ast}^{\bs{\beta}}\iota_{+}\cM/V_{\ast}^{\bs{\alpha}}\iota_{+}\cM,K)=0,\]
and this shows the desired injectivity.
\end{proof}
\begin{corollary}\label{cor: change of functions}
Assume $\cM=\cM[(f_1\cdots f_r)^{-1}]$. Suppose $f_j' = \prod_{i=1}^r f_i^{d_{ij}}$ for some $\{d_{ij}\}_{1\leq i\leq r,1\leq j\leq r'} \subset \mb{Z}_{\geq 0}$ and the zero loci of $f_1 \cdots f_r$ and $f_1' \cdots f_{r'}'$ coincide. Then for $\bs{\alpha}' \in \mb{R}_{\geq 0}^{r'}$ and $\bs{\alpha} = (\sum_j d_{1j} \alpha_j', \ldots, \sum_j d_{rj}\alpha_{j}')$, one has an isomorphism of $\sD_X[\bs{s}']$-modules:
\begin{equation}\label{eqn: Valpha' is the tensor of Valpha with Cs'}V^{\bs{\alpha}}\iota_+\mc{M} \otimes_{\mb{C}[\bs{s}]} \mb{C}[\bs{s}']\xrightarrow{\sim} V^{\bs{\alpha}'}\iota'_+\mc{M},\end{equation}
where $\iota' \colon X \to X \times \mb{C}^{r'}$ is the graph embedding of the $f_j'$  and the tensor product are taken over $s_i\mapsto \sum_jd_{ij}s_j'$. 
\end{corollary}

\begin{proof}
By Proposition \ref{prop:change of functions}, it suffices to prove that the following map is injective:
\[ V^{\bs{\alpha}}\iota_+\mc{M} \otimes_{\mb{C}[\bs{s}]} \mb{C}[\bs{s}']\to \iota'_+\mc{M}.\] 

We break the proof into two cases. First we assume for each $i$, there exists a $j$ such that $d_{ij}>0$. Now suppose $\bs{\beta}\leq \bs{\alpha}$ is separated by a single wall $\bs{L}^{-1}(\gamma)$. Choose $i_0$ such that $L_{i_0}> 0$. Pick $j_0$ such that $d_{i_0j_0}>0$. Since $L_i\geq 0$ for every $i$, we have
\[ \sum_{i=1}^r L_id_{ij_0}\geq L_{i_0}d_{i_0j_0}>0.\]
So $\sum_i L_id_{ij_0}\neq 0$ and we can apply Corollary \ref{cor: subtle injectivity} to obtain an injection
\[ V_{\ast}^{\bs{\alpha}}\iota_+\mc{M} \otimes_{\mb{C}[\bs{s}]} \mb{C}[\bs{s}']\hookrightarrow V_{\ast}^{\bs{\beta}}\iota_+\mc{M} \otimes_{\mb{C}[\bs{s}]} \mb{C}[\bs{s}'].\] 
As $\cM=\cM[(f_1\cdots f_r)^{-1}]$, taking the colimit over all $\bs{\beta}\leq \bs{\alpha}$, we have an injection
\[ V_{\ast}^{\bs{\alpha}}\iota_+\mc{M} \otimes_{\mb{C}[\bs{s}]} \mb{C}[\bs{s}']\hookrightarrow \varinjlim_{\bs{\beta} \leq \bs{\alpha}} V_*^{\bs{\beta}}\iota_+\mc{M} \otimes_{\mb{C}[\bs{s}]}\mb{C}[\bs{s}'] = \iota_+\mc{M} \otimes_{\mb{C}[\bs{s}]} \mb{C}[\bs{s}'] = \iota'_{+}\cM.\]
Since $\bs{\alpha}\in \R^r_{\geq 0}$, $V^{\bs{\alpha}}\iota_+\mc{M}=V_{\ast}^{\bs{\alpha}}\iota_+\mc{M}$ by Corollary \ref{cor:extending V*}, so this proves the desired injectivity in this case.

In the general case, let us consider the decomposition $f_i=\prod_{m=1}^{r''} g_m^{e_{mi}}$ of $f_i$ into irreducible factors, where each $g_m$ is an irreducible factor of some $f_i$ and $e_{mi}\in \Z_{\geq 0}$. By construction, for each $m$, there exists $i$ so that $e_{mi}>0$. Let $\bs{\alpha}''=(\sum_i e_{mi}\alpha_i)\in \R^{r''}_{\geq 0}$ and let $\iota''$ be the graph embedding of $g_1,\ldots,g_{r''}$. By the previous case, we have an isomorphism of $\sD_X[\bs{s}]$-modules:
\[ V^{\bs{\alpha}''}\iota''_{+}\cM\otimes_{\C[\bs{s}'']}\C[\bs{s}]\xrightarrow{\sim} V^{\bs{\alpha}}\iota_{+}\cM,\]
where the tensor product is taken over $s''_m \mapsto \sum_i e_{mi}s_i$. Using $f_j'=\prod_i f_i^{d_{ij}}$, we can write $f_j'=\prod_m g_m^{c_{mj}}$ with $c_{mj}\colonequals \sum_ie_{mi}d_{ij}$. Since $f_1\cdots f_r$ and $f_1'\cdots f_{r'}'$ have the same zero loci, for each $m$ there must exist a $j$ such that $c_{mj}\neq 0$. Hence we also have
 \[ V^{\bs{\alpha}''}\iota''_{+}\cM\otimes_{\C[\bs{s}'']}\C[\bs{s}']\xrightarrow{\sim} V^{\bs{\alpha}'}\iota'_{+}\cM,\]
 where the tensor product are taken over $s''_m\mapsto \sum_j c_{mj}s_j'=\sum_ie_{mi}(\sum_j d_{ij}s_j')$. So
 \begin{align*}
 V^{\bs{\alpha}}\iota_{+}\cM\otimes_{\C[\bs{s}]}\C[\bs{s}']&\cong ( V^{\bs{\alpha}''}\iota''_{+}\cM\otimes_{\C[\bs{s}'']}\C[\bs{s}])\otimes_{\C[\bs{s}]}\C[\bs{s}']\\
 &\cong V^{\bs{\alpha}''}\iota''_{+}\cM\otimes_{\C[\bs{s}'']}\C[\bs{s}']\\
 &\cong V^{\bs{\alpha}'}\iota'_{+}\cM. \qedhere
 \end{align*}
\end{proof}

We now turn to the proof of Theorem \ref{thm:flatness}. We will use the following general criterion for flatness.

\begin{lem} \label{lem:flatness by fibres}
Let $R$ be a Noetherian commutative ring of finite Krull dimension. Suppose that $M \in \mrm{D}^-(R)$ is a complex of $R$-modules such that $\mrm{H}^i(M) = 0$ for $i > 0$ and
\begin{equation}\label{eqn: vanishing condition} \mrm{H}^i(\operatorname{Frac}(R/\mf{p}) \overset{\mrm{L}}\otimes_R M) = 0 \quad \text{for every prime ideal $\mf{p} \subset R$ and $i \neq 0$}. \end{equation}
Then $\mrm{H}^i(M) = 0$ for all $i \neq 0$ and $M = \mrm{H}^0(M)$ is a flat $R$-module.
\end{lem}
\begin{proof}
We first show that $\mrm{H}^i(M) = 0$ for all $i \neq 0$ by induction on $\dim R$. Note that, since any non-zero $R$-module has a non-zero localisation at some prime ideal, we may assume without loss of generality that $R$ is local with maximal ideal $\mf{m}$. If $R$ is a field then the conclusion is true by the assumption \eqref{eqn: vanishing condition}. If $R$ is not a field, by Noetherian induction we may assume that the conclusion holds for $R/I$ for every non-zero ideal $I \subset R$. Choose a non-zero $f \in \mf{m}$. Letting $I_f \subset R$ be the annihilator of $f$, we have a long exact sequence
\[ \cdots\to  \mrm{H}^{i - 1}(R/(f) \overset{\mrm{L}}\otimes_R M) \to \mrm{H}^i(R/I_f \overset{\mrm{L}}\otimes_R M) \to \mrm{H}^i(M) \to \mrm{H}^i(R/(f) \overset{\mrm{L}}\otimes_R M) \to \cdots\]
Since the claim holds for $R/(f)$ by Noetherian induction, we have $\mrm{H}^i(R/(f) \overset{\mrm{L}}\otimes_R M) = 0$ for $i < 0$ and hence
\begin{equation} \label{eq:flatness by fibres 1}
\mrm{H}^i(R/I_f \overset{\mrm{L}}\otimes_R M) \overset{\sim}\to \mrm{H}^i(M), \quad \textrm{for $i<0$}.
\end{equation}
If $I_f \neq 0$, then by Noetherian induction we have $\mrm{H}^i(R/I_f \overset{\mrm{L}}\otimes_R M) = 0$ and hence $\mrm{H}^i(M) = 0$ for $i < 0$. If $I_f = 0$, then \eqref{eq:flatness by fibres 1} can be identified with the morphism
\[ \mrm{H}^i(M) \xrightarrow{f} \mrm{H}^i(M).\]
Since this is an isomorphism, we therefore have $\mrm{H}^i(M) = \mrm{H}^i(M)[f^{-1}] = \mrm{H}^i(M[f^{-1}])$. But $M[f^{-1}]$ is an $R[f^{-1}]$-module satisfying the assumptions of the lemma and, since $R$ is local and $f \in \mf{m}$, we have $0 \leq \dim R[f^{-1}] < \dim R$. So by induction on dimension, we have $\mrm{H}^i(M[f^{-1}]) = 0$ and hence $\mrm{H}^i(M) = 0$ for $i < 0$ as claimed. 

Finally, to prove flatness, we have shown that for every ideal $I\subset R$ one has 
\[ \mrm{Tor}_i^R(R/I, M) = \mrm{H}^{-i}(R/I \overset{\mrm{L}}\otimes_R M) = 0 \quad \text{for $i > 0$},\]
and hence $M = \mrm{H}^0(M)$ is flat as claimed. 
\end{proof}

\begin{proof}[Proof of Theorem \ref{thm:flatness}]
We first prove that $V_{\ast}^{\bs{\alpha}}\iota_+\mc{M}$ is flat over $\mb{C}[\bs{s}]$. By Lemma \ref{lem:flatness by fibres}, it is enough to show that, for every prime ideal $\mf{p} \subset \mb{C}[\bs{s}]$, we have
\begin{equation}\label{eqn: vanishing for tensor with C(p)} \mrm{H}^i(\mb{C}(\mf{p}) \overset{\mrm{L}}\otimes_{\mb{C}[\bs{s}]} V^{\bs{\alpha}}_{\ast}\iota_+\mc{M}) =\mrm{H}^i(\mb{C}(\mf{p}) \overset{\mrm{L}}\otimes_{\mb{C}[\bs{s}]_{\mf{p}}}(V^{\bs{\alpha}}_{\ast}\iota_+\mc{M})_\mf{p}) = 0 \quad \text{for $i \neq 0$},\end{equation}
where $\mb{C}(\mf{p}) := \mrm{Frac}(\mb{C}[\bs{s}]/\mf{p})$. We note that the cohomology vanishes for $i > 0$ by construction, so it suffices to show that it vanishes for $i < 0$.

Let us regard $\mb{C}(\mf{p}) \overset{\mrm{L}}\otimes_{\mb{C}[\bs{s}]_{\mf{p}}}(V_{\ast}^{\bs{\alpha}}\iota_+\mc{M})_\mf{p}$ as a complex of coherent $\ms{D}_X[\bs{s}]_{\mf{p}} = \ms{D}_X \otimes \mb{C}[\bs{s}]_{\mf{p}}$-modules. Since the relative holonomicity is preserved under localization in $\C[\bs{s}]$ \cite[Remark 3.2.1]{ZerolociBernsteinSatoideal} and taking cohomology,  by Proposition \ref{prop:V is relative holonomic} each cohomology sheaf of this complex is also  relative holonomic over $\C[\bs{s}]_{\mf{p}}$ and its singular support must be contained in $L\times \spec \mb{C}(\mf{p}) \subset L \times \spec \mb{C}[\bs{s}]_{\mf{p}}$ for some Lagrangian $L\subset T^{\ast}X$. Hence, each cohomology sheaf is holonomic as a $\ms{D}_X[\bs{s}]_{\mf{p}}$-module in the sense of Definition \ref{defn:localised holonomic}. By Proposition \ref{prop:holonomic duality}, the duality functor commutes (up to a shift) with taking cohomology sheaves when restricted to the category of complexes with holonomic cohomology. So for all $i$ we have 
\begin{equation} \label{eq:flatness 1}
\mb{D}_{\mathit{hol}}\mrm{H}^{i}(\mb{C}(\mf{p}) \overset{\mrm{L}}\otimes_{\mb{C}[\bs{s}]_{\mf{p}}}(V_{\ast}^{\bs{\alpha}}\iota_+\mc{M})_\mf{p}) = \mrm{H}^{-i + \codim \mf{p}}\left(\mb{D}_{\ms{D}_X[\bs{s}]_{\mf{p}}}(\mb{C}(\mf{p}) \overset{\mrm{L}}\otimes_{\mb{C}[\bs{s}]_{\mf{p}}}(V_{\ast}^{\bs{\alpha}}\iota_+\mc{M})_\mf{p})\right).
\end{equation}
Using the Koszul resolution of $\mb{C}(\mf{p})$ as a $\mb{C}[\bs{s}]_{\mf{p}}$-module and Theorem \ref{thm:duality}, one has
\begin{align*}
\mb{D}_{\ms{D}_X[\bs{s}]_{\mf{p}}}(\mb{C}(\mf{p}) \overset{\mrm{L}}\otimes_{\mb{C}[\bs{s}]_{\mf{p}}}(V_{\ast}^{\bs{\alpha}}\iota_+\mc{M})_\mf{p}) 
&= \mb{C}(\mf{p}) \overset{\mrm{L}}\otimes_{\mb{C}[\bs{s}]_{\mf{p}}}(\mb{D}_{\ms{D}_X[\bs{s}]} V_{\ast}^{\bs{\alpha}}\iota_+\mc{M})_{\mf{p}}[-\codim \mf{p}] \\
&= \mb{C}(\mf{p}) \overset{\mrm{L}}\otimes_{\mb{C}[\bs{s}]_{\mf{p}}}\left((V^{ - \bs{\alpha} + \epsilon \bs{1}}_*\iota_+\mb{D}\mc{M})_{\bs{s} \mapsto -\bs{s}}\right)_{\mf{p}}[-\codim \mf{p}].
\end{align*}
Since this has vanishing $\mrm{H}^j$ for $j > \codim \mf{p}$, and since the functor $\mb{D}_{\mathit{hol}}$ is faithful by Proposition \ref{prop:basic duality}, we conclude by \eqref{eq:flatness 1} that \eqref{eqn: vanishing for tensor with C(p)}  holds.

Similarly, to prove flatness of $V_{\ast}^{\bs{\alpha}}\iota_+\mc{M}/V_{\ast}^{\bs{\beta}}\iota_+\mc{M}$ over $\mb{C}[U]$, enlarging $U$ if necessary, we may assume that $U + \mb{C}\bs{L} = \operatorname{span}\{s_1, \ldots, s_r\}$. Since $V_{\ast}^{\bs{\alpha}}\iota_+\mc{M}/V_{\ast}^{\bs{\beta}}\iota_+\mc{M}$ is annihilated by a product of operators of the form $\bs{L}(\bs{s}) + \gamma'$, it follows from Proposition \ref{prop:V is relative holonomic} that for any prime ideal $\mf{p}' \subset \mb{C}[U]$, each cohomology group of 
\[ \mb{C}(\mf{p}') \overset{\mrm{L}}\otimes_{\mb{C}[U]_{\mf{p}'}}\left(\frac{V_{\ast}^{\bs{\alpha}}\iota_+\mc{M}}{V_{\ast}^{\bs{\beta}}\iota_+\mc{M}}\right)_{\mf{p}'}\]
is holonomic over $\ms{D}_X[\bs{s}]_{\mf{p}'} = \ms{D}_X \otimes \mb{C}[\bs{s}]_{\mf{p}'} = \ms{D}_X[\bs{s}] \otimes_{\mb{C}[U]} \mb{C}[U]_{\mf{p}'}$, where $\mb{C}(\mf{p}') = \mrm{Frac}(\mb{C}[U]/\mf{p}')$. (Note that the ring $\mb{C}[\bs{s}]_{\mf{p}'}$ is equidimensional, so this makes sense.) So arguing as above using Proposition \ref{prop:holonomic duality} and Theorem \ref{thm:duality}, we deduce that
\[ \mrm{H}^i\left(\mb{C}(\mf{p}') \overset{\mrm{L}}\otimes_{\mb{C}[U]}\frac{V_{\ast}^{\bs{\alpha}}\iota_+\mc{M}}{V_{\ast}^{\bs{\beta}}\iota_+\mc{M}}\right) = \mrm{H}^i\left(\mb{C}(\mf{p}') \overset{\mrm{L}}\otimes_{\mb{C}[U]_{\mf{p}'}}\left(\frac{V_{\ast}^{\bs{\alpha}}\iota_+\mc{M}}{V_{\ast}^{\bs{\beta}}\iota_+\mc{M}}\right)_{\mf{p}'}\right) = 0\]
for $i \neq 0$. So $V_{\ast}^{\bs{\alpha}}\iota_+\mc{M}/V_{\ast}^{\bs{\beta}}\iota_+\mc{M}$ is flat over $\mb{C}[U]$ by Lemma \ref{lem:flatness by fibres} as claimed.
\end{proof}

\section{The Hodge filtration and the $V$-filtration}\label{sec: Hodge and multivariate V}
In this section, we specialise our study of the multivariate $V$-filtration further from general holonomic $\ms{D}$-modules to the more structured setting of mixed Hodge modules. We recall some of the general theory of mixed Hodge modules in \S\ref{sec: recollection}; the main new feature for our purposes is the addition of a very special good filtration, the \emph{Hodge filtration}. In \S\ref{subsec:hodge localisation}, we show (Proposition \ref{prop:hodge !* formula})  that the $V$-filtration formula for the $\ast$-localisation along a simple normal crossings divisor holds also with the Hodge filtration in place. Using this, we show in \S\ref{subsec:magic formula} that, conversely, the multivariate $V_*$-filtration on the graph embedding of a mixed Hodge module can be recovered from the Hodge filtrations on a sequence of auxiliary mixed Hodge modules; this is a multivariate version of our construction in \cite{DY25} for the single-variable $V$-filtration. As applications, we deduce in \S\ref{subsec:properties of Hodge and V} that the Hodge filtration and the multivariate $V$-filtration interact rather well, at least if we restrict to the positive orthant: for example, we prove Hodge-filtered versions of the flatness theorem (Theorem \ref{thm:hodge flatness}), the monomial change of functions (Theorem \ref{thm:filtered change of functions}) and strictness of the direct images under proper maps (Theorem \ref{thm:multistrict pushforward}).

\subsection{Recollection on mixed Hodge modules}\label{sec: recollection}
In this subsection, we recall some aspects of the theory of mixed Hodge modules.

Let us first fix some conventions, as the phrase ``mixed Hodge module'' can mean a few different things depending on context. If $X$ is a complex manifold, we will write $\mhm_{\mb{Q}}(X)$ for the $\mb{Q}$-linear abelian category of polarisable mixed Hodge modules on $X$ defined by Saito in \cite[\S 2.17]{Saito90}. We will refer to the objects simply as $\mb{Q}$-mixed Hodge modules. For example, if $X$ is a point, then $\mhm_{\mb{Q}}(X)$ is the category of polarisable $\mb{Q}$-mixed Hodge structures \cite[Theorem 3.9]{Saito90}. For general $X$, an object in $\mhm_{\Q}(X)$ is specified by a tuple
 \[ (\cM,F_{\bullet}\cM,W_{\bullet}\cM,K),\]
where $\cM$ is a regular holonomic $\sD_X$-module, $F_{\bullet}\cM$ is a good filtration over the order filtration $F_{\bullet}\sD_X$, $W_{\bullet}\cM$ is a finite filtration by $\sD_X$-modules and $K$ is a $\Q$-perverse sheaf such that $K\otimes_{\Q}\C\cong \textrm{DR}_X(\cM)$. These data are required to satisfy a complicated set of inductive conditions amounting roughly to the requirement that any specialisation to a point is a $\mb{Q}$-mixed Hodge structure.

\begin{rmk}
We work in this paper primarily with \emph{analytic} mixed Hodge modules on complex manifolds. There is also a more restrictive notion of \emph{algebraic} mixed Hodge modules on a complex algebraic variety $X$ \cite[\S4]{Saito90}. In the polarisable setting considered here, one can define the category of algebraic mixed Hodge modules on $X$ to be the full subcategory of analytic mixed Hodge modules that can be extended to some (hence any) proper algebraic compactification of $X$. In particular, all our results proved in the analytic setting hold also in the algebraic one.
\end{rmk}

In this paper, it will be convenient also to work with larger coefficient fields than $\mb{Q}$. Here are two ways to achieve this. First, if $\mb{K} \subset \mb{C}$ is an algebraic extension of $\mb{Q}$, then we can define a $\mb{K}$-linear category $\mhm_\mb{K}(X)$ by formally extending scalars in $\mhm_{\mb{Q}}(X)$ from $\mb{Q}$ to $\mb{K}$; see \cite[\S 1.4]{DY25}. In this setting, most theorems extend formally from the case of $\mb{Q}$-mixed Hodge modules to $\mb{K}$-mixed Hodge modules. If $\mb{K} = \mb{C}$, on the other hand, then one also has the $\mb{C}$-linear category $\mhm_\mb{C}(X)$ of \emph{complex mixed Hodge modules}, which is the topic of the (unfinished) book project \cite{MHMproject} (see also \cite[Appendix A.1]{DV22} and \cite[\S 2.1]{DV23} for some general discussion). This theory, which sits somewhere between Saito's theory and Mochizuki's theory of mixed twistor $\ms{D}$-modules \cite{MochizukiMTM}, is developed in parallel to the $\mb{Q}$-linear theory, but with $\mb{C}$-Hodge structures as the prototype instead of $\mb{Q}$-Hodge structures.

For the results we will consider here, the complex theory $\mhm_\mb{C}(X)$ is the most natural context. We will therefore write simply
\[ \mhm(X) := \mhm_\mb{C}(X)\]
and refer to its objects simply as mixed Hodge modules. Since there is no complete reference at the time of writing, however, we will have to take certain expected properties of this theory on faith; we will spell these out when they occur, and will indicate when needed how to modify the statements and proofs of our results in the better-founded theory over algebraic extensions of $\mb{Q}$.

Let us first record how the theories with different coefficients are related:

\begin{prop} \label{prop:field extension}
Let $\mb{K} = \mb{C}$ or an algebraic extension of $\mb{Q}$. Then there is an exact base-extension functor
\[ \mb{K} \otimes_{\mb{Q}} - \colon \mhm_\mb{Q}(X) \to \mhm_{\mb{K}}(X).\]
\end{prop}

This is true by definition when $\mb{K}$ is algebraic over $\mb{Q}$. When $\mb{K} = \mb{C}$ we take it to be an expected property of the theory of complex mixed Hodge modules.

Our main interest in mixed Hodge modules will be their underlying $\ms{D}$-modules equipped with Hodge and weight filtrations. For future reference, we record the main properties of this structure. In the statement below, we write $\mrm{MFW}(\ms{D}_X)_{rh}$ for the category of tuples $(\mc{M}, F_\bullet\mc{M}, W_\bullet\mc{M})$, where $\mc{M}$ is a regular holonomic $\ms{D}_X$-module, $F_\bullet\mc{M}$ is a good filtration and $W_\bullet\mc{M}$ is a filtration by $\ms{D}_X$-submodules.

\begin{prop} \label{prop:hodge forgetful functor}
Let $\mb{K}$ be $\mb{C}$ or an algebraic extension of $\mb{Q}$. Then for any complex manifold $X$, there is a functor
\begin{align*}
\mhm_\mb{K}(X) &\to \mrm{MFW}(\ms{D}_X)_{rh}\\
\mc{M} &\mapsto (\mc{M}, F_\bullet\mc{M}, W_\bullet\mc{M})
\end{align*}
such that the composition with the forgetful functor to $\coh(\ms{D}_X)$ is faithful and exact. Moreover, if $\mc{M} \to \mc{N}$ is a morphism in $\mhm_\mb{K}(X)$, then the morphism $(\mc{M}, F_\bullet \mc{M}, W_\bullet\mc{M}) \to (\mc{N}, F_\bullet\mc{N}, W_\bullet\mc{N})$ is bi-strict. The forgetful functors to $\mrm{MFW})(\ms{D}_X)_{rh}$ commute with the coefficient extensions of Proposition \ref{prop:field extension}.
\end{prop}

This is a standard fact in the case $\mb{K} = \mb{Q}$ (cf.\ \cite[Proposition 5.1.14]{Saito88}), while the extension to $\mb{K}$ algebraic over $\mb{Q}$ is explained briefly in \cite[\S 1.4]{DY25}. We take it as a desideratum of the complex theory (which should follow by essentially the same arguments as the rational case). The category $\mhm_\mb{K}(X)$ also comes equipped with an auto-equivalence $\mc{M} \mapsto \mc{M}(1)$, called \emph{Tate twist}, acting on the underlying filtered $\ms{D}$-modules by
\[ (\mc{M}(1), W_\bullet\mc{M}(1), F_\bullet\mc{M}(1)) = (\mc{M}, W_{\bullet + 2}\mc{M}, F_{\bullet - 1}\mc{M}).\]
This is true by construction in all settings.

One of the main points of constructing the category $\mhm_\mb{K}(X)$ is that it has some particularly good functoriality properties. We will recall a few of these below.

First, if $f \colon X \to Y$ is a proper morphism of complex manifolds, then we have Laumon's filtered direct image functor
\[ f_+ \colon \mrm{D}^b_{\mrm{coh}}(\ms{D}_X, F_\bullet) \to \mrm{D}^b_{\mrm{coh}}(\ms{D}_Y, F_\bullet) \]
given by
\[ f_+(\mc{M}, F_\bullet) = \mrm{R}f_*((\ms{D}_{Y \leftarrow X}, F_\bullet) \overset{\mrm{L}}\otimes_{(\ms{D}_X, F_\bullet)} (\mc{M}, F_\bullet)),\]
where we endow
\[ \ms{D}_{Y \leftarrow X} = \mrm{Diff}(f^{-1}\omega_Y, \omega_X) \cong f^{-1}\ms{D}_Y \otimes_{f^{-1}\mc{O}_Y} \omega_{X/Y} \]
with the order filtration shifted to begin in degree $\dim Y - \dim X$, and $\overset{\mrm{L}}\otimes_{(\ms{D}_X, F_\bullet)}$ means the derived tensor product of Rees modules over $R_F\ms{D}_X$. One of the fundamental properties of the theory of mixed Hodge modules is:

\begin{thm}
Assume that $f \colon X \to Y$ is a projective morphism. Then there exist functors
\[ \mc{H}^i f_+ \colon \mhm_\mb{K}(X) \to \mhm_\mb{K}(Y) \]
for $i \in \mb{Z}$, such that for any $\mc{M} \in \mhm_\mb{K}(X)$, the filtered complex $f_+(\mc{M}, F_\bullet)$ is strict and we have natural isomorphisms
\[ \mc{H}^if_+(\mc{M}, F_\bullet) \cong (\mc{H}^if_+\mc{M}, F_\bullet) \]
for all $i$.
\end{thm}

The case $\mb{K} = \mb{Q}$ is \cite[Theorem 2.14 and 2.18]{Saito90}, while the case where $\mb{K}$ is algebraic over $\mb{Q}$ follows by extension of scalars. The case $\mb{K} = \mb{C}$ follows by the same argument by reduction to the pure case, which is \cite[Theorem 14.3.1]{MHMproject}. Note that in the setting of algebraic mixed Hodge modules, the cohomological functors $\mc{H}^if_+$ can be upgraded to a functor between derived categories \cite[Theorem 4.3]{Saito90}, but this is not known in the analytic setting.

As an elementary special case, when $f \colon X \to Y$ is a closed immersion, we have $\mc{H}^if_+ = 0$ for $i \neq 0$, so we get an honest exact functor
\[ f_+ = \mc{H}^0f_+\colon \mhm_\mb{K}(X) \to \mhm_{\mb{K}}(Y).\]

\begin{prop}[Kashiwara's equivalence for mixed Hodge modules]
For a closed immersion $f$ as above, the functor $f_+$ is fully faithful, with image equal to the full subcategory of mixed Hodge modules on $Y$ whose underlying $\ms{D}_Y$-modules are supported in $X$.
\end{prop}

The full faithfulness of $f_+$ is a relatively easy consequence of the definitions of the functor using the analogous statement for the underlying $\ms{D}$-modules. The statement about the essential image is baked into the inductive definition of mixed Hodge module.

As alluded to above, the category of mixed Hodge modules is defined so that it also behaves well under specialisation. More precisely, it has the following good property with respect to the $V$-filtration.

Suppose that $(\mc{M}, F_\bullet\mc{M})$ is a coherent $\ms{D}_X$-module with good filtration. Let $f \colon X \to \mb{C}$ be a holomorphic function and let $\iota \colon X\to X\times \C$ be the graph embedding. The direct image $\iota_{+}\cM$ and its $V$-filtration $V^{\alpha}\iota_{+}\cM$ inherit Hodge filtrations by
\[ F_{p+1}\iota_{+}\cM = \sum_{r+s\leq p}F_r\cM\otimes \d_t^{s}, \quad F_{\bullet}V^{\alpha}\iota_{+}\cM=F_{\bullet}\iota_{+}\cM\cap V^{\alpha}\iota_{+}\cM.\]
We also set $F_{\bullet}\gr^{\alpha}_V\iota_{+}\cM=F_{\bullet}V^{\alpha}\iota_{+}\cM/F_{\bullet}V^{>\alpha}\iota_{+}\cM$. 

\begin{definition}\label{defn: strict A-specializability}
Let $A \subset \mb{C}$ be a $\mb{Q}$-linear subspace containing $\mb{Q}$. We say that $(\mc{M}, F_\bullet)$ is \emph{strictly $A$-specialisable} along $f$ (cf.\ \cite[Definition 10.6.1]{MHMproject}) if
\begin{enumerate}
\item the $V$-filtration $V^{\bullet}\iota_{+}\cM$ is defined over $A$, (in the sense of e.g.\ Definition \ref{defn:multivariate V-filtration}),
\item \label{itm: t strict} $t(F_kV^{\alpha}\iota_{+}\cM)=F_kV^{\alpha+1}\iota_{+}\cM$ for $\alpha>0,k\in \Z$,
\item \label{itm: dt strict} $\partial_t(F_k\gr^{\alpha}_V\iota_{+}\cM)=F_{k +1}\gr^{\alpha-1}_V\iota_{+}\cM$ for $\alpha<1,k\in \Z$.
\end{enumerate}
\end{definition}

\begin{prop} \label{prop:mhm strict specialisation}
Let $\mb{K} = \mb{C}$ (resp.\ an algebraic extension of $\mb{Q}$ containing all roots of unity), let $A = \mb{R}$ (resp.\ $\mb{Q}$) and let $\mc{M} \in \mhm_\mb{K}(X)$. Then for any holomorphic function $f$, $(\mc{M}, F_\bullet\mc{M})$ is strictly $A$-specialisable along $f$,
\[ (\gr_V^\alpha\iota_+\mc{M}, F_\bullet\gr_V^\alpha\iota_+\mc{M}) \]
underlies an object $\gr_V^\alpha\iota_+\mc{M}$ in $\mhm_\mb{K}(X)$ for all $\alpha \in A$, and the maps
\[ t \colon \gr_V^0 \iota_+\mc{M} \to \gr_V^1 \iota_+\mc{M} \quad \text{and} \quad \partial_t \colon \gr_V^0\iota_+\mc{M}(1) \to \gr_V^1\iota_+\mc{M} \]
underlie morphisms in $\mhm_\mb{K}(X)$.
\end{prop}

Proposition \ref{prop:mhm strict specialisation} will be true by definition when $\mb{K} = \mb{C}$ (cf.\ \cite[Definition 14.2.2]{MHMproject} for the pure case). This is almost the case for $\mb{Q}$-mixed Hodge modules as well, except that one only has $\mb{Q}$-mixed Hodge modules $\psi_f\mc{M}$ and $\phi_f\mc{M}$ underlying
\[ \left(\bigoplus_{\alpha \in \mb{Q} \cap (0, 1]} \gr_V^\alpha\iota_+\mc{M}, F_\bullet\right) \quad \text{and} \quad \left(\bigoplus_{\alpha \in \mb{Q} \cap [0, 1)} \gr_V^\alpha\iota_+\mc{M}, F_\bullet\right)\]
respectively, both equipped with a semisimple monodromy operator acting by $\exp(-2\pi i \alpha)$ on $\gr_V^\alpha \iota_+\mc{M}$. After tensoring with a field $\mb{K}$ containing the roots of unity, one can then extract the eigenspaces to get the desired $\mb{K}$-mixed Hodge module structures on the $\gr_V^\alpha\iota_+\mc{M}$.

Since the underlying $\ms{D}$-modules are unchanged under field extensions, we deduce:

\begin{cor}\label{cor: MHM fully A specializable}
For $\mb{K} = \mb{C}$ (resp.\ any algebraic extension of $\mb{Q}$), the underlying $\ms{D}_X$-module $\mc{M}$ is fully $\mb{R}$ (resp.\ $\mb{Q}$)-specialisable.
\end{cor}

One of the main purposes of Proposition \ref{prop:mhm strict specialisation} is to pin down how Hodge filtrations are allowed to extend along open immersions. Suppose that $D \subset X$ is a divisor and write $j \colon U = X \setminus D \to X$ for the inclusion of the complement. Then we have a restriction functor
\[ j^* \colon \mhm_\mb{K}(X) \to \mhm_\mb{K}(U).\]
One defines the full subcategory $\mhm_\mb{K}(U)_X \subset \mhm_\mb{K}(U)$ of \emph{mixed Hodge modules extendable to $X$} to be the essential image of $j^*$. 

\begin{prop}
In the setting above, the functor
\[ j^* \colon \mhm_\mb{K}(X) \to \mhm_\mb{K}(U)_X \]
has fully faithful left and right adjoints
\[ j_!, j_* \colon \mhm_\mb{K}(U)_X \to \mhm_\mb{K}(X).\]
\end{prop}

This is again part of the definition. Note that one also has left and right adjoints as above for \emph{regular} holonomic $\ms{D}$-modules, constructed ultimately using Deligne's theory of canonical extensions for connections with regular singularities, and the functors for mixed Hodge modules are compatible with these by definition.

For our purposes, we will need a somewhat explicit description of the behaviour of $j_!$ and $j_*$ on the underlying Hodge-filtered $\ms{D}$-modules. In order to give the formula, observe that there is a functor
\begin{align*}
\mhm(U)_X &\to \mrm{MFW}(\ms{D}_X(*D)) \\
\mc{M} &\mapsto (\mc{M}(*D), F_\bullet^*\mc{M}(*D), W_\bullet^*\mc{M}(*D)),
\end{align*}
where $\mrm{MFW}(\ms{D}_X(*D))$ is the category of $\ms{D}_X(*D)$-modules with good filtration $F_\bullet^*$ relative to $F_\bullet\ms{D}_X(*D) \colonequals (F_\bullet\ms{D}_X)(*D)$ and filtration $W_\bullet^*$ by $\ms{D}_X(*D)$-submodules, given by the formula
\[ (\mc{M}(*D), F_\bullet^*\mc{M}(*D), W_\bullet^*\mc{M}(*D)) \colonequals (\mc{M}'(*D), (F_\bullet\mc{M}')(*D), (W_\bullet\mc{M}')(*D))\]
for any $\mc{M}' \in \mhm(X)$ such that $j^*\mc{M}' = \mc{M}$. This is well-defined since, for any two such extensions $\mc{M}'$ and $\mc{M}''$, the kernel and cokernel of the tautological maps $\mc{M}' \to j_*\mc{M}$ and $\mc{M}'' \to j_*\mc{M}$ will be supported in $D$ and hence they and all subsheaves become zero after localising at $D$. Note also that $\mc{M}(*D) = j_*\mc{M}$ as $\ms{D}_X$-modules, but that the Hodge and weight filtrations are different.

Now, let us suppose for simplicity that the divisor $D$ is smooth. Then for $\mc{M} \in \mhm(U)_X$, we have a $V$-filtration $V^\bullet \mc{M}(*D)$ of $\mc{M}(*D)$ along $D$. The following formulas now follow from Proposition \ref{prop:mhm strict specialisation} (see, e.g.\ \cite[Propositions 11.3.3 and 11.4.2]{MHMproject}).

\begin{prop} \label{prop:hodge extension}
We have
\begin{equation}\label{eqn: filtration on *extension} \left(j_*\mc{M},F_{\bullet}\right)=(\sD_X,F_{\bullet})\otimes_{(V^0\sD_X,F_{\bullet})}(V^0\mc{M}(\ast D),F^*_{\bullet}),\end{equation}
and
\begin{equation}\left(j_!\mc{M},F_{\bullet}\right)=(\sD_X,F_{\bullet})\otimes_{(V^0\sD_X,F_{\bullet})}(V^{>0}\mc{M}(\ast D),F^*_{\bullet}),\end{equation}
where $F_\bullet^* V^\alpha\mc{M}(*D) \colonequals V^\alpha \mc{M}(*D) \cap F_\bullet^*\mc{M}(*D)$.
\end{prop}

\begin{rmk} \label{rmk:hodge lower bound}
If $F_p\mc{M} = 0$, then for any extension $\mc{M}'$ of $\mc{M}$ to $X$, $F_p\mc{M}'$ is supported in $D$, so $F_p^*\mc{M}(*D) = (F_p\mc{M}')(*D) = 0$. Since $F_\bullet \ms{D}_X$ is zero in negative degrees, Proposition \ref{prop:hodge extension} implies that $F_pj_!\mc{M} = F_pj_*\mc{M} = 0$ also.
\end{rmk}

We have so far discussed a collection of nice properties satisfied by mixed Hodge modules and operations to construct new mixed Hodge modules out of old ones. In order for this to be a useful theory, however, one needs a source of examples as a starting point.

\begin{thm}[{\cite[Theorem 3.27]{Saito90}}] \label{thm:vmhs is mhm}
Let $X$ be a complex manifold and let $U \subset X$ be a Zariski open dense subset. If $\mb{V} = (\mb{V}_\mb{Q}, \mc{V}, F^\bullet \mc{V}, W_\bullet\mc{V})$ is an admissible graded-polarisable variation of $\mb{Q}$-mixed Hodge structure on $U$ relative to $X$, then there is an associated object in $\mhm_\mb{Q}(U)_X$ with underlying bi-filtered $\ms{D}_X$-module
\[ (\mc{V}, F^{-\bullet}\mc{V}, W_{\bullet - \dim X}\mc{V}).\]
\end{thm}

One expects an analogous theorem for complex mixed Hodge modules and variations of complex mixed Hodge structure, but we will not need such a strong statement.

For simplicity, let us recall the definition of admissible variations of $\mb{Q}$-mixed Hodge structure only in the special case where $X$ is a curve and the underlying local system has unipotent local monodromy.

\begin{definition} \label{defn:admissible vmhs}
For $X$ a curve, we say \cite[(3.4), (3.5), (3.13)]{SZ85} (cf.\ also \cite[(3.1.5),(3.1.6)]{Saito17}) that $\mb{V}$ as in Theorem \ref{thm:vmhs is mhm} is an \emph{admissible graded-polarisable unipotent variation of $\mb{Q}$-mixed Hodge structure} if
\begin{enumerate}
\item $\mb{V}_\mb{Q}$ is a $\mb{Q}$-local system on $U$ with unipotent monodromy around points in $X \setminus U$,
\item $\mc{V}$ is a flat holomorphic connection on $U$ with an identification $\mrm{DR}(\mc{V}) \cong \mb{V}_\mb{Q} \otimes_{\mb{Q}} \mb{C}[1]$,
\item $F^\bullet \mc{V}$ is a finite decreasing filtration by holomorphic sub-bundles satisfying Griffiths transversality,
\item $W_\bullet \mc{V}$ is a finite increasing filtration by flat sub-bundles, arising from a corresponding filtration on $\mb{V}_\mb{Q}$,
\item for each $w \in \mb{Z}$, $(\Gr^W_w\mb{V}_\mb{Q}, \Gr^W_w\mc{V}, F^\bullet\Gr^W_w\mc{V} \colonequals (F^\bullet \mc{V} \cap W_w\mc{V}/F^\bullet \mc{V} \cap W_{w - 1}\mc{V}))$ is a polarisable variation of Hodge structure of weight $w$,
\item $\gr_F^p\gr^W_w \tilde{\mc{V}}$ is a free $\mc{O}_X$-module for any $p,w\in \mathbb{Z}$, where $\tilde{\mc{V}}$ is the Deligne canonical extension of $\mc{V}$ and $W_\bullet \tilde{\mc{V}}$ and $F^\bullet \tilde{\mc{V}}$ are the intersections of $\tilde{\mc{V}}$ with the direct images of the weight and Hodge filtrations on $\mc{V}$, and
\item for each $x \in X \setminus U$, there exists a relative monodromy filtration on $(\tilde{\mc{V}}_x,W_{\bullet}\tilde{\mc{V}}_x)$ for the logarithm of the monodromy operator.
\end{enumerate}
\end{definition}

The following basic example will play a big role in what follows.

\begin{example} \label{eg:nilpotent orbit}
Let $n > 0$. Then the flat connection $\mc{V} = \mc{O}_{\mb{C}^\times}[s]t^s/(s^n)$ on $U = \mb{C}^\times$ underlies an admissible graded-polarisable unipotent variation of $\mb{Q}$-mixed Hodge structure relative to $X = \mb{C}$, with Hodge and weight filtrations
\[ F^p \left(\frac{\mc{O}_{\mb{C}^\times}[s]t^s}{(s^n)}\right) = \sum_{j \leq -p}\frac{\mc{O}_{\mb{C}^{\times}}s^jt^s}{(s^n)}  \quad \text{and} \quad W_w\left(\frac{\mc{O}_{\mb{C}^{\times}}[s]t^s}{(s^n)}\right) = \sum_{2k \geq -w}\frac{\mc{O}_{\mb{C}^{\times}}s^kt^s}{(s^n)}.\]
To see this, observe that the monodromy around $0$ is the unipotent operator $\exp(-2\pi i s)$, which indeed preserves the weight filtration. We get a flat $\mb{Q}$-structure by declaring the $\mb{Q}$-structure on the fibre at $1$ to be
\[ \mb{V}_{\mb{Q}, 1} = \sum_j (2\pi i)^j\mb{Q}s^j t^s.\]
Griffiths transversality of the Hodge filtration is immediate from the definition, and the associated graded $\Gr^W_w$ is the trivial variation of Hodge structure $\mb{Q}(j)$ if $w = -2j$ for $0 \leq j < n$ and $0$ otherwise, which is indeed polarisable. To check admissibility (the last two conditions in Definition \ref{defn:admissible vmhs}), observe that the Deligne canonical extension of $\mc{V}$ is $\tilde{\mc{V}} = \mc{O}_{\mb{C}}[s]t^s/(s^n)$, so
\[ \gr_F^p\gr^W_w \tilde{\mc{V}} = \begin{cases} \mc{O}_{\mb{C}}s^j t^s, & \text{if $w = -2j$, $p = -j$, $0 \leq j < n$}, \\ 0, &\text{otherwise,}\end{cases}\]
which is indeed a vector bundle. Finally, the logarithm of the monodromy operator $\mrm{N} = -s$ satisfies $\mrm{N}W_w\tilde{\mc{V}}_0 \subset W_{w - 2}\tilde{\mc{V}}_0$, so the relative monodromy weight filtration exists and is equal to $W_\bullet$ by, e.g.\ \cite[Proposition 2.14]{SZ85} (this is also immediate from the definition).
\end{example}

We will also need some examples of mixed Hodge modules with non-unipotent monodromy.

\begin{prop} \label{prop:basic local systems}
If either $\alpha \in \mb{Q}$ and $\mb{K}$ is an algebraic extension of $\mb{Q}$ containing $e^{-2\pi i \alpha}$, or $\alpha \in \mb{R}$ and $\mb{K} = \mb{C}$, then there exists a $\mb{K}$-mixed Hodge module on $\mb{C}^\times$, extendable to $\mb{C}$, whose underlying $\ms{D}_{\mb{C}^\times}$-module $\mc{O}_{\mb{C}^\times} t^\alpha$ with Hodge and weight filtrations
\[ W_w \mc{O}_{\mb{C}^\times}t^\alpha = \begin{cases} \mc{O}_{\mb{C}^\times}t^\alpha, & \text{if $w \geq 1$}, \\ 0, & \text{otherwise,} \end{cases} \quad \text{and} \quad F_p \mc{O}_{\mb{C}^\times}t^\alpha = \begin{cases} \mc{O}_{\mb{C}^\times}t^\alpha, & \text{if $p \geq 0$}, \\ 0, & \text{otherwise.} \end{cases}\] 
\end{prop}

In the first case, one can construct such a mixed Hodge module by extracting the $e^{-2\pi i \alpha}$-eigenspace of monodromy inside the direct image of the constant mixed Hodge module under the $d$-fold covering $\mb{C}^\times \to \mb{C}^\times$, where $d$ is chosen so that $d \alpha \in \mb{Z}$. In the second case, we take this to be a desideratum of the theory. Note that the connection $\mc{O}_{\mb{C}^\times}t^\alpha$ is unitary so underlies a trivial polarisable variation of complex Hodge structure, and hence a polarisable complex Hodge module \cite[Theorem 16.2.1]{MHMproject}.

Finally, we will also need two more mixed Hodge module constructions: pullbacks and tensor products. If $f \colon X \to Y$ is a smooth morphism of relative dimension $d$, then there is a pullback functor
\[ f^\circ = f^*[d] = f^![-d](-d) \colon \mhm_\mb{K}(Y) \to \mhm_\mb{K}(X) \]
given on the underlying bifiltered $\ms{D}$-modules by
\[ f^\circ \mc{M} = f^*\mc{M}, \quad W_wf^\circ \mc{M} = f^*W_{w - \dim X + \dim Y}\mc{M}, \quad F_pf^\circ \mc{M} = f^*F_p\mc{M},\]
where on the right hand sides of the equations $f^*$ means the pullback of sheaves of $\mc{O}$-modules. This is true by construction \cite[\S 2.17]{Saito90}. If, on the other hand, $f \colon X \to Y$ is instead a closed immersion, then we have cohomological functors
\[ \mc{H}^i f^*, \mc{H}^i f^! \colon \mhm_\mb{K}(Y) \to \mhm_\mb{K}(X),\]
lifting the analogous functors for regular holonomic $\ms{D}$-modules, which can be written down using Kashiwara's equivalence and either the functors $j_!j^*$, $j_*j^*$ for the inclusion $j$ of the complement of a divisor \cite[Proposition 2.19]{Saito90}, or nearby and vanishing cycles \cite[Corollary 2.24]{Saito90}. (As for direct images, these cohomological functors can be upgraded to well-defined functors at the derived category level in the algebraic case \cite[\S4.4]{Saito90}.)

We also have an external tensor product functor
\[ \boxtimes \colon \mhm_\mb{K}(X) \times \mhm_\mb{K}(Y) \to \mhm_{\mb{K}}(X \times Y) \]
lifting the usual external tensor product for $\ms{D}$-modules and given on weight and Hodge filtrations by
\[ W_w(\mc{M} \boxtimes \mc{N}) = \sum_{w_1 + w_2 \leq w} W_{w_1}\mc{M} \boxtimes W_{w_2}\mc{N}, \quad F_p(\mc{M} \boxtimes \mc{N}) = \sum_{p_1 + p_2 \leq p} F_{p_1} \mc{M} \boxtimes F_{p_2}\mc{N}.\]
The existence of the external tensor product is a non-trivial theorem in the case $\mb{K} = \mb{Q}$ (and hence the case $\mb{K}$ algebraic over $\mb{Q}$) \cite[Theorem 3.28]{Saito88}; we again take this as a desideratum of the complex theory. If $\mc{M}, \mc{N} \in \mhm_{\mb{K}}(X)$, their tensor product is now defined by
\[ \mc{M} \otimes \mc{N} = \mc{H}^{-\dim X} \Delta^*(\mc{M} \boxtimes \mc{N}) \in \mhm_\mb{K}(X).\]
Since the definition involves a non-smooth pullback, the behaviour of the Hodge filtration under $\otimes$ is somewhat complicated in general. In the case where the $\ms{D}$-module underlying $\mc{N}$ is a connection (i.e.\ $\mc{N}$ is a variation of mixed Hodge structure), however, the obvious formula holds:

\begin{lemma}\label{lem: tensor product formula}
If $\mc{M}, \mc{N} \in \mhm_\mb{K}(X)$ and $\cN$ is a flat connection, then 
\[ F_p(\cM\otimes \cN)=\sum_{q+r\leq p} F_q\cM\otimes F_r\cN, \quad \text{for all $p\in \Z$}.\]
\end{lemma}

\begin{proof}
Let us consider the maps
\[ 
T^{\ast}(X\times X) \xleftarrow{p_2} X\times_{X\times X} T^{\ast}(X\times X) \xrightarrow{d\Delta} T^{\ast}X
\]
Since $\cN$ is a flat connection, its singular support is the zero section $X \subset T^*X$. 
So the restricted map 
\begin{equation*}\label{eqn: restriction of dDelta} d\Delta \colon p_2^{-1}\left(\textrm{SS}(\cM\boxtimes \cN)\right) = \mrm{SS}(\mc{M}) \to T^{\ast}X\end{equation*}
is a closed immersion, hence finite. 
In other words, the diagonal embedding $\Delta$ is non-characteristic for the $\sD$-module underlying $\cM\boxtimes \cN$.  Therefore, we can apply Lemma \ref{lemma: hodge filtration on noncharacteristic pullback} below to obtain \[F_{p}\cH^{-\dim X}\Delta^{\ast}(\cM\boxtimes \cN)=\Delta^{\ast}(F_p(\cM\boxtimes \cN))=\sum_{q+r\leq p} F_q\cM\otimes F_r\cN.\qedhere\]
\end{proof}
\begin{lemma}\label{lemma: hodge filtration on noncharacteristic pullback}
Let $\cM$ be a mixed Hodge module on $X$. If $i\colon Z\to X$ is a closed embedding of a smooth subvariety and is non-characteristic for the $\sD$-module underlying $\cM$, then 
\[ F_p\cH^{-(\dim X-\dim Z)}i^{\ast}\cM = i^{\ast}(F_p\cM), \quad \text{for all $p\in \Z$}. \]
\end{lemma}

\begin{proof}
This is well-known, see the proof of \cite[Lemma 9.5]{Schnell16} for example. We recall the proof here for reader's convenience. One can factor $i:Z\to X$ into a sequence of closed embedding of codimension 1 of the form $i_r:Z_r\to Z_{r+1}$. By \cite[Lemma 3.5.4]{Saito88}, the closed embedding $i_r$ is still non-characteristic for the pullback of $\cM$ to $Z_{r+1}$. Hence we can assume $Z$ is defined by a single function $t$. In this case, one can show that, c.f. \cite[Lemma 3.5.6]{Saito88},
\[ V^{\alpha}\cM=t^{\lceil \alpha-1\rceil}\cdot \cM, \quad \text{for all $\alpha \in \R$},\]
where $V^{\bullet}\cM$ is the $V$-filtration along $t$. Furthermore, there is an isomorphism
\[\cH^{-1}i^{\ast}\cM\cong \textrm{ker}\left(\d_t: \gr^1_V\cM\to  \gr^0_V\cM\right)=\gr^1_V\cM \in \MHM(Z),\]
(see \cite[(2.24.3)]{Saito90} for example), and we conclude
\begin{equation*}
F_p\cH^{-1}i^{\ast}\cM\cong F_p\gr^1_V\cM=\frac{V^{1}\cM\cap F_p\cM}{V^{>1}\cM\cap F_p\cM}=\frac{F_p\cM}{t\cdot F_p\cM}=i^{\ast}(F_p\cM). \qedhere 
\end{equation*}
\end{proof}

We will also need a version of Lemma \ref{lem: tensor product formula} for the $F^*_\bullet$-filtrations.

\begin{lemma} \label{lem: * tensor product formula}
If $D \subset X$ is a divisor with complement $U$, $\mc{M}, \mc{N} \in \mhm_\mb{K}(U)_X$ are extendable mixed Hodge modules and $\mc{N}$ is a flat connection on $U$, then
\[ F_p^*(\cM\otimes \cN)(*D) =\sum_{q+r\leq p} F_q^*\cM(*D)\otimes F_r^*\cN(*D), \quad \text{for all $p\in \Z$}.\]
\end{lemma}
\begin{proof}
Choose extensions $\mc{M}'$ and $\mc{N}'$ of $\mc{M}$ and $\mc{N}$ to $X$ and consider the $\mc{O}_X$-coherent subsheaves
\[ F_p(\mc{M}' \otimes \mc{N}') \quad \text{and} \quad \sum_{q + r \leq p} \mrm{im}(F_q\mc{M}' \otimes F_r\mc{N}' \to \mc{M}' \otimes \mc{N}')\]
of $\mc{M}' \otimes \mc{N}'$. By Lemma \ref{lem: tensor product formula}, these are equal when restricted to the complement of $D$, so their localisations at $D$ must also be equal by coherence. This gives the lemma.
\end{proof}

\subsection{Localisation of mixed Hodge modules via the multivariate $V$-filtration} \label{subsec:hodge localisation}

In this subsection, we give a formula for the Hodge filtration on the extension of a mixed Hodge module across a simple normal crossings divisor, in terms of the multivariate $V$-filtration. This simultaneously generalises \eqref{eqn: filtration on *extension} from smooth divisors to simple normal crossings, and Proposition \ref{prop:* formula} from $\ms{D}$-modules to mixed Hodge modules.

In what follows, we fix a mixed Hodge module $\mc{M}$ on a complex manifold $X$ and a simple normal crossings divisor $D = D_1 \cup \cdots \cup D_r \subset X$. We write $U = X \setminus D$ and $j \colon U \to X$ for the inclusion. For $\bs{\alpha} \in \mb{R}^r$, we write $F_\bullet V^{\bs{\alpha}}\mc{M} = F_\bullet \mc{M} \cap V^{\bs{\alpha}}\mc{M}$, where $F_\bullet\mc{M}$ is the Hodge filtration.

First, we need a multivariate generalisation of the strict specialisability (or, more precisely, Definition \ref{defn: strict A-specializability}\eqref{itm: t strict}).
\begin{lem} \label{lem:hodge * formula}
If $\mc{M} \in \mhm(U)_X$ then, for each $i$, the morphism
\begin{equation} \label{eq:hodge * formula 1}
t_i \colon (V^{\bs{\alpha}}j_*\mc{M}, F_\bullet) \to (V^{\bs{\alpha} + \bs{1}_i}j_*\mc{M}, F_\bullet)
\end{equation}
is strict injective for $\alpha_i \geq 0$. In particular, if $\bs{\alpha}\in \R^r_{\geq 0}$, the inclusion map
\begin{equation} \label{eq:hodge * formula 2}
\left(V^{\bs{\alpha}}j_*\cM,F_{\bullet}\right)\to (\mc{M}(*D), F_\bullet^*)
\end{equation}
is strict.
\end{lem}
\begin{proof}
Since $j_*\mc{M} = (j_*\mc{M})(*D)$, the map $t_i$ above is clearly injective. It therefore suffices to show that
\[ F_pV^{\bs{\alpha}}\mc{M} = t_i^{-1} (F_pV^{\bs{\alpha} + \bs{1}_i}\mc{M}).\]
Since $F_pV^{\bs{\alpha}}j_*\mc{M} = F_pj_*\mc{M} \cap V^{\bs{\alpha}}j_*\mc{M}$ by definition and since $V^{\bs{\alpha} + \bs{1}_i}j_*\mc{M} = t_iV^{\bs{\alpha}}j_*\mc{M}$ by Propositions \ref{prop:weak specialisation} and \ref{prop:* criterion}, this is equivalent to showing that $m \in F_pj_*\mc{M}$ as long as $m \in V^{\bs{\alpha}}j_*\mc{M}$ and $t_i m \in F_p j_*\mc{M}$. But by Proposition \ref{prop:V-filtration for sub divisor}, $V^{\bs{\alpha}}j_*\mc{M} \subset V_{D_i}^{\alpha_i}j_*\mc{M}$, so this follows from the analogous claim for $V_{D_i}^{\alpha_i}j_*\mc{M}$. For $\alpha_i > 0$ this is immediate from Proposition \ref{prop:mhm strict specialisation}. 
 For $\alpha_i = 0$, we have a diagram
\[
\begin{tikzcd}
0 \ar[r] & F_p V^{>0}_{D_i}j_*\mc{M} \ar[r] \ar[d, "t_i", "\sim"' {anchor=south, rotate=90}] & F_p V^0_{D_i}j_*\mc{M} \ar[r] \ar[d, "t_i"] & F_p \gr_{V_{D_i}}j_*\mc{M} \ar[r] \ar[d, "t_i"] & 0 \\
0 \ar[r] & F_p V^{>1}_{D_i}j_*\mc{M} \ar[r] & F_pV^1_{D_i}j_*\mc{M} \ar[r] & F_p\gr_{V_{D_i}}^1 j_*\mc{M} \ar[r] & 0
\end{tikzcd}
\]
with exact rows. Since $j_*\mc{M} = (j_*\mc{M})(*D_i)$, $t_i \colon \gr_{V_{D_i}}^{0}j_*\mc{M} \to \gr_{V_{D_i}}^{1}j_*\mc{M}$ is an isomorphism at the level of $\ms{D}_{D_i}$-modules, it is also an isomorphism of mixed Hodge modules. So the vertical arrow on the right is an isomorphism and hence the middle arrow is an isomorphism as well. This proves the desired strictness of \eqref{eq:hodge * formula 1}.

To deduce the strictness of \eqref{eq:hodge * formula 2}, note that $m \in V^{\bs{\alpha}}j_*\mc{M}$ lies in $F_p^*\mc{M}(*D)$ if and only if $\bs{t}^{\bs{n}} m$ lies in $F_pj_*\mc{M}$ for some $\bs{n} \in \mb{Z}_{\geq 0}^r$, which implies $m \in F_p V^{\bs{\alpha}}j_*\mc{M}$ by repeated application of \eqref{eq:hodge * formula 1}.
\end{proof}

\begin{prop} \label{prop:hodge !* formula}
For $\mc{M} \in \mhm(U)_X$, the canonical map from left to right defines an isomorphism
\[  (\ms{D}_X, F_\bullet) \overset{\mrm{L}}\otimes_{(V^{\bs{0}}\ms{D}_X, F_\bullet)} (V^{\bs{0}}\mc{M}(*D), F_\bullet^*) = (\ms{D}_X, F_\bullet) \overset{\mrm{L}}\otimes_{(V^{\bs{0}}\ms{D}_X, F_\bullet)} (V^{\bs{0}}j_*\mc{M}, F_\bullet) \cong (j_*\mc{M}, F_\bullet). \]
in the filtered derived category $\mrm{D}^b_{\mrm{coh}}(\ms{D}_X, F_\bullet)$.
\end{prop}
\begin{proof}
We argue in a similar way to Proposition \ref{prop:* formula}, but with Hodge filtrations in place. For all $\bs{n}\in \Z^r$ and each $i$, we first note that the inclusion $V^{\bs{n} + \bs{1}_i}\ms{D}_X \to V^{\bs{n}}\ms{D}_X$ is strict with respect to $F_\bullet$, and that we have a strict exact sequence of filtered right $V^{\bs{0}}\ms{D}_X$-modules, induced by \eqref{eqn: Vn-1i/Vn},
\[ 0 \to (V^{\bs{0}}\ms{D}_X, F_{\bullet - k}) \xrightarrow{t_i} (V^{\bs{0}}\ms{D}_X, F_{\bullet - k}) \to \left(\frac{V^{\bs{n}}\ms{D}_X}{V^{\bs{n} + \bs{1}_i}\ms{D}_X}, F_\bullet\right)  \to 0, \]
where the second map is the operator $\prod_{j = 1}^r P_j$ for
\[ P_j = \begin{cases} t_j^{n_j}, & \text{if $n_j \geq 0$}, \\ \partial_{t_j}^{-n_j}, & \text{if $n_j < 0$}\end{cases} \]
and $k = \sum_{j = 1}^r \deg P_j= \sum_{j = 1}^r \max\{0, -n_j\}$. Since the map
\[ t_i \colon (V^{\bs{0}}j_*\mc{M}, F_\bullet) \to (V^{\bs{0}}j_*\mc{M}, F_\bullet) \]
is strict injective by Lemma \ref{lem:hodge * formula}, we conclude that
\[ \left(\frac{V^{\bs{n}}\ms{D}_X}{V^{\bs{n} + \bs{1}_i}\ms{D}_X}, F_\bullet\right) \overset{\mrm{L}}\otimes_{(V^{\bs{0}}\ms{D}_X, F_\bullet)} (V^{\bs{0}}j_*\mc{M}, F_\bullet) = \mrm{Cone}((V^{\bs{0}}j_*\mc{M}, F_{\bullet - k}) \xrightarrow{t_i} (V^{\bs{0}}j_*\mc{M}, F_{\bullet - k})) \]
is a strict complex concentrated in degree $0$. We also know that, for $\bs{n} \in \mb{Z}_{\geq 0}^r$, the map
\[ \bs{t}^{\bs{n}}\colon (V^{\bs{0}}\ms{D}_X, F_\bullet) \overset{\sim}\to (V^{\bs{n}}\ms{D}_X, F_\bullet) \]
given by multiplication by $\bs{t^n}$ on the left is an isomorphism of right $V^{\bs{0}}\ms{D}_X$-modules. Combining these statements, we deduce that for all $\bs{n}$, 
\[ (V^{\bs{n}}\ms{D}_X, F_\bullet) \overset{\mrm{L}}\otimes_{(V^{\bs{0}}\ms{D}_X, F_\bullet)} (V^{\bs{0}}j_*\mc{M}, F_\bullet) \]
is a strict complex in degree $0$, and that the map
\[ (V^{\bs{n}}\ms{D}_X, F_\bullet) \overset{\mrm{L}}\otimes_{(V^{\bs{0}}\ms{D}_X, F_\bullet)} (V^{\bs{0}}j_*\mc{M}, F_\bullet) \to (V^{\bs{n} - \bs{1}_i}\ms{D}_X, F_\bullet) \overset{\mrm{L}}\otimes_{(V^{\bs{0}}\ms{D}_X, F_\bullet)} (V^{\bs{0}}j_*\mc{M}, F_\bullet) \]
is a strict injection. Taking the colimit over $\bs{n}\in \Z^r$ and applying Proposition \ref{prop:* formula}, 
we get a filtration $F_\bullet'j_*\mc{M}$ such that
\begin{equation} \label{eq:hodge !* formula 1}
 (V^{\bs{n}}\ms{D}_X, F_\bullet) \overset{\mrm{L}}\otimes_{(V^{\bs{0}}\ms{D}_X, F_\bullet)} (V^{\bs{0}}j_*\mc{M}, F_\bullet) = (V^{\bs{n}}j_*\mc{M}, F_\bullet').
\end{equation}

It remains to show that $F_\bullet'j_*\mc{M}$ coincides with the Hodge filtration $F_\bullet j_*\mc{M}$. We prove this by induction on the number $r$ of components of $D$. This is vacuous if $r = 0$ (i.e.\ if $D = \emptyset$). If $r > 0$, we may suppose by induction that the two filtrations agree on the complement of, say, $D_r$. This implies that $t_r^k F_\bullet \subset F_\bullet'$ and $t_r^k F_\bullet' \subset F_\bullet$ for some $k > 0$. Now, from \eqref{eq:hodge !* formula 1} 
it follows by Proposition \ref{prop:V-filtration for sub divisor} that
\[ t_r \colon (V_{D_r}^nj_*\mc{M}, F_\bullet') \to (V_{D_r}^{n + 1}j_*\mc{M}, F_\bullet') \]
is an isomorphism for $n \geq 0$, as this morphism is the colimit of the filtered isomorphisms
\[t_r \colon (V^{\bs{n}',n}\ms{D}_X, F_\bullet) \overset{\mrm{L}}\otimes_{(V^{\bs{0}}\ms{D}_X, F_\bullet)} (V^{\bs{0}}j_*\mc{M}, F_\bullet) \to (V^{\bs{n}',n+1}\ms{D}_X, F_\bullet) \overset{\mrm{L}}\otimes_{(V^{\bs{0}}\ms{D}_X, F_\bullet)} (V^{\bs{0}}j_*\mc{M}, F_\bullet) \]
over $\bs{n}' \in \mb{Z}^{r - 1}$. Since the same is true for $F_\bullet$ by strict specialisability along $D_r$ (see Definition \ref{defn: strict A-specializability}\eqref{itm: t strict}), we claim that this implies
 \begin{equation}\label{eqn: F and F' agree on V0}F_\bullet V_{D_r}^{0}j_*\mc{M} = F'_\bullet V_{D_r}^0j_*\mc{M}. \end{equation}
Indeed, we have 
\[ F_{\bullet}V_{D_r}^{k}j_*\mc{M}=t^k_{r}F_{\bullet}V_{D_r}^{0}j_*\mc{M} \subset F'_{\bullet}V_{D_r}^{k}j_*\mc{M}\]
The same argument gives the reverse inclusion, so in fact
\[ F_\bullet V_{D_r}^{k}j_*\mc{M} = F'_\bullet V_{D_r}^kj_*\mc{M}.\]
This implies \eqref{eqn: F and F' agree on V0} by applying $t_r^{-k}$ on both sides.
Now, \eqref{eq:hodge !* formula 1} and Proposition \ref{prop:V-filtration for sub divisor} also imply that
\[ (j_*\mc{M}, F_\bullet') = (\ms{D}_X, F_\bullet) \otimes_{(V^0_{D_r}\ms{D}_X, F_\bullet)} (V_{D_r}^0j_*\mc{M}, F_\bullet'). \]
Since the same is again true for $F_\bullet$ by \eqref{eqn: filtration on *extension}, we deduce that $F_\bullet j_*\cM = F_\bullet'j_*\cM$ as claimed.
\end{proof}

\begin{remark}\label{remark: !-extension multivariate}
In fact, by a duality argument \`a la Theorem \ref{thm:duality}, Proposition \ref{prop:* formula} also implies the dual formula
\[ \ms{D}_X \overset{\mrm{L}}\otimes_{V^{\bs{0}}\ms{D}_X} V^{\epsilon \bs{1}}\mc{M}(*D) = j_!\mc{M}.\]
So a similar argument to Proposition \ref{prop:hodge !* formula} implies the formula
\[ (\ms{D}_X, F_\bullet) \overset{\mrm{L}}\otimes_{(V^{\bs{0}}\ms{D}_X, F_\bullet)} (V^{\epsilon\bs{1}}\mc{M}(*D), F^*_\bullet) \cong (j_!\mc{M}, F_\bullet) \]
for the Hodge filtration on $j_!\mc{M}$. We will not need this fact, however, so we leave the details to the interested reader.
\end{remark}

\subsection{From the Hodge filtration to the $V$-filtration}\label{subsec:magic formula}

In this subsection, we prove a multivariable version of the main result of \cite{DY25}, which expresses the multivariate $V$-filtration on a mixed Hodge module to the Hodge filtrations on free-monodromic local systems.
\subsubsection{The single variable case}
We begin by recalling the basic input into the construction. Consider $\mb{C}^\times$ with coordinate $t$ and the flat connections
\[ \mc{E}_n \colonequals \frac{\mc{O}_{\mb{C}^\times}[s]t^s}{(s^n)}\]
for $n \in \mb{Z}_{>0}$; see \S \ref{subsec:malgrange-mellin}. Recall from Example \ref{eg:nilpotent orbit} that each $\mc{E}_n$ carries an admissible variation of $\mb{Q}$-mixed Hodge structure. So by Theorem \ref{thm:vmhs is mhm} we obtain lifts of $\mc{E}_n$ to objects in $\mhm(\mb{C}^\times)_\mb{C}$.
For $\alpha \in \mb{R}$, we can also form the variation of complex mixed Hodge structure
\[ \mc{E}_{n, \alpha} \colonequals  \frac{\mc{O}_{\mb{C}^\times}[s]t^s}{(s + \alpha)^n} = \mc{E}_n \otimes \mc{O}_{\mb{C}^\times}t^{-\alpha}\]
by tensoring with the unitary local system $\mc{O}_{\mb{C}^\times}t^{-\alpha}$ with its Hodge structure given by Proposition \ref{prop:basic local systems}---here we must restrict to $\alpha \in \mb{Q}$ in order to work within $\mhm_{\bar{\mb{Q}}}$. By Lemma \ref{lem: tensor product formula}, $\mc{E}_{n, \alpha}$ is again an object in $\mhm(\mb{C}^\times)_{\mb{C}}$. Chasing through the various constructions, the weight and Hodge filtrations on $\mc{E}_{n, \alpha}$ are given by 
\[ W_{\ell} \mc{E}_{n, \alpha} = \sum_{2k \geq -\ell + 1} \mc{O}_{\mb{C}^\times}s^kt^{s + \alpha}, \quad F_p \mc{E}_{n, \alpha} = \sum_{k \leq p} \mc{O}_{\mb{C}^\times} s^kt^{s + \alpha}.\]
Moreover, since the quotient map $\mc{E}_{n + 1} \to \mc{E}_n$ is a morphism of admissible variations of $\mb{Q}$-mixed Hodge structure, we conclude that
\begin{equation}\label{eqn: En+1toEn} \mc{E}_{n + 1, \alpha} \to \mc{E}_{n, \alpha} \end{equation}
is a morphism in $\mhm(\mb{C}^\times)_{\mb{C}}$ for all $n$. Similarly, since multiplication by $s$ defines a morphism of variations of mixed Hodge structure $s \colon \mc{E}_n(1) \to \mc{E}_{n + 1}$, we also have that
\begin{equation} \label{eqn: EntoEn+1}
 s + \alpha \colon \mc{E}_{n, \alpha}(1) \to \mc{E}_{n + 1, \alpha}
\end{equation}
is a morphism in $\mhm(\mb{C}^\times)_{\mb{C}}$.

\begin{defn}
The \emph{free-monodromic local system on $\mb{C}^\times$ with monodromy $\alpha$} is the formal inverse limit
\[ \mc{O}_{\mb{C}^\times}[[s + \alpha]] t^s := \varprojlim_n \mc{E}_{n, \alpha} \in \operatorname{Pro} \mhm(\mb{C}^\times)_\mb{C}.\]
\end{defn}

Free-monodromic local systems (with and without Hodge structures) appear implicitly in Beilinson's construction of the nearby cycles and maximal extension functors \cite{Beilinson} and are playing an increasingly important role in representation theory (e.g.\ \cite{bezrukavnikov-yun} and \cite{DMB}). While for many of these applications it is important to regard this as an object in the category of pro-objects as above, this is only a matter of conceptual convenience here: for the construction below, it suffices to work with $(\mc{E}_{n, \alpha})_{n \in \mb{Z}_{> 0}}$ as an inverse system.

Now let $X$ be a complex manifold, $f \colon X \to \mb{C}$ a function with zero locus $D = f^{-1}(0)$ and $\mc{M}$ a mixed Hodge module on $U := f^{-1}(\mb{C}^\times)$. Then
\[ \frac{\mc{M}[s]f^s}{(s + \alpha)^n} = \mc{M} \otimes f^*\mc{E}_{n, \alpha} \]
is a well-defined object in $\mhm(U)$. Since $f^*\mc{E}_{n, \alpha}$ is a flat connection, by Lemma \ref{lem: tensor product formula} one has
\[ F_k\frac{\mc{M}[s]f^s}{(s + \alpha)^n}=\sum_{p+q\leq k} \frac{F_p\cM s^q f^s}{(s+\alpha)^n}.\]
Here we have used that the pullback of a connection is always non-characteristic, so the Hodge filtration on $f^*\mc{E}_{n, \alpha}$ is given by the obvious formula by Lemma \ref{lemma: hodge filtration on noncharacteristic pullback}. We form the formal inverse limit
\[ \mc{M}[[s + \alpha]]f^s := \varprojlim_n \frac{\mc{M}[s]f^s}{(s + \alpha)^n} \in \pro \mhm(U).\]
If $\mc{M}$ is extendable from $U$ to $X$, then so is each $\mc{M} \otimes f^*\mc{E}_{n, \alpha}$. So letting $j \colon U \to X$ be the inclusion, we may form the $*$-extension
\[ j_*\mc{M}[[s + \alpha]]f^s := \varprojlim_n j_*\frac{\mc{M}[s]f^s}{(s + \alpha)^n} \in \pro \mhm(X).\]

As a pro-$\ms{D}_X$-module, $j_*\mc{M}[[s + \alpha]]f^s$ is a completion of the non-holonomic $\ms{D}_X$-module $\mc{M}(*D)[s]f^s$, presented as an inverse limit of holonomic modules. 
As has been observed previously \cite{DV23, DMB}, the Hodge filtration provides a canonical ``decompletion'' of this $\ms{D}_X$-module as follows. For fixed $p$, we consider the inverse system
\[ F_p\left(j_*\frac{\mc{M}[s]f^s}{(s + \alpha)^n} \right)\]
and set
\[ F_p j_*\mc{M}[[s + \alpha]]f^s \colonequals \varprojlim_n F_p \left(j_*\frac{\mc{M}[s]f^s}{(s + \alpha)^n} \right).\]
As discussed in the introduction, this defines an interesting $\sD_X[s]$-submodule
\[ \bigcup_p F_pj_*\mc{M}[[s + \alpha]]f^s \subset j_*\mc{M}[[s + \alpha]]f^s.\]

\begin{lemma}\label{lemma: stabilisation of the inverse limit}
For any fixed $p$, the inverse system above stablises, i.e.
\[ F_p j_*\mc{M}[[s + \alpha]]f^s= F_p \left(j_*\frac{\mc{M}[s]f^s}{(s + \alpha)^n} \right),\quad \textrm{for $n\gg 0$}.\]
\end{lemma}
\begin{proof}
Note that the map \eqref{eqn: En+1toEn} induces a short exact sequence of mixed Hodge modules
\begin{equation}\label{eqn: ses for stablisation} 0 \to j_*\cM(n) \xrightarrow{s^n} j_*\frac{\mc{M}[s]f^s}{(s + \alpha)^{n+1}} \to j_*\frac{\mc{M}[s]f^s}{(s + \alpha)^n} \to 0.\end{equation}
For $n\gg 0$, one has $F_pj_*\cM(n)=F_{p-n}j_*\cM=0$, because the Hodge filtration on $j_*\cM$ is bounded below (see Remark \ref{rmk:hodge lower bound}). This proves the statement.\end{proof}

The main theorem of \cite{DY25} is\footnote{We only proved the algebraic case, but the proof therein can be adapted to the analytic situation.} 

\begin{thm} \label{thm:univariate magic formula}
In the setting above, we have a natural filtered $\sD_X[s]$-isomorphism
\[ \left(\bigcup_p F_p j_*\mc{M}[[s + \alpha]]f^s, F_\bullet\right) = (V^{\alpha}\iota_+\mc{M}(*D), F_{\bullet + 1})\]
for $\alpha \geq 0$, where $\iota\colon X\to X\times \C$ is the graph embedding of $f$. 
\end{thm}
In fact, we proved slightly more: the isomorphism holds for any $\alpha\in \R$ \cite[Theorem 1.1]{DY25} on the level of $\sD_X[s]$-modules (without the filtrations).

\subsubsection{The multivariable case}
We generalise Theorem \ref{thm:univariate magic formula} to several functions as follows. Suppose now that $f_1, \ldots, f_r \colon X \to \mb{C}$ are holomorphic functions on a complex manifold $X$, cutting out a divisor $D = \{f_1 \cdots f_r = 0\}$ with complement $U = X \setminus D$ and that $\mc{M} \in \mhm(U)_X$ is a mixed Hodge module extendable to $X$. In this case, for $\bs{\alpha} = (\alpha_1, \ldots, \alpha_r) \in \mb{R}^r$, we consider the free-monodromic local system on the torus $(\mb{C}^\times)^r$ defined by
\[ \mc{O}_{(\mb{C}^\times)^r}[[\bs{s} + \bs{\alpha}]]\bs{t}^{\bs{s}} := \varprojlim_n (\mc{E}_{n, \alpha_1} \boxtimes \cdots \boxtimes \mc{E}_{n, \alpha_r}) \in \pro \mhm((\mb{C}^\times)^r)_{\mb{C}^r}\]
and the associated pro-mixed Hodge module
\[ \mc{M}[[\bs{s} + \bs{\alpha}]]\bs{f^s} := \varprojlim_n \frac{\mc{M}[\bs{s}]\bs{f^s}}{((s_1 + \alpha_1)^n, \ldots, (s_r + \alpha_r)^n)}\in \pro \mhm(U)_X\]
where
\[ \frac{\mc{M}[\bs{s}]\bs{f^s}}{((s_1 + \alpha_1)^n, \ldots, (s_r + \alpha_r)^n)} \colonequals \mc{M} \otimes (f_1, \ldots, f_r)^*(\mc{E}_{n, \alpha_1} \boxtimes \cdots \boxtimes \mc{E}_{n, \alpha_r})\]
as a mixed Hodge module. Here $(f_1,\ldots,f_r)$ denotes the induced map $U\to (\mb{C}^{\times})^r$. As in the single variable case, we can extend to $X$ and use the Hodge filtration to decomplete by setting
\begin{equation}\label{eqn: Fp of multivariate magic formula} F_p j_*\mc{M}[[\bs{s} + \bs{\alpha}]]\bs{f^s} = \varprojlim_n F_p\left(j_*\frac{\mc{M}[\bs{s}]\bs{f^s}}{((s_1 + \alpha_1)^n, \ldots, (s_r + \alpha_r)^n)}\right),\end{equation}
obtaining a filtered $\ms{D}_X[\bs{s}]$-submodule
\[ \bigcup_p F_p j_*\mc{M}[[\bs{s} + \bs{\alpha}]]\bs{f^s} \subset j_*\mc{M}[[\bs{s} + \bs{\alpha}]]\bs{f^s}=\varprojlim_n j_{\ast}\frac{\mc{M}[\bs{s}]\bs{f^s}}{((s_1 + \alpha_1)^n, \ldots, (s_r + \alpha_r)^n)}.\]
\begin{remark}\label{remark: stablization for several variable}
By the same argument as in Lemma \ref{lemma: stabilisation of the inverse limit}, for each $p$, the inverse system \eqref{eqn: Fp of multivariate magic formula} stabilises for large $n$.
\end{remark}

The main theorem of this subsection is:

\begin{thm} \label{thm:multivariate magic formula}
For $\bs{\alpha} \in \mb{R}^r$, we have a canonical isomorphism of filtered $\ms{D}_X[\bs{s}]$-modules
\begin{equation}\label{eqn: * version of multivariate magic}
\left(\bigcup_p F_p j_*\mc{M}[[\bs{s} + \bs{\alpha}]]\bs{f^s}, F_\bullet\right) \cong (V^{\bs{\alpha}}_*\iota_+\mc{M}(*D), F_{\bullet + r}^*),\end{equation}
where $\iota \colon X\to X\times \C^r$ is the graph embedding of $f_1,\ldots,f_r$ and $F^*_{\bullet}V^{\bs{\alpha}}_*\iota_+\mc{M}(*D)=F^*_{\bullet}\iota_{+}\mc{M}(\ast D)\cap V_*^{\bs{\alpha}}\iota_{+}\mc{M}(*D)$. In particular, for each $p$, 
\[ F_{p+r}^*V^{\bs{\alpha}}_*\iota_+\mc{M}(*D)\cong F_p\left(j_*\frac{\mc{M}[\bs{s}]\bs{f^s}}{((s_1 + \alpha_1)^n, \ldots, (s_r + \alpha_r)^n)}\right), \quad \textrm{if $n\gg 0$}.\]
\end{thm}

\begin{remark}
As stated, Theorem \ref{thm:multivariate magic formula} gives a formula for the $V_*$-filtration and the localised Hodge filtration $F_\bullet^*$. Recall, however, that for $\bs{\alpha} \in \mb{R}_{\geq 0}^r$, we have
\[ (V_*^{\bs{\alpha}}\iota_+\mc{M}(*D), F_\bullet^*) = (V^{\bs{\alpha}}\iota_+j_*\mc{M}, F_\bullet)\]
by Corollary \ref{cor:extending V*} and Lemma \ref{lem:hodge * formula}, so this also gives a formula for the actual Hodge and $V$-filtrations in this regime (if $r=1$, it recovers Theorem \ref{thm:univariate magic formula}). Moreover, by Propositions \ref{prop:weak specialisation} and \ref{prop:multivariate V extension}, the whole multivariate $V$-filtration $V^{\bullet}\iota_{+}\cM(\ast D)$ can be recovered from the free-monodromic local system on the left of \eqref{eqn: * version of multivariate magic} via the formula \eqref{eqn: whole V filtration from Vres}.
\end{remark}

The proof of Theorem \ref{thm:multivariate magic formula} takes up the rest of the subsection. As a first step, let us record the following filtered version of the Malgrange-Mellin transform.

\begin{prop} \label{prop:filtered malgrange}
The Malgrange-Mellin transform $\mc{M}(*D)[\bs{s}]\bs{f^s} \cong \iota_+\mc{M}(*D)$ restricts to an isomorphism
\[ F_p^*\mc{M}(*D)[\bs{s}]\bs{f^s} \cong F_{p + r}^*\iota_+\mc{M}(*D),\]
for all $p$, where
\[ F_p^*\mc{M}(*D)[\bs{s}]\bs{f^s} \colonequals \sum_{j + k_1 + \cdots + k_r \leq p}F_j^*\mc{M}(*D) s_1^{k_1} \cdots s_r^{k_r}\bs{f^s} \subset \mc{M}(*D)[\bs{s}]\bs{f^s}.\]
\end{prop}
\begin{proof}
Since each $f_i^{-1}$ preserves $F_p^*\mc{M}(*D)$ by definition, this follows directly from the usual formulas for the Malgrange-Mellin transform and its inverse as in the single variable case \cite[Lemma 2.3]{DY25}. Note that the shift by $r$ comes from the definition of the Hodge filtration on $\iota_+\mc{M}$.
\end{proof}

Using Proposition \ref{prop:filtered malgrange}, we check that Theorem \ref{thm:multivariate magic formula} holds after localisation at $D$. For simplicity of notation, let us write
\[ \mc{M}_n = j_*\frac{\mc{M}[\bs{s}]\bs{f^s}}{((s_1 + \alpha_1)^n, \ldots, (s_r + \alpha_r)^n)} \]
from now on. Since $(f_1, \ldots, f_r)^*(\mc{E}_{n, \alpha_1} \boxtimes \cdots \boxtimes \mc{E}_{n, \alpha_r})$ is a vector bundle on $U = X \setminus D$, the localised Hodge filtration on $\mc{M}_n$ is given by the naive formula
\begin{equation} \label{eq:f^s localised hodge}
F_p^*\mc{M}_n = \sum_{j + k_1 + \cdots + k_r \leq p}F_j^*\mc{M}(*D) s_1^{k_1} \cdots s_r^{k_r}\bs{f^s}.
\end{equation}
by Lemma \ref{lem: * tensor product formula}.
Since the filtration $F_\bullet^*\mc{M}(*D)$ is the localisation of a good filtration $F_\bullet j_*\mc{M}$, it is bounded below. 
So the polynomial order in the $s_i$ of the sections of \eqref{eq:f^s localised hodge} are bounded independent of $n$. So in fact,
\[ F_p^*\mc{M}[[\bs{s} + \bs{\alpha}]]\bs{f^s}(*D) \colonequals (F_pj_*\mc{M}[[\bs{s} + \bs{\alpha}]])(*D) = F_p^*\mc{M}(*D)[\bs{s}]\bs{f^s} \subset \mc{M}[\bs{s}]{\bs{f^s}}(*D).\]
Taking the union over all $p$ and applying Proposition \ref{prop:filtered malgrange}, we conclude:

\begin{lem} \label{lem:localised magic formula}
We have an isomorphism of filtered $\ms{D}_X(*D)[\bs{s}]$-modules
\[ \left(\bigcup_p F_p^*\mc{M}[[\bs{s} + \bs{\alpha}]]\bs{f^s}(*D), F_\bullet^*\right) \cong (\iota_+\mc{M}(*D), F_{\bullet + r}^*).\]
\end{lem}

Now, by construction, we have a natural localisation map
\begin{equation} \label{eq:inverse limit inclusion}
\left(\bigcup_p F_pj_*\mc{M}[[\bs{s} + \bs{\alpha}]]\bs{f^s}, F_\bullet\right) \to \left(\bigcup_p F_p^*\mc{M}[[\bs{s} + \bs{\alpha}]]\bs{f^s}(*D), F_\bullet^*\right).
\end{equation}
To prove Theorem \ref{thm:multivariate magic formula}, we will show that \eqref{eq:inverse limit inclusion} is strict injective and that the image coincides with $V^{\bs{\alpha}}_*\iota_+\mc{M}(*D)$ under the isomorphism of Lemma \ref{lem:localised magic formula}.

\begin{lem} \label{lem:inverse limit strict inclusion}
The morphism \eqref{eq:inverse limit inclusion} is strict and injective.
\end{lem}
\begin{proof}
Since the target is the localisation of the source at $f_1 \cdots f_r$, the claim is equivalent to the assertion that
\[ f_i \colon F_pj_*\mc{M}[[\bs{s} + \bs{\alpha}]]\bs{f^s} \to F_p j_*\mc{M}[[\bs{s} + \bs{\alpha}]]\bs{f^s}\]
is injective for all $i$ and all $p$. For the sake of notational simplicity, let us prove this for $i = r$. For fixed $p$, choosing $n$ sufficiently large, we have
\[ F_pj_*\mc{M}[[\bs{s} + \bs{\alpha}]]\bs{f^s} = F_p\mc{M}_n = F_p\left( j_{r*}\frac{\mc{N}[s_r]f_r^{s_r}}{((s_r + \alpha_r)^n)}\right),\]
where
\[ \mc{N} = j'_*\frac{\mc{M}[s_1, \ldots, s_{r - 1}]f_1^{s_1} \cdots f_{r - 1}^{s_{r - 1}}}{((s_1 + \alpha_1)^n, \ldots, (s_{r - 1} + \alpha_{r - 1})^n)}\]
for $j' \colon U \to X \setminus D_r$ and $j_r \colon X \setminus D_r \to X$ the inclusions. Now, by \cite[Lemma 3.5]{DV23}, the present lemma holds for a single function. Applying this to the mixed Hodge module $\mc{N}$ and the function $f_r$, we have that $f_r$ acts injectively on
\[ F_p\left( \frac{\mc{N}[s_r]f_r^{s_r}}{((s_r + \alpha_r)^m)}(*D_r)\right)\]
for $m \gg 0$. But if $m \geq n$ then we have surjections
\[ F_p\mc{M}[[\bs{s} + \bs{\alpha}]] \bs{f^s}(*D) \to F_p\left( \frac{\mc{N}[s_r]f_r^{s_r}}{((s_r + \alpha_r)^m)}(*D_r) \right)\to F_p \left(\frac{
\mc{N}[s_r]f_r^{s_r}}{((s_r + \alpha_r)^n)}(*D_r) \right)\]
such that the composition is an isomorphism. So in fact both arrows are isomorphisms and hence $f_r$ acts injectively on $F_p\mc{M}[[\bs{s} + \bs{\alpha}]]\bs{f^s}(*D)$ as claimed.
\end{proof}

It remains to prove that that the image of \eqref{eq:inverse limit inclusion} coincides with $V^{\bs{\alpha}}_*\iota_+\mc{M}(*D)$. We first treat the case where the divisor $D$ has simple normal crossings and the functions $f_i = x_i$ are coordinates cutting out the irreducible components.

\begin{lem} \label{lem:snc magic formula}
In the setting above, the image of \eqref{eq:inverse limit inclusion} coincides with $V^{\bs{\alpha}}_*\iota_+\mc{M}(*D)$.
\end{lem}
\begin{proof}
By Proposition \ref{prop:hodge !* formula}, we have a filtered isomorphism 
\begin{equation} \label{eq:snc magic formula 2}
\begin{aligned}
 &\left(\cM_n, F_\bullet\right)= (\ms{D}_X, F_\bullet) \otimes_{(V^{\bs{0}}\ms{D}_X, F_\bullet)}  \left(V^{\bs{0}}\cM_n, F_\bullet^*\right).
 \end{aligned}
\end{equation}
Now, we claim that
\begin{equation} \label{eq:snc magic formula 1}
V^{\bs{0}}_*\cM_n = \frac{(V^{\bs{\alpha}}_*\mc{M}(*D))[\bs{s}]\bs{f^s}}{((s_1 + \alpha_1)^n, \ldots, (s_r + \alpha_r)^n)}.
\end{equation}
Indeed, it is immediate from the definitions that the filtration defined by
\[ U^{\bs{\beta}}\cM_n \colonequals \frac{(V_*^{\bs{\alpha} + \bs{\beta}}\mc{M}(*D))[\bs{s}]\bs{f^s}}{((s_1 + \alpha_1)^n, \ldots, (s_r + \alpha_r)^n)}\]
satisfies the conditions of Definition \ref{definition: V*} and hence is equal to the $V_*$-filtration on $\cM_n$. So \eqref{eq:snc magic formula 1} follows. Substituting into \eqref{eq:snc magic formula 2}, we get
\begin{equation} \label{eq:snc magic formula 3}
\begin{aligned}
 &\left(\cM_n, F_\bullet\right)= (\ms{D}_X, F_\bullet) \otimes_{(V^{\bs{0}}\ms{D}_X, F_\bullet)}  \left(\frac{(V^{\bs{\alpha}}_*\mc{M}(*D))[\bs{s}]\bs{f^s}}{((s_1 + \alpha_1)^n, \ldots, (s_r + \alpha_r)^n)}, F_\bullet^*\right).
 \end{aligned}
\end{equation}
Now, fixing $p$, one can argue as in Lemma \ref{lemma: stabilisation of the inverse limit} to show that, for $n \gg 0$,
\begin{align*}
&F_p\left((\ms{D}_X, F_\bullet) \otimes_{(V^{\bs{0}}\ms{D}_X, F_\bullet)} \left(\frac{(V^{\bs{\alpha}}_*\mc{M}(*D))[\bs{s}]\bs{f^s}}{((s_1 + \alpha_1)^n, \ldots, (s_r + \alpha_r)^n)}, F_\bullet^*\right)\right) \\
& \qquad \qquad \qquad \qquad = F_p\left((\ms{D}_X, F_\bullet) \otimes_{(V^{\bs{0}}\ms{D}_X, F_\bullet)} (V^{\bs{\alpha}}_*\mc{M}(*D))[\bs{s}]\bs{f^s}, F_\bullet^*)\right).
\end{align*}
Taking the union over all $p$, \eqref{eq:snc magic formula 3} therefore yields
\[ \bigcup_p F_p \mc{M}[[\bs{s} + \bs{\alpha}]]\bs{f^s}(*D) = \bigcup_p \varprojlim_{n} F_p\cM_n=\ms{D}_X \otimes_{V^{\bs{0}}\ms{D}_X} (V^{\bs{\alpha}}_*\mc{M}(*D))[\bs{s}]\bs{f^s}.\]
By Proposition \ref{prop:V-filtration on graph}, the inclusion into $\iota_+\mc{M}(*D)$ identifies the right hand side with $V^{\bs{\alpha}}_*\iota_+\mc{M}(*D)$, so we are done.
\end{proof}

\begin{proof}[Proof of Theorem \ref{thm:multivariate magic formula}]
To complete the proof, we just need to show that Lemma \ref{lem:snc magic formula} holds in the general case where $f_1, \ldots, f_r$ are arbitrary holomorphic functions. We reduce this to the simple normal crossings case using an additional graph embedding. Let us write $X' = X \times \mb{C}^r$ and $i \colon X \to X'$ for the graph embedding (we use this notation to distinguish from the graph embedding $\iota \colon X \to X \times \mb{C}^r$ appearing in the Malgrange-Mellin transform). We also write $x_1, \ldots, x_r$ for the coordinate functions on $\mb{C}^r$ (as a factor of $X'$), $D' = \{x_1 \cdots x_r = 0\}$ and $j' \colon X' \setminus D' \to X'$ for the inclusion. Then, as in \S\ref{subsec:direct image}, we have the direct image functors, both denoted $i_+$, from $\ms{D}_X$-modules to $\ms{D}_{X'}$-modules, and from $V^{\bs{0}}\ms{D}_{X \times \mb{C}^r}$-modules to $V^{\bs{0}}\ms{D}_{X' \times \mb{C}^r}$-modules. These two functors are fully faithful and compatible with the Malgrange-Mellin transform after localising at $D$, so it suffices to show that
\begin{equation} \label{eq:multivariate magic formula 1}
i_+\left(\bigcup_p F_pj_*\mc{M}[[\bs{s} + \bs{\alpha}]]\bs{f^s}\right) = i_+(V^{\bs{\alpha}}_*\iota_+\mc{M}(*D))
\end{equation}
as subsheaves of their common localisation at $D$.

To prove \eqref{eq:multivariate magic formula 1}, observe that, for fixed $p$ and $n \gg 0$, the kernel of the quotient map $\bigcup_q F_qj_*\mc{M}[[\bs{s} + \bs{\alpha}]]\bs{f^s} \to \mc{M}_n$ has $F_p \mc{K}_n = 0$. So
\begin{equation} \label{eq:multivariate magic formula 4}
F_pi_+\left(\bigcup_q F_q j_*\mc{M}[[\bs{s} + \bs{\alpha}]]\bs{f^s}\right) = F_pi_+\mc{M}_n.
\end{equation}
Since the functor $i_+$ for filtered $\ms{D}_X$-modules is compatible with the functor $i_* \colon \mhm(X) \to \mhm(X')$, we have
\[ i_+\mc{M}_n = i_*j_*\left(\frac{\mc{M}[\bs{s}]\bs{f^s}}{((s_1 + \alpha_1)^n, \ldots, (s_r + \alpha_r)^n))}\right) = j'_*\left(\frac{(i'_*\mc{M})[\bs{s}]\bs{x^s}}{((s_1 + \alpha_1)^n, \ldots, (s_r + \alpha_r)^n)}\right),\]
as mixed Hodge modules and hence as filtered $\ms{D}_{X'}$-modules, where $i' \colon X \setminus D \to X' \setminus D'$ is the restriction of $i$.  Applying \eqref{eq:multivariate magic formula 4} and the analogous claim for $i'_+\mc{M}$ and taking the union of the filtered pieces, we deduce that
\begin{equation} \label{eq:multivariate magic formula 2}
i_+\left(\bigcup_p F_p j_*\mc{M}[[\bs{s} + \bs{\alpha}]]\bs{f^s}\right) = \bigcup_p F_p j'_*(i'_*\mc{M})[[\bs{s} + \bs{\alpha}]]\bs{x^s}.
\end{equation}
Similarly, by Theorem \ref{thm:multivariate V direct image} applied to the proper morphism $i$, we have
\begin{equation} \label{eq:multivariate magic formula 3}
i_+(V^{\bs{\alpha}}_*\iota_+\mc{M}(*D)) = V^{\bs{\alpha}}_*\iota_+i_+\mc{M}(*D').
\end{equation}
The right hand sides of \eqref{eq:multivariate magic formula 2} and \eqref{eq:multivariate magic formula 3} coincide by Lemma \ref{lem:snc magic formula}, so \eqref{eq:multivariate magic formula 1} holds. 
This completes the proof of the theorem.\end{proof}

\begin{rmk} \label{rmk:! magic formula}
Similarly to Theorem \ref{thm:multivariate magic formula} we also have the dual formula
\begin{equation} \label{eq:! magic formula}
 \left(\bigcup_p F_p j_!\mc{M}[[\bs{s} + \bs{\alpha}]]\bs{f^s}, F_\bullet\right) \cong (V^{\bs{\alpha} + \epsilon\bs{1}}_*\iota_+\mc{M}(*D), F_{\bullet + r}^*),
\end{equation}
where $F_p j_!\mc{M}[[\bs{s} + \bs{\alpha}]]\bs{f^s}$ is defined analogously to the $j_*$ version as in \eqref{eqn: Fp of multivariate magic formula}. Indeed, setting $f = f_1 \cdots f_r$, we have for all $\alpha \in \mb{R}$ and all $n$ a natural isomorphism of pro-mixed Hodge modules
\[ \mc{M}[\bs{s} + \bs{\alpha}, s + \alpha]]\bs{f^s}f^s \cong \mc{M}[[\bs{s} + \bs{\alpha} + \alpha\bs{1}]]\bs{f^s}[[s + \alpha]]\]
sending $s$ to $s$ and $s_i$ to $s_i - s$. Arguing as in the proof of \cite[Theorem 2.11]{DMB}, we deduce that
\[ F_p j_!\mc{M}[[\bs{s} + \bs{\alpha}]]\bs{f^s} = F_p j_*\mc{M}[[\bs{s} + \bs{\alpha} + \epsilon\bs{1}]]\bs{f^s}.\]
So \eqref{eq:! magic formula} follows from Theorem \ref{thm:multivariate magic formula}.
\end{rmk}

\subsection{Properties of the Hodge and $V$-filtrations} \label{subsec:properties of Hodge and V}
In this section, we use Theorem \ref{thm:multivariate magic formula} to deduce Hodge-filtered refinements of several results in the previous sections.

Let $f_1, \ldots, f_r$ be holomorphic functions on a complex manifold $X$ and let $\iota\colon X\to X\times \C^r$ be the associated graph embedding. Let $D=\textrm{div}(f_1\cdots f_r)$ and $U\colonequals X \setminus D$, with open embedding $j\colon U\to X$. We fix a mixed Hodge module $\mc{M}$ on $X$ such that $\mc{M} = \mc{M}[(f_1 \cdots f_r)^{-1}]$ as a $\ms{D}$-module (or, equivalently, $\mc{M} = j_*j^*\mc{M}$).

We begin with a filtered version of Theorem \ref{thm:flatness}.

\begin{thm} \label{thm:hodge flatness}
For any $\bs{\alpha} \in \mb{R}^{r}$, the stalks of $\gr^{F^*}V_*^{\bs{\alpha}}\mc{M}$ and $V_*^{\bs{\alpha}}\mc{M}$ are free over $\C[\bs{s}]$. Moreover, if $\bs{\alpha}\leq \bs{\beta}$ are separated by a single wall $W \subset \mb{R}^r$, then the stalks of $\gr^{F^*}(V_{\ast}^{\bs{\alpha}}\iota_{+}\mc{M}/V_{\ast}^{\bs{\beta}}\iota_{+}\mc{M})$ are free relative to $W$.
\end{thm}

Here we have used the following terminology, by analogy with Definition \ref{defn:relative flatness}.

\begin{defn} \label{defn:relative freeness}
Let $M$ be a $\mb{C}[\bs{s}]$-module and $W = \bs{L}^{-1}(\gamma) \subset \mb{R}^r$. We say that $M$ is \emph{free relative to $W$} if $M$ is free over $\mb{C}[U]$ for any linear subspace $U \subset \mrm{span}\{s_1, \ldots, s_r\}$ not containing $\bs{L}(\bs{s})$.
\end{defn}

The proof of Theorem \ref{thm:hodge flatness} uses the following basic fact about graded modules over polynomial rings.

\begin{lemma}\label{lemma: flatness over polynomial ring}
Let $M$ be a bounded below graded $\C[\bs{s}]$-module. Then the following statements are equivalent.
\begin{enumerate}
\item $M$ is a free $\C[\bs{s}]$-module. 
\item $M$ is a flat $\C[\bs{s}]$-module.
\item $\mrm{Tor}^{\C[\bs{s}]}_i(M,\mb{C})=0$ for $i>0$, where we regard $\mb{C}$ as a $\C[\bs{s}]$-module by letting the $s_i$ act by zero.
\end{enumerate}
\end{lemma}
\begin{proof}
It suffices to prove $(3)\Rightarrow (1)$. Choose a homogeneous $\C$-basis of the graded $\C[\bs{s}]$-module $M\otimes_{\C[\bs{s}]}\C$ and lift this along the surjective map $M\twoheadrightarrow M\otimes_{\C[\bs{s}]}\mb{C}$ to a set of homogeneous elements $\{m_{\gamma}\}_{\gamma\in I}$ of $M$. We claim that $\{m_{\gamma}\}_{\gamma\in I}$ are free generators of $M$. Consider the natural morphism of graded $\C[\bs{s}]$-modules $\pi \colon \bigoplus_{\gamma\in I}\C[\bs{s}]\cdot m_{\gamma}\to M$. By our choice, we have $\mrm{im}(\pi)\otimes_{\C[\bs{s}]}\mb{C}=M\otimes_{\C[\bs{s}]}\mb{C}$. Using the Tor exact sequences, the condition (3) implies that
\[ \ker(\pi)\otimes_{\C[\bs{s}]} \mb{C}=0=\mrm{coker}(\pi)\otimes_{\C[\bs{s}]}\mb{C}.\]
Since the grading on $M$ is bounded below, the same holds for $\ker(\pi)$ and $\mrm{coker}(\pi)$. But for any bounded below graded $\mb{C}[\bs{s}]$-module $N$, if $N\otimes_{\C[\bs{s}]}\mb{C}=0$, then $N=0$. So $\ker(\pi)=0=\mrm{coker}(\pi)$ and so $\pi$ is an isomorphism. In particular, $M$ is a free $\C[\bs{s}]$-module, i.e.\ $(1)$ holds.
\end{proof}

We will also make repeated use of the following elementary fact.

\begin{lemma} \label{lem:filtered koszul}
Let $(M, F_\bullet) \to (N, F_\bullet)$ be a morphism of complexes of filtered $\mb{C}[\bs{s}]$-modules. If $F_{p'} M \to F_{p'}N$ is an isomorphism for all $p' \leq p$, then
\[ F_{p'}(\mb{C}\overset{\mrm{L}}\otimes_{\mb{C}[\bs{s}]} M) \to F_{p'}(\mb{C} \overset{\mrm{L}}\otimes_{\mb{C}[\bs{s}]} N) \]
is an isomorphism for all $p' \leq p$.
\end{lemma}
\begin{proof}[Proof of Theorem \ref{thm:hodge flatness}]
Observe that for a bounded below exhaustively filtered $\mb{C}[\bs{s}]$-module $(M, F_\bullet)$ (where $F_\bullet M$ is compatible with the order filtration on $\mb{C}[\bs{s}]$), if $\gr^FM$ is free over $\mb{C}[\bs{s}] = \gr^F \mb{C}[\bs{s}]$ then $M$ is also free over $\mb{C}[\bs{s}]$. Indeed, if we lift free homogeneous generators $m_i \in \gr^F_{p_i} M$ to elements $\tilde{m}_i$ in $F_{p_i}M$, then the map $\bigoplus_i \mb{C}[\bs{s}]\tilde{m}_i \to M$ induces an isomorphism on $\gr^F$ and is hence itself an isomorphism. So to prove the theorem, it is enough to prove the statements for $\gr^{F^*}$. 

We first prove that the stalks of $\gr^{F^*} V^{\bs{\alpha}}_*\iota_+\mc{M}$ are free over $\mb{C}[\bs{s}]$. By Lemma \ref{lemma: flatness over polynomial ring}, it suffices to show that 
\begin{equation}\label{eq:hodge flatness 1}
\mrm{H}^i(\mb{C} \overset{\mrm{L}}\otimes_{\mb{C}[\bs{s}]}\gr^{F^*}V_*^{\bs{\alpha}}\iota_+\mc{M}) = 0 \quad \text{for every $i<0$},\end{equation}
where $s_i$ acts on $\mb{C}$ by zero. We will show this using Theorem \ref{thm:multivariate magic formula}.

Consider the morphism $j^*\mc{M}[\bs{s}]\bs{f^s} \to \mc{M}_n$, where, to simplify the notation, we write
\[ \mc{M}_n\colonequals \frac{j^{\ast}\cM[\bs{s}]\bs{f^s}}{\left((s_1+\alpha_1)^n,\ldots (s_r+\alpha_r)^n\right)}.\]
By Lemma \ref{lem:filtered koszul}, for fixed $p$, we can choose $n$ large enough so that
\[ F_{p'}(\mb{C} \overset{\mrm{L}}\otimes_{\mb{C}[\bs{s}]} j^*\mc{M}[\bs{s}]\bs{f^s}) \to F_{p'}(\mb{C} \overset{\mrm{L}}\otimes_{\mb{C}[\bs{s}]} \mc{M}_n)\]
is an isomorphism for $p' \leq p$. Here in the ungraded setting we let $s_i$ act on $\mb{C}$ by $-\alpha_i$. Since the source has $\mrm{H}^i = 0$ for $i < 0$, we conclude that
\[ \mrm{H}^i(F_{p'}(\mb{C} \overset{\mrm{L}}\otimes_{\mb{C}[\bs{s}]} \mc{M}_n))  = 0 \quad \text{for $i < 0$ and $p' \leq p$}.\]
Observe that $\mb{C}\overset{\mrm{L}}\otimes_{\mb{C}[\bs{s}]} \mc{M}_n$ can be written as the Koszul complex for the morphisms $s_i + \alpha_i \colon \mc{M}_n(1) \to \mc{M}_n$; in particular, it underlies a complex of mixed Hodge modules. Taking the direct image we conclude by Remark \ref{rmk:hodge lower bound} that
\[ \mrm{H}^i(\gr^F_{p'}(\mb{C} \overset{\mrm{L}}\otimes_{\mb{C}[\bs{s}]} j_*\mc{M}_n)) = 0 \quad \text{for $i < 0$ and $p' \leq p$}.\]
Applying Lemma \ref{lem:filtered koszul} again, we conclude that
\begin{equation} \label{eq:hodge flatness 2}
 \mrm{H}^i\left(\gr^F\left(\mb{C} \overset{\mrm{L}}\otimes_{\mb{C}[\bs{s}]} \bigcup_p F_pj_*\mc{M}[[\bs{s} + \bs{\alpha}]]\bs{f^s}\right)\right) = 0 \quad \text{for $i < 0$}.
\end{equation}
But by Theorem \ref{thm:multivariate magic formula}, the left hand side of \eqref{eq:hodge flatness 2} agrees with the left hand side of \eqref{eq:hodge flatness 1}, so this proves \eqref{eq:hodge flatness 1}; here we have used the standard fact that if $(M,F_{\bullet})$ and $(N,F_{\bullet})$ are filtered modules over $(\C[\bs{s}],F_{\bullet})$, then there is an isomorphism
\[ \gr^F\left((M,F_{\bullet})\overset{\mrm{L}}{\otimes}_{(\C[\bs{s}],F_{\bullet})}(N,F_{\bullet})\right)\cong \gr^FM\overset{\mrm{L}}{\otimes}_{\gr^F\C[\bs{s}]}\gr^FN\]
in the derived category of graded $\gr^F\C[\bs{s}]$-modules.

To prove that $\gr^{F^*}(V_*^{\bs{\alpha}}\iota_+\mc{M}/V_*^{\bs{\beta}}\iota_+\mc{M})$ is free relative to the wall, let us assume for simplicity that $\bs{\alpha}$ lies on the separating wall so that $V_*^{\bs{\beta}}\iota_+\mc{M} = V_*^{\bs{\alpha} + \epsilon \bs{1}}\iota_+\mc{M}$. Fix a linear subspace $U \subset \mrm{span}\{s_1, \ldots, s_r\}$. Observing that every element of $U$ defines a morphism of complex mixed Hodge modules $\mc{M}_n(1) \to \mc{M}_n$, we deduce that the derived tensor product
\[ \mb{C} \overset{\mrm{L}}\otimes_{\mb{C}[U]} \mc{M}_n \]
again underlies a complex of mixed Hodge modules. Arguing as above using Theorem \ref{thm:multivariate magic formula} and the $!$-variant of Remark \ref{rmk:! magic formula}, we deduce that, for fixed $p$, there exists $n$ such that
\begin{equation} \label{eq:hodge flatness 3}
F_{p' + r}^*(\mb{C}\overset{\mrm{L}}\otimes_{\mb{C}[U]} V^{\bs{\alpha}}_*\iota_+\mc{M}/V^{\bs{\alpha} + \epsilon \bs{1}}_*\iota_+\mc{M}) = F_{p'}(\mb{C} \overset{\mrm{L}}\otimes_{\mb{C}[U]} \mrm{Cone}(j_!\mc{M}_n \to j_*\mc{M}_n))
\end{equation}
for $p' \leq p$. Since the right hand side of \eqref{eq:hodge flatness 3} is a complex of mixed Hodge modules, it is strict. So the filtered complex $\mb{C} \overset{\mrm{L}}\otimes_{\mb{C}[U]} V^{\bs{\alpha}}\iota_+\mc{M}/V^{\bs{\alpha} + \epsilon \bs{1}}\iota_+\mc{M}$ is strict as well. But the underlying complex has vanishing $\mrm{H}^i$ for $i < 0$ by Theorem \ref{thm:flatness}, so we conclude
\[ \mrm{H}^i(\gr^{F^*}(V_*^{\bs{\alpha}}\iota_+\mc{M}/V_*^{\bs{\beta}}\iota_+\mc{M})) = 0 \quad \text{for $i < 0$}.\]
We deduce the theorem by applying Lemma \ref{lemma: flatness over polynomial ring}.
\end{proof}

\begin{remark}
In the proof of Theorem \ref{thm:hodge flatness}, we used the fact that any element of $\mb{C}\text{-span}\{s_1, \ldots, s_r\}$ defines a morphism of complex mixed Hodge modules $\mc{E}_n(1) \to \mc{E}_n$. While this is true on the nose in the complex category, it is only true for elements of $\bar{\mb{Q}}\text{-span}\{s_1, \ldots, s_r\}$ in $\mhm_{\bar{\mb{Q}}}$. For the above proof to work in Saito's theory, we must therefore restrict the subspaces $U$ in Definition \ref{defn:relative freeness} to those defined over $\bar{\mb{Q}}$.
\end{remark}

The next result is the filtered version of Corollary \ref{cor: change of functions}.

\begin{thm} \label{thm:filtered change of functions}
Suppose that $f_j' = \prod_{i = 1}^r f_i^{d_{ij}}$, $j = 1, \ldots, r'$, for some $d_{ij} \in \mb{Z}_{\geq 0}$ and assume that the zero loci of $f_1 \cdots f_r$ and $f_1' \cdots f_{r'}'$ coincide. Then for any $\bs{\alpha}'\in \R^{r'}$, setting $\bs{\alpha} = (\sum_j d_{1j} \alpha_j', \ldots, \sum_j d_{rj}\alpha_{j}')\in \R^r$, we have
\[\gr^{F^*}V_*^{\bs{\alpha}}\iota_+\mc{M}(-r)\otimes_{\mb{C}[\bs{s}]} \mb{C}[\bs{s}']= \gr^{F^*}V_*^{\bs{\alpha}'}\iota_+'\mc{M}(-r'),\]
where $\iota' \colon X \to X \times \mb{C}^{r'}$ is the graph embedding of the $f_j'$ and the tensor product are taken over $s_i \mapsto \sum_j d_{ij}s_j'$.
\end{thm}

\begin{proof}
Since $\gr^{F^*}V_*^{\bs{\alpha}}\iota_+\mc{M}$ is flat over $\mb{C}[\bs{s}]$, we may replace the tensor product with the derived tensor product in the theorem. Since
\[ V_*^{\bs{\alpha}'}{\iota'}_+\mc{M} = V^{\bs{\alpha}}_*\iota_+\mc{M}\overset{\mrm{L}}\otimes_{\mb{C}[\bs{s}]}\mb{C}[\bs{s}'] \]
by Corollary \ref{cor: change of functions}, it suffices to show that the morphism
\[ (V_*^{\bs{\alpha}}{\iota}_+\mc{M}, F_{\bullet + r}^*) \overset{\mrm{L}}\otimes_{(\mb{C}[\bs{s}], F_\bullet)} (\mb{C}[\bs{s}'], F_\bullet) \to (\iota'_+\mc{M}, F_{\bullet + r'}^*) \]
is strict. As in the proof of Corollary \ref{cor: change of functions}, this reduces to showing that
\[ (V_*^{\bs{\beta}}{\iota}_+\mc{M}, F_{\bullet + r}^*) \overset{\mrm{L}}\otimes_{(\mb{C}[\bs{s}], F_\bullet)} (\mb{C}[\bs{s}'], F_\bullet) \to (V_*^{\bs{\alpha}}{\iota}_+\mc{M}, F_{\bullet + r}^*) \overset{\mrm{L}}\otimes_{(\mb{C}[\bs{s}], F_\bullet)} (\mb{C}[\bs{s}'], F_\bullet)\]
is strict injective whenever $\bs{\alpha} \leq \bs{\beta}$ are separated by a single wall. Since $\gr$ commutes with derived tensor products, this is equivalent to showing that
\[ \gr^{F^*}V_*^{\bs{\beta}}{\iota}_+\mc{M} \overset{\mrm{L}}\otimes_{\mb{C}[\bs{s}]}\mb{C}[\bs{s}'] \to \gr^{F^*}V_*^{\bs{\alpha}}{\iota}_+\mc{M} \overset{\mrm{L}}\otimes_{\mb{C}[\bs{s}]}\mb{C}[\bs{s}']\]
is injective, which follows from the flatness relative to the wall of $\gr^{F^*}(V_*^{\bs{\alpha}}\iota_+\mc{M}/V_*^{\bs{\beta}}\iota_+\mc{M})$ as in the proof of Corollary \ref{cor: change of functions}.
\end{proof}

Finally, we prove a filtered version of Theorem \ref{thm:multivariate V direct image}, which is the multivariate version of \cite[Proposition 3.3.17]{Saito88} (see also \cite[Proposition 5.1]{DY25}). Let $X,Y$ be complex manifolds and $\pi \colon Y \to X$ a projective morphism. Suppose that $\mc{M}$ is a mixed Hodge module on $Y$. Recall that the direct image of a mixed Hodge module $\mc{M}$ on $Y$ is given at the level of filtered $\ms{D}$-modules by 
\begin{equation} \label{eq:laumon}
(\pi_*\mc{M}, F_\bullet) \colonequals \mathbf{R}\pi_*\left((\ms{D}_{X \leftarrow Y}, F_\bullet) \overset{\mrm{L}}{\otimes}_{(\ms{D}_Y, F_\bullet)} (\mc{M}, F_\bullet)\right),
\end{equation}
where we give the bimodule $\ms{D}_{Y \leftarrow X}$ the filtration by order of differential operator starting in degree $\dim X - \dim Y$. Let $f_1,\ldots,f_r$ be holomorphic functions on $X$ and write $\iota$ for the graph embeddings of both $f_1,\ldots,f_r$ and $f_1\circ \pi,\ldots,f_r\circ\pi$.

In the statement below, we say that an $\mb{R}^r \times \mb{Z}$-filtered complex $(M, F_\bullet V^\bullet M)$ is \emph{bistrict} if each $(V^{\bs{\alpha}}M, F_\bullet)$ is a strict complex with respect to $F_\bullet$ and the natural maps
\[ \mrm{H}^i(V^{\bs{\beta}}M, F_\bullet) \to \mrm{H}^i(V^{\bs{\alpha}}M, F_\bullet)\]
are strict injections. Note that this does not impose any strictness condition on the interaction of the $r$ different $\mb{R}$-indices.

\begin{thm} \label{thm:multistrict pushforward}
Suppose $f_i\circ \pi$ acts on $\cM$ bijectively for each $i$. Then the $\mb{R}^r \times \mb{Z}$-filtered complex
\[ \pi_+(V^{\bullet}_*\iota_{+}\cM,F_{\bullet}^*)\colonequals \mathbf{R}\pi_*\left((\ms{D}_{X \leftarrow Y}, F_\bullet) \overset{\mrm{L}}\otimes_{(\ms{D}_Y, F_\bullet)} (V^\bullet_*\iota_{+}\mc{M}, F_\bullet^*)\right) \]
is bistrict with respect to the filtrations $F_\bullet^*$ and $V^\bullet_*$. In particular,
\[ (V^{\bs{\alpha}}_*\iota_+\mc{H}^i(\pi_*\mc{M}), F_\bullet^*) = \mc{H}^i\pi_{\ast}(V^{\bs{\alpha}}_*\iota_{+}\cM,F_{\bullet}^*), \quad \text{for all $i\in \Z$ and all $\bs{\alpha} \in \mb{R}^r$}.\]
\end{thm}
\begin{proof}
Let us write
\[ \mc{M}_{n, \bs{\alpha}} = \frac{j^{\ast}\cM[\bs{s}]\bs{f^s}}{\left((s_1+\alpha_1)^n,\ldots (s_r+\alpha_r)^n\right)}.\]
By Theorem \ref{thm:multivariate magic formula}, for each $p$ there exists $n$ such that
\[ F_{p' + r} ^*V_*^{\bs{\alpha}}\iota_+\mc{M} = F_{p'}j_*\mc{M}_n \quad \text{for all $p' \leq p$ and all $\bs{\alpha}$}\]
and hence
\begin{equation} \label{eq:multistrict pushforward 1}
F_{p' + r}^*\pi_+V_*^{\bs{\alpha}}\iota_+\mc{M} = F_{p'}\pi_*j_*\mc{M}_n \quad \text{for all $p' \leq p + \dim Y - \dim X$ and all $\bs{\alpha}$}.
\end{equation}
Since the right hand side is the direct image of a mixed Hodge module, it is a strict complex by Saito's direct image theorem. So $\pi_+V_*^{\bs{\alpha}}\iota_+\mc{M}$ is strict as well. Similarly, to prove strict injectivity of $\mc{H}^i\pi_+V^{\bs{\beta}}_*\iota_+\mc{M} \to \mc{H}^i\pi_+V^{\bs{\alpha}}_*\iota_+\mc{M}$ we may assume without loss of generality that $\bs{\beta} = \bs{\alpha} + \epsilon \bs{1}$. Then arguing as above using Remark \ref{rmk:! magic formula}, we have
\begin{equation} \label{eq:multistrict pushforward 2}
 F_{p' + r}^*\pi_+V_*^{\bs{\alpha} + \epsilon \bs{1}}\iota_+\mc{M} = F_{p'}\pi_*j_!\mc{M}_n \quad \text{for all $p' \leq p + \dim Y - \dim X$ and all $\bs{\alpha}$}.
\end{equation}
Since
\[ \mc{H}^i(\pi_*j_!\mc{M}_n) \to \mc{H}^i(\pi_*j_*\mc{M}_n) \]
underlies a morphism of mixed Hodge modules, it is strict with respect to $F_\bullet$. So \eqref{eq:multistrict pushforward 1} and \eqref{eq:multistrict pushforward 2} imply that $\mc{H}^i\pi_+V^{\bs{\beta}}_*\iota_+\mc{M} \to \mc{H}^i\pi_+V^{\bs{\alpha}}_*\iota_+\mc{M}$ is strict with respect to $F_\bullet^*$. Since it is also injective by Theorem \ref{thm:multivariate V direct image}, we are done.
\end{proof}

\subsection{Connection to mixed multiplier ideals and Bernstein-Sato ideals} 

We conclude this section with a brief discussion of the relationship between the multivariate $V$-filtration and more classical invariants.

Let $X$ be a complex manifold, $f_1,\cdots,f_r$ be holomorphic functions on $X$, $D=\textrm{div}(f_1\cdots f_r)$, and $\iota$ be the graph embedding. We first reinterpret the mixed multiplier ideals of $f_i$ in terms of the lowest piece in the Hodge filtration on $V^{\bs{\alpha}}\iota_{+}\cO_X$. This is a multivariate generalisation of a result of Budur-Saito \cite[Theorem 0.1]{BS05}. 
\begin{corollary} \label{cor:log canonical}
 For $\bs{\alpha}\in \R^r_{> 0}$ and $0<\epsilon\ll 1$, one has an equality of ideals 
\[ \mc{J}\left(X,\sum_{i = 1}^r(\alpha_i - \epsilon) \mrm{div}(f_i)\right)= F_r V^{\bs{\alpha}}\iota_{+}\cO_X, \quad h\mapsto h\bs{f}^{\bs{s}}.\]
Here we regard $F_rV^{\bs{\alpha}}\iota_+\mc{O}_X$ as an ideal via the inclusion $F_rV^{\bs{\alpha}}\iota_+\mc{O}_X \subset F_r\iota_+\mc{O}_X = \mc{O}_X$. In particular, $\bs{f}^{\bs{s}} \in V^{\bs{\alpha}}\iota_{+}\cO_X$ if and only if the pair $(X, \sum_{i = 1}^r\alpha_i \mrm{div}(f_i))$ is log canonical. 
\end{corollary}
\begin{proof}
Since  $\cO_X \in \MHM_{\Q}(X)$, the multivariate $V$-filtration $V^{\bullet}\iota_{+}\cO_X$ is defined over $\Q$ by Corollary \ref{cor: MHM fully A specializable}. Thus by the semi-continuity of both sides with respect to $\bs{\alpha}$, we can assume $\bs{\alpha}\in \Q^r_{>0}$. Choose an integer $N\gg0$ such that $(d_1,\ldots,d_r)\colonequals N\bs{\alpha}\in \Z^r$. Applying Theorem \ref{thm:filtered change of functions} to $\cO_X(\ast D)$, $g=f_1^{d_1}\cdots f_r^{d_r}$ and $\alpha'=1/N$, and using Corollary \ref{cor:localising V} and Lemma \ref{lem:hodge * formula}, we obtain
\[ \gr^FV^{\bs{\alpha}}\iota_{+}\cO_X(-r)\otimes_{\C[\bs{s}]}\C[s]=\gr^FV^{1/N}\iota_{g,+}\cO_X(-1),\]
where $\iota_g \colon X\to X\times \C$ is the graph embedding of $g$ and the tensor product is taken over $s_i\mapsto d_is$; note that this is an equality of ideals by construction. Evaluating at the lowest nonzero Hodge level gives
\[ F_rV^{\bs{\alpha}}\iota_{+}\cO_X \cong F_1V^{1/N}\iota_{g,+}\cO_X,\]
where $h\bs{f}^{\bs{s}}$ is mapped to $hg^s$. We finish the proof by using \cite[Theorem 0.1]{BS05}:
\[ \mc{J}\left(X,\sum_{i = 1}^r(\alpha_i - \epsilon) \mrm{div}(f_i)\right)= \cJ\left(X,(1/N-\epsilon)\textrm{div}(g)\right)= F_1V^{1/N}\iota_{g,+}\cO_X, \quad h\mapsto hg^s.\qedhere\]
\end{proof}

We also have the following relation between the multivariate $V$-filtration and the Bernstein-Sato ideal, generalising the case $r = 1$.
Denote by
\[ \textrm{Exp}\colon \C^r \to (\mb{C}^{\times})^r, \quad (\alpha_1,\ldots,\alpha_r)\mapsto (e^{2\pi i\alpha_1},\ldots, e^{2\pi i\alpha_r}).\]
\begin{prop}\label{prop: Bernsteinideal and Vfiltration}
We have
\begin{equation} \label{eq:bernstein vs V 1}
\mathrm{Exp}(\mathrm{Zero}(B_F))=\mathrm{Exp}\overline{\left(\{-\bs{\alpha}\mid \bs{\alpha}\in \R^r_{\geq 0}, \, V^{\bs{\alpha}}\iota_{+}\cO_X(\ast D)\neq V^{\bs{\alpha}+\epsilon\bs{1}}\iota_{+}\cO_X(\ast D)\}\right)},
\end{equation}
where $\overline{(-)}$ denotes the Zariski closure in $\mb{C}^r$.
\end{prop}
\begin{proof}
By Corollary \ref{cor:extending V*}, we can replace the right hand side of \eqref{eq:bernstein vs V 1} with $\mrm{Exp}$ of the Zariski closure of the jumping values of $V_*$. Since $V_*^\bullet \mc{O}_X(*D)$ is defined over $\mb{Q}$, it follows from the definition and Theorem \ref{thm:flatness} that this Zariski closure is the zero locus of the annihilator in $\mb{C}[\bs{s}]$ of
\begin{equation} \label{eq:bernstein vs V 3}
 V^{\bs{\alpha}}_*\iota_+\mc{O}_X(*D)/V^{\bs{\alpha} + \bs{1}}_*\iota_+\mc{O}_X(*D) = V^{\bs{\alpha}}_*\iota_+\mc{O}_X(*D)/\bs{t}V^{\bs{\alpha}}_*\iota_+\mc{O}_X(*D),
\end{equation}
where $\bs{t} = t_1 \cdots t_r$. Now, observe that
\[ V^{\bs{\alpha}}_*\iota_+\mc{O}_X(*D) \quad \text{and} \quad \ms{D}_X[\bs{s}]\bs{f^s} \subset \iota_+\mc{O}_X(*D) \]
are $\ms{D}_X[\bs{s}]$-coherent subsheaves that become equal after inverting $\bs{t}$. So \eqref{eq:bernstein vs V 3} is an extension of finitely many sheaves of the form
\begin{equation} \label{eq:bernstein vs V 2}
 \frac{V^{\bs{\alpha}}_*\iota_+\mc{O}_X(*D) \cap \bs{t}^{\bs{n}}\ms{D}_X[\bs{s}]\bs{f^s}}{V^{\bs{\alpha}}_*\iota_+\mc{O}_X(*D) \cap \bs{t}^{\bs{n} + \bs{1}}\ms{D}_X[\bs{s}]\bs{f^s} + \bs{t}V^{\bs{\alpha}}_*\iota_+\mc{O}_X(*D) \cap \bs{t^n}\ms{D}_X[\bs{s}]\bs{f^s}}.
 \end{equation}
Since
\[ B_F = \mrm{Ann}_{\mb{C}[\bs{s}]} \ms{D}_X[\bs{s}]\bs{f^s}/\bs{t}\ms{D}_X[\bs{s}]\bs{f^s}\]
by definition, the terms \eqref{eq:bernstein vs V 2} are all annihilated by some integer translate of $B_F$. So \eqref{eq:bernstein vs V 3} is annihilated by a product of such ideals, proving the inclusion $\supset$ in \eqref{eq:bernstein vs V 1}. Interchanging the roles of $V_*^{\bs{\alpha}}\iota_+\mc{O}_X(*D)$ and $\ms{D}_X[\bs{s}]\bs{f^s}$ gives the reverse inclusion, so we are done.
\end{proof}

\section{The Strong Monodromy Conjecture for hyperplane arrangements}\label{sec: strong monodromy conjecture}

In this section, as a demonstration of the general theory developed in \S\S\ref{sec: multivariate Vfiltration}--\ref{sec: Hodge and multivariate V}, we prove the Strong Monodromy Conjecture for hyperplane arrangements and its multivariable version due to Budur.

\subsection{Multivariate $V$-filtration of hyperplane arrangements}\label{sec: multivariate V hyp arr}

In this subsection, we establish a structural lemma about the multivariate $V$-filtration of the graph embedding of hyperplane arrangements that we will use in the proof

Assume that $f=\prod_{i=1}^r f_i^{d_i}\in \C[x_1,\ldots,x_n]$ defines a hyperplane arrangement in  $X=\C^n$, where $f_i$ are linear and $d_i\in \Z_{>0}$. Set $D\colonequals f^{-1}(0)=\sum_{i=1}^r d_iD_i$, where $D_i=\mathrm{div}(f_i)$ are the irreducible components of $D$.

We recall some terminologies and basic facts following \cite[\S 2]{BMT}. A hyperplane arrangement  is \emph{central} if each $f_i$ is homogeneous. A linear subspace $W\subset \C^n$ is called an \emph{edge}, if it is an intersection of some hyperplanes $D_i$. A central hyperplane arrangement is \emph{essential} if $\{0\}$ is an edge, and is \emph{indecomposable} if there is no linear change of coordinates on $\C^n$ such that $f$ can be written as the product of two non-constant polynomials in disjoint sets of variables. An edge $W$ is \emph{dense} if the induced hyperplane arrangement  $\textrm{im}\{\cup_{D_i\supseteq W} D_i \to \C^n/W\}$ is essential and indecomposable. We assume throughout this section that
\[ \textrm{$D$ is central, essential, and indecomposable}.\]
Under this assumption, let $\iota \colon X\to X\times \C^r$ be the graph embedding of $f_1,\ldots,f_r$, and let $\mc{M} = \iota_{+} \mc{O}_{X}$ be the direct image $\sD$-module. Consider the multivariate $V$-filtration $V^{\bullet}\cM$, with respect to the coordinates on $\C^r$. Via the $\sD_X[\bs{s}]$-module isomorphism
\[ \iota_{+}\cO_X(*D)\cong \cO_X(*D)[\bs{s}]\bs{f^s}\]
from \eqref{eqn: Malgrange isomorphism}, we will identify $\cM$ as a sub-$\sD_X[\bs{s}]$-module of $\cO_X(*D)[\bs{s}]\bs{f^s}$. For $\bs{\alpha} \in \mb{R}^r$, since $V^{\bs{\alpha}}\mc{M}$ is a $\C[\bs{s}]$-module, we can define an ideal $I^{\bs{\alpha}} \subset \mb{C}[\bs{s}]$ by
\begin{equation}\label{eqn: definition of Ialpha} \mb{C}[\bs{s}]\bs{f^s} \cap V^{\bs{\alpha}}\mc{M} = I^{\bs{\alpha}}\bs{f^s}.\end{equation}
In other words, $I^{\bs{\alpha}}=\{p(\bs{s})\in \C[\bs{s}] \mid p(\bs{s)}\bs{f^s}\in V^{\bs{\alpha}}\cM\}$. The main result of this section is:
\begin{lemma} \label{lem: n/d reduction}
If $\bs{\alpha}\in \R^r_{>0}$ satisfies $\sum_{i=1}^r \alpha_i =n$, then for any $0<\epsilon \ll 1$, the image of the natural map
\begin{equation}\label{eqn: map in n/d reduction}
I^{\bs{\alpha}} \to \frac{V^{\bs{\alpha}}\mc{M}}{V^{\bs{\alpha} + \epsilon\bs{1}} \mc{M}} \otimes_{\mb{C}[\bs{s}]} \mb{C}[s],\quad p(\bs{s}) \mapsto [p(\bs{s})\bs{f^s}]\otimes 1,
\end{equation}
is nonzero, where the tensor product is taken via $s_i \mapsto \alpha_i s$.
\end{lemma}
\begin{remark}\label{remark: localized version of n/d reduction}
Note that if $\mf{m}=(s_1+\alpha_1,\ldots,s_r+\alpha_r)$, then the target of \eqref{eqn: map in n/d reduction} is annihilated by some power of $\mf{m}$: this follows immediately from the definition of $V^{\bs{\alpha}}\mc{M}$ and the fact that the multivariate $V$-filtration on $\mc{M}$ is defined over $\mb{Q}$. In particular, \eqref{eqn: map in n/d reduction} is also nonzero after localising at $\mf{m}$. We will use this fact below in \S \ref{subsec:multi SMC}.
\end{remark}

The rest of this subsection is devoted to the proof of Lemma \ref{lem: n/d reduction}.

The first key point is to control $V^{\bs{\bullet}}\cM$ via an explicit log resolution of $(X,D)$. Denote by $\cG_{\mrm{dense}}$ the set of dense edges of $D$. Let $\pi \colon \tilde{X} \to X$ be the iterated blow-up along the proper transform of the union of edges in $\cG_{\mrm{dense}}$ of dimension $d$ as $d$ ranges from $0$ to $n-1$. Note that for a fixed dimension, the blow up of the union is the same as the blow up the proper transform of each dense edge of the same dimension successively. Since $D$ is central, it is known that (see e.g. \cite[\S 2]{BMT})  $\pi$ is a log resolution of $(X,D)$, and
\begin{equation}\label{eqn: resolution data for hyperplane arrangement} K_{\tilde{X}/X}=\sum_{W\in \cG_{\mrm{dense}}} (\textrm{codim}_{X}W-1) E_W, \quad \pi^{\ast}D_i=\sum_{W\in S} a_{iW} E_W,\end{equation}
where $E_W$ denotes the exceptional divisor over $W$, $a_{iW}=1$ if $W\subset D_i$ and is $0$ otherwise. In particular,
\[ \pi^{\ast}\left(\sum_{i=1}^r \alpha_iD_i\right)=\sum_{W\in \cG_{\mrm{dense}}}\left(\sum_{D_i\supseteq W} \alpha_i\right)E_W\]

We will prove Lemma \ref{lem: n/d reduction} by a wall-crossing argument, i.e.\ we will fix the tensor product $s_i\mapsto \alpha_i s$ but vary $\bs{\alpha}$ in other parts of \eqref{eqn: map in n/d reduction}, along the hyperplane $\{\bs{\gamma}\in \R^r\mid \sum_i\gamma_i=n\}$. We first prove the deformed statement for certain very special points where it becomes easy and deduce Lemma \ref{lem: n/d reduction} using general properties of $V^\bullet\mc{M}$.

The following language will be convenient.
\begin{definition}\label{dfn: good}
The \emph{log canonical polytope} is the set
\[ \left\{\bs{\beta} \in \mb{R}^r_{>0} \mid \sum_{D_i \supseteq W}\beta_i \leq \codim_X W \text{ for all dense edges $W$}\right\}.\]
We say that a vector $\bs{\beta}\in \Q^r_{>0}$ in the log canonical polytope is \emph{adapted to the edge $\{0\}$} if
\[ \sum_{i = 1}^r \beta_i = n \quad \text{and} \quad \sum_{D_i \supseteq W} \beta_i \not\in\mb{Z}\]
for any dense edge $W \neq \{0\}$.
\end{definition}
Note that, by \eqref{eqn: resolution data for hyperplane arrangement}, the log canonical polytope is precisely the set of vectors $\bs{\beta}$ such that the pair $(X, \sum_i \beta_i D_i)$ is log canonical. 

In the next two lemmas, we establish the existence of a vector adapted to $\{0\}$ (see also \cite[Lemmas 4.4, 4.5 and 4.6]{BSZ25}).

\begin{lemma}\label{lemma: basis JW}
    For any edge $W\neq \{0\}$, one has $\dim \mathrm{span}_{\C}\{ f_i \mid f_i|_W\neq 0\}>\dim W$.  
\end{lemma}
\begin{proof}
Consider the following vector spaces in $(\C^n)^{\ast}$: \begin{align*}
V_1=\textrm{span}_{\C}\{ f_i \mid f_i|_W=0\}, \quad V_2=\textrm{span}_{\C}\{ f_i \mid f_i|_W\neq 0\}.\end{align*}
Since $D$ is essential, $V_1+V_2 = \textrm{span}_{\C}\{f_i \mid 1\leq i\leq r\}=(\C^n)^{\ast}$. Hence
\[ \dim V_1 + \dim V_2=\dim (V_1+V_2) +\dim V_1\cap V_2=n+\dim V_1\cap V_2.\]
We claim that $V_1\cap V_2\neq \{0\}$. Otherwise, $V_1\oplus V_2 =(\C^n)^{\ast}$. Since $W \neq \{0\}$ we have $V_2\neq \{0\}$, so after a linear change of coordinates, we can write $f$ as the product of two linear functions with separate variables. This contradicts the indecomposability of $D$, so $\dim V_2 >n-\dim V_1 = \dim W$ as claimed.
\end{proof}

\begin{lemma}\label{lemma: existence of good vector}
There exists a vector $\bs{\beta}$ adapted to $\{0\}$.
\end{lemma}
\begin{proof}
Let us say a subset $J\subset \{1,\ldots,r\}$ is a \emph{basis} if $\{f_j \mid j\in J\}$ forms a basis of $(\C^n)^{\ast}$. For any basis $J$, we define a vector $\bs{\beta}^{J}$ by setting $\bs{\beta}^{J}_i$ to be $1$ if $i\in J$ and is $0$ otherwise. Clearly, the vector $\bs{\beta} = \bs{\beta}^J$ satisfies
\begin{equation}\label{eqn: weaker version of good vector} \sum_{D_i\supseteq W} \beta_i= \begin{cases}n, \quad &\textrm{if $W=\{0\}$},\\
\leq \textrm{codim}_{X}W, \quad &\textrm{if $W\neq \{0\}$}.\end{cases}\end{equation}
Let $(\lambda_J)\subset \Q_{>0}$ be a tuple of positive rational numbers, indexed by all bases $J$, such that $\sum_J \lambda_J=1$. Clearly, the vector $\sum_J \lambda_J \bs{\beta}^{J}$ still satisfies \eqref{eqn: weaker version of good vector}. The plan is to make a suitable choice of $(\lambda_J)$ such that the vector $\sum_J \lambda_J \bs{\beta}^{J}$ satisfies the desired property. To see that this is possible, fix an edge $W \neq \{0\}$ and consider the map
\begin{equation} \label{eq:good vector 1}
 \left\{ (\lambda_J \in \mb{R}_{>0})\, \left| \, \sum_J \lambda_J = 1\right.\right\} \to \mb{R}_{>0}; \quad (\lambda_J) \mapsto \sum_{D_i \supset W} \sum_J \lambda_J\beta^J_i.
\end{equation}
Since $W \neq \{0\}$, if $J_W$ is a basis extending a basis of $\textrm{span}_{\C}\{f_i \mid f_i|_W\neq 0\}$ then by Lemma \ref{lemma: basis JW}, we must have
\[ \sum_{D_i\supseteq W} \beta^{J_W}_i \leq n-\dim \textrm{span}_{\C}\{f_i \mid f_i|_W\neq 0\}<\codim_{X}W.\]
Conversely, if $J_W'$ is a basis extending a basis of $\mrm{span}_{\C}\{f_i \mid f_i|_W = 0\}$ then
\[ \sum_{D_i\supseteq W} \beta^{J_W'}_i  = \codim_{X}W.\]
So, varying $\lambda_{J_W}$ and $\lambda_{J_W'}$ while keeping their sum equal, we conclude that the map \eqref{eq:good vector 1} is open. So the set of rational points in the domain whose image is not an integer is dense. Taking a point $(\lambda_J)$ in the intersection over all $W$ of these sets, we conclude that $\bs{\beta} = \sum_J \lambda_J \bs{\beta}^J$ is a vector adapted to $\{0\}$.
\end{proof}

Now we establish the deformed version of Lemma \ref{lem: n/d reduction} for a vector $\bs{\beta}$ adapted to $\{0\}$.
\begin{lemma}\label{lemma: beta variant}
Suppose $\bs{\beta}$ adapted to $\{0\}$ and $\bs{\alpha}\in \R^r_{>0}$. Then $I^{\bs{\beta}}=\C[\bs{s}]$ and the image of the natural map 
\begin{equation}\label{eqn: map associated to beta}I^{\bs{\beta}}\to \frac{V^{\bs{\beta}}\cM}{V^{\bs{\beta}+\epsilon\bs{1}}\cM}\otimes_{\C[\bs{s}]}\C[s],\end{equation}
is nonzero, where the tensor product is taken over $s_i\mapsto \alpha_is$.
\end{lemma}
\begin{proof}
Since $\bs{\beta}$ is in the log canonical polytope by assumption, the pair $(X, \sum_i \beta_i D_i)$ is log canonical. So $\bs{f^s} \in V^{\bs{\beta}}\mc{M}$ by Corollary \ref{cor:log canonical} and hence $I^{\bs{\beta}} = \mb{C}[\bs{s}]$. It is therefore enough to show that the image $\bs{f^s}$ of $1 \in I^{\bs{\beta}}$ is non-zero in
\[ \frac{V^{\bs{\beta}}\cM}{V^{\bs{\beta}+\epsilon\bs{1}}\cM}\otimes_{\C[\bs{s}]}\C[s].\]
Now, by Lemma \ref{lem:hodge * formula}, we have $\bs{f^{s}}\in F_r^*V^{\bs{\beta}}\cM = F_rV^{\bs{\beta}}\mc{M}$. Since, $F_r$ is the lowest non-zero piece of the Hodge filtration, by the Koszul complex, the natural map
\begin{equation} \label{eq:beta variant 1}
F_r\left(\frac{V^{\bs{\beta}}\cM}{V^{\bs{\beta}+\epsilon\bs{1}}\cM}\right)\to F_r\left(\frac{V^{\bs{\beta}}\cM}{V^{\bs{\beta}+\epsilon\bs{1}}\cM}\overset{\mrm{L}}\otimes_{\C[\bs{s}]}\C[s]\right)
\end{equation}
is an isomorphism. So it is enough to show that the target of \eqref{eq:beta variant 1} is a strict complex or, equivalently, that its associated graded has cohomology in degree $0$ only. Note that we cannot apply Theorem \ref{thm:filtered change of functions} directly to conclude this as $\bs{\beta}$ is not necessarily a multiple of $\bs{\alpha}$. Instead, observe that
\begin{equation} \label{eq:beta variant 2}
\left(\frac{V^{\bs{\beta}}\mc{M}}{V^{\bs{\beta} + \epsilon \bs{1}}\mc{M}}, F_\bullet\right) \cong \left(\frac{V^{\bs{0}}_*\iota_+\mc{O}_X(*D)\bs{f}^{-\bs{\beta}}}{V_*^{\epsilon \bs{1}}\iota_+\mc{O}_X(*D)\bs{f^{-\bs{\beta}}}}, F_\bullet^*\right)
\end{equation}
as filtered $\ms{D}_X[\bs{s}]$-modules, up to translating the action of $\mb{C}[\bs{s}]$ by $\bs{\beta}$. Indeed, \eqref{eq:beta variant 2} follows without the filtrations from Corollary \ref{cor:localising V} and the definition of $V_*$, while the agreement of the filtrations follows from Lemma \ref{lem:hodge * formula} again. Since translations in $\bs{s}$ act trivially on $\gr^F\mb{C}[\bs{s}]$, we conclude by applying Theorem \ref{thm:filtered change of functions} to the right hand side of \eqref{eq:beta variant 2}.
\end{proof}

We now proceed to the proof of Lemma \ref{lem: n/d reduction}.

\begin{proof}[Proof of Lemma \ref{lem: n/d reduction}]

Choose a vector $\bs{\beta}\in \Q^r_{>0}$ adapted to $\{0\}$, which exists by Lemma \ref{lemma: existence of good vector}. The plan is to relate the map \eqref{eqn: map in n/d reduction} to the map \eqref{eqn: map associated to beta} in Lemma \ref{lemma: beta variant}, after localising along the maximal ideal $\mf{m} =(s_1+\alpha_1,\ldots,s_r+\alpha_r)\subset \mb{C}[\bs{s}]$.

Applying Example \ref{ex: set of walls via log resolution} to the log resolution $\pi$, we know that the set
\begin{equation}\label{eqn: set of walls for n/d conjecture} \mc{W}= \left\{ H_{W, k} \,\left|\,  H_{W, k} \colonequals\left\{\bs{\gamma} \in \mb{R}^r \, \left|\,\sum_{i \in I_W} \gamma_i = k\right.\right\}, \, W \in \cG_{\mrm{dense}},k \in \mb{Z}\right.\right\},\end{equation}
forms a set of walls for the filtration $\{V_{\ast}^{\bs{\alpha}}\mc{M}\}_{\bs{\alpha}\in \R^r}$. So it forms a set of walls for the filtration $\{V^{\bs{\alpha}}\mc{M}\}_{\bs{\alpha}\in \R^r_{>0}}$. By assumption, both vectors $\bs{\alpha},\bs{\beta}$ lie on the wall $H_{\{0\},n}$. Consider the finite set of walls
\[ \mc{W}_{\mf{m}}=\{ H\in \mc{W} \mid \bs{\alpha}\in H\}.\]
Let $\sigma_{1, \mf{m}}$ and $\sigma_{2, \mf{m}}$ be the chambers of $\mc{W}_{\mf{m}}$ containing $\bs{\beta}$ and $\bs{\beta} + \epsilon \bs{1}$ respectively, and let $\sigma_1 \subset \sigma_{1, \mf{m}}$ and $\sigma_2 \subset \sigma_{2, \mf{m}}$ be the unique chambers of $\mc{W}$ such that $\bs{\alpha}$ lies in the closures of $\sigma_1,\sigma_2$, respectively; see Figure \ref{fig:n/d walls} for an illustration. Under this setup, we claim that
\[ V^{\bs{\alpha} + \epsilon \bs{1}}\mc{M} \subset V^{\sigma_2}\mc{M}\subset V^{\sigma_1}\mc{M}\subset V^{\bs{\alpha}}\cM.\]
The first and third inclusions follow from the fact that $\bs{\alpha}\in \overline{\sigma}_2$ and $\bs{\alpha}\in \overline{\sigma}_1$; see Remark \ref{remark: containment of closure}(1). For the second inclusion, observe that the set of walls of $\mc{W}$ separating $\sigma_1$ and $\sigma_2$ is precisely the set of walls of $\mc{W}_\mf{m}$ containing $\bs{\beta}$. Since $\bs{\beta} + \epsilon \bs{1}$ and hence $\sigma_2$ lies above any such wall, the second inclusion follows from Remark \ref{remark: containment of closure}(2). 
Furthermore, Lemma \ref{lem:localised chambers} implies that \begin{equation}\label{eq:n/d localised V}
(V^{\bs{\beta}}\cM)_{\mf{m}}=(V^{\sigma_1}\cM)_{\mf{m}},\quad (V^{\bs{\beta}+\epsilon \bs{1}}\cM)_{\mf{m}}=(V^{\sigma_2}\cM)_{\mf{m}}.
\end{equation}

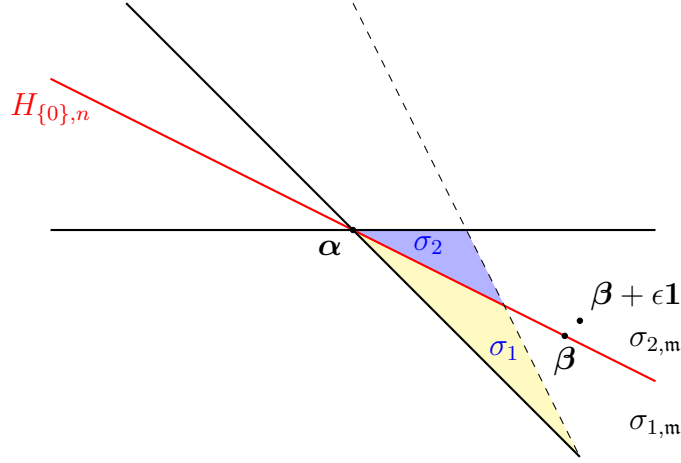
\begin{figure} 
\begin{tikzpicture}[scale=2]

  \fill[yellow!30] (0,0) -- (1, -0.5) -- (1.5, -1.5) -- cycle;
  \fill[blue!30] (0,0) -- (0.75, 0) -- (1, -0.5) -- cycle;
  \draw[red, thick] (-2, 1) -- (2, -1);         
  \draw[thick] (-1.5, 1.5) -- (1.5, -1.5);      
  \draw[thick] (-2, 0) -- (2, 0);          

  \draw[dashed] (0, 1.5) -- (1.5,-1.5);         

  \filldraw (1.4, -0.7) circle (0.5pt) node[below] {$\bs{\beta}$};
  \filldraw (1.5, -0.6) circle (0.5pt) node[above right] {$\bs{\beta}+\epsilon\bs{1}$};
  
  \filldraw (0,0) circle (0.5pt) node[below left] {$\bs{\alpha}$};

  \node at (0.5, -0.1) {$\textcolor{blue}{\sigma_{2}}$};
  \node at (2, -0.75) {$\sigma_{2,\mf{m}}$};
  \node at (2, -1.3) {$\sigma_{1,\mf{m}}$};
  \node at (-2, 0.8) {$\textcolor{red}{H_{\{0\},n}}$};
    \node at (1, -0.8) {$\textcolor{blue}{\sigma_{1}}$};
\end{tikzpicture}

\caption{The chambers $\sigma_i$ and $\sigma_{i, \mf{m}}$. In the picture, $\mc{W}_{\mf{m}}=\{\textrm{solid lines}\}$, $\mc{W}=\mc{W}_{\mf{m}}\cup \{\textrm{dashed lines}\}$, $\sigma_1$ is the yellow chamber and $\sigma_2$ is the purple chamber.}
\label{fig:n/d walls}
\end{figure}

As in \eqref{eqn: definition of Ialpha}, we define an ideal $I^{\sigma_1}\subset \mb{C}[\bs{s}]$ by $\mb{C}[\bs{s}]\bs{f^s}\cap V^{\sigma_1}\cM=I^{\sigma_1}\bs{f^s}$. Then $I^{\sigma_1} \subset I^{\bs{\alpha}}$ and we have a natural map
\begin{equation} \label{eq:n/d conjecture 1}
I^{\sigma_1} \to \frac{V^{\sigma_1}\mc{M}}{V^{\sigma_2}\mc{M}} \otimes_{\mb{C}[\bs{s}]} \mb{C}[s],
\end{equation}
 which fits into the following commutative diagram
\[\begin{tikzcd}
    I^{\sigma_1}\arrow[d,hook] \arrow[r,"\eqref{eq:n/d conjecture 1}"] &\frac{V^{\sigma_1}\mc{M}}{V^{\sigma_2}\mc{M}} \otimes_{\mb{C}[\bs{s}]} \mb{C}[s]\arrow[dr,"\eta"]\\
    I^{\bs{\alpha}} \arrow[r,"\eqref{eqn: map in n/d reduction}"] & \frac{V^{\bs{\alpha}
    }\mc{M}}{V^{\bs{\alpha}+\epsilon\bs{1}
    }\mc{M}} \otimes_{\mb{C}[\bs{s}]} \mb{C}[s] \arrow[r,"\delta",->>] & \frac{V^{\bs{\alpha}
    }\mc{M}}{V^{\sigma_2}\mc{M}} \otimes_{\mb{C}[\bs{s}]} \mb{C}[s]
\end{tikzcd}
\]
Since $\alpha_i>0$ for all $i$, for any separating wall $\bs{L}^{-1}(\gamma)$ of $\bs{\alpha}$ and $\sigma_1$ we must have $\sum_i L_i\alpha_i\neq 0$, so by Corollary \ref{cor: subtle injectivity} we have an injection
\[ V^{\sigma_1}\mc{M}\otimes_{\mb{C}[\bs{s}]}\mb{C}[s] \to  V^{\bs{\alpha}}\mc{M}\otimes_{\mb{C}[\bs{s}]}\mb{C}[s].\]
It follows that the map $\eta$ is injective. Therefore, to prove that the map \eqref{eqn: map in n/d reduction} is nonzero, it suffices to prove the same holds for the map \eqref{eq:n/d conjecture 1}.

To this end, we will prove that the localisation of \eqref{eq:n/d conjecture 1} at $\fm$ is nonzero, i.e.\ that
\[ (I^{\sigma_1})_{\mf{m}} \to \left(\frac{V^{\sigma_1}\mc{M}}{V^{\sigma_2}\mc{M}} \otimes_{\mb{C}[\bs{s}]} \mb{C}[s]\right)_{\mf{m}}\]
is nonzero. By \eqref{eq:n/d localised V}, the map above can be identified with the natural map
\begin{equation}\label{eqn: localised Ibeta to Vbeta}
(I^{\bs{\beta}})_{\mf{m}} \to \left(\frac{V^{\bs{\beta}}\mc{M}}{V^{\bs{\beta}+\epsilon\bs{1}}\mc{M}} \otimes_{\mb{C}[\bs{s}]} \mb{C}[s]\right)_{\mf{m}}.\end{equation}
Now note that since $\bs{\beta}$ is adapted to $\{0\}$, by Definition \ref{dfn: good} we have that $\bs{\beta}$ and $\bs{\beta}+\epsilon\bs{1}$ are separated by the unique wall $H_{\{0\},n}=\{\sum_{i=1}^r \gamma_i=n\}$.  Since $V^{\bullet}\cM$ is defined over $\mathbb{Q}$ by Corollary \ref{cor: MHM fully A specializable}, the operator $n+\sum_{i=1}^r s_i$ therefore acts nilpotently on $V^{\bs{\beta}}\cM/V^{\bs{\beta}+\epsilon\bs{1}}\cM$. Because $\sum_{i=1}^r \alpha_i=n$,  under the tensor product $s_i\mapsto \alpha_i s$, the operator $s+1=\frac{1}{n}\left(n+\sum_{i=1}^r(\alpha_is)\right)$ acts nilpotently on $(V^{\bs{\beta}}\cM/V^{\bs{\beta}+\epsilon\bs{1}}\cM)\otimes_{\C[\bs{s}]}\C[s]$ and the ideal $\mf{m}$ becomes $(s+1)$. Therefore
\[ \left(\frac{V^{\bs{\beta}}\mc{M}}{V^{\bs{\beta}+\epsilon\bs{1}}\mc{M}} \otimes_{\mb{C}[\bs{s}]} \mb{C}[s]\right)_{\mf{m}}=\left(\frac{V^{\bs{\beta}}\cM}{V^{\bs{\beta}+\epsilon\bs{1}}\cM}\otimes_{\C[\bs{s}]}\C[s]\right)_{(s+1)}=\frac{V^{\bs{\beta}}\cM}{V^{\bs{\beta}+\epsilon\bs{1}}\cM}\otimes_{\C[\bs{s}]}\C[s].\]
But by Lemma \ref{lemma: beta variant}, the image of the map
\[ I^{\bs{\beta}} \to \frac{V^{\bs{\beta}}\mc{M}}{V^{\bs{\beta} + \epsilon \bs{1}}\mc{M}} \otimes_{\mb{C}[\bs{s}]} \mb{C}[s]\]
is nonzero, and hence the same is true for \eqref{eqn: localised Ibeta to Vbeta}. We conclude that the map \eqref{eq:n/d conjecture 1} is also nonzero, which finishes the proof of Lemma \ref{lem: n/d reduction}.
\end{proof}

\subsection{The single-variable case}\label{sec: single SMC}
In this section, we give the proof of the Strong Monodromy Conjecture for hyperplane arrangements.

We first recall the definition of the topological zeta function. Assume $f\in \C[x_1,\ldots,x_n]$, and set $D=\textrm{div}(f)$. Let $\pi \colon \tilde{\C}^n\to \C^n$ be a log resolution of $(\C^n,D)$ and $I$ be the set of irreducible components of $\pi^{-1}D$, i.e. the set of exceptional divisors of $\pi$. For any $J\subset I$, denote by $E_J^{\circ}=\cap_{j\in J}E_j-\cup_{i\in I\setminus J}E_i$. The (global) \emph{topological zeta function} of $f$ is defined by
\begin{equation}\label{eqn: global top zeta} Z^{\textrm{global}}_{f,\mathrm{top}}(s)\colonequals\sum_{J\subset I} \chi\left(E_J^{\circ}\right)\cdot \prod_{E\in J} \frac{1}{\textrm{ord}_E(D)\cdot s+\textrm{ord}_E(K_{\tilde{\C}^n/\C^n})+1},\end{equation}
where $\chi(-)$ denotes the Euler characteristic and $\textrm{ord}_E(-)$ is the vanishing order along $E$. Assume $f(0)=0$, we also have the \emph{local topological zeta function} of $f$ at $0$:
\begin{equation}\label{eqn: top zeta} Z_{f,\mathrm{top}}(s)\colonequals\sum_{J\subset I} \chi\left(E_J^{\circ}\cap \pi^{-1}(0)\right)\cdot \prod_{E\in J} \frac{1}{\textrm{ord}_E(D)\cdot s+\textrm{ord}_E(K_{\tilde{\C}^n/\C^n})+1}.\end{equation}

If $f$ defines a hyperplane arrangement, we can always make a linear change of variables so that $f(0)=0$. In this case, we have a local version of Theorem \ref{thm:intro SMC}.
\begin{thm}\label{thm: SMC main text} If $f\in \C[x_1,\ldots,x_n]$ defines a hyperplane arrangement and $f(0)=0$, then any pole $s_0$ of $Z_{f,\mathrm{top}}(s)$ is a root of $b_f(s)$.
\end{thm}

As explained in \cite[Theorem 1.3 and \S 2.6]{BMT}, Theorem \ref{thm: SMC main text} and Theorem \ref{thm:intro SMC}\footnote{Furthermore, by \cite[Theorem 2.8]{BMT}, the assertion for the $p$-adic zeta function holds for \emph{every} prime.} follow from the $n/d$ conjecture \cite[Conjecture 1.2]{BMT}, which we prove below.\begin{thm}\label{thm: n/d text}
Let $f\in \mb{C}[x_1, \ldots, x_n]$ of degree $d$ such that $f^{-1}(0)_{\mathrm{red}}$ defines an essential indecomposable central hyperplane arrangement in $\mb{C}^n$. Then $b_f(-n/d) = 0$.
\end{thm}

\begin{proof}
Write $f=\prod_{i=1}^r f_i^{d_i}$, where $f_i$ are linear and homogeneous. The essential assumption implies $r\geq n$. If $r = n$, then $D=\textrm{div}(f)$ has normal crossings, contradicting the indecomposability of $D$. Hence $n/d\in (0,1)$.

Let $\iota_{f}\colon X \to X\times \C, \iota\colon X\to X\times \C^r$ be the graph embedding of $f$ and $f_1,\ldots,f_r$, respectively. Consider the $V$-filtration on $\mc{N} \colonequals \iota_{f,+}\mc{O}_{X}$ along the coordinate on $\C$ and the multivariate $V$-filtration on $\cM\colonequals \iota_{+}\mc{O}_X$ along the coordinates on $\C^r$. Using Corollary \ref{cor:localising V} and applying Corollary \ref{cor: change of functions} to $\{f_1,\ldots,f_r\}, f'=f = \prod_i f_i^{d_i}, \alpha'=n/d$ and $\bs{\alpha}\colonequals n/d\cdot (d_1,\ldots,d_r)\in \R^r$, we obtain
\[ V^{\bs{\alpha}}\mc{M} \otimes_{\mb{C}[\bs{s}]} \mb{C}[s]\xrightarrow{\sim} V^{n/d}\cN, \quad V^{\bs{\alpha} + \epsilon\bs{1}} \mc{M} \otimes_{\mb{C}[\bs{s}]} \mb{C}[s] \xrightarrow{\sim} V^{>n/d}\cN,\]
where the tensor product is taken via $s_i\mapsto d_is$. So by Corollary \ref{cor: subtle injectivity} we have an isomorphism
\[ \frac{V^{\bs{\alpha}}\mc{M}}{V^{\bs{\alpha}+\epsilon\bs{1}}\mc{M}}\otimes_{\mb{C}[\bs{s}]} \mb{C}[s]\cong \frac{V^{n/d}\cN}{V^{>n/d}\cN}.\]

Clearly the vector $\bs{\alpha}\in \R^r_{>0}$ satisfies the assumption of Lemma \ref{lem: n/d reduction}. So there exists $p(\bs{s}) \in \mb{C}[\bs{s}]$ such that 
\begin{equation}\label{eqn: qsfs not in Vn/d} p(d_1s,\ldots,d_rs)\prod_{i=1}^rf_i^{d_is}=q(s) f^s \in V^{n/d}\mc{N}\setminus V^{>n/d}\cN,\end{equation}
where $q(s)\colonequals p(d_1s, \ldots, d_r s)$. Note that $f^s\in V^{>0}\cN$ (by e.g.\ Corollary \ref{cor:log canonical}) and $n/d<1$, so $f^{s + 1} \in t\cdot V^{>0}\mc{N}\subset V^{>n/d}\mc{N}$. Choose a differential operator $P\in \sD_X[s]$ such that $P\cdot f^{s+1}= b_f(s)\cdot f^s$. Since $V^{>n/d}\cN$ is a $\sD_X[s]$-module, we have
\[ b_f(s)q(s) f^s = q(s)P f^{s + 1} \in V^{>n/d}\mc{N},\]
i.e.\ the class of $b_f(s)q(s) f^s$ in $\gr_V^{n/d}\mc{N}$ is zero. Suppose $b_f(-n/d)\neq 0$. Since $s+\beta$ acts invertibly on $\gr_V^{n/d}\mc{N}$ if $\beta\neq n/d$, we deduce from \eqref{eqn: qsfs not in Vn/d} that the class of $q(s)f^s$ in $\gr_V^{n/d}\mc{N}$ is zero. But this contradicts with $q(s) f^s\not\in V^{>n/d}\cN$. Hence $b_f(-n/d)=0$, as claimed. 
\end{proof}

\begin{remark}\label{remark: archimedean zeta function does not work}

A different strategy for proving Theorem \ref{thm: n/d text}, which has been proposed elsewhere (cf.\ \cite[Question 1.12]{Shi-Zuo24}) would be to show that $-n/d$ is a pole of the \emph{Archimedean} zeta function $Z_f$ (see, e.g.\ \cite{Shi-Zuo24} or \cite{DLY} for the definition); since $n/d < 1$, this would imply $b_f(-n/d)=0$ by Bernstein's functional equation for $Z_f$. This approach is used, for example, in  \cite{Shi-Zuo24} to prove a generic version of the $n/d$-conjecture. We show below, however, that $-n/d$ need not always be a pole of $Z_f$.

Consider $f=xy(x-y)z^2(x-z)^4$ from \cite[Example A.1]{BSY11}, found by Veys, where the authors proved that $-n/d=-1/3$ is not a pole of the local topological zeta function of $f$. We claim that $-\alpha \colonequals -1/3$ is also not a pole of the Archimedean zeta function $Z_f$. By \cite[Theorem 1.2]{DLY}, it is enough to show that $-\alpha$ is a simple root of $b_f(s)$ and that $f^{-\alpha} \in  \ms{D}_X \cdot f^{-\alpha + 1}$, where $X=\C^3$. For the first condition, computing $b_f(s)$ using Macaulay2 gives that $-1/3$ is indeed a simple root.  For the second condition, we apply the weight filtration algorithm \cite[Algorithm 5.9]{LY26}: let $(\cO_X)_f$ be the localization of $\cO_X$ along $f$, then one can check that $f^{-\alpha}\in W_{\dim X}\left((\cO_X)_f\cdot f^{-\alpha})\right)$. Denote by $j:U=\{f\neq 0\}\hookrightarrow X$ the open embedding, then $W_{\dim X}\left((\cO_X)_f\cdot f^{-\alpha}\right)=j_{!\ast}(\cO_U\cdot f^{-\alpha})$. Since $\alpha\in (0,1)$, using \cite[Theorem 2.4]{DLY} and \cite[Proposition 5.10]{DLY}, one can show that $f^{-\alpha+1}\in  j_{!\ast}(\cO_U\cdot f^{-\alpha})$. Hence by the simplicity of  $j_{!\ast}(\cO_U\cdot f^{-\alpha})$ as $\sD_X$-module, one has  $j_{!\ast}(\cO_U\cdot f^{-\alpha})=\sD_X\cdot f^{-\alpha+1}$. This establishes the second condition.
\end{remark}

Now we explain how to adapt the proof of Theorem \ref{thm: n/d text} to the setting of Theorem \ref{thm: roots n+1/d}.
\begin{proof}[Proof of Theorem  \ref{thm: roots n+1/d}]
Let $\bs\beta\in\Q_{>0}^r$ be adapted to $\{0\}$ in the sense of Definition \ref{dfn: good}. The proof of Lemma \ref{lemma: beta variant} can be adapted to prove the following.
\begin{lemma}\label{lem:affine}
Let $\bs u=(u_1,\ldots,u_r)\in\Z_{\ge0}^r$, then the class of $\bs{f^s}$ has nonzero image in $V^{\bs\beta}\cM/V^{\bs\beta+\epsilon\bs{1}}\cM\otimes_{\C[\bs{s}]}\C[s]$, where $s_i\longmapsto d_i s+u_i$.
\end{lemma}
Fix an integer $n\le k<d$ and choose $\bs u\in\Z_{\ge0}^r$ such that $|\bs u|:=\sum_i u_i=k-n$. Set
\[ \bs\gamma:=\bs\beta+\bs u, \qquad g_{\bs u}:=\prod_{i=1}^r f_i^{u_i}, \qquad \bs\alpha:=\frac{k}{d}(d_1,\ldots,d_r). \]
Then $\sum_i\gamma_i=\sum_i\alpha_i=k$, and for every dense edge $W\ne\{0\}$, $\sum_{D_i\supseteq W}\gamma_i\notin\Z$. Thus, $\bs\gamma$ and $\bs\gamma+\epsilon\bs{1}$ are separated by the unique wall $H_{\{0\},k}=\left\{\bs\eta\in\R^r\,\middle|\,\sum_i\eta_i=k\right\}$.

\begin{lemma}\label{lem:higher-wall}
There exists $p(\bs s)\in \C[\bs{s}]$ such that $p(\bs s)g_{\bs u}\bs{f^s}\in V^{\bs\alpha}\cM$ and its class has nonzero image in $V^{\bs\alpha}\cM/V^{\bs\alpha+\epsilon\bs{1}}\cM\otimes_{\C[\bs{s}]}\C[s]$, where $s_i\longmapsto d_i s$.
\end{lemma}

\begin{proof}
By Proposition \ref{prop:weak specialisation}, iterated for the coordinates with $u_i>0$, multiplication by $t^{\bs u}$ induces an isomorphism
\begin{equation*}\label{eq:t-u-quotient} t^{\bs u}:\frac{V^{\bs\beta}\cM}{V^{\bs\beta+\epsilon\bs{1}}\cM}\xrightarrow{\ \sim\ }\frac{V^{\bs\gamma}\cM}{V^{\bs\gamma+\epsilon\bs{1}}\cM}. \end{equation*}
Under the Malgrange--Mellin transform, formula \eqref{eqn: ti action} gives $t^{\bs u}\bs{f^s}=g_{\bs u}\bs{f^s}$ and, for any $m\in \cM$, $t^{\bs u}\bigl(a(\bs s)m\bigr)=a(\bs s+\bs u)t^{\bs u}m$. Thus $t^{\bs{u}}$ induces an isomorphism
\begin{equation*} \frac{V^{\bs\beta}\cM}{V^{\bs\beta+\epsilon\bs{1}}\cM}\otimes_{s_i\mapsto d_i s+u_i}\C[s]\xrightarrow{\ \sim\ }\frac{V^{\bs\gamma}\cM}{V^{\bs\gamma+\epsilon\bs{1}}\cM}\otimes_{s_i\mapsto d_i s}\C[s]. \end{equation*}
By Lemma \ref{lem:affine}, the image of $\bs{f^s}$ on the left hand side is nonzero. Therefore $[g_{\bs u}\bs{f^s}]\ne0$ on the right hand side. Then the wall-crossing argument in the proof of Lemma \ref{lem: n/d reduction} proves the desired statement.
\end{proof}
From now on, we use the same notation as in the proof of Theorem \ref{thm: n/d text}. Fix $n\le k<d$ and choose $\bs u$ as above. Corollary \ref{cor: subtle injectivity} and Lemma \ref{lem:higher-wall} give a polynomial $q(s)\in \C[s]$ such that
\begin{equation*} q(s)g_{\bs u}f^s\in V^{k/d}\cN\setminus V^{>k/d}\cN. \end{equation*}
Since $f^s\in V^{>0}\cN$ and $k/d<1$, we have $f^{s+1}\in V^{>k/d}\cN$. Choose $P(s)\in \sD_X[s]$ such that $P(s)f^{s+1}=b_f(s)f^s$. Then
\[ b_f(s)q(s)g_{\bs u}f^s=q(s)g_{\bs u}P(s)f^{s+1}\in V^{>k/d}\cN. \]
The same argument as in Theorem \ref{thm: n/d text} then shows that $b_f(-k/d)=0$.
\end{proof}

\subsection{The multivariable case} \label{subsec:multi SMC}
We prove Theorem \ref{thm: multivariate n/d intro}, i.e.\ Budur's multivariate $n/d$ conjecture \cite[Conjecture 1.17]{Budur15} and hence the multivariable Strong Monodromy Conjecture for arbitrary factorisations of hyperplane arrangements \cite[Conjecture 1.13]{Budur15}.

We first recall the topological zeta function in the multivariate setting. Let $F=(f_1,\ldots,f_k)$ with $0\neq f_j\in\C[x_1,\ldots,x_n]$ and set $F_j=\textrm{div}(f_j)$. Let $\pi:\tilde{\C}^n\to \C^n$ be a log resolution of $(\C^n,\sum_{j=1}^k F_j)$, and let $I$ be the set of exceptional divisors. The (global) \emph{topological zeta function} of $F$ is defined by
\begin{equation}\label{eqn: multi topological zeta function} Z^{\mrm{global}}_{F,\mathrm{top}}(\bs{s})=\sum_{J\subset I} \chi(E_J^{\circ})\cdot \prod_{E\in J}\frac{1}{\sum_{j=1}^k\textrm{ord}_E(F_j)s_j+\textrm{ord}_E(K_{\tilde{\C}^n/\C^n})+1},\end{equation}
where $\chi(-)$ denotes the Euler characteristic. We establish the following multivariable Strong Monodromy Conjecture for hyperplane arrangements. 
\begin{thm}\label{thm: multi SMC}
Let $F=(h_1,\ldots,h_k)$ where $h_j$ defines a hyperplane arrangement in $\C^n$. Then the polar locus of $Z^{\mrm{global}}_{F,\mathrm{top}}(\bs{s})$ is contained in the zero locus of $B_F$.
\end{thm}
By \cite[Theorem 1.18]{Budur15}, Theorem \ref{thm: multi SMC} follows from the following multivariate $n/d$ conjecture, which we prove below.
\begin{thm}\label{thm: multi n/d}
Let $F=(h_1,\ldots,h_k)$ where $h_j$ is a central hyperplane arrangement of degree $d_i$ in $\C^n$ and $(h_1\cdots h_k)_{\textrm{red}}$ is central, essential and indecomposable, then  $\{d_1s_1+\cdots+d_ks_k+n=0\}\subset \textrm{Zero}(B_F)$.
\end{thm}

As in the proof of Theorem \ref{thm: n/d text}, Theorem \ref{thm: multi n/d} is proved by analysing the multivariate $V$-filtration associated to the tuple $(h_1,\ldots,h_{k})$. Set $\prod_{j=1}^k h_j=\prod_{i=1}^r f_i^{d_i}$, where $f_i$ are linear. The basic observation is that this $V$-filtration arises as the specialization of the $V$-filtration of $(f_1,\ldots,f_r)$, as described in Corollary \ref{cor: change of functions}. Furthermore, in Lemma \ref{lem: n/d reduction}, we already have certain control over the specialisation process to the $V$-filtration of the full product $\prod_{j=1}^k h_j$. Therefore, by combining these two facts, we can also gain control over the intermediate object, namely,  the $V$-filtration associated to the tuple $(h_1,\ldots,h_k)$.

Since the proof of Theorem \ref{thm: multi n/d} will rely on Corollary \ref{cor: subtle injectivity} and Lemma \ref{lem: n/d reduction}, to be compatible with the notations there, we set $k=r'$, $\bs{s}'=(s_1',\ldots,s_{r'}')$.  Set $h_j = \prod_{i=1}^r f_i^{d_{ij}}$, which is of degree $d_j'\colonequals\sum_{i=1}^{r} d_{ij}$ and write $\bs{h^{s'}}$ as $h_1^{s'_1}\cdots h_{r'}^{s_{r'}'}$. Let $D=\textrm{div}(f_1\ldots f_r)$ be the induced central hyperplane arrangement and let $\pi\colon \tilde{X}\to X=\C^n$ be the log resolution of $(X,D)$ associated to the set of dense edges of $D$, as in \S \ref{sec: single SMC}. 

\begin{proof}[Proof of Theorem \ref{thm: multi n/d}]
Take any $\bs{\alpha}'\in \R^{r'}$ with $0<\alpha_j'<1$ for each $j$ and $\sum_{j=1}^{r'} d'_i\alpha'_i=n$. Since such vectors are Zariski dense in the hyperplane $\{\bs{\alpha}'\in \C^{r'}\mid \sum_{j=1}^{r'}d'_j\alpha_j'=n\}$, it suffices to show that $b(-\bs{\alpha}')=0$ for any $b(\bs{s}')\in B_F$.
 
Consider the graph embedding $\iota\colon X\to X\times \C^{r}$ associated to $(f_1,\ldots,f_r)$ and the multivariate $V$-filtration $V^{\bullet}\cM$ on $\cM=\iota_{+}\cO_X$ along the coordinates of $\C^r$. Similarly, consider the $V$-filtration $V^{\bullet}\cM'$ on the graph embedding $\cM'=\iota'_{+}\cO_X$ along the coordinates, where $\iota'$ is the graph embedding of $(h_1,\ldots,h_{r'})$. Define $\bs{\alpha}\in \R^r_{>0}$ by 
\[\alpha_i\colonequals \sum_{j=1}^{r'}d_{ij}\alpha_j'.\]
It follows that $ \sum_{i=1}^r \alpha_i=\sum_{j=1}^{r'}\alpha_j' \sum_{i=1}^rd_{ij}=\sum_j d_j'\alpha_j'=n$. So by Lemma \ref{lem: n/d reduction} (see Remark \ref{remark: localized version of n/d reduction}), for $0<\epsilon\ll 1$,  there exists $0\neq p(\bs{s})\in \C[\bs{s}]$ such that $p(\bs{s})\bs{f^s}\in V^{\bs{\alpha}}\cM$ and its class in
\[ \left(\frac{V^{\bs{\alpha}}\cM}{V^{\bs{\alpha}+\epsilon\bs{1}}\cM}\otimes_{\C[\bs{s}]} \C[s_0]\right)_{\mf{m}} =  \left(\frac{V^{\bs{\alpha}}\cM}{V^{\bs{\alpha}+\epsilon\bs{1}}\cM}\otimes_{\C[\bs{s}]} \C[s_0]\right)_{(s_0 + 1)}\]
is nonzero, where the tensor product is taken over $s_i\mapsto \alpha_is_0$ and $\mf{m}=(s_1+\alpha_1,\ldots,s_r+\alpha_r)$. Notice that there is a surjection
\[\frac{V^{\bs{\alpha}}\cM}{V^{\bs{\alpha}+\epsilon\bs{1}}\cM}\otimes_{\C[\bs{s}]} \C[\bs{s}']\twoheadrightarrow \frac{V^{\bs{\alpha}}\cM}{V^{\bs{\alpha}+\epsilon\bs{1}}\cM}\otimes_{\C[\bs{s}]} \C[\bs{s}']\otimes_{\C[\bs{s}']}\C[s_0]= \frac{V^{\bs{\alpha}}\cM}{V^{\bs{\alpha}+\epsilon\bs{1}}\cM}\otimes_{\C[\bs{s}]} \C[s_0],\]
where the extra tensor products are taken over $s_i\mapsto \sum_{j=1}^r d_{ij}s_j'$ and $s_j'\mapsto \alpha_j's_0$. Let $\mf{m}'=(s_1'+\alpha_1',\ldots,s_{r'}'+\alpha'_{r'})$. Since $\mf{m}' \subset \mb{C}[\bs{s}']$ is the pre-image of $(s_0 + 1) \subset \mb{C}[s_0]$, this induces a surjection
\[ \left(\frac{V^{\bs{\alpha}'}\cM'}{V^{\bs{\alpha}'+\epsilon\bs{1}}\cM'}\right)_{\mf{m}'}=\left(\frac{V^{\bs{\alpha}}\cM}{V^{\bs{\alpha}+\epsilon\bs{1}}\cM}\otimes_{\C[\bs{s}]} \C[\bs{s}']\right)_{\mf{m}'}\twoheadrightarrow \left(\frac{V^{\bs{\alpha}}\cM}{V^{\bs{\alpha}+\epsilon\bs{1}}\cM}\otimes_{\C[\bs{s}]} \C[s_0]\right)_{(s_0 + 1)},\]
where the first isomorphism follows from \eqref{eqn: Valpha' is the tensor of Valpha with Cs'}. We conclude that the class of $p(\bs{s})\bs{f^{s}}$ in
\[ \left(\frac{V^{\bs{\alpha}}\cM}{V^{\bs{\alpha}+\epsilon\bs{1}}\cM}\otimes_{\C[\bs{s}]} \C[\bs{s}']\right)_{\mf{m}'}\]
is nonzero.  Set $s_i=\sum_{j=1}^{r'} d_{ij}s_j'$, it follows that there is a polynomial $q(\bs{s}')\colonequals p(\bs{s})$ such that 
\begin{equation}\label{eqn: property of q} [q(\bs{s'})\bs{h^{s'}}]_{\mf{m}'}\in \left(V^{\bs{\alpha}'}\cM'\right)_{\mf{m}'}\setminus \left(V^{\bs{\alpha}'+\epsilon\bs{1}}\cM'\right)_{\mf{m}'}. \end{equation}

Since the support of $V^{\bs{\alpha'}}\cM'/V^{\bs{\alpha'}+\epsilon\bs{1}}\cM'$ is contained in $\bigcup_{j} h_j^{-1}(0)$, one can find $m\in \Z_{>0}$ such that 
\[q(\bs{s}')\cdot \prod_{j} h_j^{s_j'+m}\in V^{\bs{\alpha'}+\epsilon\bs{1}}\cM'.\]
On the other hand, by \cite[Theorem (2)]{Gyoja93}, there exists an element $b_{\textrm{Gyoja}}(\bs{s}')\in B_F$ such that $b_{\textrm{Gyoja}}(\bs{s}')=\prod_{t=1}^N (a_{t1}s_1'+\cdots +a_{tr'}s_{r'}'+b_t)$, where $a_{tj}\in \N$ and $b_t\in \Q_{>0}$ for any $t,j$. Applying the defining differential equations for $b_{\textrm{Gyoja}}(\bs{s}')$ and $b(\bs{s}')$ repeatedly, one has
\[ b(\bs{s}')\cdot b_{\textrm{Gyoja}}(\bs{s}'+\bs{1})\cdots b_{\textrm{Gyoja}}(\bs{s}'+(m-1)\bs{1})\cdot [q(\bs{s}')\bs{h^{s'}}]_{\mf{m'}} \in (V^{\bs{\alpha'}+\epsilon\bs{1}}\cM')_{\mf{m'}},\]
where $b(\bs{s}'+\bs{1})\colonequals b(s_1'+1,\ldots,s_{r'}'+1)$.  Since $0<\alpha'_j<1$ for all $1\leq j\leq r'$, our choice of $b_{\textrm{Gyoja}}(\bs{s}')$ implies that $b_{\textrm{Gyoja}}(-\bs{\alpha'}+\ell\bs{1})\neq 0$ for all $0<\ell\leq m-1$. So $b_{\textrm{Gyoja}}(\bs{s}'+\ell \bs{1})\not\in \mf{m'}$ and thus acts invertibly on $(V^{\bs{\alpha'}}\cM'/V^{\bs{\alpha'}+\epsilon\bs{1}}\cM')_{\mf{m'}}$ for any such $\ell$. 
So from \eqref{eqn: property of q}, we must have
\[ b(\bs{s}')\in \mf{m'}.\]
In other words, $b(-\bs{\alpha}')=0$, as claimed. This finishes the proof.
\end{proof}

\bibliographystyle{alpha}
\bibliography{Vfiltration}{}

\vspace{\baselineskip}

\footnotesize{
\textsc{School of Mathematics and Statistics, University of Melbourne, Parkville, VIC, 3010, Australia} \\
\indent \textit{E-mail address:} \href{mailto:dougal.davis1@unimelb.edu.au}{dougal.davis1@unimelb.edu.au}

\vspace{\baselineskip}

\textsc{Department of Mathematics, University of Kansas, 1450 Jayhawk Blvd, Lawrence, KS 66045, United States} \\
\indent \textit{E-mail address:} \href{mailto:ruijie.yang@ku.edu}{ruijie.yang@ku.edu} 
}
\end{document}